\documentclass[aos]{imsart}

\RequirePackage{amsthm,amsmath,amsfonts,amssymb}
\RequirePackage[numbers]{natbib}
\RequirePackage[colorlinks,citecolor=blue,urlcolor=blue]{hyperref}

\usepackage{dsfont}

\startlocaldefs
\numberwithin{equation}{section}
\theoremstyle{plain}
\newtheorem{conjecture}{Conjecture}
\newtheorem{theorem}{Theorem}[section]
\newtheorem{corollary}{Corollary}[section]
\newtheorem{proposition}{Proposition}[section]
\newtheorem{lemma}{Lemma}[section]
\theoremstyle{definition}
\newtheorem*{assumption}{Assumption}
\newtheorem{remark}{Remark}[section]

\def\R{\mathbb{R}}
\def\K{[K]}
\def\A{\mathcal{A}}
\def\J{\mathcal{J}}
\def\C{\mathcal{C}}

\def\P{\mathbb{P}}
\def\E{\mathbb{E}}
\def\L{L}
\def\onee{\mathds{1}}
\def\boh{\mathbf{h}}
\def\bohk{\mathbf{h}_k}
\def\bog{\mathbf{g}}

\def\RLK{\mathbb{R}^K\times L_2^{K^2}}
\def\RLKO{\mathbb{R}^K\times L_{2,\boh^0}}
\def\RLKOp{\mathbb{R}^K\times L_{2,\boh^0}^+}
\def\bth{\boldsymbol{\theta}}
\def\lti{\lfloor \log(T)\rfloor}
\def\VERT{\vert\hspace{-0.04cm} \Vert}
\DeclareMathOperator{\supp}{supp}
\DeclareMathOperator{\Span}{Span}

\endlocaldefs

\begin{document}

\begin{frontmatter}
\title{The Bernstein-von Mises theorem and efficiency for semiparametric inference in multivariate Hawkes processes}
\runtitle{Semiparametric BvM for Hawkes processes}

\begin{aug}
\author[A]{\fnms{Maël}~\snm{Duverger}\ead[label=e1]{duverger@ceremade.dauphine.fr}}
\and
\author[A]{\fnms{Judith}~\snm{Rousseau}\ead[label=e2]{rousseau@ceremade.dauphine.fr}} 

\address[A]{CEREMADE,  University Paris Dauphine - PSL\printead[presep={,\ }]{e1,e2}}

\end{aug}

\begin{abstract}
In this paper, we study semiparametric inference for linear multivariate Hawkes processes, a class of point processes widely used to describe self and mutually exciting phenomena. We establish a convolution theorem giving the best limiting distribution for a regular estimator of smooth functional. Then, in the Bayesian setting, we prove a semiparametric Bernstein-von Mises (BvM) theorem for nonparametric random series priors. We apply this result to histogram and wavelet based priors. Taken together, the convolution and BvM theorems show that, from a frequentist point of view, semiparametric Bayesian procedures have asymptotically the optimal behavior. Deriving the BvM property for random series priors led us to prove $\L_{2}$ posterior contraction, complementing for these priors the results of Donnet, Rivoirard and Rousseau (2020).
\end{abstract}



\end{frontmatter}
\tableofcontents


\section{Introduction}
\subsection{Hawkes processes}

Multivariate Hawkes processes are point processes introduced by Alan G. Hawkes in 1971 \cite{hawkes:71}. Let $N$ be a marked point process on $\R$ with marks in $\K:=\{1,...,K\}$. For a real Borel set $A$, $N(A)$ is the number of occurrences in $A$, $N^k(A)$ is the number of occurrences in $A$ with mark $k$ and for $t\geq 0$, $N_t^k:=N^k([0, t])$. Let $\mathcal{G} = (\mathcal{G}_t)_{t\in \R}$ be the internal history, the stochastic (predictable) intensity of the process is given by 
\begin{equation}
    \lambda_t^k   =  \underset{\Delta\searrow 0}{\lim}\frac{\P\big(N^k(]t,t+\Delta]) = 1 \vert  \mathcal{G}_{t-}\big) }{\Delta}, \hspace{0.3cm}k\in \K. \nonumber
\end{equation}
For a linear multivariate  Hawkes process $N$ with marks in $\K$, we have
\begin{equation}\label{intensity_linear}
    \lambda_t^k(\nu_k, \bohk) = \nu_k + \sum_{l=1}^K \int_{-\infty}^{t-} h_{l,k}(t-u)dN_u^l\hspace{0.1cm}, \hspace{0.3cm}k\in \K,
\end{equation}
$\nu_k >0$ is the background rate and the functions $\bohk = (h_{l,k})_{l\in\K}$,  called interaction functions, are non negative and supported on $[0,+\infty[$.  The parameters of a Hawkes process are thus $f=(\nu,\boh)$  with $\nu=(\nu_1,...,\nu_K)\in ]0,+\infty[^K$ and $\boh$ is the $K\times K $ matrix whose $k$-th column is $\boh_k= (h_{1,k}, ..., h_{K,k})$.

Hawkes processes are used in a wide range of fields such as seismology \cite{Ogata_88, verejones_ozaki}, genomics \cite{gusto_al, reynaud_bouret_schbath}, neuroscience \cite{bonnet_dion_lemler,reynaud_genom}, finance \cite{bacry_delattre_hoffmann,bacry_muzy_mastro, embrechts_al} and cyber security \cite{hillairet_bessy,hillairet_reveillac_rosenbaum}.

There is a vast literature on statistical inference in Hawkes processes. In parametric models, starting from the seminal works \cite{Ogata_78} and \cite{Kutoyants} (Chapter 4), several extensions and generalisations have been proposed,  see for instance \cite{bonnet_martinez_sangnier_1,bonnet_martinez_sangnier_2} and the references therein. In particular, nonparametric estimation was first studied in \cite{reynaud_bouret_schbath} and then a Lasso-type estimator was proposed in \cite{hansen_al}. Shortly after,  an \textit{edge screening} approach, with a low computational cost, to recover the graph of interaction was proposed in \cite{chen_al}. In the Bayesian setting,  nonparametric posterior contraction was established in \cite{donnet_concentration,Sulem_concentration}, extensions to settings where $K$ is high dimensional have been recently proposed in \cite{rousseau_highdim} while variational Bayes methods for Hawkes processes have been studied in \cite{Sulem_variational}, see also references therein such as \cite{zhou2020}.

However, none of these works address the important question of semiparametric inference. Certain functionals are of special interest in the analysis of a Hawkes process such as the background rate $\nu_k$ or the interaction matrix $\rho$ defined by  $\rho_{l,k} = \int h_{l,k}(x)dx$. Indeed, in the univariate case, the matrix $\rho$ reduces to $\int h$ which corresponds to the average number of offspring "produced" by a point when the Hawkes process is viewed as a Poisson cluster process, see \cite{hawkes_oakes}. Moreover, there is a causal interpretation of $\rho$: one can say that the process $N^l$ does not "Granger-cause" $N^k$ if  $\rho_{l,k} =0$, see \cite{eichler_al_causal}. It is notably worthwhile to test the nullity of some coefficients $\rho_{l,k}$, as done for parametric models  \cite{bonnet_dion_sadeler, lotz} and for a nonparametric model in \cite{kutoyants_dachian} (in the univariate case). However the nonparametric result in \cite{kutoyants_dachian} requires to know the background rate $\nu$ and is specific to the univariate case, where under the null hypothesis the process is a Poisson process. For a multivariate Hawkes process, the process is Poisson only if $\rho_{l,k}=0$ for all $(l,k)\in \K^2$, having only some $\rho_{l,k}=0$ is not sufficient to find back a Poisson process and proposing such a nonparametric testing procedure in the multivariate case is both of practical and theoretical interest. In \cite{Sulem_concentration} the authors estimate consistently $\delta_{l,k} = \onee_{\rho_{l,k}\neq 0}$ using a spike and slab prior on $\rho_{l,k}$ but under the unsatisfyingly  strong assumption on the  slab prior: $\Pi( \rho_{l,k} \lesssim \epsilon_T |\delta_{l,k}=1 ) = o(e^{-\kappa T \epsilon_T^2})$, where  $\epsilon_T$ is the nonparametric posterior concentration rate of $f$. In \cite{Sulem_variational} the authors use a thresholding procedure to estimate $\delta$, however the threshold is based on the nonparametric posterior concentration rate $\epsilon_T$ and needs to be pre-defined. Deriving a Bernstein-von Mises property for $\rho_{l,k}$ allows to define a thresholding procedure, fully data driven and sharp, as explained after Corollary \ref{main_coro}.

Beyond testing nullity, there are other general key questions. Which functionals can be estimated at a fast rate? What is the efficiency theory in these models? What is the limiting behavior of the marginal posterior distribution of a smooth functional of $f$? In this paper, we aim to fill these gaps, which are important both from a frequentist and Bayesian perspective.

\subsection{Our contribution}
We consider the same asymptotic framework as in \cite{donnet_concentration,Sulem_concentration} where a stationary Hawkes process $N$ with \textit{true} parameters $f^0$ is observed over a time window of order $T$ and $T\rightarrow  +\infty$. We are interested in the semiparametric estimation of \textit{smooth} functionals $\psi(f^0)\in \R$ of the infinite dimensional parameter $f^0$. We first derive a convolution theorem which states that for a regular estimator $\hat\psi_T$, the asymptotic distribution of $\sqrt{T}(\hat\psi_T -\psi(f^0))$ can be written as the convolution between some probability distribution and a Gaussian distribution with optimal variance $V_0$, denoted $\mathcal{N}(0,V_0)$. An estimator achieving the limit distribution $\mathcal{N}(0,V_0)$ is said to be efficient. This derivation is done following the standard semiparametric theory as presented in Chapter 25 of \cite{VDV_AS}: a nonparametric LAN expansion is proved, giving rise to a LAN inner product that is used to define $V_0$. Parametric LAN expansion for point processes has been studied in \cite{Kutoyants}  (see Chapter 4 and the references within) where a central limit theorem tailored to this purpose is proved.

Secondly, based on the general approach of \cite{castillo_rousseau} and using the recent results of \cite{donnet_concentration,Sulem_concentration}, we study the semiparametric Bernstein-von Mises (BvM) property for nonparametric priors on $f$. This property says that, given a smooth functional $\psi$ and an efficient estimator $\hat\psi_T$, the marginal posterior distribution of $\sqrt{T}(\psi(f)-\hat\psi_T)$ converges weakly, in probability, to a $\mathcal{N}(0,V_0)$ as  $T\rightarrow +\infty$. In this way, by the convolution theorem, semiparametric Bayesian inference is efficient from a frequentist point of view.
In the context of priors on the
$h_{l,k}$ based on expansions on a basis, we obtain in Theorem \ref{main_theorem} an approximation of the posterior
distribution of functional $\psi(f)$ as a mixture of Gaussian distributions. To understand when
this mixture of Gaussian distributions corresponds to
a single Gaussian distribution and thus that the BvM property holds, we need to study more precisely the least favourable direction and show that it can be well approximated under the prior distribution. The difficulty here is that the least favourable direction is non explicit and depends on the Palm distribution of the process, see Lemma \ref{explicit_gamma}. Nevertheless we show, under Hölder smoothness assumptions on the functions $h^0_{l,k}$, that the least favourable direction is Hölder smooth (see Corollary \ref{g0_palm_holder}). This allows us to derive explicit sufficient conditions for the BvM to hold. Moreover, in order to derive the asymptotic distribution of the marginal posterior, we also need to obtain posterior contraction rates in $L_2$, which is of independent interest.

The semiparametric BvM property allows efficient and flexible estimation of smooth functionals of the parameters. In particular, $ 1 - \alpha$ Bayesian credible intervals for $\psi(f)$ are then  asymptotically $1-\alpha$ frequentist confidence intervals. It also allows to study  the consistency of Bayes tests for $\rho_{l,k}=0$. Moreover, since the least favourable direction is not explicit, constructing concrete confidence intervals with a frequentist procedure is not straightforward, whereas Bayesian procedures provide an automatic approach using the quantiles of the marginal posterior distribution.

Our work is part of a broader line of research aimed at extending the BvM theorem to more general settings. Over the last two decades, the BvM theorem for parametric models (see Section 10 of \cite{VDV_AS} for a statement with minimal assumptions) has been generalized in various ways. In particular, \cite{rousseau_rivoirard}, \cite{castillo_ptrf} and  \cite{castillo_rousseau} (upon which we build) investigated the BvM property for semiparametric models while nonparametric version of the BvM property was proposed in \cite{castillo_nickl}. We also point out that the BvM phenomena for parametric models that are non regular and exhibit boundary effects was analyzed in \cite{bvm_truncated}. More recently, the semiparametric and the nonparametric BvM properties have been investigated in more complex models than those initially considered in the seminal papers \cite{ castillo_ptrf,castillo_nickl,castillo_rousseau,rousseau_rivoirard}. The nonparametric version has been notably studied for  nonlinear inverse problems  \cite{nickl_IP1,nickl_IP2} and for multidimensional diffusions \cite{nickl_ray}. Semiparametric BvM has been studied, among others,  in Hidden Markov Models (HMM) \cite{moss_rousseau}, for multidimensional diffusions \cite{giordano_ray}, in a missing data model (with a causal inference equivalent formulation) \cite{ray_vdv} and for some linear inverse problems \cite{magra_vdv_vz}.  In our case, as mentioned earlier, one of the main difficulties for the Hawkes model lies in the study of the  regularity of the least favourable direction. The latter is defined as the image of the usual $\L_2$ Riesz representor of the functional by the inverse of the so called \textit{information operator} which is at first glance non explicit in this model. In Section \ref{subsec_explicit_gamma}, we obtain an explicit formula for this operator using Palm calculus and then we prove in Section \ref{sec_lfeq} that under a Hölder regularity assumption on $h_{l,k}^0$, the least favourable direction is (almost) as regular as the $\L_2$ Riesz representor. Similar issues concerning the information operator and the least favourable direction have also arisen recently in some PDEs and inverse problems models, see \cite{giordano_ray, monard_nickl_paternain, nickl_IP2}.

The paper is organized as follows. In Section \ref{sec_semi_param_eff} we present the results on the semiparametric efficiency (the LAN expansion, the convolution theorem and the derivation of an explicit formula for the information operator). Section \ref{BVM} is dedicated to Bayesian inference, we state a theorem for the $L_2$ posterior contraction, then we present our main results on the BvM property and we illustrate them on histogram and wavelet based priors. The study of the regularity of the least favourable direction is presented in Section \ref{sec_lfeq}. In the Appendix, we briefly introduce Palm calculus giving the results we use in this paper. We conclude the paper with a discussion of possible extensions of our results in Section \ref{sec:disc}. Sections \ref{proof_section_semiparam} to \ref{appendix_omega} contain all the proofs and technical lemmas.

\subsection{Setup and notations}\label{not_assump}

The parameters of a Hawkes process are $f=(\nu,\boh)$ with $\nu=~(\nu_1,...,\nu_K)\in ]0,+\infty[^K$ and $\boh$ is the $K\times K $ matrix of functions whose $k$-th column is $\boh_k= (h_{1,k}, ..., h_{K,k})$. The intensity of the $k$-th process depends only on $f_k=(\nu_k, \boh_k)$. To avoid confusion, we will always denote in bold matrices or vectors of functions. Throughout this paper we assume that the interactions functions $h_{l,k}$ are supported on a bounded interval $[0,A]$ with $A$ known. It is a common assumption (see \cite{donnet_concentration} and \cite{hansen_al}) and it implies that the Hawkes process has a finite memory, some renewal properties (see \cite{Sulem_concentration} and \cite{Costa_renewal}) and its intensity is equal to 
\begin{equation}\label{intensity_linear_A}
    \lambda_t^k(f_k) = \nu_k + \sum_{l=1}^K \int_{t-A}^{t-} h_{l,k}(t-u)dN_u^l, \hspace{0.3cm}k\in \K.
\end{equation}
By Theorem 7 of \cite{bremaud_massoulie_96}, if the $K\times K$ matrix $\rho$ with entry $(l,k)$ equal to $\rho_{l,k}=\int_0^A h_{l,k}(x)dx$ has a spectral radius $r(\rho)$ strictly smaller than $1$, then there exists a unique stationary distribution for the
multivariate Hawkes process $N$ with stochastic intensity given by (\ref{intensity_linear_A}). Hence, we define the following  sets of parameters
\begin{align}
    &\mathcal{\bar H} = \Big\{\boh = (h_{l,k})_{(l,k)\in \K^2}; \hspace{0.1cm} \supp(h_{l,k}) \subseteq [0,A], \hspace{0.1cm} \forall (l,k)\in \K^2 \Big\},\nonumber\\
    &\mathcal{H} = \Big\{\boh\in \mathcal{\bar H}:\hspace{0.1cm} \forall (l,k)\in \K^2, h_{l,k}\geq 0 , \hspace{0.1cm}r(\rho)<1\Big\}. \nonumber
\end{align}
 The statistical model we consider is $\mathcal{P} = \big\{\P_{f}, \hspace{0.1cm}f=(\nu, \boh)\in ]0,+\infty[^K\times \mathcal{H}\big\}$, which is identifiable. We assume that the true generating process of $N$, $\mathbb P_0$,  is associated to a \textit{true} parameter $f^0\in ]0,+\infty[^K\times \mathcal{H}$, so that $\P_0 =\P_{f^0}$;   $\L_2(\P_0) $ denotes the space of random variables square integrable with respect to $\P_0$ and an expectation under $\P_0$ is denoted $\E_0$.  We denote by $\P_0(.\vert \mathcal{G}_{0-})$ the conditional distribution of $\P_0$ given $\mathcal{G}_{0-}$.
 
 We assume that we observe on an interval $[-A,T]$ the stationary multivariate linear Hawkes process with parameters $f^0$. For some parameters $f$, the log-likelihood on $[0, T]$ conditionally on $\mathcal{G}_{0-}$ (or equivalently on $\sigma\big(N^k([s,0[), -A\leq s <0, k \in \K\big)$) is given by
\begin{align}\label{likelihood}
    L_T\big(f\big) =  \sum_{k=1}^K \int_0^T \log\big(\lambda_t^k(f)\big) dN_t^k - \int _0^T \lambda_t^k(f) dt.
\end{align}
$\P_{f}$ is the distribution of a Hawkes process parameterized by $f$ and we have
\begin{align}
    d\P_f(.\vert \mathcal{G}_{0^-}) = e^{L_T(f) - L_T(f^0)}d\P_0(.\vert \mathcal{G}_{0^-}).\nonumber
\end{align}

To avoid a certain "boundary" issue explained just after Theorem \ref{th_convol} (see also the discussion after Theorem \ref{main_theorem}),  we will consider ReLU nonlinear multivariate Hawkes processes (the processes of interest being still the linear ones) for which the functions $(h_{l,k})_{l,k}$ can take negative values and the intensity of the $k$-th process is given by
\begin{align}\label{intensity_non_linear}
    \lambda_t^{k}(f_k)  = \Big(\nu_k + \sum_{l=1}^K \int_{t-A}^{t-} h_{l,k}(t-u)dN_u^l\Big)_+\hspace{0.1cm},
\end{align}
with still $\nu_k>0$. Note that if the functions $(h_{l,k})_{l,k}$ are non negative, a ReLU nonlinear Hawkes  process and a linear Hawkes process both parameterized by $f$ have the same intensity. Linear Hawkes processes are thus particular cases of nonlinear ReLU Hawkes processes. Set $h^+_{l,k}=\max(h_{l,k},0)$ and $h^-_{l,k}=-\min(h_{l,k},0)$, we denote by $\rho_+$ the matrix with entry $(l,k)$ equal to $\int_0^A h_{l,k}^+(x)dx$ and we set $\mathcal{H}_R = \{\boh\in \mathcal{\bar H}:\hspace{0.1cm} r(\rho_+)<1\}$. By Theorem 7 of \cite{bremaud_massoulie_96}, given $f\in ]0,+\infty[^K\times \mathcal{H}_R$, there exists a unique stationary distribution with intensity (\ref{intensity_non_linear}). Now, we set
\begin{align}
    &\mathcal{F}_R = \Big\{(\nu,\boh)\in ]0,+\infty[^K\times \mathcal{H}_R: \hspace{0.1cm} \underset{k\in \K}{\min}( \nu_k -\underset{l\in \K}{\max} \Vert h^-_{l,k}\Vert_\infty)>0\Big\}, \nonumber
\end{align}
and  $\mathcal{P}_R = \big\{\P_f, f\in \mathcal{F}_R\big\}$. By Proposition 2.3 of \cite{Sulem_concentration},  the model $\mathcal{P}_R$ is identifiable.

The convergences in distribution, in probability and almost sure are respectively indicated by $\overset{d}{\rightarrow}$, $\overset{\P}{\rightarrow}$ and $\overset{a.s.}{\rightarrow}$. $\mathcal{N}(\varepsilon,\mathcal{A}, d)$ is the covering number of a set $\mathcal{A}$ by balls of radius $\varepsilon$ in terms of metric $d$. The symbol $\lesssim$ will denote an inequality up to a positive multiplicative constant that might be universal or depend only on the true parameter $f^0$,  and $a\asymp b$ means that $a\lesssim b \lesssim a$.  For $p\in [1,+\infty]$, $(\L_p([0,A]), \Vert.\Vert_p)$ is the usual Lebesgue normed space on $[0,A]$ and we simply write $\L_p = L_p([0,A])$. The product space $\bigotimes_{i=1}^n \L_p$ is denoted $\L_p^{n}$ (typically $n=K$ or $n=K^2$). Any norm on the space $\R
^K\times L_p^{K^2}$ induces a norm on $\L_p^{K^2}$ and on $\L_p$ and we  keep the same notation for the induced norms. In particular, for $p=1,2$, we have
\begin{align}
    \Vert (\nu,\boh)\Vert_p = \Big( \sum_{k=1}^K \vert \nu_k\vert_p^p + \sum_{l=1}^K\sum_{k=1}^K \Vert h_{l,k}\Vert_p^p\Big)^{1/p}\hspace{0.2cm},\nonumber
\end{align}
as well as  $\Vert (\nu,\boh)\Vert_\infty = \max (\hspace{0.1cm}\max\{\vert \nu_k\vert, k\in \K\},\hspace{0.1cm} \max\{\Vert h_{l,k}\Vert_\infty, (l,k)\in \K^2\})$. 

We also recall some standard notations for Hölder and Besov spaces. Let $I$ be a sub-interval of $\R$, possibly equal to $\R$, for $\beta>0$ we denote by $C^\beta(I)$ the usual Hölder space of order $\beta$ on $I$ as defined for example in Section 1.2.1 of \cite{tsybakov} (a function belonging to this space will sometimes be referred to as 
"$\beta$-Hölder regular").  When $\beta\notin \mathbb{N}$ , $C^\beta(I)$ is equal to the Besov space $\mathcal{B}^\beta_{\infty,\infty}(I)$ and when $\beta\in\mathbb{N},$ $C^\beta(I)\subsetneq \mathcal{B}^\beta_{\infty,\infty}(I)$. The spaces $\mathcal{B}^\beta_{\infty,\infty}(I)$ are also called the Hölder-Zygmund spaces. 

Finally, let $\boh$ and $\bog$ be two $K\times K$ matrices of functions, $\boh .\bog$ and $\boh./\bog$ are the $K\times K$ matrices whose entries $(l,k)$ are the functions $h_{l,k}  g_{l,k}$  and $h_{l,k}/g_{l,k}$ (provided that $g_{l,k}\neq 0$), respectively. Given $\gamma: \R\rightarrow \R$, the matrix with entry $(l,k)$ equal to $\gamma(h_{l,k})$ is denoted $\gamma(\boh)$.

\section{Semiparametric efficiency}\label{sec_semi_param_eff}

In this Section, we study the efficiency theory on the estimation of smooth functionals $\psi(f)$ as presented in \cite{VDV_AS} for i.i.d. models and in \cite{mcneney} for non i.i.d. models. To do so, we need to first derive the LAN property. Then, using the convolution theorem, we find the optimal asymptotic Gaussian distribution for the class of regular estimators. Finally, we obtain, using Palm calculus, an explicit expression for the \textit{information operator} which can be useful to study the optimal variance for regular estimators.

\subsection{LAN expansion} \label{subsec_LAN}
We expand the log-likelihood at second order around $(\nu^0,\boh^0)$. Let $(\xi,\bog)\in \R^K\times \L_2^{K^2}$, here and after the quantity $\tilde \lambda_t^k(\xi,\bog)$ is given by (\ref{intensity_linear_A}), namely
\begin{equation}
    \tilde \lambda_t^k(\xi_k, \bog_k) = \xi_k + \sum_{l=1}^K \int_{t-A}^{t-} g_{l,k}(t-u)dN_u^l.\nonumber
\end{equation}
To simplify the presentation in this paragraph, we take $\bog$  such that for $T$ large enough the linear path $\boh^0 +\bog/\sqrt{T}$ is in $\mathcal{H}$. We write $\log(1+x) = x - \frac{x^2}{2}  + x^2R(x), \hspace{0.2cm}x>-1$, and simple algebra (see the proof of Lemma \ref{lemma_LAN}) implies 
\begin{align} \label{likelihood_expansion}
    &L_T\Big(\nu^0+ \frac{\xi}{\sqrt{T}}, \boh^0+ \frac{\bog}{\sqrt{T}}\Big) - L_T\big(\nu^0, \boh^0\big) =W_T(\xi,\bog) -  \frac{\Vert (\xi,\bog)\Vert_L^2}{2} + R_T\big(\frac{\xi}{\sqrt{T}},\frac{\bog}{\sqrt{T}}\big),
\end{align}
with
\begin{align}
    &W_T(\xi,\bog) = \frac{1}{\sqrt{T}}\sum_{k=1}^K \int_0^T \frac{\tilde \lambda_t^k(\xi_k, \bog_k)}{\lambda_t^k(f^0_k)} \big(dN_t^k - \lambda_t^k(f^0_k) dt\big),\\
    &\Vert (\xi,\bog)\Vert_L^2 = \sum_{k=1}^K\E_0\bigg[\frac{\tilde\lambda_A^k(\xi_k,\bog_k)^2}{\lambda_A^k(f^0_k)} \bigg],\nonumber
\end{align}
and the remainder $R_T\big(\frac{\xi}{\sqrt{T}},\frac{\bog}{\sqrt{T}}\big)$ is given by (\ref{remainderLAN_def}) in Section \ref{sec:pr:LAN_2}. Next, we set
\begin{align}
    \L_{2,h^0_{l,k}} := \Big\{ g:[0,A]\rightarrow\R \hspace{0.1cm},\hspace{0.1cm}\int_0^A\frac{g^2(x)}{\nu_0^k+h^0_{l,k}(x)} dx<+\infty\Big\}\hspace{0.2cm}\text{and}\hspace{0.2cm}
    \L_{2,\boh^0} :=\bigotimes_{(l,k)\in \K^2} \L_{2,h^0_{l,k}}   .\nonumber
\end{align}
The canonical product norm on the space $\RLKO$ is written $\Vert .\Vert_{(2,\boh^0)}$, 
\begin{align}
    \Vert (\xi,\bog)\Vert_{(2,\boh^0)}^2 = \sum_{k=1}^K \xi_k^2 + \sum_{(l,k)\in \K^2} \int_0^A \frac{g_{l,k}^2(x) }{\nu_0^k + h^0_{l,k}(x)}dx.\nonumber
    \end{align}
Let $\L_{2,h^0_{l,k}}^+:=\big\{ g\in \L_{2,h^0_{l,k}}, g(x)\geq 0 \hspace{0.1cm}\text{when }\hspace{0.1cm} h^0_{l,k}(x) =0\big\}$ and $ \L_{2, \boh^0}^+ :=\bigotimes_{(l,k)\in \K^2} \L_{2,h^0_{l,k}}^+$. The latter set is a convex cone, its linear span is $\L_{2,\boh^0}$ and when the functions $h^0_{l,k}$ are bounded away from $0$, $\L_{2, \boh^0}^+=\L_{2,\boh^0}$. Given the first order term $W_T$, we define the operator $\mathcal{S}$ by
\begin{equation}\label{def_score}
\begin{aligned}
    \mathcal{S} &:\hspace{0.1cm} \R^K \times \L_{2,\boh^0} \xrightarrow{\hspace{0.5cm}}\mathcal{S}\big(\R^K \times \L_{2,\boh^0}\big)\\
    &\hspace{0.3cm}\big(\xi,\bog \big)\xrightarrow{\hspace{0.5cm}}  \bigg(\int_0^1 \tilde\lambda_t^k(\xi_k,\bog_k) \big(\frac{dN_t^k}{\lambda_t^k(f^0_k)} -dt\big)\bigg)_{k\in\K} .
\end{aligned}
\end{equation}
This linear operator defines a bilinear map over $\RLKO$ by 
\begin{align}
    \langle (\xi,\bog), (\xi',\bog')\rangle_L :=\E_0 \Big[\mathcal{S}(\xi,\bog)^T \mathcal{S}(\xi',\bog')\Big] = \sum_{k=1}^K \E_0\bigg[ \frac{\tilde\lambda_A^k(\xi_k,\bog_k)\tilde\lambda_A^k(\xi_k',\bog_k')}{\lambda_A^k(f^0_k)}\bigg].\nonumber
\end{align}
 $\Vert (\xi,\bog)\Vert_L^2 = \langle (\xi,\bog), (\xi,\bog)\rangle_L$, we then show in the following lemma that  $\langle\cdot{,}\cdot\rangle_L$ is an inner product, we endow the space $\RLKO$ with this LAN inner product. We call $\Vert . \Vert_L$ the LAN norm,  it plays a major role in semiparametric efficiency, see Chapter 25 of \cite{VDV_AS}, \cite{mcneney} and \cite{castillo_rousseau}.   The linear operator $\mathcal{S}$ can be viewed as a score operator (see again Chapter 25 of \cite{VDV_AS} for instance) for our non-i.i.d. setting in the sense that $\mathcal{S}$ is independent of $T$ and verifies $\Vert .\Vert_L = \Vert \mathcal{S}(.)\Vert_{L_2(\P_0)}$. Our choice of $\mathcal{S}$ is inspired by the first order term $W_T$, other choices are possible as soon as they verify the previous properties.

\begin{lemma}\label{eq}
    The bilinear map $\langle\cdot{,}\cdot\rangle_L$ is an inner product on $\RLKO$ and induces the LAN norm $\Vert .\Vert_L$. In addition,  $\Vert.\Vert_L$ is equivalent to the canonical norm $\Vert .\Vert_{(2,\boh^0)}$ over $\RLKO$. Consequently, $\mathcal{S}$ is a bounded, bijective, linear operator between the Hilbert spaces $\big(\RLKO, \Vert .\Vert_{(2,\boh^0)}\big)$ and $\big(\mathcal{S}(\RLKO),  \Vert . \Vert_{\L_2(\P_0)}\big)$.
\end{lemma}
We can now state the result concerning the LAN expansion. Here the LAN property is in the sense of definition 2.1 of \cite{mcneney}. 

\begin{lemma}\label{lemma_LAN} Let $(\nu^0,\boh^0)$ be in $]0,+\infty[^K\times \mathcal{H}$. The model $\mathcal{P}$ verifies the LAN property at $\P_0$ with respect to $\mathcal{S}(\RLKOp)$. 
\end{lemma}

 Lemmas \ref{eq} and \ref{lemma_LAN} are proved respectively in Section \ref{sec:pr:LAN_1} and Section \ref{sec:pr:LAN_2}. Lemma \ref{eq} shows that the perturbations must belong to $\RLKOp$ and since $\mathcal{S}$ is linear, $\mathcal{S}(\RLKOp)$ is the maximal tangent set (see Chapter 25 of \cite{VDV_AS} for further details).  In the proof of Lemma \ref{lemma_LAN}, to ensure that we stay in the model $\mathcal{P}$ we take a nonlinear path, instead of the linear one used in (\ref{likelihood_expansion}), given by
\begin{align}\label{nonlinear_path_model_lin}
(\xi(T), \bog(T))_{\mathcal{P}}:=\bigg(\big(\nu_k^0 + \frac{\xi_k}{\sqrt{T}}\big)_{k\in \K} ,\Big(\Big(h_{l,k}^0 + \frac{1}{a_T}\arctan\big(\frac{a_T g_{l,k}}{\sqrt{T}}\big)\Big)_+\Big)_{(l,k)\in \K^2}\bigg),
\end{align}
with $(a_T)_T$ a positive sequence such that $a_T\rightarrow +\infty$ and $a_T/\sqrt{T}\rightarrow 0$ as $T$ goes to $+\infty$. As $\arctan(x) \sim x$ in $0$, this nonlinear path is asymptotically equivalent to the linear path of (\ref{likelihood_expansion}) and one can check that
the perturbed parameters defined by (\ref{nonlinear_path_model_lin}) are in $]0,+\infty[^K\times \mathcal{H}$ for $T$ large enough, since  $\arctan$ is bounded.

\subsection{Convolution theorem} \label{subsec_convolution}

Using the LAN property, we can now obtain the convolution theorem for smooth functionals. For simplicity, we restrict to the case of real-valued functionals but the results presented in this paper can be directly extended to $\R^d$-valued functionals, $d\geq 2$. Similarly to \cite{VDV_AS} and \cite{mcneney}, we consider that the functional $\Psi: \mathcal{P}\rightarrow \R$  is differentiable at $\P_0$, i.e.  there exists $\tilde \psi^0_L = (\xi^0_L,\bog^0_L)\in \RLKO$ such that for all $(\xi,\bog) \in\RLKOp$,
\begin{align}
\label{def_diff_vdv}
    \sqrt{T}\bigg(\Psi\big(\P_{(\xi(T),\bog(T))_{\mathcal{P}}}\big)-\Psi\big(\P_{0}\big)\bigg) \xrightarrow[T\rightarrow +\infty]{}\langle (\xi,\bog) , \tilde \psi^0_L\rangle_L\hspace{0.1cm},
\end{align}
so that $\tilde \psi^0_L$ is the least favourable direction. Since the model $\mathcal{P}$ is identifiable, $\Psi(\P_f) =\psi(f)$ and from now on we work directly with $\psi(f)$.  The LAN norm is equivalent to the norm $\Vert .\Vert_{(2,\boh^0)}$ (Lemma \ref{eq}) so (\ref{def_diff_vdv}) is equivalent to  the existence of $\tilde \psi^0_2=(\xi^0_2,  \bog^0_2 )\in \RLKO$ such that for all $(\xi,\bog) \in \RLKOp$,
\begin{align}\label{diff_l2}
\sqrt{T}\bigg(\psi\big((\xi(T),\bog(T))_\mathcal{P}\big) -\psi(f^0)\bigg)\xrightarrow[T\rightarrow +\infty]{}\langle (\xi,\bog) , \tilde \psi^0_2\rangle_2\hspace{0.1cm},
\end{align}
and, denoting $\mathcal{S}^*$ the adjoint operator of $\mathcal{S}$ and letting $\Gamma:=\mathcal{S}^*\mathcal{S}$,  we have
\begin{align}\label{eq_least_fav}
   \tilde \psi^0_2 = \Gamma(\tilde \psi^0_L ).
\end{align}
$\Gamma$ is called the \textit{information operator}, see Section 25.5 of \cite{VDV_AS}. For instance, the functionals
\begin{align}\label{typical_functionals}
    \psi^b_{k}(\nu,\boh) = \nu_k;\hspace{0.2cm}\psi_{l,k}^2(\nu,\boh) = \int_0^A h_{l,k}^2(x) dx \hspace{0.2cm}\text{and} \hspace{0.2cm}      \psi_{l,k}^a(\nu,\boh) = \int_0^A a(x)h_{l,k}(x) dx\hspace{0.1cm},
\end{align}
with $a\in L_{2,h^0_{l,k}}$ and such that $\int_0^A \vert a(x)h^0_{l,k}(x)\vert dx<+\infty$, verify (\ref{diff_l2}) (and (\ref{def_diff_vdv})) with  $\tilde \psi^0_2$ respectively given by 
$(e_{k}, 0)$, $(0,2h^0_{l,k} E_{l,k})$ and $(0, aE_{l,k})$ where  $E_{l,k}$ is the $K\times K$ matrix that contains functions all equal to $0$ except the one at entry $(l,k)$ equal to the identity and $e_{k}$ is the vector of length $K$ with all element equal to $0$ except the $k$-th one equal to $1$.

The LAN property and differentiability are the two assumptions of the convolution theorem (Theorem  2.4 in \cite{mcneney}) for regular estimators (see again \cite{mcneney} or \cite{VDV_AS} for a definition).

\begin{theorem}\label{th_convol}
 Let $f^0\in]0,+\infty[^K\times \mathcal{H}$. Consider  a functional $\psi :]0,+\infty[^K\times \mathcal{H}\rightarrow \R$  differentiable at $f^0$ in the sense of (\ref{diff_l2}) with least favourable direction $\tilde \psi^0_L$. The asymptotic distribution of every regular estimator of $\psi(f^0)$ can be written as the convolution between $\mathcal{N}\big(0, \Vert \tilde \psi^0_{L}\Vert_{L}^2\big)$ and some other probability distribution.
\end{theorem}

Note that we have chosen a path asymptotically equivalent to the additive path $\boh^0 + \bog/\sqrt{T}$ instead of a path equivalent to the multiplicative path $\boh^0(1 + \bog/\sqrt{T})$ so that the functions $h_{l,k}^0$ can be perturbed on the sets $\{x: h_{l,k}^0(x)=0\}$. Note also that Theorem 12 of \cite{mcneney} is written for linear tangent spaces.  When the functions $h^0_{l,k}$ are not all bounded away from $0$, the tangent set $\mathcal{S}(\RLKOp)$ is not a linear space but a convex cone and the convolution theorem remains valid in this more general case, see \cite{vdv_88}.  The convolution theorem implies that a regular estimator is asymptotically
efficient if its limit distribution is $\mathcal{N}\big(0, \Vert \tilde \psi^0_L\Vert_L^2\big)$. Furthermore, by Proposition 5 of \cite{mcneney}, a regular estimator $\hat \psi_T$ is efficient if and only if
\begin{align} \label{eq_efficient_estimator}
     \hat \psi_T = \psi(f^0) + \frac{W_T(\tilde \psi^0_L)}{\sqrt{T}} + o_{\P_0}\big( \frac{1}{\sqrt{T}}\big).
\end{align}

We point out that in some models with non open parameter set, it is possible to modify an efficient estimator by leveraging the property that defines the boundary of the parameter set and obtaining a new estimator, regular in the interior and with "better performance" on the boundary. For this reason, the notion of regular estimators and the convolution theorem can be seen as less relevant when the true parameters are on the boundary of the parameters set, see \cite{vdv_88} for instance the parametric example 2.15. In our model $\mathcal{P}$, if the functions $\boh^0$ are not bounded away from $0$, then $\boh^0$ is on the boundary of $\mathcal{H}$.

Moreover, the BvM property (presented in the following Section \ref{BVM}) can suffer from boundary effects  when the prior on $\boh$ is supported on $ \mathcal{H}$ and the true functions $\boh^0$ are not bounded away from $0$  (so are on the boundary of the prior support, see in particular the discussion just after Theorem \ref{main_theorem}). To avoid this boundary issue, we will consider priors supported on a larger set $\mathcal{F}_R^e$ that contains $\mathcal{F}_R$ which corresponds to the "enlarged model" $\mathcal{P}_R$ (both defined in Section \ref{not_assump}). Indeed $f^0=(\nu^0,\boh^0)\in ]0,+\infty[^K\times \mathcal{H}$ cannot lie on the boundary of $\mathcal{F}_R$, even when the functions $\boh^0$ are not bounded away from $0$.  We thus also derive semiparametric efficiency at $\P_0$ in the  model $\mathcal{P}_R$ so that the BvM property we obtain for these “enlarged” priors is accompanied by a convolution theorem.

The following study of the efficiency theory in the nonlinear ReLU Hawkes processes model  $\mathcal{P}_R$ is limited to the case where the distribution of interest $\P_0$ is the one of a linear Hawkes process. Deriving the efficiency theory for the full model $\mathcal{P}_R$ is out of the scope of this article and may require new tools and ideas to handle the nonlinearity.

In the model $\mathcal{P}_R$, when $\mathbf h^0\geq 0$ we consider perturbations $(\xi, \bog)\in \RLKO$ and we use the path
\begin{align}\label{gen_non_linear_path_relu}
    (\xi(T), \bog(T))_{\mathcal{P}_R}:=\bigg(\big(\nu_k^0 + \frac{\xi_k}{\sqrt{T}}\big)_{k\in \K} \hspace{0.05cm}, \hspace{0.05cm}\Big(h_{l,k}^0 + \frac{1}{a_T}\arctan\big(\frac{a_T g_{l,k}}{\sqrt{T}}\big)\Big)_{(l,k)\in \K^2}\bigg) .
\end{align}
One can check  that for $T$ large enough $(\xi(T), \bog(T))_{\mathcal{P}_R}\in \mathcal{F}_R$. Then, we have the following lemma on the LAN expansion and the convolution theorem at $f^0$ in the model $\mathcal{P}_R$.

\begin{lemma}\label{lemma_LAN_convol_RELU}
 Let $(\nu^0,\boh^0)$ be in $]0,+\infty[^K\times \mathcal{H}$. The model $\mathcal{P}_R$ verifies the LAN property at $\P_0$ with respect to $\mathcal{S}(\RLKO)$. Moreover, consider a functional $\psi: \mathcal{F}_R\rightarrow \R$ differentiable at $\P_0$ in the sense that there exists $\tilde \psi^0_L = (\xi^0_L,\bog^0_L)\in \RLKO$ such that for all $(\xi,\bog) \in\RLKO$,
 \begin{align}
    \sqrt{T}\bigg(\psi\big((\xi(T),\bog(T))_{\mathcal{P}_R}\big)-\psi(f^0)\bigg) \xrightarrow[T\rightarrow +\infty]{}\langle (\xi,\bog) , \tilde \psi^0_L\rangle_L\nonumber.
\end{align}
    Then, the asymptotic distribution of every regular estimator of $\psi(f^0)$ can be written as the convolution between $\mathcal{N}(0, \Vert \tilde \psi^0_L\Vert_L^2)$ and some other probability distribution.
\end{lemma}
 This result is proved in Section \ref{sec:pr:LAN_2}, with similar arguments as in model $\mathcal{P}$.
 
 From now on, we assume that $\boh^0$ belongs to $\L_\infty^{K^2}$, with this assumption the space of perturbations $\L_{2,\boh^0}$ is equal to $\L_2^{K^2}$ and the norms $\Vert .\Vert_{2, \boh^0}$ and $\Vert . \Vert_2$ are equivalent.
 
 \subsection{Explicit formula of the information operator}\label{subsec_explicit_gamma}
 
In this section, we show how Palm calculus allows to obtain an explicit expression for the information operator $\Gamma$ (Lemma \ref{explicit_gamma}). In particular, this implies that the least favourable direction satisfies the fixed-point equation (\ref{converse_expr}). This is of interest in order to study the optimal variance for regular estimators $\Vert \tilde \psi^0_L \Vert_L^2$. This will also serve as our starting point in Section \ref{sec_lfeq}, where we study the regularity of the least favourable direction to derive explicit conditions for the BvM property. In the Appendix, we briefly introduce Palm calculus and we state the useful results that we use,  giving each time references for the proofs and further details.

Let $(\xi,\bog)\in \R^K\times \L_2^{K^2}$ and set $(\xi',\bog') = \Gamma(\xi,\bog)= \mathcal{S}^*\mathcal{S}(\xi,\bog) $. By definition of the LAN norm, we have for all $(\bar \xi, \bar \bog)\in \R^K\times \L_2^{K^2}$, $\langle (\bar \xi, \bar \bog), (\xi',  \bog')\rangle_2 = \langle (\bar \xi, \bar \bog) ,(\xi,  \bog)\rangle_L$ which can be written as
\begin{align}
    \sum_{k=1}^K \bar \xi_k \xi'_{k}+ \sum_{j=1}^{K} \langle  \bar g_{j,k}, g'_{j,k}\rangle_2 = \sum_{k=1}^K \E_0\bigg[\frac{\tilde \lambda_A^k(\bar \xi_k, \bar \bog_k)\tilde \lambda_A^k(\xi_{k}, \bog_{k})}{\lambda_A^k(f^0_k)}\bigg]. \nonumber
\end{align}
Taking  $\bar \bog=0$ and $\bar \xi$ with all entries equal to $0$ except for the entry $k$ equal to $1$, we find that  $\xi'_k = \E_0\big[\frac{\tilde \lambda_A^k(\xi_k,\bog_{k})}{\lambda_A^k(f^0_k)}\big]$. Then, let $(l,k)\in \K^2$, take $\bar \xi=0$ and $\bar \bog$ with all entries equal to $0$ except for the entry $(l,k)$ equal to some function $\bar g_{l,k}\in \L_2$, we have
\begin{align}
     \langle \bar g_{l,k}, g'_{l,k}\rangle_2 &= 
      \xi_k\E_0\Big[\int_0^A \bar g_{l,k}(A-u) \frac{dN_u^l}{\lambda_A^k(f^0_k)}\Big] +\E_0\bigg[\int_0^A \bar g_{l,k}(A-u) g_{l,k} (A-u)  \frac{dN_u^l}{\lambda_A^k(f^0_k)}\bigg]\nonumber\\ 
      &\hspace{2cm}+\sum_{j=1}^K \E_0\bigg[\int_0^A\int_0^A\onee_{(u,l)\neq (s,j)} \bar g_{l,k}(A-u) g_{j,k} (A-s)  \frac{dN_s^jdN_u^l}{\lambda_A^k(f^0_k)}\bigg]. \label{main_decomp_palm}
\end{align}
Now,  for $(l,k,j)\in \K^3$, we define the function $p_{l,k}: [0,A]\rightarrow \R$ by 
\begin{align}
    p_{l,k}(A-u):=\E_0^{(u,l)}\Big[\frac{1}{ \lambda_A^{k}(f^0_k)}\Big],\label{def_plk}
\end{align}
and the operator $\zeta_{l,j,k} :\L_2\rightarrow \L_2$ by 
\begin{align}
    \zeta_{l,j,k}(g)(A-u) := \displaystyle \E^{(u,l)}_0\bigg[ \int_0^A \onee_{(s,j)\neq (u,l)} g(A-s) \frac{dN_s^j}{ \lambda_A^{k}(f^0_k)}\bigg],\label{def_zeta}
\end{align}
where $\E_0^{(u,l)}$ is an expectation under the Palm distribution $\P^{(u,l)}_0$ of $N$, see the Appendix. Then, using (\ref{int_eq_palm}),  we can rewrite the right hand side of (\ref{main_decomp_palm})  with the functions $p_{l,k}$ and $\zeta_{l,j,k}(g_{l,k})$:
\begin{align}
    \langle \bar g_{l,k}, g'_{l,k}\rangle_2= \langle \bar g_{l,k}, \mu_l^0\xi_{k}p_{l,k}\rangle_2 + \langle \bar g_{l,k}, \mu^0_lg_{l,k}p_{l,k}\rangle_2 + \langle \bar g_{l,k}, \mu_l^0\sum_{j=1}^K \zeta_{l,j,k}(g_{j,k})\rangle_2.\nonumber
\end{align}
This holds for any $\bar g_{l,k}\in \L_2$, thus $g'_{l,k} = \mu_l^0\xi_k p_{l,k} +\mu_l^0g_{l,k} p_{l,k}+ \mu_l^0 \sum_{j=1}^K \zeta_{l,j,k}(g_{j,k})  $. We  have thus proved the following lemma.

\begin{lemma}\label{explicit_gamma}
Let $(\xi,\bog)\in \R^K\times \L_2^{K^2}$ and set $(\xi',\bog') = \Gamma(\xi,\bog) $, for all $(l,k)\in \K^2$:
\begin{align}
    \xi'_k = \E_0\big[\frac{\tilde \lambda_A^k(\xi_k,\bog_{k})}{\lambda_A^k(f^0_k)}\big]\hspace{0.1cm},\hspace{0.2cm} g'_{l,k} =\mu_l^0 \xi_k p_{l,k}+ \mu_l^0 g_{l,k}  p_{l,k} + \mu_l^0 \sum_{j=1}^K \zeta_{l,j,k}(g_{j,k}), \nonumber
\end{align}
\end{lemma}

Conversely,  $(\xi,\bog)$ verifies the following fixed point equation:
\begin{equation}
\begin{aligned}\label{converse_expr}
    &\xi_k =\E_0\big[\lambda_A^k(f^0_k)^{-1}\big]^{-1}\bigg(\xi'_{k}-\E_0\Big[\frac{\tilde \lambda_A^k(0,\bog_{k})}{\lambda_A^k(f^0_k)}\Big]\bigg)\hspace{0.1cm},\hspace{0.2cm}\forall k\in \K\hspace{0.1cm},\\
    &g_{l,k} = \frac{g'_{l,k} - \mu^0_l\sum_{j=1}^K \zeta_{l,j,k}(g_{j,k}) }{\mu^0_l p_{l,k}} - \xi_{k} \hspace{0.1cm},\hspace{0.2cm}\forall (l,k)\in \K^2.
\end{aligned}
\end{equation}
Note that the right hand side of (\ref{converse_expr}) is well defined because for each $(l,k)\in\K^2$, the function $p_{l,k}$ is lower bounded by some positive constant $c_{l,k}$ (see Lemma \ref{lemma_lb_p}). In particular, $(\xi^0_L, \bog^0_L)$ verifies this fixed point equation with $(\xi',\bog')= (\xi^0_2,\bog^0_2)$.

In the following section, we apply the above results to Bayesian semiparametric estimation of $\psi(f^0)$ and we study the BvM property in this context.

\section{Main results for Bayesian inference}\label{BVM}

We recall that we observe on an interval $[-A,T]$ a stationary and linear Hawkes process with parameters $f^0$. We consider a Bayesian nonparametric approach and we put a prior $\Pi$ on $f\in \mathcal{F}_R^e$ defined by 
\begin{align*}
    \mathcal{F}_R^e = \Big\{(\nu,\boh)\in ]0,+\infty[^K\times \mathcal{\bar H}: \hspace{0.1cm} \underset{k\in \K}{\min}( \nu_k -\underset{l\in \K}{\max} \Vert h^-_{l,k}\Vert_\infty)>0,\hspace{0.1cm} \VERT \rho_+\VERT_1\leq H\Big\}
\end{align*}
for some $H\geq 2\VERT \rho^0\VERT_1$, where $\VERT. \VERT_1$ is the matrix norm related to the norm $\Vert.\Vert_1$. The prior is thus allowed to put some mass on parameters $f=(\nu,\boh)$ with some $h_{l,k}$ taking negative values; as in the previous section, $\P_f$ is then the distribution of a ReLU nonlinear Hawkes process. As in \cite{donnet_concentration}, the posterior distribution is given by:
\begin{align}
    \Pi(B\vert N) = \frac{\int_B e^{L_T(f)}d\Pi(f)}{\int_{\mathcal{F}_R^e}e^{L_T(f)}d\Pi(f)}\hspace{0.1cm},\nonumber
\end{align}
with $L_T(f)$ the log-likelihood (\ref{likelihood}). We study, under $\P_0$, the behavior of the posterior distribution when $T\rightarrow +\infty$. More specifically, we are interested in the marginal posterior distribution of $\sqrt{T}(\psi(f)-\hat\psi_T)$, that we denote $\mathcal{L}^\Pi\big(\sqrt{T}(\psi(f) - \hat\psi_T)\vert N\big)$,  with $\hat\psi_T$ as in (\ref{eq_efficient_estimator}). Let $d_{BL}$ be the bounded Lipschitz metric on the space of probability measures on $\R$, we say that we have the (semiparametric) Bernstein-von Mises (BvM) property when
\begin{align}\label{def_bvm}
    d_{BL}\Big(\mathcal{L}^\Pi\big(\sqrt{T}(\psi(f) - \hat\psi_T)\vert N\big), \mathcal{N}(0,\Vert \tilde \psi^0_L\Vert_L^2\big)\Big)\xrightarrow[T\rightarrow+\infty]{\P_0} 0.
\end{align}
Since the $d_{BL}$ metric metrizes weak convergence, the BvM property  (\ref{def_bvm}) means that with probability going to 1 with respect to $\P_0$, the posterior distribution of $\sqrt{T} (\psi(f) - \hat \psi_T)$ converges weakly to a $\mathcal N(0, \|\tilde \psi_0\|_L^2)$. This has many consequences, among others,  the weak convergence in (3.1) implies the uniform convergence of the c.d.f (see the Supplement of \cite{castillo_rousseau}). In particular, let $a_T$ and $b_T$ denote  the quantiles at level $\alpha/2$ and $1-\alpha/2$, respectively, of the marginal posterior of $\psi(f)$, we have 
\begin{align*}
    \P_0\big(\psi(f^0)\in [a_T,b_T]\big)\rightarrow 1-\alpha,
\end{align*}
so that the credible interval $[a_T,b_T]$ is asymptotically a confidence interval at level $\alpha$ of  minimal asymptotic length (see Section 1.5 of \cite{castillo_ptrf}). To obtain $a_T$ and $b_T$, we do not need to know $\tilde \psi_L^0$ (which is not explicit in our model, see Section \ref{sec_lfeq}) and thus the Bayesian procedure can be viewed as an algorithm for computing asymptotic confidence intervals. We will see later (Corollary \ref{coro_posterior_mean}), that in our setting (\ref{def_bvm}) implies, under some mild conditions, the efficiency of the posterior mean.

A first step to prove the BvM property (\ref{def_bvm}) is to obtain posterior contraction on $f^0$ with respect to some metric $d$ and rate $\epsilon_T\rightarrow 0$: $\Pi(f:d(f,f^0)\leq \epsilon_T \vert N) \xrightarrow[]{\P_0} 1$ as $T\rightarrow +\infty$. Posterior contraction in $\L_1$ norm has been studied by \cite{donnet_concentration, rousseau_highdim, Sulem_concentration}.  To prove the BvM property, we need a similar result but with respect to  the $\L_2$ norm. 

We first present the family of  prior models we consider. Then, we give assumptions from which we derive a $\L_2$ posterior contraction rate, which have an interest on its own. Next, with some additional assumptions, we state our most important results, those on the BvM property. Finally, these results are illustrated on specific prior models.

\subsection{Prior model}\label{subsec_prior_construct}

We choose a prior distribution $\Pi$ on the parameters $f=(\nu, \boh)$ of the form $\Pi \propto (\Pi_\nu\otimes\Pi_{\boh})\onee_{(\nu,\boh)\in \mathcal{F}_R^e}$. $\Pi_\nu$ is a probability distribution on $]0,+\infty[^K$ with a positive and continuous density and such that $\Pi_\nu(\Vert \nu \Vert_\infty\geq x)\lesssim x^{-a}$ for some $a>1$ and positive $x$ large enough. For the functional part, the prior is based on  finite random series,  these models are widely used to obtain adaptive priors, see for instance Section 10.4 of \cite{bnp_book}. For all integer $J\geq 1$, consider a family of linearly independent functions $B_J = (b_1^J, ...,b_J^J)\in\L_\infty^J$,  in $\L_2$. The prior $\Pi_\boh$ is built by choosing first a random $J\geq 1$ from $\Pi_J$, then drawing random coefficients $\bth= (\theta_{l,k})_{(l,k)\in \K^2}$ with $\theta_{l,k}\in \R^{J}$ from $\Pi_{\bth|J}$, and finally setting
\begin{align}
    \tilde h_{l,k} = \theta_{l,k}^TB_J =\sum_{i=1}^J \theta_{l,k}^i b_i^J\hspace{0.1cm},\hspace{0.3cm}\forall (l,k)\in \K^2\hspace{0.1cm},\hspace{0.2cm}\boh = \varphi(\tilde \boh) = \big( \varphi(\tilde h_{l,k})\big)_{(l,k)\in \K^2}\hspace{0.1cm},\nonumber
\end{align}
where $\varphi:\R\rightarrow \R$ is a given function, non-decreasing and $L$-Lipschitz for some $L>0$, typically used to ensure that $h_{l,k}\geq0$. $\Pi_\boh$ is thus induced by the prior $\Pi_{J,\bth}$ on $(J, \bth)$ which is defined hierarchically: $d\Pi_{J,\bth} =d\Pi_{\bth\vert J} d\Pi_J$.  To write shortly the $K^2$ equalities $\tilde h_{l,k} = \theta_{l,k}^TB_J$, we will use the notation $\tilde \boh = \bth^TB_J$. For the sake of simplicity, we consider the case where $J_{l, k } =J$ for all $(l,k)\in \K^2$. Our results extend to the more general case easily. 

We set $\mathcal{\tilde H}(j):= \Span(b_1^j,...,b_j^j)$ and we assume that there exists $R\geq 2$ such that $\mathcal{\tilde H}(j)\subseteq ~\mathcal{\tilde H}(Rj)$ for all $j\geq 1$, this is weaker than assuming that $(\mathcal{\tilde H}(j))_j$ is increasing. We also impose that for all integer $j\geq 1$ and all $\tilde h_\theta = \theta^TB_j\in \mathcal{\tilde H}(j)$,
\begin{align}
    & \Vert \theta\Vert_2\asymp \gamma(j) \Vert    \tilde h_\theta\Vert_2\hspace{0.1cm}, \label{condi_rho} \\
    &\Vert  \tilde h_\theta \Vert_\infty \lesssim \sqrt{j}\Vert \tilde h_\theta\Vert_2\hspace{0.1cm}, \label{condi_sup}
\end{align}
for some positive and monotone sequence $\gamma$  such that $\gamma$ and $1/\gamma$ have at most polynomial growth. Note that if $(b_i)_{i\in \mathbb{N}}$ is an orthornormal basis of $\L_2$ and $B_j =  (b_i)_{i\leq j}$ (truncated orthonormal family), then $(\mathcal{\tilde H}(j))_j$ is increasing, (\ref{condi_rho}) is verified with $\gamma(j)=1$ and, by Cauchy-Schwarz inequality,  (\ref{condi_sup}) is verified if $\Vert \sum_{i=1}^j b_i^2\Vert_\infty\lesssim j$. The latter is for instance true for Fourier basis or for a wavelet basis with bounded and compactly supported mother wavelet. Also, if $B_j$ is a B-splines basis associated to the regular partition of $[0,A]$ into $j$ bins, then for all $j\geq 2$, $\mathcal{\tilde H}(j) \subset \mathcal{\tilde H}(2j)$ and (\ref{condi_rho}) and (\ref{condi_sup}) are verified with  $\gamma(j)=\sqrt{j}$ (see Lemma E.6 of \cite{bnp_book}).

Finally, we assume there exists some $c_1>0$ such that
\begin{align}\label{cond_piJ}
    \Pi_J(J=j) \asymp e^{-c_1j\log(j)}\hspace{0.1cm},
\end{align}
and for $\Pi_{\bth\vert J}$, we assume that there exists $U:\R^+\rightarrow \R^+$ such that for any $j\geq1$ and $M>0$, 
 \begin{align}\label{sieve_theta}
     \Pi_{\bth\vert j} (\Vert \bth\Vert_\infty \geq M)\lesssim j e^{-U(M)}.
 \end{align}
 
These two conditions are common for random series priors (see again Section 10.4 of \cite{bnp_book}).

We introduce some notations. Given $B_j$, we denote by $P^{j}_2$ and $P^j_{L}$ the orthogonal projections on $\R^K\times \mathcal{\tilde H}(j)^{K^2}$ in terms of $\L_2$ and LAN inner products respectively. Note that  $P^{j}_2(\tilde \boh):= P^{j}_2(0, \tilde \boh)$ is the orthogonal projection of  $\tilde \boh$ on $\mathcal{\tilde H}(j)^{K^2}$ in terms of  the $L_2$ norm. When it exists, we denote $\varphi^{-1}(\boh^0)$ by $\tilde \boh^0$ and $(\nu^0, \tilde \boh^0)$ by $\tilde f^0$. 

\subsection{Posterior contraction in $L_1$ and $L_2$ norms}

We now give the assumptions and the results concerning the posterior contraction, first in $L_1$ norm and then in $L_2$ norm. We begin with an assumption on $\boh^0$.

\begin{assumption}[A] \label{as_A} We assume that for all $(l,k)\in \K^2$, there exists a function $\tilde h^0_{l,k}\in ~L_\infty([0,A])$ such that $h^0_{l,k} = \varphi(\tilde h^0_{l,k})$.
\end{assumption} 

 Note that Assumption \hyperref[as_A]{(A)} allows $\varphi(x) =x\onee_{x\geq c_0}$ provided that for all $(l,k)\in \K^2$, $h^0_{l,k}\geq~c_0$ and in this case $\tilde h^0_{l,k} = h^0_{l,k}$. We also assume that $\boh^0$ verifies the following assumption.

\begin{assumption}[P1]
\label{as_P1}  For $c_2>0$ and $c_3\in \mathbb{N}$ and $c_1$ as in (\ref{cond_piJ}), set $c_4:=5c_3(c_1+c_2)+11 $ and $J_0:=4c_4/(c_1c_3)\geq 1$. We assume that there exist $c_2$ and $c_3$ as before, an integer valued sequence $(\bar J_T)_T$ and a positive sequence $(\bar\varepsilon_T)_T$  with $\bar\varepsilon_T \rightarrow 0$, $\log(T)^3\lesssim T\bar\varepsilon_T^2$, such that for $T$ large enough and for all $j\in [\bar J_T,RJ_0\bar J_T]$:
\begin{align}\label{approx_prior_B}
    \Pi_{\bth\vert j}\big(\bth\in \R^{K^2j}:\forall (l,k)\in \K^2,  \tilde h^0_{l,k}\leq \theta_{l,k}^TB_{j} \leq \tilde h^0_{l,k} + \frac{\bar\varepsilon_T}{L}\big) \geq e^{-c_2j\log(T)},
\end{align}
and such that
\begin{align}\label{P1_rate}
     \bar J_T= c_3\lceil T\bar \varepsilon_T^2 \log(T)\rceil\hspace{0.3cm}\text{and}\hspace{0.3cm} c_4T\bar\varepsilon_T^2 \log(T)^2 \leq U(M_T)  ,
\end{align}
for some $M_T\rightarrow +\infty$ with at most polynomial growth.
\end{assumption}
Assumption  (P1) is the transposition of the assumptions (A0), (A1) and (A2) in \cite{Sulem_concentration} to the case of sieve priors and has been for instance verified in \cite{Sulem_concentration} for log-spline based priors. We consider other bases functions in Section \ref{app_priors}. To prove it, it is enough to show that the functions $\tilde h^0_{l,k} + \bar \varepsilon_T/2L$ can be approximated in sup-norm at rate $\bar \varepsilon_T/4L$ by functions in $\text{Span}(B_{j})$ for any $j\geq \bar J_T$ with $\bar J_T$ as in (3.7), as long as  for some $R>0$ and all  $j\leq ~ RJ_0\bar J_T$, $\Vert b^{RJ_0\bar J_T}_j\Vert_\infty \leq w(R\bar J_T)$ with $w$ having at most polynomial growth and for instance if the prior given $J$ on $\bth$ is such that  $(\theta_j)_{ j \leq J}\stackrel{iid}{\sim} \Pi_\theta$ with  positive and continuous density with respect to Lebesgue measure on a large enough ball, and with sub-exponential tails.

 We set
\begin{align}\label{def:JTepsT}
    \varepsilon_T = \log(T)\bar\varepsilon_T\hspace{0.2cm}\text{and}
    \hspace{0.2cm}J_T= J_0\bar J_T,
\end{align}
We can now state the first posterior contraction result in terms of the  $\L_1$ norm. 

\begin{proposition}\label{th_concentration_L1}
    Let $N$ be a stationary, linear and multivariate Hawkes process with parameters $f^0 = (\nu^0, \boh^0)$ such that $\boh^0\in\L_\infty^{K^2}$. Let $\Pi$ be a prior distribution on $\mathcal{F}_R^e$ constructed as in Section \ref{subsec_prior_construct} and verifying the assumptions therein. Assume \hyperref[as_P1]{(P1)} and if $\varphi$ is not the identity function, assume also \hyperref[as_A]{(A)}. Then, for some $C_1>0$
    \begin{align}
        \Pi\big(\Vert f-f^0\Vert_1\geq C_1\varepsilon_T\big \vert N\big) \xrightarrow[T\rightarrow +\infty]{\P_0}0.\nonumber
    \end{align}
\end{proposition}

 This proposition is a direct consequence of Proposition 3.5 of \cite{Sulem_concentration} (case 1) for random series priors, we prove it  in Section \ref{sec:pr:concentration}. Note that Assumption (\ref{condi_sup}) is not used in this proof. Note also that in \cite{Sulem_concentration}, $L_1$ posterior contraction is obtained under the condition $\VERT \rho^0\VERT_1<1$ and the prior has to be supported on $\mathcal{F}_R\cap \{f: \VERT \rho_+ \VERT_1<1\}$ (similar assumptions were made in \cite{donnet_concentration}). Thanks to recent results in \cite{leblanc_exp1}, our Proposition \ref{th_concentration_L1} holds when we only assume the weaker stationarity condition $r(\rho^0)<1$ and when the prior is supported on the more flexible set $\mathcal{F}^e_R$.
 
 \begin{remark}\label{rk_approx}
    The neighborhood $ \tilde h^0_{l,k}\leq \theta_{l,k}^TB_{j} \leq \tilde h^0_{l,k} + \bar\varepsilon_T/L$ we use in Assumption \hyperref[as_P1]{(P1)} is one-sided because this is required in  \cite{Sulem_concentration} to handle the nonlinearity of ReLU when $h_{l,k}$ takes negative values (see Lemma S7.1). If the prior puts all its mass on non negative functions $\boh$, it is possible to use the  two-sided neighborhood $\tilde h^0_{l,k} - \bar\varepsilon_T/L\leq \theta_{l,k}^TB_{j} \leq \tilde h^0_{l,k} +~ \bar\varepsilon_T/L$ by applying instead Lemma 2 in \cite{donnet_concentration} in the proof of Proposition \ref{th_concentration_L1}. 
 \end{remark}

Next, to derive the $L_2$ posterior contraction rate when $\varphi$ is not the identity function, we need an other assumption on the sup-norm of the $h_{l,k}$ drawn from the prior. For this assumption, we first define for  $C_1>0$ introduced in Proposition \ref{th_concentration_L1},  the sets
\begin{align}
    &\mathcal{W}^1_T(j):= \bigg\{f = (\nu,\boh)= (\nu, \varphi(\bth^TB_j))\in \mathcal{F}_R :  \Vert \bth\Vert_\infty\leq M_T\hspace{0.1cm}, \Vert f-f^0\Vert_1\leq C_1\varepsilon_T\bigg\}\hspace{0.1cm},\nonumber\\
    &\mathcal{W}^1_T = \bigcup_{j\leq J_T} \mathcal{W}^1_T(j) \nonumber.
\end{align}
The assumption is then the following.

\begin{assumption}[P2]\label{P2}\hfill When the link function $\varphi$ is not the identity function, it is assumed that there exist $G>0$ and a positive sequence $(r_T)_T$ such that for $T$ large enough,
\begin{align}
    &\Pi\big(\mathcal{W}^1_T\cap \{\hspace{0.1cm}\Vert \boh\Vert_\infty \geq r_T\}\big) \lesssim  e^{-c_4T\varepsilon_T^2} \label{P2_h+}\hspace{0.1cm},\\
    &\Pi\big(\mathcal{W}^1_T\cap \{ \Vert \boh - \boh^0\Vert_2\leq  \log(T) r_T\varepsilon_T, \hspace{0.1cm} \Vert \boh-\boh^0 \Vert_\infty \geq G\}\big) \lesssim  e^{-c_4T\varepsilon_T^2}.\label{P2_G}
\end{align}
with $c_4$ defined in \hyperref[as_P1]{(P1)}.
\end{assumption}
 The $L_2$ posterior contraction rate is then bounded by $\log(T) \varepsilon_T$ as shown below.

\begin{theorem}\label{th_concentration_L2}
    Let $N$ be a stationary, linear and multivariate Hawkes process with parameters $f^0 = (\nu^0, \boh^0)$ such that $\boh^0\in\L_\infty^{K^2}$. Let $\Pi$ be a prior distribution on $\mathcal{F}_R^e$ constructed as in Section \ref{subsec_prior_construct} and verifying the assumptions therein. Assume \hyperref[as_P1]{(P1)} holds and let $\varepsilon_T$ and $J_T$, defined in \eqref{def:JTepsT}, be such that $\varepsilon_T\sqrt{J_T \log(T)}\rightarrow 0$, and when $\varphi$ is not the identity function,  assume also \hyperref[as_A]{(A)} and \hyperref[P2]{(P2)}. Then,
    \begin{align}
        \Pi\big(f:\Vert f-f^0\Vert_2\geq \log(T)\varepsilon_T\vert N\big) \xrightarrow[T\rightarrow +\infty]{\P_0}0.\nonumber
    \end{align}
\end{theorem}
The proof of Theorem \ref{th_concentration_L2} is provided in Section \ref{sec:pr:concentration}. When $\varphi$ is the identity function, the assumptions in Theorem \ref{th_concentration_L2} are similar to those used in \cite{donnet_concentration} and \cite{Sulem_concentration} for the weaker $\L_1$-norm with  the additional condition $\varepsilon_T\sqrt{J_T\log(T)}\rightarrow 0$.  We will see in Section \ref{app_priors}, that under $\beta$-Hölder regularity, the latter condition is equivalent to $\beta >1/2$.

 Let us briefly comment on Assumption \hyperref[P2]{(P2)} which is required when $\varphi$ is not the identity function. First, with the preceding $L_1$ posterior result we have $ \Pi(\mathcal{W}^1_T\vert N)\rightarrow 1$ under $\P_0$, and for this reason the assumption only concerns $f\in \mathcal{W}^1_T$. Secondly, this assumption is used for the "$L_2$" tests (with type 1 and 2 errors that decrease exponentially) needed for the $L_2$ posterior contraction, see Section \ref{sec_tests}. More precisely, the study of the $L_2$ tests involves a lower bound on  ratios of the type $\Vert \boh - \boh^0\Vert_2/\Vert \boh - \boh^0\Vert_\infty$ and \hyperref[P2]{(P2)} allows to deal with these ratios. Note that there is a trade off between (\ref{P2_h+}) and (\ref{P2_G}): the smaller $r_T$, the easier it is to verify (\ref{P2_G}), but conversely, the harder it is to verify (\ref{P2_h+}). In Lemma \ref{verif_P2_A_wavelets}, for priors constructed on a wavelet basis in Section \ref{app_priors}, we verify \hyperref[P2]{(P2)} for a certain class of link functions $\varphi$.

\subsection{The BvM property}

The BvM property we derive relies on the previous $L_2$ posterior contraction rate but, when $\varphi$ is not the identity function, we also need that the posterior on $\tilde \boh$ concentrates around $\tilde \boh^0$ at a rate of order $\varepsilon_T$ up to $\log(T)$ terms. For this, we strengthen Assumption \hyperref[as_A]{(A)}. To do so, we denote by  $I_0$ the convex hull of the set $\big(\cup_{(l,k)\in \K^2}  h_{l,k}^0([0,A])\big)$ and for $\epsilon>0$, $I_0(\epsilon)$ denotes the $\epsilon$-neighborhood of $I_0$. 

\begin{assumption}[A']\label{as_A'}  We assume that for some $\epsilon>0$, the inverse of the link function $\varphi^{-1}$ is well defined on  $I_0(\epsilon)$ and that $\varphi$ is three times continuously differentiable on $\varphi^{-1}(I_{0}(\epsilon))$ with a positive first derivative. Moreover, we assume that 
\begin{align}
    &\Pi\big(\mathcal{W}^1_T\cap \{ \Vert \boh - \boh^0\Vert_2\leq  \log(T) \varepsilon_T; \exists (l,k)\in \K^2, Range( h_{l,k})\not\subseteq I_{0}(\epsilon)\}\big) \lesssim  e^{-c_4T\varepsilon_T^2}.\label{A'_eps}
\end{align}
\end{assumption}

 This assumption is made so that if $\varphi$ is not the identity function, we can "linearize" it   around the $\tilde h^0_{l,k}$ by controlling the remainder terms. Note that for the function $\varphi(x) = x\onee_{x\geq c_0}$ to verify Assumption \hyperref[as_A']{(A')}, it is necessary that the functions $h^0_{l,k}$ are bounded away from $c_0$. As \hyperref[P2]{(P2)}, in Lemma \ref{verif_P2_A_wavelets}, for priors constructed on a wavelet basis in Section \ref{app_priors}, we verify \hyperref[as_A']{(A')} for a certain class of link functions $\varphi$. From now on, for $f=(\nu, \boh)= (\nu, \varphi(\tilde \boh))$, we denote $(\nu,\tilde \boh)$ by $\tilde f$. Let $\mathcal{J}_T\subseteq  \{1,..., J_T\}$ be such that $\Pi_J(\mathcal{J}_T\vert N)\rightarrow1$ under $\P_0$ and let $M>0$. We define
\begin{align}
    \A_{T}(j) = \big\{f\in \mathcal{W}^1_T(j):\Vert  \tilde f - \tilde f^0\Vert_2\leq M\log(T)\varepsilon_T\big\}\hspace{0.2cm},\hspace{0.3cm}\A_{T}:= \bigcup_{j\in \mathcal{J}_T}  \A_{T}(j).\label{def_A}
\end{align}
With Assumption \hyperref[as_A']{(A')}, we show in Lemma \ref{lemma_concentration_A} that $\Pi(\A_T\vert N)\rightarrow 1$ under $\P_0$.

Now, regarding the functional $\psi$, we recall that $\psi$ satisfies (\ref{def_diff_vdv}) with least favourable direction $\tilde \psi_L^0= (\xi^0_L, \bog^0_L)$  if   and only if it satisfies (\ref{diff_l2}) with Riesz representor $\tilde \psi^0_2= (\xi^0_2, \bog^0_2)$ and  that $\tilde \psi^0_2 = \Gamma(\tilde \psi^0_L)$. For $f\in \A_T$, we write
\begin{align}\label{psi_diff_LAN} 
    \psi(f) =\psi(f^0) + \langle f-f^0, \tilde \psi_2^0\rangle_2 + r(f, f^0) .
\end{align}

When $\varphi$ is not the identity and verifies Assumption \hyperref[as_A']{(A')}, $ \boldsymbol{\bar\varphi}^0 :=  (\varphi'(\tilde h^0_{l,k}), (l,k)\in \K^2)$. Since these functions   are bounded away from $0$ and $+\infty$, the bilinear map  $\langle (\xi,\bog), (\xi',\bog')\rangle_{L,\varphi} :=\E_0\big[\mathcal{S}_\varphi(\xi,\boldsymbol{\bar\varphi}^0.\bog)^T \mathcal{S}_\varphi(\xi',\boldsymbol{\bar\varphi}^0.\bog')\big]$ is an inner product on $\RLK$. Moreover, the induced norm $\Vert . \Vert_{L,\varphi}$ is equivalent to $\Vert.\Vert_2$ on $\RLK$. Note that if $\varphi$ is the identity function, then $\bar\varphi^0_{l,k}=1$ for all $(l,k)\in \K^2$ and $\langle\cdot{,}\cdot\rangle_{L,\varphi}= \langle\cdot{,}\cdot\rangle_{L}$. Let $P^{j}_{L,\varphi}$ be the orthogonal projection operator on $\R^K\times \mathcal{\tilde H}(j)^{K^2}$ in terms of inner product $\langle\cdot{,}\cdot\rangle_{L,\varphi}$ and set
\begin{align}
    \tilde f^{0,j} =  P_{L,\varphi}^j(\tilde f^0)\hspace{0.1cm},\hspace{0.2cm}\tilde \psi^0_{L,\varphi} = (\xi^0_L, \bog^0_L./\boldsymbol{\bar\varphi}^0) \hspace{0.3cm}\text{and} \hspace{0.3cm}\tilde \psi^{0,j}_{L,\varphi} = (\xi^{0,j}_{L,\varphi}, \bog^{0,j}_{L,\varphi}) = P^{j}_{L,\varphi}(\tilde \psi^0_{L,\varphi}).\nonumber
\end{align}
With these notations and definitions, we can state our general BvM theorem.
  \begin{theorem}\label{main_theorem}
    Consider a stationary multivariate Hawkes process $N$ observed on $[-A,T]$, with parameters $f^0= (\nu^0,\boh^0)\in ]0,+\infty[^K\times \mathcal{H}$ such that $\boh^0\in \L_\infty^{K^2}$. Let $\Pi$ be a prior distribution on $f$ defined as in (\ref{subsec_prior_construct}) and that verifies  \hyperref[as_P1]{(P1)} and, in case $\varphi$ is not the identity function, \hyperref[as_A']{(A')} and \hyperref[P2]{(P2)}. Assume also that for some $\delta>0$
    \begin{align}\label{condi_rate} \sqrt{T}\varepsilon_T^{2(1-\delta)}\xrightarrow[T\rightarrow +\infty]{}0 \hspace{0.4cm}\text{and}\hspace{0.4cm}\displaystyle \log(T)^7\sqrt{J_T}\varepsilon_T \xrightarrow[T\rightarrow +\infty]{} 0.
    \end{align}
     Let $\psi$ be a functional that verifies (\ref{diff_l2}) with $\bog^0_2\in \L_\infty^{K^2}$, $\sup\{\Vert  \bog^{0,j}_{L,\varphi}\Vert_\infty, \hspace{0.1cm}j\geq 1\}<+\infty$, 
    \begin{align}\label{weak_cond_approx}
        \underset{j\in \mathcal{J}_T}{\max} \Vert \bog^0_L./\boldsymbol{\bar\varphi}^0 - P^j_{2}(\bog^0_L./\boldsymbol{\bar\varphi}^0)\Vert_{2}\rightarrow 0,
    \end{align}
    and such that $r(f,f^0)$ in  \eqref{psi_diff_LAN} verifies
    \begin{align}\label{rem_func}
        \sqrt{T}\underset{f\in \A_T}{\sup}\vert r(f,f^0)\vert \xrightarrow[T\rightarrow+\infty]{}0.
    \end{align}
    Finally, for $u\in \R$ and $f =(\nu, \varphi(\tilde \boh))$, let $f_{u,j}  = \big(\nu - u\xi^{0,j}_{L,\varphi}/\sqrt{T}, \varphi(\tilde\boh - u \bog^{0,j}_{L,\varphi}/\sqrt{T})\big)$. Assume that for all $u$ in a neighborhood of $0$,
    \begin{align}\label{change_var_cond}
         \sum_{j\in \mathcal{J}_T} \frac{\int_{\A_{T}(j) }e^{L_T(f_{u,j})}d\Pi_{f\vert j}(f)}{\int_{\A_{T}(j)}e^{L_T(f)}d\Pi_{f\vert j}(f)} \Pi_J(j\vert N)\xrightarrow[T\rightarrow +\infty]{\P_0} 1.
    \end{align}
    Then, 
    \begin{align}\label{gen_bvm}
    d_{BL}\bigg(\mathcal{L}^\Pi\big(\sqrt{T}(\psi(f) - \hat\psi_T - \mathcal{B}_{J,T})\big\vert N\big) ,  \mathcal{N}\big(0, \Vert \tilde \psi^{0}_{L}\Vert_{L}^2\big)\bigg)\xrightarrow[T\rightarrow+\infty]{\P_0} 0,
    \end{align}
with $\mathcal{B}_{J,T}= -\langle \tilde f^0 - \tilde f^{0,J} , \tilde \psi^{0}_{L,\varphi} -\tilde \psi^{0,J}_{L,\varphi} \rangle_{L,\varphi} + \frac{1}{\sqrt{T}}W_{T}\big(  (\tilde \psi^{0,J}_{L,\varphi}-\tilde \psi^0_{L,\varphi}).\boldsymbol{\bar\varphi}^0\big)$.
 \end{theorem}
 
  This theorem is proved in Section \ref{sec:pr:main_th}. Roughly speaking,  (\ref{gen_bvm}) says that the marginal posterior distribution of $\sqrt{T}\big(\psi(f) - \hat\psi_T \big)$ can be approximated by a mixture of Gaussian distributions  as $T\rightarrow +\infty$:
\begin{align}
    \mathcal{L}^\Pi\Big(\sqrt{T}\big(\psi(f) - \hat\psi_T\big)\Big\vert N\Big) \approx \sum_{j\in \mathcal{J}_T}\mathcal{N}\big(\sqrt{T}\mathcal{B}_{j,T}, \Vert \tilde \psi^{0}_{L}\Vert_{L}^2) \Pi_J(j\vert N).\nonumber
\end{align}

 Let us comment on the conditions of Theorem \ref{main_theorem}. The condition $\sqrt{T}\varepsilon_T^{2(1-\delta)}\rightarrow  0$ is standard in semiparametric statistics (see for instance \cite{castillo_rousseau}).  The second condition in (\ref{condi_rate}) is  more technical and is  used in Lemma \ref{control_rem}, but  essentially boils down to the first part of (\ref{condi_rate}). In Section \ref{app_priors}, we show that in case of $\beta$-Hölder regularity, (\ref{condi_rate}) is equivalent to $\beta >1/2$. 
 
 The functional smoothness condition (\ref{diff_l2}) is standard and (\ref{rem_func}) comes from \cite{castillo_rousseau}. Regarding the condition on $\Vert  \bog^{0,j}_{L,\varphi}\Vert_\infty$, by Lemma \ref{lemma_linf}, since $\bog^0_2\in \L_\infty^{K^2}$,  we have  $\bog^0_L\in \L_\infty^{K^2}$. Nevertheless, this condition is not trivial because $ \bog^{0,j}_{L,\varphi}$ is a projection in terms of LAN norm. We will see in Section \ref{app_priors} that $\sqrt{j}\Vert \bog^0_L./\boldsymbol{\bar\varphi}^0 - P^j_{2}(\bog^0_L./\boldsymbol{\bar\varphi}^0 )\Vert_2 = O(1)$ as $j\rightarrow +\infty$  implies this condition (see Lemma \ref{lemma_LAN_proj_bound}). The latter implies in particular that condition (\ref{weak_cond_approx}) is verified. Besides, condition (\ref{weak_cond_approx}) is weak since the rate of convergence towards $0$ can be arbitrarily slow. As $\bog^0_L./ \boldsymbol{\bar \varphi}^0\in \L_2^{K^2}$,  when $B_j$ is a truncated total family of $\L_2$, this condition is verified as soon as $\min\{j, j\in \J_T\} \rightarrow +\infty$.
 
 Note that in condition (\ref{change_var_cond}),  when $\varphi$ is the identity, contrary wise to Theorem 2.1 of \cite{castillo_rousseau}, the perturbation is not $\boh - u \bog^0_L/\sqrt{T}$  but $\boh - u \bog^{0,j}_{L} /\sqrt{T}$ which makes it much easier to verify. A consequence is that the posterior may be biased since it is centered at $\hat\psi_T + \mathcal{B}_{J,T}$ instead of $\hat\psi_T$. A similar strategy was used in \cite{castillo_rousseau} in  the density model for random histograms priors.
 
 One way to verify (\ref{change_var_cond}) is to perform for each $j\in \mathcal{J}_T$ the change of variable $f\rightarrow f_{u,j}$ at the numerator of the ratio in the $j$-th summand and then to show that $d\Pi_{\tilde \boh\vert j}(\tilde \boh +~ u\bog^{0,j}_{L,\varphi})/d\Pi_{\tilde \boh\vert j} (\tilde \boh) $ goes uniformly to $1$ on $\A_T(j)$. This might not be true if the prior $\Pi_{h_{l,k}}$ puts all its mass on non-negative functions while $h^0_{l,k}$ is not bounded away from $0$. The behavior of the ratios in (\ref{change_var_cond}), and consequently the behavior of  $\mathcal{L}^\Pi\big(\sqrt{T}(\psi(f) - \hat\psi_T)\vert N\big)$, are not clear in this boundary case. This is an interesting problem, both theoretically and for applications since it is not always realistic in practice to assume that the functions $h^0_{l,k}$ are bounded away from $0$. This problem has been addressed in \cite{bvm_truncated} for a broad class of parametric statistical models. They find that when the parameter is on the boundary of the prior support, the limiting distribution is (under some assumptions) a truncated Gaussian. We believe that something similar happens for the conditional distribution of the whole parameter
$f$ at least  in some specific  cases. This is left for future work.

When the $h^0_{l,k}$ are not assumed to be bounded away from $0$, to circumvent this issue we can use a prior which puts some mass on functions taking negative values. This is in particular how we proceed in Section \ref{app_priors} for the random histograms prior.

\begin{remark}
It is possible to replace (\ref{condi_rho}) by $\gamma_1(j)\Vert \tilde h_\theta\Vert_2\lesssim \Vert \theta\Vert_2\lesssim \gamma_2(j)\Vert \tilde h_\theta\Vert_2$ for some positive sequences $\gamma_1$ and $\gamma_2$ such that $\gamma_2/\gamma_1$ has the same properties as $\gamma$ in (\ref{condi_rho}) but then the second condition in (\ref{condi_rate}) becomes $\log(T)^7\sqrt{J_T}\varepsilon_T \max \big(\frac{\gamma_2(J_T)}{\gamma_1(J_T)}, \frac{\gamma_1(J_T)}{ \gamma_2(J_T)}\big)=o(1)$ (see the entropy computations after (\ref{first_chain}) in the proof of Lemma \ref{control_rem}). Similarly, $\sqrt{j}$ can be replaced by some positive sequence $\eta(j)\rightarrow+\infty$ (with at most polynomial growth) in (\ref{condi_sup}), the ideas remain the same but  many conditions have to be adapted.
\end{remark}

\begin{remark}\label{rk_bias} Lemma \ref{lemma_rk}, provided in Section \ref{sec:lemma:bvm}, shows that under the assumption
\begin{align}\label{medium_cond}
    \log(T)^3\underset{j\in \mathcal{J}_T}{\max} \Vert \bog^0_L./\boldsymbol{\bar\varphi}^0 - P^j_{2}(\bog^0_L./\boldsymbol{\bar\varphi}^0)\Vert_{2}\rightarrow 0,
\end{align}
which is slightly stronger than (\ref{weak_cond_approx}), the bias term $\mathcal{B}_{J,T}$ in (\ref{gen_bvm}) can be replaced by
\begin{align}
    \mathcal{B}_J:=-\langle \tilde f^0 - \tilde f^{0,J} , \tilde \psi^{0}_{L,\varphi} -\tilde \psi^{0,J}_{L,\varphi} \rangle_{L,\varphi}.\nonumber
\end{align}
Moreover, $\vert \mathcal{B}_J\vert \lesssim \log(T)\varepsilon_T\Vert \bog^0_{L}. /\boldsymbol{\bar\varphi}^0 - P^{J}_2(\bog^0_{L}. /\boldsymbol{\bar\varphi}^0 )\Vert_2$.
\end{remark}

To have the BvM property, we thus need $\sqrt{T}\mathcal{B}_{j} \xrightarrow[]{} 0$ uniformly in $j\in \mathcal{J}_T$ and so we obtain the following corollary.

\begin{corollary}\label{main_coro} Under the assumptions of Theorem \ref{main_theorem}, if
\begin{align}\label{stronger_cond_bvm}
    \log(T)\sqrt{T}\varepsilon_T\underset{j\in \mathcal{J}_T}{\max}\Vert \bog^0_{L}. /\boldsymbol{\bar\varphi}^0 - P^{j}_2(\bog^0_{L}. /\boldsymbol{\bar\varphi}^0 )\Vert_2\rightarrow 0, 
\end{align}
then the posterior distribution has the BvM property (\ref{def_bvm}).
\end{corollary}
 Under the assumptions of Corollary \ref{main_coro}, when $\varepsilon_T$ is the optimal $L_2$ rate (up to a logarithmic factor), the Bayesian procedure has the \textit{plug-in} property as in \cite{bickel_ritov}: the posterior distribution concentrates "optimally" on $f^0$ and  simultaneously estimates a large class of smooth functionals efficiently at rate $1/\sqrt{T}$. In particular, when \eqref{stronger_cond_bvm} holds for the functionals $\psi(f) = \rho_{l,k}$ $(l,k\in \K)$, then one can devise a thresholding approach for estimating the graph of interactions, i.e. $\delta_{l,k} = \mathds{1}_{\rho_{l,k}>0}$ in the form 
  $ \hat \delta_{l,k}  = 0$ if and only if $0 \in  C_{\alpha}^\pi$ the $1-\alpha$ equal tail credible interval of $\rho_{l,k}$, with $\alpha=\alpha_T =o(1)$ slowly. By construction and by Corollary  \ref{main_coro}, if $\rho_{l,k}^0=0$, then $\P_0 (\hat \delta_{l,k}= \delta_{l,k}^0 )= 1- o(1) $.  Interestingly,  also as soon as $\rho_{l,k}^0\gg 1/\sqrt{T}$, $\P_0( \hat \delta_{l,k}\neq \delta_{l,k}^0) =o(1)$. The detection rate is therefore much better than the ones obtained for instance in \cite{Sulem_concentration, rousseau_highdim, Sulem_variational}.

Condition (\ref{stronger_cond_bvm}) is the key condition for the BvM to hold in our framework. It requires that the least favourable direction can be \textit{well} approximated by functions $\bth^T B_j$ for $j\in \mathcal{J}_T$. To verify (\ref{stronger_cond_bvm}), our strategy consists in showing that if $\boh^0$ and $\bog^0_2$ belong to some functional class then so does $\bog^0_L$  and that for such class (\ref{stronger_cond_bvm}) is verified 
so that BvM holds. This is not trivial because $\bog^0_L$ is not explicit (see Section \ref{sec_lfeq}). Nevertheless, we show in Corollary~\ref{g0_palm_holder} that if the functions $\boh^0$ and $\bog^2_0$ are $\beta$-Hölder regular (in the sense specified at the end of Section \ref{not_assump}),  then the functions $\bog^0_L$ are also (almost) $\beta$-Hölder regular. Then under reasonable   assumptions on the prior, $\varepsilon_T \asymp T^{-\beta /(2 \beta+ 1)} (\log T)^\alpha$ for some $\alpha \geq 0$  and to ensure that (\ref{stronger_cond_bvm}) holds and  thus that the BvM property is verified, it is enough to have  $\beta >1/2$ as well as for some $\alpha '\in \mathbb R$,
$\Pi ( J \gtrsim T^{1/(2\beta+1)}(\log T)^{\alpha'} | N ) = o_{\P_0}(1)$. In particular we can either choose a priori $J$ deterministically, i.e. as the Dirac mass 
 $ \Pi_J = \delta_{(T^{1/(2\beta+1)}J_1)}$, which leads to a non adaptive posterior distribution or assume that the functions $h_{l,k}^0$ verify the polished tail condition as in \cite{rousseau_szabo_confidence}.

In \cite{castillo_rousseau}, in  the density model and for specific functionals the authors prove that 
\begin{align}\label{dream_ineq}
    \Vert \bog^0_L -  P_2^j(\bog^0_L)\Vert_2\lesssim \Vert \boh^0 -  P_2^j(\boh^0)\Vert_2\hspace{0.1cm},\hspace{0.5cm} j\geq j_0,
\end{align}
which implies that (\ref{stronger_cond_bvm}) holds as soon as $\log(T)^2\sqrt{T}\varepsilon_T^2=o(1)$.  This is particularly interesting since it implies that BvM holds for hierarchical priors which lead to adaptive posterior contraction rates on $f$. For other functionals, on the contrary (\ref{dream_ineq}) does not hold and the posterior may not verify BvM due to a non-vanishing bias term (see in particular Proposition 2 in \cite{castillo_rousseau}).

Here the situation is more complex since $\bog^0_L$ does not have an explicit expression and it is unclear for which functionals (\ref{dream_ineq}) holds. Moreover, it is also much more complicated to construct a counterexample, similarly to Proposition 2 in \cite{castillo_rousseau},  of a functional for which under an adaptive prior the bias term does not vanish since again the least favourable direction is not explicit. Without lower bound on the bias term, we do not know whether Condition (\ref{stronger_cond_bvm}) is sharp. 
\begin{remark}\label{rk_derminist}
 Theorem \ref{main_theorem} and Corollary \ref{main_coro} are stated for adaptive priors where $J\sim~ \Pi_J$ that satisfies (\ref{cond_piJ}), but they also hold for non-adaptive prior where $J= J_T= J_T(\beta)$ is deterministic and thus we can choose $\mathcal{J}_T= \{J_T\}$. In particular, Condition (\ref{stronger_cond_bvm}) becomes $\log(T)\sqrt{T}\varepsilon_T \Vert \bog^0_{L}. /\boldsymbol{\bar\varphi}^0 - P^{J_T}_2(\bog^0_{L}. /\boldsymbol{\bar\varphi}^0 )\Vert_2\to 0 $.
\end{remark}

As asked by a referee, it is of interest to understand if the Bernstein von Mises property leads to efficiency of the posterior mean $\bar \psi_T$, for a given a functional $\psi$:
\begin{align*}
    \bar \psi_T := \frac{\int \psi(f)e^{L_T(f)}d\Pi(f)}{\int e^{L_T(f)}d\Pi(f)}.
\end{align*}
Showing the efficiency of $\bar \psi_T$ amounts to showing both the weak convergence (\ref{def_bvm}) as well as convergence of the first moment, the latter often being established by a uniform integrability argument. This has already been done in \cite{nickl_IV_book} (Theorem 4.1.10) in the context of PDE inverse problems but the arguments of the proof are general and it gives the following corollary.
\begin{corollary}\label{coro_posterior_mean} Under the assumptions of Theorem \ref{main_theorem} and condition (\ref{stronger_cond_bvm}), if Assumptions  \hyperref[as_P1]{(P1)}, \hyperref[P2]{(P2)} and \hyperref[as_A']{(A')} can be verified by replacing $c_4$ with $2c_4$ in the definition of $J_0$ and in the exponential bounds, and if there exists $\epsilon>0$ such that
\begin{align}
    T^{1+\epsilon}e^{-T\varepsilon_T^2}\int \big\vert \psi(f)-\psi(f^0)\big\vert^{2+2\epsilon}d\Pi(f) \xrightarrow[T\rightarrow +\infty]{}0,\label{cond_func_b}
\end{align}
then the posterior mean $\bar \psi_T$ is an efficient estimator of $\psi(f^0)$.
\end{corollary}
For the sake of completeness, in Section \ref{sec:lemma:bvm}, we give a short sketch of the proof and briefly justify why the arguments of \cite{nickl_IV_book} can be applied in our model. Verifying \hyperref[as_P1]{(P1)}, \hyperref[P2]{(P2)} and \hyperref[as_A']{(A')} with $2c_4$ instead of $c_4$ should be immediate. Condition (\ref{cond_func_b}) is a mild condition that is satisfied for instance as soon as $E_\Pi[ \vert \psi(f)\vert^{2+2\epsilon}]<+\infty$ and is easy to verify. For  example light enough tails on the prior $\Pi_\nu$ on $\nu_k$  (see  Section \ref{subsec_prior_construct}) leads to \eqref{cond_func_b}, or functionals like $\int h_{l,k}$ which are bounded under the prior also trivially verify \eqref{cond_func_b}. As with the construction of confidence intervals (mentioned just after (\ref{def_bvm})), it is not clear how to construct efficient estimators, because the usual methods (e.g. the efficient score equation, see again Chapter 25 of \cite{VDV_AS}) require an explicit efficient influence function. 

We apply in the following section our results to some specific random series priors.

\subsection{Application on specific prior models}\label{app_priors}

We present two types of random series priors, one constructed on histogram bases and the other on truncated wavelet bases, that satisfy the assumptions of the BvM results of Section \ref{BVM}. Recall the standard notations for Hölder and Besov spaces at the end of Section \ref{not_assump}. Throughout this section,  to simplify the presentation we consider that the following assumption on the true parameters $f^0$ and on the functional holds. 

\begin{assumption}[S]\label{as_S} 
The true parameters $f^0\in ]0,+\infty[^K\times \mathcal{H}$ are such that the functions $\boh^0$ are in $C^\beta([0,A])$ for some $\beta >1/2$. Moreover, the functional $\psi$ verifies (\ref{diff_l2}), can be expanded as in (\ref{psi_diff_LAN}) with $r(f,f^0) = O(\Vert f-f^0\Vert_2^2)$ and is such that the functions $\bog^0_{2}$ are also in $C^\beta([0,A])$. 
\end{assumption}

The condition on the remainder of the functional is slightly more demanding than what is required in Theorem 3.2. For instance, the functionals  $\psi^b_k$, $\psi^2_{l,k}$ and $\psi^a_{l,k}$ with $a\in C^\beta([0,A])$  (all three defined in (\ref{typical_functionals}))  verify the functional conditions in  \hyperref[as_S]{(S)}. By Corollary \ref{g0_palm_holder}, when $\beta\in~]1/2,1]$, the functions $\bog^0_L$ are in $\mathcal{B}^\beta_{\infty,\infty}([0,A])$ under Assumption \hyperref[as_S]{(S)}. When $\beta>1$, assuming the Conjecture \ref{conj_beta1} and under Assumption \hyperref[as_S]{(S)},  the functions $\bog^0_L$ are in $\C^{\beta'}([0,A])$ for any $\beta'<\beta$ by Corollary \ref{g0_holder_beta1}.

For ease of presentation, we specify the prior distribution of $\bth\vert J$: $\Pi_{\bth\vert J} = \otimes_{(l,k)\in \K^2}\Pi_{\theta\vert J}$ and $\Pi_{\theta\vert J} = \otimes_{1\leq j\leq J}  \bar \Pi_{\theta,j}$  for some distributions $(\bar \Pi_{\theta,j})_j$ that have densities $(\pi_{\theta,j})_j$ with respect to Lebesgue measure. Similarly, we take $\Pi_\nu = \otimes_{1\leq k\leq K}  \bar \Pi_\nu$ with $\bar \Pi_\nu$ a probability distribution on $]0,+\infty[$ with a density denoted $\pi_\nu$ assumed to be positive and continuously differentiable on $]0,+\infty[$ and such that $\bar \Pi_\nu([x;+\infty[)\lesssim x^{-a}$ for $x$ large enough.

\subsubsection{Random histogram priors}
We assume that $\beta \in ]1/2,1]$. We take for $B_j$ the histogram basis associated to the regular partition of $[0,A]$ into $j$ bins:
\begin{align*}
    B_i^j (x)= \onee\Big\{x\in \Big[\frac{A(i-1)}{j},\frac{Ai}{j}\Big[\hspace{0.05cm}\Big\}\hspace{0.1cm}, \hspace{0.1cm}1\leq i\leq j-1\hspace{0.1cm}; \hspace{0.1cm} B_j^j (x)= \onee\Big\{x\in \Big[\frac{A(j-1)}{j},A\Big]\Big\}. 
\end{align*}
This is a particular case of B-splines basis (as defined in Appendix E of \cite{bnp_book}) when $q=~1$. $(B_j)_j$ verifies the conditions of Section \ref{subsec_prior_construct}  with $\gamma(j)= \sqrt{j}$. We then choose $\varphi(x)=x$ and for all $j\geq 1$, $\pi_{\theta,j}= \pi_{\theta} $  where $\pi_\theta$ is positive and continuously differentiable on $\R$. For simplicity, we also consider that $\pi_\theta$ has sub-exponential tails. Note that to encode directly the constraint that $\VERT \rho\VERT_1\leq H$ (in the nonnegative case), one can also consider prior densities proportional to $\pi_\theta(x)\mathds{1}_{x\leq C} $ with $C$ chosen arbitrarily large. Moreover, 
in the ReLU case, the priors on the functions $h_{l,k}$  can depend on $\nu_k$ through the constraint  that their support is included in $(-\nu_k, \infty)$, which corresponds to choosing for $\theta_{l,k}^j $  a prior density proportional to $\pi_{\theta}(x)\mathds{1}_{-\nu_k<x\leq C}$.  We verify the assumptions of Theorem \ref{main_theorem}. First, (\ref{sieve_theta}) is verified with $U(M)\asymp M$ and Assumption \hyperref[as_P1]{(P1)} is verified with in particular
\begin{align}\label{eps_T_beta}
    \bar \varepsilon_T(\beta) \asymp T^{\frac{-\beta}{2\beta+1}}\log(T)^{\frac{\beta}{2\beta+1}} \hspace{0.1cm}, \hspace{0.3cm}\bar J_T(\beta)\asymp T^{\frac{1}{2\beta+1}} \log(T)^{\frac{3\beta+1}{2\beta+1}} ,
\end{align}
see for instance Lemma 4 in \cite{castillo_rousseau_add} (to verify (\ref{approx_prior_B}), we can approximate in sup-norm the functions $h^0_{l,k}+\varepsilon_T(\beta)/2$ at rate $\varepsilon_T(\beta)/4$ with $c\bar J_T(\beta)$ bins and $c>0$ chosen correctly, see Remark \ref{rk_approx}). When $\beta>1/2$, (\ref{condi_rate}) is verified and, under Assumption \hyperref[as_S]{(S)}, (\ref{rem_func}) is also verified. Moreover, since the functions $\bog^0_L$ are in $\mathcal{B}^\beta_{\infty,\infty}([0,A])$, we have $\Vert \bog^0_L - P_2^j(\bog^0_L )\Vert_2\lesssim j^{-\beta'}$ with $\beta'=\beta$ when $\beta<1$ and $\beta'<1$ when $\beta=1$. Thus, by Lemma \ref{lemma_LAN_proj_bound}, the condition $\sup\{\Vert \bog^{0,j}_{L,\varphi}\Vert_\infty, j\geq 1\}<+\infty$ is verified and, letting $j_T :=\min\{j, j\in \mathcal{J}_T\}$,  (\ref{weak_cond_approx}) is verified when $j_T\rightarrow+\infty$, whatever the rate.  By Lemma \ref{change_var_lemma}, if the functions $\boh^0$ are bounded away from $\kappa$,  then (\ref{change_var_cond}) is verified. Finally, Condition (\ref{stronger_cond_bvm}) of Corollary \ref{main_coro} holds as soon as 
\begin{align}\label{conjT}
    \log(T)\sqrt{T}\varepsilon_T(\beta) j_T^{-\beta'}\rightarrow 0,
\end{align}
We thus have proved the following corollary.

\begin{corollary}\label{histo_coro}
Let  $f^0$ and $\psi$ that verify Assumption \hyperref[as_S]{(S)} with $\beta \in ]1/2,1]$. Let $\Pi$ be a prior constructed as above on the regular histogram bases with $\kappa$ such that the functions $\boh^0$ are bounded away from $\kappa$. If $j_T\rightarrow +\infty$, then (\ref{gen_bvm}) holds. If in addition Condition (\ref{conjT}) is satisfied, then the posterior distribution has the BvM property (\ref{def_bvm}).
\end{corollary}

 Examples of densities $\pi_\theta$ are the shifted exponential, truncated Laplace and truncated Gaussian densities. In particular, if the functions $\boh^0$ are bounded away from $0$, one can take $\kappa=0$ and the standard exponential distribution is a possible choice. Otherwise, one has to take $\kappa<0$ and the prior puts some mass on functions taking negative values, in other words it is a prior in the nonlinear ReLU model. A similar result holds for B-splines for any $\beta>1/2$.

\subsubsection{Wavelet priors} Let $\beta> 1/2$. Let $\big\{\Phi_{i,v} , i\in\{-1,0\}\cup\mathbb{N}, v\in \{0,...,\bar v_i\}\big\}$ be a boundary-corrected Daubechies wavelet basis on $[0,A]$ of regularity $S\in \mathbb{N}$ with $S>\beta$ (see Section 4.3.5 of \cite{gine_nickl_book}). For all $i\geq -1$,  we have $0\leq \bar v_i  \leq C_v A2^{i}$ for some $C_v>0$. Let $ \bar B_I:= \big\{\Phi_{i,v} , i\in\{-1,0, 1, ..., I\}, v\in \{0,...,\bar v_i \}\big\}$ (the truncated basis at $I$). To be in line with our conventions of Section \ref{subsec_prior_construct}, we note that we can rewrite $\bar B_I$ as a family $B_J=(b_1^J,...,b_J^J)$ with $J=c(I):=\sum_{i=-1}^{I}\bar v_i$. However, we  work with the resolution level $I$ instead of $J$ as it is commonly done with wavelets. We thus put a prior on $I$ which induces a prior on $J$. The sequences $(\bar I_T(\beta))_T$ and $(I_T(\beta))_T$ and the set $\mathcal{I}_T$ are then defined analogously to $(\bar J_T(\beta))_T$, $(J_T(\beta))_T$ and $\mathcal{J}_T$ respectively. To simplify the computations, given $\bth = (\bth_{l,k}, (l,k)\in \K^2)$ with for each $(l,k)\in \K^2$, $\bth_{l,k} = (\theta_{l,k}^i , i\in \{-1,0,..., I\})$ and for each  $i\geq -1$, $\theta^i_{l,k} = (\theta^{i,v}_{l,k}, v=0,...,\bar v_i) \in \R^{\bar v_i}$ (and thus $\bth_{l,k} \in \R^{c(I)}$ and $\bth \in \R^{K^2c(I)}$), we set
\begin{align}\label{not_wav}
    \bth^{i,T}_{l,k}\bar B_I^i := \sum_{v=0}^{\bar v_i}\theta_{l,k}^{i,v} \psi_{i,v}\hspace{0.1cm} \bth_{l,k} ^T\bar B_I = \sum_{i=-1}^I\bth^{i,T}_{l,k}\bar B_I^i \hspace{0.2cm}\text{and}\hspace{0.2cm} \bth^T\bar B_I = (\bth_{l,k}^T\bar B_I, (l,k)\in \K^2).
\end{align}
 These truncated bases verify the conditions of Section \ref{subsec_prior_construct}  with $\gamma= 1$. Contrary to histograms, there is no obvious choice of the wavelet coefficients that ensures the non negativeness of the functions. We can either consider the ReLU nonlinear Hawkes model or choose a nonlinear $\varphi\geq 0$. We study both. 

In the ReLU nonlinear model, we thus choose $\varphi(x)=x$. For the prior on the wavelets coefficients, we take for all $i\geq -1$ and $v\in \{0,...,\bar v_i\}$, $\pi_{\theta,i,v} = \pi_\theta$ where $\pi_\theta$ is a density  which is  positive and continuously differentiable on $\R$ with, for simplicity,  sub-exponential tails (for example a Laplace or a Gaussian density). Then, similarly to what we did for random histograms, one can check, assuming Conjecture \ref{conj_beta1} (we also briefly explain this Conjecture at the end of Section \ref{sec_lfeq}) when $\beta>1$ and using standard approximation property of (boundary-corrected) wavelets (see again Section 4.3.5 of \cite{gine_nickl_book}), that \hyperref[as_P1]{(P1)},  (\ref{condi_rate}) and (\ref{rem_func}) are verified with $\bar \varepsilon_T(\beta)$ as in (\ref{eps_T_beta}) and $\bar I_T(\beta)$ such that $2^{\bar I_T(\beta)}\asymp~T^{\frac{1}{2\beta+1}} \log(T)^{\frac{3\beta+1}{2\beta+1}}$. Moreover, Lemma \ref{lemma_LAN_proj_bound} allows again to verify  that $\sup\{\Vert \bog^{0,i}_{L,\varphi}\Vert_\infty, i\geq -1\}<+\infty$,   (\ref{change_var_cond}) is verified by Lemma \ref{change_var_lemma} and, letting $i_T :=\min\{i, i\in \mathcal{I}_T\}$,  (\ref{weak_cond_approx}) is verified when $i_T\rightarrow~+\infty$.  We obtain the following corollary.

\begin{corollary}\label{wavelets_coro}
Let  $f^0$ and $\psi$ that verify Assumption \hyperref[as_S]{(S)} with $\beta >1/2$. If  $\beta>1$, assume in addition the Conjecture \ref{conj_beta1}. Let $\Pi$ be a prior constructed as above on the truncated wavelet bases in the ReLU model. If $i_T\rightarrow +\infty$, then (\ref{gen_bvm}) holds. If in addition we have $\log(T)\sqrt{T}\varepsilon_T(\beta) 2^{-i_T\beta'}\rightarrow 0$ for some $\beta'<\beta$, then the posterior distribution has the BvM property (\ref{def_bvm}). 
\end{corollary}

Now, we present  wavelets priors with nonlinear $\varphi\geq 0$ and in order to verify Assumption \hyperref[as_A']{(A')}, we assume that the functions $\boh^0$ are bounded away from $0$. We also assume that $\beta>\beta_m$ for some $\beta_m\geq 1/2$ defined in Corollary \ref{wavelets_coro_varphi}. We choose $\varphi\geq 0$, globally non-decreasing and Lipschitz as well as  infinitely differentiable with a positive first derivative on $\varphi^{-1}(I_0(\epsilon))$ for some $\epsilon >0$ that can be taken arbitrarily small. For example possible choices are $\varphi(x)=x_+$ or $\varphi(x)=\log(1+e^x)$ (the \textit{softplus} function). Then, recall that since the functions $\tilde \boh^0$ are in $\mathcal{B}^\beta_{\infty,\infty}([0,A])$ with $\beta >\beta_m$, for $i\geq 0$ the order of magnitude of their wavelets coefficients is at most $2^{-i( \beta_m +1/2)}$. For the prior on the wavelets coefficients, we consider two cases. In case (i),  we choose $\bar \Pi_{-1,r}=\mathcal{N}(0, 1)$ and for $i\geq 0$, $\bar \Pi_{i,r}=\mathcal{N}(0, 2^{-i(2 \beta_m +1)})$. In case (ii),  we take $\bar \Pi_{-1,r}=\mathcal{U}([-C_0, C_0])$ (with $C_0>0$ chosen large enough and that depends only on $\boh^0$, $\varphi$ and on the Daubechies wavelet basis) and for $i\geq 0$, $\bar \Pi_{i,r}=\mathcal{U}([-\log(T)^2 2^{-i/2}, \log(T)^2 2^{-i/2}])$   Then, as for the wavelet prior in the nonlinear ReLU model, we verify easily with the same $\bar \varepsilon_T(\beta)$ and $\bar I_T(\beta)$ the conditions required in Theorem \ref{main_theorem} when $\varphi$ is the identity function. Moreover, we have to verify here Assumptions \hyperref[P2]{(P2)} and \hyperref[as_A']{(A')} since $\varphi$ is not the identity function; we do it in Lemma \ref{verif_P2_A_wavelets}.  We thus have the following corollary.

\begin{corollary}\label{wavelets_coro_varphi}
Let  $f^0$ and $\psi$ that verify Assumption \hyperref[as_S]{(S}) with $\beta>\beta_m$ for some $\beta_m\geq~ 1/2$. If $\beta>1$, assume in addition the Conjecture \ref{conj_beta1}. Assume also that the functions $\boh^0$ are bounded away from $0$. Let $\Pi$ be a prior constructed as above on the truncated wavelet bases with $\varphi$ that is infinitely differentiable and with a positive first derivative on $\varphi^{-1}(I_0(\epsilon))$ for some $\epsilon >0$ that can be taken arbitrarily small. Take $\beta_m= 5/6$ in case (i) and $\beta_m= 1/2$ in case (ii). If $i_T\rightarrow +\infty$, then (\ref{gen_bvm}) holds. If in addition we have $\log(T)\sqrt{T}\varepsilon_T(\beta)2^{-i_T \beta'}\rightarrow 0$ for some $\beta'<\beta$, then the posterior distribution has the BvM property (\ref{def_bvm}).
\end{corollary}

In case (i), the condition $\beta>\beta_m=5/6$ comes from the verification of  \hyperref[P2]{(P2)} and \hyperref[as_A']{(A')} (see Lemma \ref{verif_P2_A_wavelets}). In case (ii), it would be natural to take for i $\geq 0$, $\bar \Pi_{i,r}=\mathcal{U}([-C_02^{-i}, C_02^{-i}])$ to "match" the decrease of the wavelet coefficients (as we did in case (i) with the Gaussian distributions), however with these uniform distributions our proof of Lemma \ref{change_var_lemma} fails and it is not clear whether the change of variable condition (\ref{change_var_cond}) is verified. 
\begin{remark}
In Assumption \href{as_S}{(S)}, we could have assumed more generally that the functions $\bog_2^0$ belong to $C^\alpha([0,A])$ for some $\alpha > 1/2$ (possibly with $\alpha \neq \beta$). The results of Section~\ref{sec_lfeq} then show, in the same way, that the functions $\bog_L^0$ are (almost) in $C^{\min(\alpha,\beta)}([0,A])$ (assuming Conjecture~\ref{conj_beta1} if $\min(\alpha,\beta) > 1$ and if regularity greater than one is needed for $\bog_L^0$). The corollaries of this section follow similarly, with the conditions on $j_T$ or $i_T$ adapted accordingly. For histograms, Condition~(\ref{conjT}) becomes $\log(T)\sqrt{T}\varepsilon_T(\beta)j_T^{-\min(\alpha,\beta)'} \to 0$, and for wavelets it becomes $\log(T)\sqrt{T}\varepsilon_T(\beta)2^{-i_T\min(\alpha,\beta)'} \to 0$. In particular, in the non-adaptive case of Remark~\ref{rk_derminist}, where $J = \bar J_T(\beta)$ is deterministic (and similarly for $I$ in the wavelet case), these conditions are satisfied whenever $\alpha > 1/2$.
\end{remark}

\section{Regularity of the least favourable direction}\label{sec_lfeq}
In this section, we study the regularity of the least favourable direction $(\xi^0_L,\bog^0_L)$ which is only known through the linear equation (\ref{eq_least_fav}): $\Gamma(\xi^0_L, \bog^0_L) = (\xi_2^0, \bog_2^0)$ where $\Gamma = \mathcal{S}^*\mathcal{S}$ is the \textit{information operator}. We use the notations $C^\beta$ and $\mathcal{B}^\beta_{\infty,\infty}$, recalled at start of Section \ref{app_priors}, for Hölder and Besov (or Hölder-Zygmund) spaces respectively. We set $C^\beta_b(\R):= C^\beta(\R)\cap L_\infty(\R)$ and $\mathcal{B}^\beta_{\infty,\infty,b}(\R):= \mathcal{B}^\beta_{\infty,\infty}\cap L_\infty(\R)$.  

First, by Lemma \ref{eq} and with standard properties of the adjoint operator,  we know that the information operator $\Gamma$ is a bounded, bijective, linear operator of $\R^K\times \L_2^{K^2}$. By Banach-Schauder theorem, its inverse is also bounded. We recall also that Lemma \ref{explicit_gamma} gives an explicit expression for the operator $\Gamma$ involving the Palm distribution of the process and  it implies that if $(\xi,\bog)=\Gamma^{-1}(\xi',\bog')$ for some $(\xi',\bog') \in \R^K\times L_2^{K^2}$, then $(\xi,\bog)$ verifies the  fixed point equation (\ref{converse_expr}). Starting from this, we study  $\Gamma$ as an operator on $\R^K \times L_\infty^{K^2}= \R^K \times L_\infty([0,A])^{K^2}$ and for technical reasons we extend it on $\R^K \times L_\infty(\R)^{K^2}$, we denote by $\Gamma^e$ its extension.  Next, using these results, we show that when the functions $\boh^0$ are in $C^\beta([0,A])$ for some $\beta \in ]0,1]$, the operator $(\Gamma^e)^{-1}$ maps $\R^K \times C^\beta([0,A])^{K^2}$ to $\R^K\times \mathcal{B}^\beta_{\infty,\infty}([0,A])^{K^2}$ (Lemma~\ref{final_lemma_operator}). Hence, it shows that when the interaction functions $\boh^0$ and the $L_2$ Riesz representor $\bog^0_2$ are in $C^\beta([0,A])$ for some $\beta \in ]0,1]$,  the functions $\bog^0_L$ are in  $\mathcal{B}^\beta_{\infty,\infty}([0,A])$. An extension of this result for $\beta>1$ is presented in Section \ref{sec:pr:beta>1} under a certain regularity assumption on the second order Palm distribution, see Conjecture (\ref{conj_beta1}). We shortly explain this conjecture at the end of this section. Palm calculus is briefly presented in the Appendix where we state the results that we use,  giving each time references for the proofs and further details.

To extend $\Gamma$ on $\R^K \times L_\infty(\R)^{K^2}$, we first extend on $\R$ the functions $\boh^0$, the functions $(p_{l,k}, (l,k)\in \K^2)$ and the operators $(\zeta_{l,j,k}, (l,j,k)\in \K^3)$ (the last two being defined by (\ref{def_plk}) and (\ref{def_zeta}) respectively). We consider bounded extensions $\boh^{0,e}$ on $\R$ of the functions $\boh^0$ such that for each $(l,k)\in \K^2$, $h^{0,e}_{l,k}\geq -\nu_k^0/2$. For $t,u\in \R$, let
\begin{align}\label{def_lambda_e}
    &\lambda_t^{k,e}(f^0_k,(u,l)):= \lambda_t^k(f^0_k) + h^{0,e}_{l,k}(t-u)\onee_{u\notin [t-A,t[} \geq \frac{\nu^0_k}{2}.
\end{align}
Then, recalling that $\E_0^{(u,l)}$ is an expectation under the Palm distribution $\P^{(u,l)}_0$ of $N$ (see the Appendix), we define the function $p_{l,k}^e: \R\rightarrow \R$ by 
\begin{align}
    p^e_{l,k}(A-u):=\E_0^{(u,l)}\Big[\frac{1}{ \lambda_A^{k,e}(f^0_k,(u,l))}\Big],\label{def_plk_e}
\end{align}
and for $x\in [0,A]$, $p_{l,k}(x) =p^{e}_{l,k}(x)$. The operator $\zeta^e_{l,j,k}: \L_2(\R)\rightarrow \L_2(\R)$ is defined by
\begin{align}
    \zeta_{l,j,k}^e (g)(A-u) := \displaystyle \E^{(u,l)}_0\bigg[ \int_0^A \onee_{(s,j)\neq (u,l) }g(A-s) \frac{dN_s^j}{ \lambda_A^{k,e}(f^0_k,(u,l))}\bigg],\label{def_zeta_e}
\end{align}
and the restriction of $\zeta_{l,j,k}^e$ to $L_2= L_2([0,A])$ (so a linear operator of $L_2$ which is well defined because for $g\in \L_2(\R) $, $\zeta_{l,j,k}^e(g)$ depends only on $g(x)$ for $x\in[0,A]$) is equal to the operator $\zeta_{l,j,k}$. With these extensions and notations, we can now state the following lemma.

\begin{lemma}\label{lemma_linf}
The information operator $\Gamma$ is a bounded, bijective linear operator of $\R^K\times~ \L_\infty^{K^2}=\R^K\times \L_\infty([0,A])^{K^2}$ and by Banach-Schauder theorem, its inverse is also a bounded operator on $\R^K\times \L_\infty^{K^2}$. 

Moreover, let $\Gamma^e:\R^K \times L^{K^2}_\infty(\R)\rightarrow \R^K\times L^{K^2}_\infty(\R)$ be defined by 
\begin{align}
    \Gamma^e (\xi, \bog) &=\begin{bmatrix}
           \xi_{k}\E_0\Big[\frac{1}{\lambda_A^k(f^0_k)}\Big]+ \E_0\Big[\frac{\tilde \lambda_A^k(0,\bog_{k})}{\lambda_A^k(f^0_k)}\Big], \hspace{0.15cm}k\in \K \\
           \mu_l^0p^e_{l,k}g_{l,k} + \mu_l^0p^e_{l,k}\xi_k + \mu_l^0 \sum_{j=1}^K \zeta^e_{l,j,k}(g_{j,k}), \hspace{0.15cm}(l,k)\in \K^2
         \end{bmatrix}.\nonumber
\end{align}
$\Gamma^e$ is  equal to $\Gamma$ on $\R^K\times \L_\infty^{K^2}([0,A])$ in the sense that $\Gamma^e(\xi,\bog)_{\vert [0,A]} = \Gamma(\xi,\bog_{\vert [0,A]})$. In addition, $\Gamma^e$ is a bounded, bijective linear operator of $\R^K\times \L_\infty^{K^2}(\R)$ with a bounded inverse.
\end{lemma}
Lemma \ref{lemma_linf} is proved in Section \ref{sec:proof:gamma_inf}.  When $\bog^0_2\in ~\L_\infty^{K^2}$, this lemma shows in particular that $\bog^0_L\in \L_\infty^{K^2}$ too. Note that $(\xi,\bog)= (\Gamma^e)^{-1}(\xi', \bog')$ for bounded functions $\bog'$ on $\R$ can be written similarly to (\ref{converse_expr}) as a fixed point equation on $\R$.

Now, with these results, we are ready to study the smoothness of $\bog^0_L$. We assume that the functions $\boh^0$ are in $C^\beta([0,A])$ for some $\beta\in ]0;1]$ and we extend these functions on $\R$ by functions $\boh^{0,e}$ that are in $C^\beta_b(\R)$ (this is possible by Whitney extension theorem). Moreover, we choose the functions $h^{0,e}_{l,k}$ such that again $h^{0,e}_{l,k}\geq -\nu^0_k/2$. We want to show that $(\Gamma^e)^{-1}(\R^K\times C^\beta_b(\R)^{K^2})\subseteq \R^K\times \mathcal{B}_{\infty, \infty,b}^\beta(\R)^{K^2}$ which would imply the announced result: $\Gamma^{-1}\big(\R^K \times C^\beta([0,A])^{K^2}\big)\subseteq\R^K\times \mathcal{B}^\beta_{\infty,\infty}([0,A])^{K^2}$. First, Lemma \ref{lemma_lb_p} and Lemma \ref{smooth_palm} show that the functions $(p^e_{l,k}, (l,k)\in \K^2)$ are in $C^\beta_b(\R)$ and bounded away from $0$, the proofs are based on some basic properties of the first order Palm distribution of the process recalled in Appendix. Then, let $ (\xi,\bog)= (\Gamma^e)^{-1}(\xi', \bog')$ with $\bog'\in C^\beta_b(\R)^{K^2}$, given the fixed point equation verified by $\bog$, if we could show that $\zeta^e_{l,j,k}(g_{j,k})\in C^\beta(\R)$ whatever $(l,j,k)\in \K^3$, it would prove the desired result. The difficulty here is that we do not have an explicit expression for $\bog$ in terms of $(\xi',\bog')$ and also that what lies behind the operators $\zeta^e_{l,j,k}$ is the second order Palm distribution of the Hawkes process (see (\ref{explicit_zeta})) which is much harder to study than the first order Palm distribution (see the Appendix for an explanation on this terminology). To circumvent this difficulty, our strategy consists in using Littlewood–Paley characterization of Besov spaces on $\mathbb{R}$ (see for example Proposition 2.2 of \cite{ker_picard}). More specifically,  we prove that the functions in $(\Gamma^e)^{-1}(\R^K\times C^\beta_b(\R)^{K^2})$ can be approximated at a "$\beta$-Hölder" rate in sup-norm with a convolution kernel (satisfying some properties given just below) we deduce the result with Proposition 2.3 of \cite{ker_picard} (this is explained in more details in the proof of the following lemma).

For the just announced proof strategy, the convolution kernel $\mathcal{K}: \R \rightarrow \R$ has to be a bounded function, supported on $[-1,1]$, symmetric about $0$, with a bounded (weak) derivative on $\R$ and such that  $\int \mathcal{K}(x)dx=1$. For $n\geq 1$, let $\mathcal{K}_n(x) = n\mathcal{K}(nx)$  and for $g:\R\rightarrow \R$, $\mathcal{K}_n g (x)= \int \mathcal{K}_n(y)g(x-y) dy$. For $(\xi,\bog)\in \R^K\times \L_\infty^{K^2}(\R)$, $\mathcal{K}_n\bog:= (\mathcal{K}_ng_{l,k}, (l,k)\in \K^2)$.

\begin{lemma}\label{final_lemma_operator}
 Assume that the functions $\boh^0$ are in $C^\beta([0,A])$ for some $\beta \in ]0,1]$, and let $(\xi,\bog)= (\Gamma^e)^{-1}(\xi', \bog')$ for some functions $\bog'$ that are in $C^\beta_b(\R)$. Then, let $\mathcal{K}$ be a kernel as above, we have $ \Vert \bog -\mathcal{K}_n\bog \Vert_\infty \leq C(\bog) n^{-\beta}$ for some constant $C(\bog)>0$ that depends only on $\bog$.  It implies that the functions $\bog$ belong to the Besov space $\mathcal{B}^\beta_{\infty,\infty}(\R)$ and thus
 \begin{align*}
     (\Gamma^e)^{-1}\big(\R^K\times C_b^\beta(\R)^{K^2}\big)\subseteq \R^K\times \mathcal{B}_{\infty, \infty,b}^\beta(\R)^{K^2}.
 \end{align*}
\end{lemma}
This lemma is proved in Section \ref{sec:pr:main_operator}. Since any function $\bog$ in $C^{\beta}([0,A])$ can be extended on $\R$ by functions  in $C^{\beta}_b(\R)$ (again by Whitney extension theorem) and since $ (\Gamma^e)^{-1}(\xi,\bog)_{\vert [0,A]} = \Gamma^{-1}(\xi,\bog_{\vert [0,A]})$, we obtain the following corollary.
 
 \begin{corollary}\label{g0_palm_holder}
If the functions $\boh^0$ are in $C^\beta([0,A])$ for some $\beta \in ]0,1]$, then $\Gamma^{-1}\big(\R^K\times C^\beta([0,A])^{K^2}\big)\subseteq \R^K\times \mathcal{B}^\beta_{\infty,\infty}([0,A])^{K^2}$.
 \end{corollary}
Thus, we conclude that when the interaction functions $\boh^0$ and the $L_2$ Riesz representor $\bog^0_2$ are in $C^\beta([0,A])$ for some $\beta \in ]0,1]$,  the functions $\bog^0_L$ are in  $\mathcal{B}^\beta_{\infty,\infty}([0,A])$ which is equal to  $C^\beta([0,A])$ for $\beta\notin \mathbb{N}$.
 
A natural question is how to extend this result to  $\beta>1$. When $\beta>1$ our previous proof strategy involves higher order Palm distributions (and not only the second order), about which we are not able to say as much in terms of regularity as for the first order Palm distribution. For this reason, we change our proof strategy and we come back to the fixed point equation (\ref{converse_expr}) verified by $\bog^0_L$. Starting from this, we show in Section \ref{sec:pr:beta>1}, assuming some regularity on the second order Palm distribution only, namely Conjecture \ref{conj_beta1}, that when the functions $\boh^0$ and $\bog^0_2$ are in $C^\beta([0,A])$ for some $\beta>1$, then the functions $\bog^0_L$ are also in $C^\beta([0,A])$ (see Corollary \ref{g0_holder_beta1}). Essentially and loosely speaking, Conjecture $1$ says that when the functions $\boh^0$ are $\beta$-Hölder regular for some $\beta>1$, then the functions
\begin{align*}
    x\rightarrow \E_0^{(x,s,l,j)}\Big[ \frac{1}{\lambda_A^k(f^0_k)}\Big]
\end{align*}
are also $\beta$-Hölder ($\E_0^{(x,s,l,j)}$ denotes an expectation with respect to the second order Palm distribution, see the Appendix). In the proof, this allows in particular, given $(l,k)$, to handle the regularity of $\sum_{j=1}^K\zeta_{l,j,k}(g_{L,j,k}^0)$ which is the most problematic term in the fixed point equation (\ref{converse_expr}) verified by $\bog_L^0$. Indeed, applying (\ref{iter_palm}), we have that
\begin{align*}
    \zeta_{l,j,k}(g^0_{L,j,k})(A-x)= \int_0^{A}g^0_{L,j,k}
    (A-s)\E_0^{(x,s,l,j)}\Big[ \frac{1}{\lambda_A^k(f^0_k)} \Big] m_{l,j}(s-x)ds,
\end{align*}
see (\ref{explicit_zeta}) for more details. Conjecture \ref{conj_beta1} is also used to handle the derivatives of the functions $p_{l,k}$.

\section{Discussion}\label{sec:disc}

In this paper we have studied  the semiparametric Bernstein-von Mises property for stationary linear multivariate Hawkes processes. Theorem \ref{main_theorem} shows that, under standard or rather weak conditions, this property holds for random series prior provided that a bias term vanishes in the limit. Then, Corollary \ref{main_coro} gives a sufficient condition for this bias term to vanish. These results are illustrated with random series priors constructed on histogram and wavelet bases. 

It is of interest to consider extensions of the model we have considered here. In particular, it is likely that our results hold for excitation functions that are not supported on a bounded interval (i.e. $A=+\infty$) since, quoting \cite{Sulem_concentration}, "the regenerative properties of the (nonlinear) Hawkes processes
also hold for processes with unbounded memory". However, it is not entirely clear how to proceed for some steps, for instance our proof of the equivalence between the LAN norm and the $L_2$ norm relies quite heavily on $A<+\infty$. Regarding the prior model, we  make extensive use of the (weak) condition (\ref{condi_sup}) on our random series priors to handle the sup-norm with just a control on the $L_2$ norm. For other prior models, one would certainly need another technique to handle the sup-norm.

Secondly, it would be of interest to extend our results to   stationary nonlinear Hawkes processes, in particular those with activation functions for which posterior contraction has been proved in \cite{Sulem_concentration} (for instance ReLU, and sigmoid functions). We believe that the same strategy could be used, with extra technicalities to handle the non linearity when studying the LAN property. 
The study of  the least favourable direction, would be significantly more complicated,  since there is no cluster representation in the nonlinear case, which we have used here.

Finally extensions towards non stationary Hawkes processes,  with parameters $(\nu, \boh)$ that might vary with time or depend on other processes, would be interesting too (these extensions were actually proposed by one referee). Typically, \cite{lotz_limit} considers Hawkes processes for which $(\nu,\boh)$ depend on time (in a certain multiplicative way for $\boh$) whereas in \cite{dion_driven}, $\nu$ is driven by a diffusion process. The statistical models studied in these articles are parametric. Assuming that the model considered is clearly defined (in particular no identifiability issues), one should be able to derive the LAN property similarly provided that certain limit theorems have been proven for the process at hand. In particular, the Birkhoff ergodic theorem in \cite{dion_driven} should be well suited for this purpose. Moreover, if possible, establishing a cluster representation for the considered process would certainly be  useful to prove the BvM property. Also  many various other extensions could be considered such as spatio-temporal Hawkes processes (for which posterior contraction has been recently proved in \cite{Sulem_spatio}, but with an asymptotic regime different from ours) or Hawkes autoregressive processes, studied in \cite{leblanc_AR}, where a Hawkes process is combined with an AR process and for which the author establishes cluster representation and derives properties on the long-time behaviour. Further extensions can be found in the recent review \cite{hawkes_review}.

\section{Proofs of Section \ref{sec_semi_param_eff}}\label{proof_section_semiparam}

 In this section, we prove the results of Section \ref{sec_semi_param_eff} on semiparametric efficiency. We first prove Lemma \ref{eq} on the equivalence between the norms $\Vert. \Vert_L$ and $\Vert .\Vert_{(2,\boh^0)}$. Then, we prove the results on the LAN expansion (Lemma \ref{lemma_LAN} and Lemma \ref{lemma_LAN_convol_RELU}).

\subsection{Proof of Lemma \ref{eq}}\label{sec:pr:LAN_1}

It is clear that $\langle\cdot{,}\cdot\rangle_L$  is a bilinear map, symmetric, positive semi-definite on $\RLKO$ and $\Vert (\xi,\bog)\Vert_L^2 = \langle (\xi,\bog), (\xi,\bog)\rangle_L$. We first show that there exist two positive constants $C_2>C_1$ such that for all $(\xi,\bog)\in \RLKO$,
\begin{align}
    &C_1 \Vert (\xi,\bog)\Vert_{(2,\boh^0)}\leq \Vert (\xi,\bog)\Vert_L \leq C_2 \Vert (\xi,\bog)\Vert_{(2,\boh^0)} , \label{(b)} 
\end{align}
We recall the definition of these norms:
\begin{align}
    \Vert (\xi,\bog) \Vert_{(2,\boh^0)}^2 =  \sum_{k=1}^K \xi_k^2 + \sum_{l=1}^K\int_0^A \frac{g^2_{l,k}(x)}{\nu_k^0 +h_{l,k}^0(x)} dx \hspace{0.25cm} \text{and} \hspace{0.25cm} \Vert (\xi,\bog) \Vert_L^2 =\sum_{k=1}^K\E_0\Big[\frac{\tilde\lambda_A^k(\xi_k,\bog_k)^2}{\lambda_A^k(f^0_k)} \Big] .\nonumber
\end{align}
We begin by the right-hand side of (\ref{(b)}), fix $ k\in \K$, 
\begin{align}
    \E_0\Big[\frac{\tilde\lambda_A^k(\xi_k,\bog_k)^2}{\lambda_A^k(f^0_k)} \Big] \leq 2 \E_0\bigg[\frac{1}{\lambda_A^k(f^0_k)}\bigg]\xi_k^2 + 2\E_0\bigg[\frac{\tilde\lambda_A^k(0,\bog_k)^2}{\lambda_A^k(f^0_k)} \bigg]. \nonumber
\end{align}
For the second term on the right-hand side, we have:
\begin{align}
    &\E_0\Big[\frac{\tilde\lambda_A^k(0,\bog_k)^2}{\lambda_A^k(f^0_k)} \Big]= \E_0\Bigg[ \frac{\big(\sum_{l=1}^K\int_0^A g_{l,k}(A-t) dN_t^l\big)^2}{\nu^{0}_k + \sum_{j=1}^K\int_0^A h_{j,k}^0(A-u) dN_u^l} \Bigg] \nonumber \\
     &\leq 2^{K-1}\sum_{l=1}^K\E_0\Bigg[ \bigg(\int_0^A \frac{\vert g_{l,k}(A-t) \vert }{\sqrt{\nu^{0}_k + h_{l,k}^0(A-t)}}\times \frac{\sqrt{\nu^{0}_k + h_{l,k}^0(A-t)}}{\sqrt{\nu^{0}_k + \sum_{j=1}^K\int_0^A h_{j,k}^0(A-u) dN_u^l}}dN_t^l\bigg)^2 \Bigg]  \nonumber \\
     &\leq 2^{K-1}\sum_{l=1}^K\E_0\Bigg[ \bigg(\int_0^A \frac{\vert g_{l,k}(A-t) \vert }{\sqrt{\nu^{0}_k + h_{l,k}^0(A-t)}}dN_t^l\bigg)^2 \Bigg] \nonumber \\
     &\leq 2^{K}\sum_{l=1}^K\E_0\Bigg[ \bigg(\int_0^A \frac{\vert g_{l,k}(A-t) \vert }{\sqrt{\nu^{0}_k + h_{l,k}^0(A-t)}}\big(dN_t^l -\lambda_t^l(f^0_l) dt\big)\bigg)^2 \Bigg] \nonumber\\
     &\hspace{3cm}+ \E_0\Bigg[ \bigg(\int_0^A \frac{\vert g_{l,k}(A-t) \vert }{\sqrt{\nu^{0}_k + h_{l,k}^0(A-t)}}\lambda_t^l(f^0_l) dt \bigg)^2 \Bigg] \nonumber \\
    &\leq 2^{K}\sum_{l=1}^K\E_0\Bigg[ \int_0^A \frac{g_{l,k}(A-t)^2}{\nu^{0}_k + h_{l,k}^0(A-t)}\lambda_t^l(f^0_l)  dt \Bigg]  \nonumber\\
    &\hspace{3cm}+\E_0\Bigg[ A\int_0^A \frac{ g_{l,k}(A-t)^2}{\nu^{0}_k + h_{l,k}^0(A-t)}\lambda_t^l(f^0_l)^2 dt  \Bigg] \nonumber \\
    &\leq 2^{K}\underset{k\in \K}{\max}\bigg(\E_0[\lambda_A^k(f^0_k)] + A\E_0\big[\lambda_A^k(f^0_k)^2\big]\bigg)\sum_{l=1}^K \int_0^A \frac{g_{l,k}(A-t)^2}{\nu_k^0 + h^0_{l,k}(A-t)} dt \nonumber \\
    &= 2^{K}\underset{k\in \K}{\max}\bigg(\E_0[\lambda_A^k(f^0_k)] + A\E_0\big[\lambda_A^k(f^0_k)^2\big]\bigg) \Vert \bog_k\Vert_{(2,\boh^0)}^2,\nonumber
\end{align}
where for the penultimate inequality we have used Jensen inequality for the second term. So we have shown the right-hand side of (\ref{(b)}) with
\begin{align}
    C_2^2 = 2\max\Big(\underset{k\in \K}{\max} \E_0\Big[1/\lambda_A^k(f^0_k)\big]\hspace{0.1cm}; \hspace{0.1cm}2^K\underset{k\in \K}{\max}\big(\E_0[\lambda_A^k(f^0_k)] + A\E_0\big[\lambda_A^k(f^0_k)^2\big]\big)\Big).\nonumber
\end{align}
For the left-hand side of (\ref{(b)}), we adapt the proof of Lemma 4 from the supplementary material of \cite{donnet_concentration} to our purpose. We define for all $l\in\K$ the event 
\begin{align}
    \mathcal{A}_l = \{N ([-A,0])=0 \text{   and   }N(]0,A])=N^l(]0,A]) =1\} .\nonumber
\end{align}
The events $(\mathcal{A}_l)_{l\in\K}$ are mutually disjoint and disjoint from $ \{N([0,A])=0\}$. We thus obtain
 \begin{align}
    \Vert (\xi,\bog)\Vert_L^2 &\geq\sum_{k=1}^K\E_0\Bigg[\Big(\onee_{\{N([0,A])=0\}} + \sum_{l=1}^K\onee_{\mathcal{A}_l}\Big) \frac{\tilde\lambda_A^k(\xi_k,\bog_k)^2}{\lambda_A^k(f^0_k)} \Bigg]\nonumber\\
    &=\sum_{k=1}^K\E_0\Bigg[\onee_ {\{N([0,A])=0\}} \frac{\xi_k^2}{\nu^{0}_k} \Bigg]+\sum_{k=1}^K \sum_{l=1}^K  \E_0\Bigg[\onee_{\mathcal{A}_l}\frac{\tilde \lambda_A^k(\xi_k, \bog_k)^2}{\lambda_A^k(f^0_k)}\Bigg].\label{init_eq_norm}
 \end{align}
 First, let $c= \P_0\big(N([0,A])=0\big)\min_{k\in \K} (\nu_k^0)^{-1} >0$, we have
 \begin{align*}
     \sum_{k=1}^K\E_0\Bigg[\onee_{\{N([0,A])=0\}} \frac{\xi_k^2}{\nu^{0}_k} \Bigg] \geq c\sum_{k=1}^K\xi_k^2.
 \end{align*}
Then, let $k\in \K$. On an event $\mathcal{A}_l$, we denote by  $U_l$ the unique point of the process of type $l$ between $[0,A]$. We have
 \begin{align*}
      \sum_{l=1}^K  \E_0\Bigg[ \onee_{\mathcal{A}_l}\frac{\tilde \lambda_A^k(\xi_k,\bog_k)^2}{\lambda_A^k(f^0_k)}\Bigg] &= \sum_{l=1}^K  \E_0\Bigg[\onee_{\mathcal{A}_l} \frac{\big(\xi_k+ g_{l,k}(A-U_l)\big)^2 }{\nu_k^0 +h_{l,k}^0(A-U_l) }\Bigg]\nonumber.
 \end{align*}
Furthermore, as in Proposition 7.3.III of \cite{DVJ1}, let $\mathds{Q}$ be the $K$-marked process consisting in $K$ independent component each being a Poisson process on $[0;+\infty[$ with constant intensity equal to $1$, the likelihood ratio of the Hawkes process relative to $\mathds{Q}$ on $[-A,A]$ is
 \begin{align}
     \mathcal{L}_A = \exp\bigg(2AK+ \sum_{k=1}^K \int_{-A}^A \log(\lambda_t^k(f^0_k)) dN_t^k - \int_{-A}^A\lambda_t^k(f^0_k) dt \bigg). \nonumber
 \end{align}
 On the event $\A_l$,
 \begin{align}
 \mathcal{L}_A &= \exp(2KA)\times\lambda_{U_l}^l(f^0)\times\exp\bigg(-\sum_{k=1}^K\Big(2A\nu^0_k + \int_{U_l}^A h_{l,k}^0(t-U_l) dt\Big)\bigg) \nonumber \\
 &\geq \nu_l^0\exp\Big( -\sum_{k=1}^K \big(2A\nu_0^k + \Vert h^0_{l,k}\Vert_1\big)\Big):= r_l ,\nonumber
\end{align}
and $r_l>0$ because the constraint $\rho(\boh^0)<1$ implies that the functions $(h^0_{l,k})_{l,k}$ are integrable. Moreover, under the distribution $\mathds{Q}$ and conditionally on $\mathcal{A}_l$, $U_l$ is distributed uniformly over $[0,A]$ so we have
 \begin{align}
     \sum_{l=1}^K\E_0\bigg[\onee_{\mathcal{A}_l}\frac{\tilde\lambda_A^k(\xi_k,\bog_k)^2}{\lambda_A^k(f^0_k)}\bigg]&= \sum_{l=1}^K  \E_{\mathds{Q}} \Bigg[\mathcal{L}_A\onee_{\mathcal{A}_l} \frac{ \big(\xi_k+g_{l,k}(A-U_l)\big)^2 }{\nu_k^0 +h_{l,k}^0(A-U_l) }\Bigg]\nonumber\\
     &\geq \sum_{l=1}^K  r_l\E_{\mathds{Q}} \Bigg[\onee_{\mathcal{A}_l} \frac{\big(\xi_k+ g_{l,k}(A-U_l)\big)^2 }{\nu_k^0 +h_{l,k}^0(A-U_l) }\Bigg]\nonumber \\
     &\geq \sum_{l=1}^K  r_l\mathds{Q} (\mathcal{A}_l) \int_0^A \frac{ \big(\xi_k+g_{l,k}(A-u)\big)^2 }{\nu_k^0 +h_{l,k}^0(A-u) }du\nonumber\\
     &\geq  c'\sum_{l=1}^K  \int_0^A \frac{ \big(\xi_k+g_{l,k}(A-u)\big)^2 }{\nu_k^0 +h_{l,k}^0(A-u) }du,\nonumber
 \end{align}
with $c' = \min\{r_l \mathds{Q}(\mathcal{A}_l), l\in \K\}>0$. Next, set
\begin{align*}
    c''_{k}:= \bigg( \sum_{l=1}^K\int_0^A \frac{1}{\nu_0^k+h^0_{l,k}(u)}du \bigg)^{-1}>0,
\end{align*}
and $c''=\min\big\{\min(1,c_k''), k\in\K\big\}$. Coming back to (\ref{init_eq_norm}), we now have that
\begin{align*}
    \Vert (\xi,\bog)\Vert_L^2 &\geq \min(c,c')\sum_{k=1}^K\xi_k^2 + \sum_{l=1}^K \int_0^A \frac{ \big(\xi_k+g_{l,k}(A-u)\big)^2 }{\nu_k^0 +h_{l,k}^0(A-u) }du\nonumber\\
    &= \min(c,c')\sum_{k=1}^K \sum_{l=1}^K  \int_0^A \frac{ c_k''\xi_k^2+\big(\xi_k+g_{l,k}(A-u)\big)^2 }{\nu_k^0 +h_{l,k}^0(A-u) }du\nonumber\\
    &\geq \min(c,c')c'' \sum_{k=1}^K \sum_{l=1}^K  \int_0^A \frac{ \xi_k^2+\big(\xi_k+g_{l,k}(A-u)\big)^2 }{\nu_k^0 +h_{l,k}^0(A-u) }du.
\end{align*}
Finally, let $c_0= (3-\sqrt{5})/2>0$, using that $\forall a,b\in \R$, $a^2+(a+b)^2\geq c_0(a^2+b^2)$, we find that
\begin{align*}
    \Vert (\xi,\bog)\Vert_L^2&\geq \min(c,c')c''c_0\sum_{k=1}^K (c_k'')^{-1}\xi_k^2 +\sum_{l=1}^K  \int_0^A \frac{g_{l,k}(A-u)^2 }{\nu_k^0 +h_{l,k}^0(A-u) }du \geq C_1^2 \Vert (\xi,\bog)\Vert_{(2,\boh^0)}^2,
\end{align*}
with $C_1^2=\min(c,c')c''c_0\min((c_k'')^{-1}, k\in \K)$. It concludes the proof for the left-hand side of (\ref{init_eq_norm}). Hence, we have proven that $\Vert.\Vert_L$ is a norm on $\RLKO$ equivalent to $\Vert.\Vert_{(2,\boh^0)}$.

In conclusion, it comes with the previous result that $\mathcal{S}$ is a bijective, bounded, linear operator between the Hilbert spaces $(\RLKO, \Vert.\Vert_{(2,\boh^0)})$ and $(\mathcal{S}(\RLKO),  \Vert . \Vert_{\L_2(\P_0)})$.

\subsection{Proofs of the LAN expansions}\label{sec:pr:LAN_2}
In this section, we prove the results on the LAN expansion, we begin with the expansion in model $\mathcal{P}$ (Lemma \ref{lemma_LAN})  and then we extend it to the model $\mathcal{P}_R$ (Lemma  \ref{lemma_LAN_convol_RELU}).

\begin{proof}[Proof of Lemma \ref{lemma_LAN}. ]  We will use in the proof the expansion (\ref{likelihood_expansion}) that we first detail here.  Let $(\nu^0,\boh^0)\in ]0,+\infty[^K\times \mathcal{H}$ and $(\xi,\bog)\in \RLKO$ such that for $T$ large enough, $\boh^0 + \bog/\sqrt{T}\in \mathcal{H}$, we have
\begin{equation} \label{likelihood_expansion_proof}
\begin{aligned}
    &L_T\Big(\nu^0+ \frac{\xi}{\sqrt{T}}, \boh^0+ \frac{\bog}{\sqrt{T}}\Big) - L_T\Big(\nu^0, \boh^0\big)  \\
    &=\sum_{k=1}^K\int_0^T\log\bigg( 1 + \frac{1}{\sqrt{T}}\frac{\tilde \lambda_t^k(\xi_k, \bog_k)}{\lambda_t^k(f^0_k)}\bigg) dN_t^k - \frac{1}{\sqrt{T}}\int_0^T \frac{\tilde\lambda_t^k(\xi_k,\bog_k)}{\lambda_t^k(f^0_k)}\lambda_t^k(f^0_k) dt  \\
    &=\frac{1}{\sqrt{T}}\hspace{-0.1cm} \sum_{k=1}^K \int_0^T \frac{\tilde \lambda_t^k(\xi_k, \bog_k)}{\lambda_t^k(f^0_k)} \big(dN_t^k\hspace{-0.05cm} -\hspace{-0.05cm}  \lambda_t^k(f^0_k) dt\big) - \frac{1}{2T}\hspace{-0.1cm}  \sum_{k=1}^K \int_0^T \frac{\tilde \lambda_t^k(\xi_k, \bog_k)^2}{\lambda_t^k(f^0_k)} \frac{dN_t^k}{\lambda_t^k(f^0_k)} \\
    &\hspace{2cm}+\frac{1}{T}\hspace{-0.1cm}\sum_{k=1}^K \int_0^T \frac{\tilde \lambda_t^k(\xi_k, \bog_k)^2}{\lambda_t^k(f^0_k)} R\Big(\frac{1}{\sqrt{T}}\frac{\tilde \lambda_t^k(\xi_k, \bog_k)}{\lambda_t^k(f^0_k)} \Big)\frac{dN_t^k}{\lambda_t^k(f^0_k)}\\
    &=W_T(\xi,\bog) - \frac{1}{2} \Vert (\xi,\bog)\Vert_L^2 + R_T\big(\frac{\xi}{\sqrt{T}},\frac{\bog}{\sqrt{T}}\big),
\end{aligned}
\end{equation}
with
\begin{align}\label{remainderLAN_def}
    &R_T\big(\frac{\xi}{\sqrt{T}},\frac{\bog}{\sqrt{T}}\big) = \Delta Q_T(\xi,\bog)+\frac{1}{T} \sum_{k=1}^K \int_0^T \frac{\tilde \lambda_t^k(\xi_k, \bog_k)^2}{\lambda_t^k(f^0_k)} R\Big(\frac{1}{\sqrt{T}}\frac{\tilde \lambda_t^k(\xi_k, \bog_k)}{\lambda_t^k(f^0_k)} \Big)\frac{dN_t^k}{\lambda_t^k(f^0_k)}, \\
    &\Delta Q_T(\xi,\bog)=\frac{1}{2}\Vert (\xi,\bog)\Vert_L^2  - \frac{1}{2T} \sum_{k=1}^K \int_0^T \frac{\tilde \lambda_t^k(\xi_k, \bog_k)^2}{\lambda_t^k(f^0_k)} \frac{dN_t^k}{\lambda_t^k(f^0_k)}, \nonumber
\end{align}
and $R$ defined for $x>-1$ by $\log(1+x) = x - \frac{x^2}{2} + x^2R(x)$. With (\ref{likelihood_expansion_proof}), we begin by proving  the LAN property when the functions $h_{l,k}^0$ are bounded away from $0$. Then, we show that this proof can be quite directly extended to the general case.

 When the functions $h_{l,k}^0$ are bounded away from $0$, $L_{2,\boh^0}^+=L_{2,\boh^0}$. Let $(\xi,\bog)\in \RLKO$, we study the behavior of the log-likelihood along the nonlinear path $(\xi(T), \bog(T))_{\mathcal{P}}$  defined by (\ref{nonlinear_path_model_lin}).  Because $a_T^{-1}\arctan(g_{l,k}a_T/\sqrt{T})$ is bounded by $\pi/2a_T \rightarrow 0$, the functions $h^0_{l,k} +~ a_T^{-1}\arctan(g_{l,k}a_T/\sqrt{T})$ are non-negative for $T$ large enough and this path is equal to (\ref{gen_non_linear_path_relu}):
\begin{align}
   \bigg(\big(\nu_k^0 + \frac{\xi_k}{\sqrt{T}}\big)_{k\in \K} \hspace{0.05cm}, \hspace{0.05cm}\Big(h_{l,k}^0 + \frac{1}{a_T}\arctan\big(\frac{a_T g_{l,k}}{\sqrt{T}}\big)\Big)_{(l,k)\in \K^2}\bigg). \nonumber
\end{align}

\textit{Likelihood expansion.} The difference of the log-likelihoods involves the difference of the intensities and by linearity of the intensity we have:
\begin{equation}\label{intensity_difference}
\begin{aligned}
    &\lambda_t^k\Big(\nu^0_k +\frac{\xi_k}{\sqrt{T}} ,\boh^0_k+\frac{1}{a_T}\arctan\big(\frac{a_T \bog_k}{\sqrt{T}}\big)\Big) - \lambda_t^k(\nu^0_k,\boh^0_k)\\
    &= \lambda_t^k\Big(\nu^0_k +\frac{\xi_k}{\sqrt{T}} ,\boh^0_k+\frac{1}{a_T}\arctan\big(\frac{a_T \bog_k}{\sqrt{T}}\big)\Big) - \tilde \lambda_t^k\Big(\nu^0_k+\frac{\xi_k}{\sqrt{T}},\boh^0_k + \frac{\bog_k}{\sqrt{T}})\Big)\\
    &\hspace{1cm}+\tilde\lambda_t^k\Big(\nu^0_k+\frac{\xi_k}{\sqrt{T}},\boh^0_k +\frac{\bog_k}{\sqrt{T}}\Big) - \lambda_t^k(\nu_k^0, \boh_k^0)  \\
    &=  \frac{1}{\sqrt{T}} \tilde \lambda_t^k\Big(0,\bog_k.\delta\big(\frac{a_T\bog_k}{\sqrt{T}}\big)\Big) + \frac{1}{\sqrt{T}}\tilde \lambda_t^k\Big(\xi_k,\bog_k\Big) .
\end{aligned}
\end{equation}
with for $x\neq 0, \delta(x)= (\arctan(x)-x)/x$ and it is continuously extended by $\delta(0) =0$. With this expression, we can expand the log-likelihood as in (\ref{likelihood_expansion_proof}) but with  some additional remainder terms coming from the nonlinearity of the path.
\begin{align}\label{expansion_arctan}
    L_T\Big((\xi(T), \bog(T))_{\mathcal{P}}\Big) - L_T(\nu^0,\boh^0) = W_T(\xi,\bog) -\frac{1}{2}\Vert (\xi,\bog)\Vert_L^2 +\bar R_T\big(\frac{\xi}{\sqrt{T}},\frac{\bog}{\sqrt{T}}\big) ,
\end{align}
with
\begin{align}
    \bar R_T\big(\frac{\xi}{\sqrt{T}},\frac{\bog}{\sqrt{T}}\big)&= \bigg(\frac{1}{2}\Vert (\xi,\bog)\Vert_L^2  - \frac{1}{2T} \sum_{k=1}^K \int_0^T \frac{\tilde \lambda_t^k(\xi_k, \bog_k)^2}{\lambda_t^k(f^0_k)} \frac{dN_t^k}{\lambda_t^k(f^0_k)}\bigg) \nonumber\\
    &\hspace{1cm}+ \frac{1}{\sqrt{T}} \sum_{k=1}^K\int_0^T \tilde \lambda_t^k\Big( 0,\bog_k.\delta(\frac{a_T\bog_k}{\sqrt{T}})\Big) \big( \frac{dN_t^k}{\lambda_t^k(f^0_k)} -dt \big) \nonumber\\
    &\hspace{1cm}- \frac{1}{2T} \sum_{k=1}^K\int_0^T \frac{\tilde\lambda_t^k\Big( 0,\bog_k.\delta(\frac{a_T\bog_k}{\sqrt{T}})\Big)^2}{\lambda_t^k(f^0_k)} \frac{dN_t^k}{\lambda_t^k(f^0_k)}\nonumber\\
    &\hspace{1cm}- \frac{1}{T} \sum_{k=1}^K\int_0^T \tilde\lambda_t^k\Big( 0,\bog_k.\delta(\frac{a_T\bog_k}{\sqrt{T}})\Big)\frac{\tilde\lambda_t^k( \xi_k, \bog_k)}{\lambda_t^k(f^0_k)} \frac{dN_t^k}{\lambda_t^k(f^0_k)} \label{non_linear_remainder}\\
    &\hspace{1cm} + \frac{1}{T}\sum_{k=1}^K \int_0^T \bigg(\frac{\tilde\lambda_t^k(\ \xi_k,\bog_k) + \tilde \lambda_t^k\big( 0,\bog_k.\delta(\frac{a_T\bog_k}{\sqrt{T}})\big)}{\lambda_t^k(f^0_k)}\bigg)^2\nonumber\\
    &\hspace{3cm}\times R\bigg(\frac{1}{\sqrt{T}} \frac{\tilde\lambda_t^k(\xi_k,a_T^{-1}\arctan\big(\frac{a_T\bog_k}{\sqrt{T}}\big)}{\lambda_t^k (f^0)}\bigg) \hspace{0.1cm}dN_t^k  \nonumber\\
    &:=  \Delta Q_T(\xi,\bog) + \bar R_{T,1}(\xi,\bog) + \bar R_{T,2}(\xi,\bog) +\bar R_{T,3}(\xi,\bog) +\bar R_{T,4}(\xi,\bog) .\nonumber
\end{align}
To prove the LAN property, we first show that $\Delta Q_T(\xi,\bog)$ goes in probability to $0$, then that $W_T(\xi,\bog)$ goes in distribution to a $\mathcal{N}\big(0,\Vert (\xi,\bog)\Vert_L^2\big)$ and finally that the other terms $ \bar R_{T,i}(\xi,\bog)$, $i\in [4]$, go in probability to $0$ too.

\textit{Term} $  \mathit{\Delta Q_T(\xi,\bog)}$. Let $(\varepsilon_t^k)_{t,k}$ be a marked, stationary and predictable (with respect to the history $\mathcal{G}$) process with a first moment. Recall that the Hawkes process $N$ we consider is ergodic (see Proposition 12.3.IX of \cite{DVJ2}).  Fix $k\in\K$, we know by Birkhoff ergodic theorem that
\begin{align}\label{ergodic}
    \frac{1}{T} \int_0^T \varepsilon_t^k dt \xrightarrow[T\rightarrow +\infty]{a.s.}\E_0[\varepsilon_0^k] .
\end{align}
Lemma 2 of \cite{Ogata_78} shows that if the integrand process has a second moment then the convergence (\ref{ergodic}) holds with respect to $dN_t^k/\lambda_t^k(f^0_k)$ instead of $dt$. Moreover, the convergence $\Delta Q_T(\xi,\bog) \rightarrow 0$ is implied by  
\begin{align}\label{cv_norme_LAN_k}
    \frac{1}{2}\E_0\Big[ \frac{\tilde\lambda_A^k(\xi_k,\bog_k)^2}{\lambda_A^k(f^0_k)}\Big]  - \frac{1}{2T}  \int_0^T \frac{\tilde \lambda_t^k(\xi_k, \bog_k)^2}{\lambda_t^k(f^0_k)} \frac{dN_t^k}{\lambda_t^k(f^0_k)} \xrightarrow[T\rightarrow +\infty]{\P_0} 0 ,
\end{align}
for all $k\in \K$. Set $\varepsilon_t^k = \tilde \lambda_t^k(\xi_k, \bog_k)^2/\lambda_t^k(f^0_k) $, this process has a first moment, not necessarily a second moment, but we can  still obtain the desired convergence in probability (instead of almost surely) which is sufficient for our purpose. Indeed, define for $M\geq 0$, $\varepsilon_{t,M}^k =\varepsilon_t^k \onee_{\varepsilon_t^k \leq M}$, this process is bounded so has a second moment and $\underset{M\rightarrow +\infty}{\lim}\varepsilon_{t,M}^k  =\varepsilon_{t}^k$ almost surely.  Let $\eta >0$, by dominated convergence theorem,  there exists $M_0>0$ such that for $M\geq M_0$, $ \E_0\big[\vert\varepsilon_0^k - \varepsilon_{0,M}^k\vert\big]\leq\eta^2/4$. Let $M\geq M_0$,
\begin{align}
    &\P_0\bigg(\Big \vert \frac{1}{T}\int_0^T \varepsilon_t^k \frac{dN_t^k}{\lambda_t^k(f^0_k)} - \E_0[\varepsilon_0^k]\Big\vert >\eta \bigg) \nonumber\\
    &\leq \P_0\bigg(\Big \vert \frac{1}{T}\int_0^T \varepsilon_t^k - \varepsilon_{t,M}^k \frac{dN_t^k}{\lambda_t^k(f^0_k)} \Big\vert >\frac{\eta}{2} \bigg)  + \P_0\bigg(\Big \vert \frac{1}{T}\int_0^T \varepsilon_{t,M}^k \frac{dN_t^k}{\lambda_t^k(f^0_k)} - \E_0[\varepsilon_{0,M}^k]\Big\vert >\frac{\eta}{4} \bigg).\nonumber
\end{align}
By Markov inequality  and stationarity,
\begin{align}
    \P_0\bigg(\Big \vert \frac{1}{T}\int_0^T \varepsilon_t^k - \varepsilon_{t,M}^k \frac{dN_t^k}{\lambda_t^k(f^0_k)} \Big\vert >\frac{\eta}{2} \bigg) &\leq \frac{2\E_0 \big[\vert \varepsilon_0^k - \varepsilon_{0,M}^k\vert\big]}{\eta}  \leq \frac{\eta}{2}. \nonumber
\end{align}
Then, because we can apply Lemma 2 of \cite{Ogata_78} to $\varepsilon^k_{t,M}$, there exists $T$ large enough such that
\begin{align}
    \P_0\bigg(\Big \vert \frac{1}{T}\int_0^T \varepsilon_{t,M}^k \frac{dN_t^k}{\lambda_t^k(f^0_k)} - \E_0[\varepsilon_{0,M}^k]\Big\vert >\frac{\eta}{4} \bigg) \leq \frac{\eta}{2}.\nonumber
\end{align}
Hence, we have proved that for all $\eta >0$, there exists $T$ large enough such that
\begin{align}
    \P_0\bigg(\Big \vert \frac{1}{T}\int_0^T \varepsilon_t^k \frac{dN_t^k}{\lambda_t^k(f^0_k)} - \E_0[\varepsilon_0^k]\Big\vert >\eta \bigg)  \leq \eta.\nonumber
\end{align}
Therefore, we have obtained (\ref{cv_norme_LAN_k}) and  $\Delta Q_T(\xi,\bog)$ goes to $0$ in probability.

\textit{First order term $\mathit{W_T(\xi,\bog)}$}. To prove the asymptotic normality of this term, we prove that the vector 
\begin{align}\label{vector_CLT}
    \bigg(\frac{1}{\sqrt{T}}\int_0^T \frac{\tilde \lambda_t^k(\xi_k, \bog_k)}{\lambda_t^k(f^0_k)} \big(dN_t^k - \lambda_t^k(f^0_k) dt\big)\bigg)_{k\in\K}
\end{align}
converges in distribution to a centered Gaussian vector with diagonal covariance matrix. To do so, we apply Theorem 2.1 of \cite{Nishiyama}. This is a slight generalisation of Corollary 4.5.1 in \cite{Kutoyants}. To be consistent with \cite{Nishiyama}, for $i\in \K$, let $H_i(t,k)$ be the marked process defined by 
\begin{equation}
     H_i(t,k)=
    \begin{cases}
      \frac{\tilde \lambda_t^i(\xi_i, \bog_i)}{\lambda_t^i(f^0_i)}  & \text{if}\ k=i \\
      0 & \text{otherwise}
    \end{cases}. \nonumber
  \end{equation}
We can rewrite the $i$-th element in (\ref{vector_CLT}) as an integral of this marked process over $[0,T]\times~ \K$:
\begin{align}
    \frac{1}{\sqrt{T}}\int_0^T \frac{\tilde \lambda_t^i(\xi_i, \bog_i)}{\lambda_t^i(f^0_i)} \big(dN_t^i - \lambda_t^i(f^0_i) dt\big) = \frac{1}{\sqrt{T}} \int_0^T \int_{\K}  H_i(t,k)\big(N(dt\times dk)  - \lambda_t^k(f^0_k)dk dt\big).\nonumber
\end{align}
If there is convergence, the coefficient $(i,j)$ of the asymptotic covariance matrix is given by the limit of
\begin{align}
     &\frac{1}{T} \int_0^T \int_{\K} H_i(t,k)H_j(t,k)\lambda_t^k(f^0_k)dkdt, \nonumber
\end{align}
as $T\rightarrow +\infty$. Obviously, for $i\neq j$, there is convergence towards $0$. By again Lemma 2 of \cite{Ogata_78}, the diagonal terms also converge and the covariance matrix is thus a diagonal matrix whose diagonal is given by 
\begin{align}
    \bigg(\E_0\Big[ \frac{\tilde \lambda_A^1(\xi_1,\bog_1)^2}{\lambda_A^1(f^0_1)}\Big],\hspace{0.05cm}...\hspace{0.05cm}, \E_0\Big[ \frac{\tilde\lambda_A^K(\xi_k,\bog_K)^2}{\lambda_A^K(f^0_K)}\Big]\bigg).\nonumber
\end{align}
Finally we have to verify a Lindeberg-type condition: $\forall \epsilon>0$,
\begin{align}
    \E_0\Bigg[\frac{1}{T}\int_0^T \sum_{k=1}^K  \frac{\tilde\lambda_t^k(\xi_k,\bog_k)^2}{\lambda_t^k(f^0_k)} \onee\Big\{ \sum_{k=1}^K \frac{\tilde\lambda_t^k(\xi_k,\bog_k)^2}{\lambda_t^k(f^0_k)} > T\epsilon\Big\} dt \Bigg ] \xrightarrow[T\rightarrow +\infty]{} 0. \nonumber
\end{align}
 By Fubini theorem, stationarity and dominated convergence theorem, this last condition is verified. Consequently, (\ref{vector_CLT}) tends to the announced Gaussian vector and $W_T(\xi,\bog)$ goes in distribution to a $\mathcal{N}\big(0,\Vert (\xi,\bog)\Vert_L^2\big)$.

\textit{Remainder terms.} We finally show that the last four terms in (\ref{non_linear_remainder}) ($\bar R_{T,i}(\xi,\bog), \hspace{0.05cm}i\in [4]$)  go in probability to 0. These terms involve the function $\delta$.  This function is continuous, $\delta(0)= 0$ and  since $g_{l,k}<+\infty$ a.s.,  we have that $\delta(a_Tg_{l,k}/\sqrt{T}) \rightarrow 0$ a.s. . In addition, $\delta$ is bounded and $\int_0^A \vert g_{l,k}(A-u)\vert dN_u^l <+\infty$ a.s., for any $(l,k)\in\K^2$. So, by dominated convergence theorem we have:
\begin{align}\label{as_conv}
    \tilde \lambda_A^k\Big(0,\bog_k.\delta\big(\frac{a_T\bog_k}{\sqrt{T}}\big)\Big) = \sum_{l=1}^K \int_0^A g_{l,k}(A-u)\delta\big(\frac{a_Tg_{l,k}(A-u)}{\sqrt{T}}\big) dN_u^l\xrightarrow[T\rightarrow +\infty]{\textit{a.s.}}0 .
\end{align} 
With (\ref{as_conv}), it is clear that $\bar R_{T,1}(\xi,\bog)$, $\bar R_{T,2}(\xi,\bog)$ and $\bar R_{T,3}(\xi,\bog)$ go to $0$. For instance, consider $\bar R_{T,1}(\xi,\bog)$ that is
\begin{align}
    \frac{1}{\sqrt{T}} \sum_{k=1}^K\int_0^T \tilde \lambda_t^k\Big( 0,\bog_k.\delta(\frac{a_T\bog_k}{\sqrt{T}})\Big) \big( \frac{dN_t^k}{\lambda_t^k(f^0_k)} -dt \big) .\nonumber
\end{align}
Fix $k\in\K$, let $\epsilon>0$, by Markov inequality and stationarity we get:
\begin{align}
    &\P_0\bigg( \Big\vert \frac{1}{\sqrt{T}} \int_0^T \tilde \lambda_t^k\Big( 0,\bog_k.\delta(\frac{a_T\bog_k}{\sqrt{T}})\Big) \big( \frac{dN_t^k}{\lambda_t^k(f^0_k)} -dt \big )\Big\vert\geq \epsilon \bigg)\nonumber\\
    &= \P_0\Bigg( \bigg( \int_0^T \tilde \lambda_t^k\Big( 0,\bog_k.\delta(\frac{a_T\bog_k}{\sqrt{T}})\Big) \big( \frac{dN_t^k}{\lambda_t^k(f^0_k)} -dt \big )\bigg)^2\geq T\epsilon^2\Bigg) \\
    &\leq \frac{1}{\epsilon^2} \E_0\Bigg[\frac{\tilde \lambda_A^k\big(0,\bog_k.\delta(\frac{a_T\bog_k}{\sqrt{T}})\big)^2}{\lambda_A^k(f^0_k)}\Bigg],\nonumber
\end{align}
and by (\ref{as_conv}) and dominated convergence theorem, this last bound goes 0 as $T$ goes to $+\infty $. Markov inequality and dominated convergence theorem also lead to the result for the following two terms ($\bar R_{T,2}(\xi,\bog)$ and $\bar R_{T,3}(\xi,\bog)$).
Finally, we treat the last term $\bar R_{T,4}(\xi,\bog)$ by  reproducing some computations of the proof of Theorem 7.2 of \cite{VDV_AS}. This term is
\begin{align}
     \sum_{k=1}^K \frac{1}{T}\int_0^T \bigg(\tilde \lambda_t^k(\ \xi_k,\bog_k) + \tilde \lambda_t^k\big( 0,\bog_k.\delta(\frac{a_T\bog_k}{\sqrt{T}})\big)\bigg)^2 R\bigg(\frac{ \tilde\lambda_t^k\big(\xi_k,a_T^{-1}\arctan\big(\frac{a_T\bog_k}{\sqrt{T}}\big)\big)}{\sqrt{T}\lambda_t^k(f^0_k)}\bigg)\frac{dN_t^k}{\lambda_t^k(f^0_k)} \nonumber
\end{align}
Fix $k\in \K$, since $\delta$ is bounded by $1$, the $k$-th term of the sum is bounded by
\begin{align}
      &\frac{4}{T}\int_0^T\hspace{-0.2cm}\tilde \lambda_t^k( \xi_k,\vert \bog_k \vert)^2 \bigg\vert R\bigg(\frac{\tilde\lambda_t^k\big(\xi_k,a_T^{-1}\arctan\big(\frac{a_T\bog_k}{\sqrt{T}}\big)\big)}{\sqrt{T}\lambda_t^k(f^0_k)}\bigg)\bigg\vert  \frac{dN_t^k}{\lambda_t^k(f^0_k)},\nonumber
\end{align}
which can be further bounded by
\begin{align}
    \underset{t_i \in N^k\vert_{[0,T]}}{\max}\bigg\vert R\bigg(\frac{ \tilde\lambda_{t_i}^k\big(\xi_k,a_T^{-1}\arctan\big(\frac{a_T\bog_k}{\sqrt{T}}\big)\big)}{\sqrt{T}\lambda_{t_i}^k(f^0_k)}\bigg)\bigg\vert\frac{4}{T}\int_0^T  \frac{\tilde \lambda^k_t(\xi_k,\vert \bog_k\vert)^2}{\lambda_t^k(f^0_k)} \frac{dN_t^k}{\lambda_t^k(f^0_k)}. \nonumber
\end{align}
Thanks to (\ref{cv_norme_LAN_k}) and because $R(x)\xrightarrow[x\rightarrow 0]{} 0$, it is enough to show that
\begin{align}
   \underset{t_i\in N^k\vert_{[0,T]}}{\max}\bigg\vert \frac{ \tilde\lambda_{t_i}^k\big(\xi_k,a_T^{-1}\arctan\big(\frac{a_T\bog_k}{\sqrt{T}}\big)\big)}{\sqrt{T}\lambda_{t_i}^k(f^0_k)}\bigg\vert\xrightarrow[T\rightarrow +\infty]{\P} 0. \nonumber
\end{align}
Let $\epsilon >0$, again by Markov inequality and stationarity,
\begin{align}
    &\P_0\Bigg(\underset{t_i\in N^k\vert_{[0,T]}}{\max}\bigg| \frac{\tilde\lambda_{t_i}^k\big(\xi_k,a_T^{-1}\arctan\big(\frac{a_T\bog_k}{\sqrt{T}}\big)\big)}{\sqrt{T}\lambda_{t_i}^k(f^0)}\bigg| > \epsilon\Bigg) \nonumber \\
    &=\P_0\Bigg(\underset{t_i\in N^k\vert_{[0,T]}}{\max}  \frac{\tilde\lambda_{t_i}^k\big(\xi_k,a_T^{-1}\arctan\big(\frac{a_T\bog_k}{\sqrt{T}}\big)\big)^2}{\lambda_{t_i}^k(f^0)^2}\nonumber\\
    &\hspace{3cm}\times\onee\Big\{\bigg|\frac{\tilde\lambda_{t_i}^k\big(\xi_k,a_T^{-1}\arctan\big(\frac{a_T\bog_k}{\sqrt{T}}\big)\big)}{\lambda_{t_i}^k(f^0)}\bigg| > \sqrt{T}\epsilon \Big\}> T\epsilon^2\Bigg) \nonumber \\
    &\leq \P_0\Bigg(\int_0^T \frac{\Big(\tilde\lambda_{t}^k\big(\xi_k,a_T^{-1}\arctan\big(\frac{a_T\bog_k}{\sqrt{T}}\big)\big)^2}{\lambda_{t}^k(f^0_k)}\nonumber\\
    &\hspace{3cm}\times\onee\Big\{\ \bigg|\frac{\tilde\lambda_{t}^k\big(\xi_k,a_T^{-1}\arctan\big(\frac{a_T\bog_k}{\sqrt{T}}\big)\big)}{\lambda_{t}^k(f^0_k)}\bigg| > \sqrt{T}\epsilon \Big\} \frac{dN_t^k}{\lambda_t^k(f^0_k)} > T\epsilon^2 \Bigg) \nonumber \\
    &= \frac{1}{\epsilon^2}\E_0\Bigg[ \frac{\tilde\lambda_{A}^k\big(\xi_k,a_T^{-1}\arctan\big(\frac{a_T\bog_k}{\sqrt{T}}\big)\big)^2}{\lambda_{A}^k(f^0_k)}\onee\big\{\ \bigg|\frac{\tilde\lambda_{A}^k\big(\xi_k,a_T^{-1}\arctan\big(\frac{a_T\bog_k}{\sqrt{T}}\big)\big)}{\lambda_{A}^k(f^0_k)}\Big| > \sqrt{T}\epsilon \bigg\}\Bigg]  ,\nonumber 
\end{align}
and the last term goes to $0$ as $T\rightarrow+\infty$ by dominated convergence theorem which terminates to prove that  $\bar R_{T,4}(\xi,\bog) \xrightarrow[]{\P}0$. Hence, we have shown that the model $\mathcal{P}$ has the LAN property at $f^0$ when the functions $h^0_{l,k}$ are bounded away from $0$.

\textit{Extension to the general case. } We no longer assume that the functions $\boh^0$ are bounded away from $0$. Previously, this assumption has only been used to show that the path (\ref{nonlinear_path_model_lin}) is in fact equal to the path (\ref{gen_non_linear_path_relu}) for $T$ large enough, and from that we have obtained the expansion (\ref{expansion_arctan}). This assumption was not used to study the asymptotic behavior of the terms of the expansion (\ref{expansion_arctan}), namely $W_T(\xi,\bog)$, $\Delta Q_T(\xi,\bog)$ and $\bar R_{T,i}(\xi,\bog)$ ($i\in [4])$. Lemma \ref{technical_lemma_LAN} shows that, without assuming the functions $\boh^0$ are bounded away from $0$,  we still have for $(\xi,\bog)\in \RLKOp$
\begin{align}
    &L_T\Big(\big(\xi(T), \bog(T)\big)_{\mathcal{P}}\Big) - L_T(\nu^0, \boh^0) \nonumber\\
    &=  W_T(\xi,\bog) -\frac{1}{2}\Vert (\xi,\bog)\Vert_L^2 + \Delta Q_T(\xi,\bog) +  \bar R_{T,1}(\xi,\bog)+ \bar R_{T,2}(\xi,\bog) + \bar R_{T,3}(\xi,\bog)+o_{\P_0}(1),\nonumber
\end{align}
and therefore the LAN expansion is proved in the general case for model $\mathcal{P}_l$.
\end{proof}

The following lemma is a technical result used in the previous proof when it is not assumed that the functions $\boh^0$ are bounded away from $0$.

\begin{lemma}\label{technical_lemma_LAN}
Let $(\xi,\bog)\in\RLKOp$,
    \begin{align}
     L_T\Big(\big(\xi(T),\bog(T)\big)_{\mathcal{P}}\Big) -  L_T(\nu^0,\boh^0) &= W_T(\xi,\bog) - \frac{1}{2}  \Vert (\xi,\bog)\Vert_L^2 + \Delta Q_T(\xi,\bog) \nonumber\\
     &\hspace{0.3cm}+\bar R_{T,1}(\xi, \bog) + \bar R_{T,2}(\xi, \bog) + \bar R_{T,3}(\xi, \bog) +o_{\P_0}(1)\nonumber
\end{align}
\end{lemma}

\begin{proof}[Proof of Lemma \ref{technical_lemma_LAN}]  To shorten the notations, we write $\xi_T$ for $\xi/\sqrt{T}$ and $\bog_T$ for $a_T^{-1}\arctan(a_T \bog/\sqrt{T})$ and we define the vector of $K$ functions $(\boh^0_k+\bog_{T,k})\onee_{\boh^0_k<-\bog_{T,k}}$ by
\begin{align}
    \big((\boh^0_k+\bog_{T,k})\onee_{\boh^0_k<-\bog_{T,k}}\big)_l (x) = (h^0_{l,k}(x)+g_{T,l,k}(x))\onee_{h^0_{l,k}(x)<-g_{T,l,k}(x)}.\nonumber
\end{align}
 First, note that $$\lambda_t^k\big(\big(\xi(T),\bog(T)\big)_{\mathcal{P}}\big) =  \tilde \lambda_t^k (\nu^0_k +\xi_{T,k}, \boh^0_k+\bog_{T,k}) - \tilde \lambda_t^k\big(0,(\boh^0_K+\bog_{T,k})\onee_{\boh^0_k<-\bog_{T,k}}\big).$$
Then, reproducing the computations that lead to the expansion (\ref{expansion_arctan}), we obtain
\begin{align}
    &L_T\Big(\big(\xi(T),\bog(T)\big)_{\mathcal{P}}\Big) - L_T(\nu^0,\boh^0)
    = W_T(\xi,\bog) - \frac{1}{2}  \Vert (\xi,\bog)\Vert_L^2 + \Delta Q_T(\xi,\bog) \nonumber\\
    &\hspace{3cm}+\bar R_{T,1}(\xi, \bog) + \bar R_{T,2}(\xi, \bog) + \bar R_{T,3}(\xi, \bog)\nonumber\\
    &\hspace{3cm}+\bar R_{T,1}'(\xi, \bog) + \bar R_{T,2}'(\xi, \bog) + \bar R_{T,3}'(\xi, \bog)  + \bar R_{T,4}'(\xi, \bog), \nonumber
\end{align}
with 
\begin{align}
    &\bar R_{T,1}'(\xi, \bog) =\sum_{k=1}^K\int_0^T \frac{\tilde \lambda_t^k(0, (\boh^0_k+ \bog_{T,k})\onee_{\boh^0_k < -\bog_{T,k}})}{\lambda_t^k(f^0_k)} \big( dN_t^k - \lambda_t^k(f^0_k)\big) \nonumber \\
    &\bar R_{T,2}'(\xi, \bog) = -\frac{1}{2}\sum_{k=1}^K\int_0^T \frac{\tilde\lambda_t^k(0, (\boh^0_k+ \bog_{T,k})\onee_{\boh^0_k < -\bog_{T,k}})^2}{\lambda_t^k(f^0_k)} \frac{dN_t^k}{\lambda_t^k(f^0_k)} \nonumber \\
    &\bar R_{T,3}'(\xi,\bog) = \sum_{k=1}^K\int_0^T \frac{\tilde\lambda_t^k(\xi_{T,k},\bog_{T,k})\tilde\lambda_t^k(0, (\boh^0_k+ \bog_{T,k})\onee_{\boh^0_k < -\bog_{T,k}})}{\lambda_t^k(f^0_k)}\frac{dN_t^k}{\lambda_t^k(f^0_k)} \nonumber \\
    &\bar R_{T,4}'(\xi,\bog)= \sum_{k=1}^K\int_0^T \frac{\tilde\lambda_t^k(\xi_T, \bog_{T,k} - (\boh^0_k+ \bog_{T,k})\onee_{\boh^0_k <- \bog_{T,k}})^2}{\lambda_t(f^0_k)^2}\nonumber\\ &\hspace{3cm}\times R\bigg(\frac{\tilde\lambda_t^k(\xi_T, \bog_{T,k} - (\boh^0_k+ \bog_{T,k})\onee_{\boh^0_k < -\bog_{T,k}})}{\lambda_t^k(f^0_k)}\bigg)dN_t^k.\nonumber
\end{align}
Now, observe that for all $(l,k)\in\K^2$, since $h^0_{l,k}\geq0$, 
\begin{align}
    \vert h^0_{l,k} + g_{T,l,k}\vert\onee_{h^0_{l,k} <-g_{T,l,k}} \leq \vert g_{T,l,k}\vert \onee_{h^0_{l,k} <-g_{T,l,k}} .\nonumber
\end{align}
Moreover, $\arctan(x) = x\kappa(x)$  with $\kappa(0)=1$ and $\kappa$ is in $[0,1]$. The arguments used to prove that $\bar R_{T,i}(\xi,\bog)=o_{\P_0}(1)$, $i\in [4]$, also show, with the two previous remarks,  that $\bar R_{T,i}'(\xi,\bog)=~o_{\P_0}(1)$, $i\in [4]$. For the sake of completeness, we present in detail the computations for $\bar R_{T,1}'(\xi,\bog)$, one can proceed similarly for the other terms. Let $k\in \K$ and $\epsilon>0$,
\begin{align}
    &\P_0\bigg(\Big\vert \int_0^T \frac{\tilde\lambda_t^k(0, (\boh^0_k+ \bog_{T,k})\onee_{\boh^0_k < -\bog_{T,k}})}{\lambda_t^k(f^0_k)} \big( dN_t^k - \lambda_t^k(f^0_k)\big)\Big\vert >\epsilon\bigg) \nonumber\\
    &\leq \frac{T}{\epsilon^2}\E_0\bigg[\frac{\tilde\lambda_A^k(0, (\boh^0_k+ \bog_{T,k})\onee_{\boh^0_k < -\bog_{T,k}})^2}{\lambda_A^k(f^0_k)} \bigg] \nonumber \\
    &\leq \frac{T}{\epsilon^2}\E_0\bigg[\frac{\tilde\lambda_A^k\big(0, \vert \bog_k.\kappa(a_T\bog_k/\sqrt{T})\vert\onee_{\boh^0_{k} <-\bog_{T,k}}/\sqrt{T}\big)^2}{\lambda_A^k(f^0_k)} \bigg] \nonumber \\
    &\leq \frac{T}{\epsilon^2}\E_0\bigg[\frac{\tilde\lambda_A^k\big(0, \vert \bog_k\vert\onee_{\boh^0_{k} <-\bog_{T,k}}/\sqrt{T}\big)^2}{\lambda_A^k(f^0_k)} \bigg] = \frac{1}{\epsilon^2}\E_0\bigg[\frac{\tilde\lambda_A^k\big(0, \vert \bog\vert\onee_{\boh^0_{k} <-\bog_{T,k}}\big)^2}{\lambda_A^k(f^0_k)} \bigg]. \nonumber 
\end{align}
Furthermore, $arctan(0) =0$, $sgn(\arctan(x))= sgn(x)$, $\bog_{T,l,k}$ is bounded by $\pi/2a_T\rightarrow 0$ and $g_{l,k}(x)\geq 0$ when $h^0_{l,k}(x)= 0$. Whence,$\onee_{h^0_{l,k} <-g_{T,l,k}}  \xrightarrow[]{a.s.} 0$, as $T\rightarrow +\infty$. Thus, by dominated convergence theorem, $\E_0\bigg[\frac{\tilde\lambda_A^k\big(0, \vert \bog_k\vert\onee_{\boh^0_{k} <-\bog_{T,k}}\big)^2}{\lambda_A^k(f^0_k)} \bigg] \xrightarrow[]{}0$ as $T\rightarrow +\infty$ and it proves that $R_{T,1}'(\xi,\bog)=o_{\P_0}(1)$.
\end{proof}

Now,  we prove Lemma \ref{lemma_LAN_convol_RELU} on the LAN expansion in model $\mathcal{P}_R$, using what we did before for the LAN expansion in model $\mathcal{P}$ (Lemma \ref{lemma_LAN}).

\begin{proof}[Proof of Lemma \ref{lemma_LAN_convol_RELU}] We fix $\alpha >0$ and we consider $\Omega_T$ the event defined in  Lemma~\ref{omega_T}. As $\P_0(\Omega_T) \rightarrow 1$ it is enough to prove the LAN expansion on the event $\Omega_T$. Moreover, let $(\xi,\bog) \in \RLKO$, on the event $\Omega_T$ we have
\begin{align}
    &\underset{k\in \K}{\min} \underset{ t\in [0,T]}{\inf} \tilde \lambda_t^k(\nu^0_k + \xi_{T,k},  h^0_k +a_T^{-1}\arctan(a_T \bog_k/\sqrt{T}))\nonumber\\
    &\geq  \underset{k\in \K}{\min} \frac{\nu^0_k}{2} -\frac{K\pi}{2a_T}\times \underset{l\in\K}{\max}\hspace{0.05cm}\underset{t\in [0,T]}{\sup}N^l([t-A,t)] \geq \underset{k\in \K}{\min} \frac{\nu^0_k}{2} -\frac{ K\pi C_\alpha \log(T)}{2a_T}\geq0, \nonumber
\end{align}
for $T$ large enough (recall that $\log(T)/a_T\rightarrow0$). In this model, the intensity  along the path (\ref{gen_non_linear_path_relu}) directed by $(\xi,\bog)$ is given by $\lambda_t^k\big( (\xi(T), \bog(T))_{\mathcal{P}_R}\big) = \tilde \lambda_t^k\big( (\xi(T), \bog(T))_{\mathcal{P}_R}\big)_+$. But with the previous inequality, on $\Omega_T$ for $T$ large enough, $ \tilde \lambda_t^k\big( (\xi(T), \bog(T))_{\mathcal{P}_R}\big)_+$ is equal to $\tilde \lambda_t^k\big( (\xi(T), \bog(T))_{\mathcal{P}_R}\big)$. Whence, we can decompose the difference of intensities as in (\ref{intensity_difference}). Then, all the arguments that follow in the proof of Lemma \ref{lemma_LAN} for the linear model $\mathcal{P}$  are still valid and the LAN property is thus proved for the model $\mathcal{P}_R$.

Then, as in the model $\mathcal{P}$, the convolution theorem comes directly with the LAN property and the differentiability of the functional by applying Theorem  2.4  of \cite{mcneney}.
\end{proof}

\section{Proofs of Section \ref{BVM}}\label{proof_subsection_main_result}

 \subsection{Proofs of the posterior contraction results} \label{sec:pr:concentration}
 
 In this section, we prove Proposition \ref{th_concentration_L1} on $L_1$ posterior contraction and Theorem \ref{th_concentration_L2} on $L_2$ posterior contraction. The results on the $L_2$ tests used in the proof of Theorem \ref{th_concentration_L2} are presented in the following Section \ref{sec_tests}. Recall the definition of $\varepsilon_T$ and $J_T$ in (\ref{def:JTepsT}).

\begin{proof}[Proof of Proposition \ref{th_concentration_L1}]
We verify hypotheses of Proposition 3.5 of \cite{Sulem_concentration} (case 1). 

Even though our prior is not supported on the set of parameters $f$ such that $\VERT \rho_+\VERT_1<1$ (Condition (C1bis) in \cite{Sulem_concentration}),  but on the set $\mathcal{F}_R^e$ and even though we just assume $r(\rho^0)<1$ and not $\VERT \rho^0\VERT_1<1$, we show that Proposition 3.5 of \cite{Sulem_concentration}  can still be applied. First,  in the cluster representation of the Hawkes process (see \cite{reynaud_roy_hawkes}), denote by $W^k$ the number of events in a cluster with an ancestor of type $k$ (if the process considered is nonlinear, $W^k$ is the number of events in a cluster of the dominating linear process $N^+$, as in the proof of Lemma 5.2 in \cite{Sulem_concentration}). The condition $\VERT \rho^0\VERT_1<1$ is used to prove that the variables $N(B)$ (with $B$ a bounded set) and $W^k$ have exponential moments under $\P_0$ but it has been shown recently in \cite{leblanc_exp1} that $r(\rho^0)<1$ is sufficient for the existence of exponential moments under $\P_0$. Secondly,  in the proof of Lemma A.4 \cite{Sulem_concentration}, the condition $\VERT \rho_+\VERT_1<1$ is  used to upper bound $\sum_l(\rho_+)_{l,k}$ by $1$ (on p.45 of the supplementary file). In our case, by definition of $\mathcal{F}_R^e$, we can upper bound this quantity by $H$, which is enough for the conclusion of Lemma A.4 to still hold. Moreover, in this lemma, $\VERT \rho_+\VERT_1<1$ is implicitly used to ensure that $C_0'<+\infty$ in (S9.53), this is also true when $f\in \mathcal{F}^e_R$. Lastly, in \cite{Sulem_concentration}, when working under $\P_f$ (for some $f$ drawn from the prior) elsewhere, it is always conditionally on the event $\tilde \Omega_T$ (defined by (25) in \cite{Sulem_concentration}) on which the number of point of the process and the number of excursions are controlled. Thus, since the concentration inequality (7.7) in \cite{hansen_al} does not require stationarity, the results under $\P_f$ in \cite{Sulem_concentration} do not require $\VERT \rho_+\VERT_1<1$ (or even $r(\rho_+)<1$). Thus, these results in \cite{Sulem_concentration} still hold in our setting and consequently we can apply Proposition 3.5.

Next, we verify the conditions of Proposition 3.5. The condition given by Equation (8) in \cite{Sulem_concentration} is true in our case since the functions $h^0_{l,k}$ are non-negative. Then we have to verify the usual type of assumptions for posterior contraction, the so-called prior mass, sieves and entropy conditions.

We begin with the prior mass condition, we have to verify that $\Pi(B_\infty(\bar \varepsilon_T))\geq e^{-(c_4-11)T \varepsilon_T^2}$ with $c_4$ defined in \hyperref[as_P1]{(P1)} and
\begin{align}
    B_\infty(\bar \varepsilon_T) := \big\{f: \forall(l,k)\in \K^2, \nu^0_k\leq \nu_k\leq \nu^0_k +\bar\varepsilon_T, h^0_{l,k} \leq h_{l,k}\leq h^{0}_{l,k}+\bar\varepsilon_T\big\}.\nonumber
\end{align}
Since $\Pi_\nu$ has a positive continuous density on $[0,+\infty[^K$, there exists a constant $r>0$ such that $\Pi_\nu\big(\nu: \nu_k^0\leq  \nu_k\leq \nu^0+  \bar \varepsilon_T, \forall k\in \K\big) \geq r( \bar \varepsilon_T)^K$, this lower bound is greater than $\exp(-c_3(c_1+c_2)T \varepsilon_T^2)$ for $T$ large enough. Moreover, by assumption \hyperref[as_P1]{(P1)}, (\ref{cond_piJ}) and (\ref{P1_rate}) we have for $T$ large enough:
\begin{align}
    &\Pi_{\tilde \boh} \big(\tilde \boh: \forall (l,k)\in \K^2,  \tilde h^0_{l,k} \leq \tilde h_{l,k}\leq \tilde h^{0}_{l,k}+\bar\varepsilon_T/L\big)\nonumber\\
    &\geq \Pi_{\tilde \boh\vert \bar J_T} \big(\tilde \boh: \forall (l,k)\in \K^2,  \tilde h^0_{l,k} \leq \tilde h_{l,k}\leq \tilde h^{0}_{l,k}+\bar\varepsilon_T/L\big)\Pi_J(J=\bar J_T)\nonumber\\
    &\geq e^{- c_2\bar J_T\log(T)}e^{-2c_1\bar J_T\log(\bar J_T)}\geq  e^{- 2c_2c_3T\varepsilon_T^2}e^{-4c_1c_3T\varepsilon_T^2 }\geq  e^{- 4c_3(c_1+c_2)T\varepsilon_T^2}\nonumber
\end{align}
As $\varphi$ is $L$-Lipschitz and non decreasing, it proves that
\begin{align*}
    (\Pi_\nu \otimes \Pi_\boh)(B_\infty(\bar \varepsilon_T)) \geq e^{-5c_3(c_1+c_2)T\varepsilon_T^2 }=e^{-(c_4-11)T\varepsilon_T^2}.
\end{align*}
Consequently, because for $T$ large enough, $B_\infty(\bar \varepsilon_T)\subseteq \mathcal{F}_R^e$,  we have for $T$ large enough $\Pi(B_\infty(\bar \varepsilon_T))\geq e^{-(c_4-11) T\varepsilon_T^2}$.

Next, we study the sieve condition.  We define
\begin{align*}
    \mathcal{\tilde H}_{T}:=\big\{\tilde \boh = \bth^TB_j,\hspace{0.1cm} \bth\in \R^{K^2j} ,\hspace{0.1cm} j\leq J_T,\hspace{0.1cm}\Vert \bth\Vert_\infty \leq M_T\big\}\hspace{0.1cm},\hspace{0.2cm}\mathcal{H}_T:= \big\{\varphi(\tilde \boh), \tilde \boh\in \mathcal{\tilde H}_{T} \big\},
\end{align*} with $M_T$ the sequence of assumption \hyperref[as_P1]{(P1)} and $\Upsilon_T := \big\{\nu: \Vert \nu \Vert_\infty \leq e^{c_4T\varepsilon_T^2}\big\}$.  We have to verify that $\Pi_\nu(\Upsilon_T^c)+ \Pi_\boh (\mathcal{H}_{T}^c)\lesssim e^{-c_4T\varepsilon_T^2}$. For $\Pi_\nu$, by assumption we have for some $a>1$,   $\Pi_\nu(\Upsilon_T^c) \lesssim e^{-ac_4T\varepsilon_T^2}$.   Then,  with direct computations Conditions  (\ref{cond_piJ}) and (\ref{sieve_theta}), we obtain that for $T$ large enough $\Pi_\boh (\mathcal{H}_{T}^c)\leq e^{-U(M_T)} +e^{- c_1J_T\log(J_T)/2}$.  But, by assumption \hyperref[as_P1]{(P1)}, $U(M_T)\geq c_4T\varepsilon_T^2$ and $c_1J_T\log(J_T)/2\geq c_1J_0\bar J_T\log(T)/4 \geq c_4T\varepsilon_T^2  $ and therefore sieve condition is verified.

Finally, there is an entropy condition to be verified: $\log(\mathcal{N}(\zeta_0\varepsilon_T, \mathcal{H}_T, \Vert.\Vert_1))\leq x_0T\varepsilon_T^2$ for some $\zeta_0>0$ and $x_0>0$. As for any $h_1, h_2 \in \mathcal{H}_T$
\begin{align*}
    \Vert h_1 - h_2\Vert_1\leq L \Vert \tilde h_1 - \tilde h_2\Vert_1\leq L\sqrt{A}\Vert \tilde h_1 - \tilde h_2\Vert_2,
\end{align*}
it is enough to show that $\log(\mathcal{N}(\zeta_0'\varepsilon_T, \mathcal{\tilde H}_T, \Vert.\Vert_2))\leq x_0T\varepsilon_T^2$ for some $\zeta_0'>0$. Then, using (\ref{condi_rho}) and the usual formula for the entropy of a ball in $\R^{K^2j}$, we obtain that for some $c>0$ 
\begin{align}
     \mathcal{N}(\zeta_0'\varepsilon_T, \mathcal{\tilde H}_T, \Vert.\Vert_2) &\leq \sum_{j=1}^{J_T} \mathcal{N}\big(\zeta_0'\varepsilon_T, \{\tilde \boh = \bth^TB_j, \bth\in \R^{K^2j}, \Vert \bth\Vert_\infty\leq M_T\},  \Vert.\Vert_2\big)\nonumber \\
     &\leq\sum_{j=1}^{J_T} \mathcal{N}(c\zeta_0'\varepsilon_T\gamma(j)/\sqrt{j}, \{\bth\in \R^{K^2j}: \Vert \bth\Vert_\infty\leq M_T \}, \Vert.\Vert_\infty)\nonumber \\
    &\leq \sum_{j=1}^{J_T} \Big(\frac{M_T\sqrt{j}}{c\zeta_0'\gamma(j)\varepsilon_T}\Big)^{K^2j}\leq \sum_{j=1}^{J_T} \Big(\max (\gamma(j), 1/\gamma(j))\frac{M_T\sqrt{j}}{c\zeta_0'\varepsilon_T}\Big)^{K^2j}\nonumber.
\end{align}
By (\ref{condi_rho}), $\gamma(j)$ is a monotone sequence. Assume that $\gamma(j)$ is non decreasing. Then, either for some $j\geq 1$, $\gamma(j)\geq 1/\gamma(1)$ ($1/\gamma$ is non increasing) and for $T$ large enough and $j\leq J_T$ we have $\max(\gamma(j), 1/\gamma(j))\leq \gamma(J_T)$, or for all $j$, $\gamma(j)\leq1/\gamma(1)$ and $\max(\gamma(j), 1/\gamma(j))\leq~ 1/\gamma(1)$. In both cases we have $\max (\gamma(j), 1/\gamma(j))\leq \gamma(J_T) + 1/\gamma(1)$. If $\gamma(j)$ is non increasing we obtain similarly that  $\max (\gamma(j), 1/\gamma(j))\leq \gamma(1) + 1/\gamma(J_T)$. Whence, we have $\max (\gamma(j), 1/\gamma(j)) \leq \gamma(1) + 1/\gamma(J_T) + \gamma(J_T) + 1/\gamma(1):=w_T$ and by assumption $w_T$ has at most polynomial growth. Thus, for some $C>0$
\begin{align}
    \mathcal{N}(\zeta_0'\varepsilon_T, \mathcal{\tilde H}_T, \Vert.\Vert_2)&\leq J_T\Big(w_T\frac{M_T\sqrt{J_T}}{c\zeta_0' \varepsilon_T}\Big)^{K^2J_T}\leq e^{CJ_T\log(T)}.\nonumber
\end{align}
By Condition (\ref{P1_rate}), we have that $J_T\log(T)\lesssim T \varepsilon_T^2$  so for some $x_0>0$, $CJ_T\log(T)\leq~ x_0 T\varepsilon_T^2$ and the entropy condition is verified. It concludes the proof of Proposition \ref{th_concentration_L1}.
 \end{proof}

\begin{proof}[ Proof of Theorem \ref{th_concentration_L2}]

Let $\C_T= \big \{f:\Vert f-f^0\Vert_2\leq \log(T)\varepsilon_T\big\}$, we have to prove that $\Pi(\C_T^c\vert N)\rightarrow 0$ under $\P_0$. By Proposition \ref{th_concentration_L1}, we have under $\P_0$,  $\Pi(\mathcal{W}_T^1\vert N)\rightarrow 1$  and thus it is enough to prove that $\Pi(\C_T^c\cap \mathcal{W}_T^1\vert N)\rightarrow  0$ under $\P_0$.  We prove this following the well known strategy of \cite{GGV_2000}  for posterior contraction. Before applying this strategy, we make some preliminary definitions and remarks.

 When $\varphi$ is the identity function, we  set $\mathcal{F}_T = \mathcal{W}_T^1$. Let $\bar r_T:=\lceil 2DC_1\sqrt{J_T}\varepsilon_T\rceil$, where $D$ is defined in Lemma  \ref{prem_1} and $C_1$ in Proposition \ref{th_concentration_L1}, we then have $\mathcal{F}_T \subset \big\{f: \Vert \boh - \boh^0\Vert_2 \leq \bar r_T\big\}$. Indeed, it holds trivially if $ \Vert \boh-\boh^0\Vert_2 \leq \varepsilon_T$  and  otherwise, by Lemma \ref{prem_1} we have
\begin{align}
    \Vert \boh -\boh^0\Vert_2 \leq \sqrt{\Vert \boh-\boh^0\Vert_1\Vert \boh-\boh^0\Vert_\infty} &\leq \sqrt{C_1\varepsilon_T\Vert \boh-\boh^0\Vert_\infty}\nonumber\\
    &\leq \sqrt{DC_1\varepsilon_T\sqrt{J_T} \big(\Vert \boh-\boh^0\Vert_2 + \varepsilon_T\big)}\nonumber\\
    &\leq \sqrt{2DC_1\varepsilon_T\sqrt{J_T} \Vert \boh-\boh^0\Vert_2}.\nonumber
\end{align}
Dividing both sides by $\sqrt{\Vert \boh -\boh^0\Vert_2}$, we find that $\Vert \boh -\boh^0\Vert_2 \leq 2DC_1\varepsilon_T\sqrt{J_T}$ and the inclusion is verified.

When $\varphi$ is not the identity function, we define $\mathcal{F}_T$ by
\begin{align}
    \mathcal{F}_T =\big\{f\in \mathcal{W}^1_T, \Vert \boh\Vert_\infty \leq r_T\big\}\cap \big\{f: \Vert \boh -\boh^0\Vert_2\geq \log(T)r_T\varepsilon_T\hspace{0.15cm}\text{or}\hspace{0.15cm} \Vert \boh -\boh^0\Vert_\infty\leq G\big\}.\nonumber
\end{align}
and we set in this case $\bar r_T:=\lceil 2Ar_T\rceil$, for $T$ large enough $\mathcal{F}_T\subseteq \{f: \Vert f-f^0\Vert_2\leq \bar r_T\}$.

Now, whatever $\varphi$, for $r\geq \underline r_T := \lti$, we let 
\begin{align*}
    \mathcal{F}_T(r) :=\big\{f\in \mathcal{F}_T, \Vert f-f^0\Vert_2\in [r\varepsilon_T, (r+1)\varepsilon_T[ \big\}.
\end{align*}
Let $(f^{r,i})_{i\leq \mathcal{N}_r}$ be the centering points of a covering of $\mathcal{F}_T(r)$ in terms of metric $\Vert.\Vert_2$ and with radius $\varepsilon_T (r+1) \Delta/ (4\eta)$ ($\Delta$ and $\eta$ are universal positive constants  defined in Section~\ref{sec_tests}). With the entropy computations done in the proof of Proposition \ref{th_concentration_L1}, we find that for some $x_0'>0$, $\mathcal{N}_r\leq \exp(x_0'T\varepsilon_T^2)$. Finally, let
\begin{align}
    D_T = \int_{\mathcal{\bar F}} e^{L_T(f) - L_T(f^0)}d\Pi(f),\label{def_DT}
\end{align} 
and recall the events $\Omega_{\tau,T}$ and $\Xi_T$ defined in Section \ref{sec_tests} and both observable. We are now ready to apply the strategy of \cite{GGV_2000}. First, we have the classical decomposition (see also Equation (24) in \cite{Sulem_concentration}):
\begin{equation}\label{classical_decomp}
\begin{aligned}
    &\E_0[\Pi(\mathcal{C}_T^c\cap \mathcal{W}^1_T\vert N)] \\
    &\leq \P_0\big( \tilde \Omega_T\cap\{D_T\leq e^{-(c_4-1)T \varepsilon_T^2}\}\big) + e^{(c_4-1)T \varepsilon_T^2}\Pi\big(\mathcal{F}_T^c\cap \mathcal{W}^1_T\big)+\P_0(\tilde \Omega_{T}^{c}\cup \Xi_T^c) \\
    &\hspace{0.4cm}+\E_0[\phi\onee_{\Omega_{\tau,T}}] + e^{(c_4-1)T\varepsilon_T^2}\sum_{r=\underline r_T}^{\bar r_T}\int_{\mathcal{F}_T(r)} \E_0\Big[\E_f\Big[ (1-\phi)\onee_{ \Xi_T\cap \Omega_{\tau,T}}\Big \vert \mathcal{G}_{0-}\Big]\Big]d\Pi(f).
\end{aligned}
\end{equation}
with $\tilde \Omega_T\subset \Omega_{\tau,T}$ defined by (25) in \cite{Sulem_concentration} and that verifies $ \P_0(\tilde \Omega_T^c)\rightarrow 0$ and with
\begin{align}
    \phi =  \underset{\underline r_T\leq r\leq \bar r_T}{\max} \hspace{0.1cm}\underset{i\leq \mathcal{N}_r}{\max}\hspace{0.1cm}\phi_{f^{r,i}},\nonumber
\end{align} 
where the tests $\phi_{f^{r,i}}$ are defined in Lemma \ref{test_l2}. The first term on the right-hand side of (\ref{classical_decomp}) goes to $0$ as $T$ goes to $+\infty$ by Equation (26) in \cite{Sulem_concentration}. The second term is equal to $0$ when $\varphi$ is the identity function and when $\varphi$ is not the identity function, we have by assumption~\hyperref[P2]{(P2)} that this term goes to $0$ as $T$ goes to $+\infty$. Then, for the third term on the right-hand side of (\ref{classical_decomp}), we can apply Lemma \ref{sets_XiT} that gives that $\P_0(\Xi_T^c)=o(1)$ and thus $\P_0(\tilde \Omega_T^c\cup\Xi_T^c)=o(1)$.   It remains to deal with the two terms involving the test $\phi$. For the type I error term, by Lemma \ref{test_l2}, we have:
\begin{align}
    \E_0[\phi \onee_{\Omega_{\tau,T}}] &\leq \sum_{r=\underline r_T}^{+\infty}\E_0\Big[\underset{i\leq \mathcal{N}_l}{\max}\hspace{0.1cm}\phi_{f^{r,i}}\onee_{\Omega_{\tau,T}}\Big]
    \leq e^{x_0'T\varepsilon_T^2}\sum_{r=\underline r_T}^{+\infty}e^{-c T\varepsilon_T^2r}\lesssim e^{\big(x_0'-c\log(T)\big)T\varepsilon_T^2}=o(1).\nonumber
\end{align}
 For the type II error term, Lemma \ref{test_l2} gives that
\begin{align}
    e^{(c_4-1)T\varepsilon_T^2}\sum_{r=\underline r_T}^{\bar r_T}\int_{\mathcal{F}_T(r)} \hspace{-0.4cm}\E_0\Big[\E_f\Big[ (1-\phi)\onee_{ \Xi_T\cap \Omega_{\tau,T}}\Big \vert \mathcal{G}_{0-}\Big]\Big]d\Pi(f) \lesssim e^{\big((c_4-1)- c \log(T)\big)T\varepsilon_T^2}=o(1).\nonumber
\end{align}
This terminates the proof of Theorem \ref{th_concentration_L2}.
\end{proof}

\subsection{On the tests for the $L_2$ posterior contraction}\label{sec_tests}
We study in this section the "$L_2$" tests which are the main new elements to prove the $L_2$ posterior contraction (see the proof of Theorem \ref{th_concentration_L2}). These tests are based on renewal properties of the Hawkes process coming from \cite{Sulem_concentration} and \cite{Costa_renewal}. In particular, the tests statistics are built on a restricted window $A_2(T) \subset~[0,T]$, introduced in \cite{Sulem_concentration}, on which the process can be decomposed as a sum of i.i.d. point processes.

Before presenting the tests, we first recall some renewal properties of the Hawkes process and make some definitions.  The sequence of random times $(\tau_n)_{n\geq 0}$ (called renewal times) is defined by
\begin{equation}
    \tau_n = \left\{
    \begin{array}{ll}
        0 & \mbox{  if  } n=0 \\
        \inf\big\{ t>\tau_{n-1}: N([t-A,t[\neq 0, N(]t-A,t])=0\big\} & \mbox{  if  } n\geq 1
    \end{array}
\right. .\nonumber
\end{equation}
By Lemma 5.1 of \cite{Sulem_concentration}, the variables $(\tau_n)_n$ are stopping times and the point processes $(N_{\vert[\tau_n,\tau_{n+1}[})_{n\geq 1}$ are i.i.d.. Let $\bar \tau := \E_0[\tau_2 - \tau_1]$. Denote by $U^{(1)}_n$ and $U^{(2)}_n$ the two first events after $\tau_n$. By Lemma 5.1 of \cite{Sulem_concentration}, there exists $v>0$ such that for all $n\geq 1$, $\E[e^{(U^{(1)}_n-\tau_n)v}]<+\infty$. For $T>0$,  let $n_T:=\max \{ n\geq 0, \tau_n\leq T\}$ and set 
\begin{align}\label{def_omega_tau}
    \Omega_{\tau,T}:=\Big\{ n_T-1\in \Big[\frac{T}{2\bar\tau},\frac{2T}{\bar \tau}\Big] \hspace{0.2cm}\text{and}\hspace{0.2cm} U^{(1)}_n-\tau_n\geq 2v^{-1}\log(T),\hspace{0.1cm} 1\leq n\leq n_T-1\Big\}.
\end{align}
With Lemma 5.1 and Lemma 5.4 of \cite{Sulem_concentration}, $\P_0(\Omega_{\tau,T})\rightarrow 0$. Next, set $\chi_n := \min(U^{(2)}_n , \tau_{n+1})$; it is a stopping time as the minimum of two stopping times. The announced restricted window $A_2(T)$ is defined by
\begin{align}
    A_2(T) :=\bigcup_{n=1}^{n_T-1} [\tau_n,\chi_n].\nonumber
\end{align}
We introduce the stochastic distance  $d_T$:
\begin{align}
    d_T^2(f,f') :=\frac{1}{T}\sum_{k=1}^K\int_0^T\onee_{A_2(T)}(t)\tilde \lambda_t^k(f_k-f'_k)^2dt = \frac{1}{T}\sum_{k=1}^K\sum_{n=1}^{n_T-1}\int_{\tau_n}^{\chi_n}\tilde \lambda_t^k(f_k-f'_k)^2dt.\nonumber
\end{align}
Note that similarly to Equation (29) from \cite{Sulem_concentration}, we have for some universal $C>0$
\begin{align}
    d_T^2(f,f') &\leq \Vert f-f'\Vert_2^2\frac{C}{T}\sum_{n=1}^{n_T-1} (A+1+\tau_{n+1}-\tau_n )\nonumber\\
    &\leq \Vert f-f'\Vert_2^2 \frac{C}{T}\big((n_T-1)(A+1) + \tau_{n_T}-\tau_1\big).\nonumber
\end{align}
and thus on the event $\Omega_{\tau,T}$ we have for some universal $\eta>0$,
\begin{align}\label{dom_dT}
    d_T^2(f,f') \leq \eta \Vert f-f'\Vert_2^2.
\end{align}

Now, as in the proof of Theorem \ref{th_concentration_L2}, we define "sieves" $\mathcal{F}_T$ depending on $\varphi$.  When $\varphi$ is the identity function, we  set $\mathcal{F}_T = \mathcal{W}_T^1$ and we divide it into
slices
\begin{align*}
    \mathcal{F}_T(r) :=\big\{f\in \mathcal{F}_T, \Vert f-f^0\Vert_2\in [r\varepsilon_T, (r+1)\varepsilon_T[ \big\}, \hspace{0.1cm}r\geq 1.
\end{align*}
It is shown in the proof of Theorem \ref{th_concentration_L2} that it is enough to take the slices with $r\leq \bar r_T:= \lceil 2DC_1\sqrt{J_T}\varepsilon_T\rceil$ ($D$ and $C_1$ being a universal positive constants) to cover $\mathcal{F}_T$.  When $\varphi$ is not the identity function, we define $\mathcal{F}_T$ by
\begin{align}
    \mathcal{F}_T =\big\{f\in \mathcal{W}^1_T, \Vert \boh\Vert_\infty \leq r_T\big\}\cap \big\{f: \Vert \boh -\boh^0\Vert_2 \geq  \log(T)r_T\varepsilon_T\hspace{0.15cm}\text{or}\hspace{0.15cm} \Vert \boh -\boh^0\Vert_\infty\leq G\big\}.\nonumber
\end{align}
The slices $(\mathcal{F}_T(r))_{r\geq 1}$ are then defined in the same way and it is enough to take the slices with $r\leq \bar r_T:=\lceil 2Ar_T\rceil$ to cover $\mathcal{F}_T$. Next, whatever $\varphi$,  let $(f^{r,i})_{i\leq \mathcal{N}_r}$ be the centering points of a covering of $\mathcal{F}_T(r)$ in terms of metric $\Vert.\Vert_2$ and with radius $\varepsilon_T (r+1) \Delta/ (4\eta)$ for some $\Delta>0$. With the entropy computations done in the proof of Proposition \ref{th_concentration_L1}, we have that  $\mathcal{N}_r\leq e^{x_0'T\varepsilon_T^2}$ for some $x_0'>0$. Then, we define the sequence of events $(\Xi_T(r))_{r\geq 1}$ by
\begin{align}
   \Xi_T(r) = \big\{d_T^2(f^0, f^{r,i})\geq \Delta \Vert f^0 -f^{r,i}\Vert_2^2, \hspace{0.1cm}\forall i\leq \mathcal{N}_r\big\}, \quad \displaystyle \Xi_T := \bigcap_{r=1}^{\bar r_T}\Xi_T(r).\nonumber
\end{align}
 and we have the following lemma.
 
\begin{lemma}\label{sets_XiT}  Under assumptions of Theorem \ref{th_concentration_L2}, there exists $\Delta>0$ such that $ \P_0(\Xi_T^c)=o(1)$.
\end{lemma}
 This lemma is proved at the end of this section. The announced $L_2$ tests are built in the following lemma with the desired exponential decay for the type II error on the event $\Xi_T$. As in the proof of Theorem \ref{th_concentration_L2}, we set $\underline r_T =\lfloor \log(T)\rfloor$.

\begin{lemma}\label{test_l2}
     Under assumptions of Theorem \ref{th_concentration_L2}, let $\underline r_T \leq r \leq \bar r_T$, $i\leq \mathcal{N}_r$ and  $f^{r,i}$ as previously. Let also $\Delta>0$ as in Lemma \ref{sets_XiT}. Define the test $\phi_{f^{r,i}}$ by
    \begin{align}
        \phi_{f^{r,i}} = \onee\Big\{\displaystyle\sum_{k=1}^K\int_0^T \onee_{A_2(T)}\tilde \lambda_t^k(f^{r,i}_k - f^0_k)\big(dN_t^k - \lambda_t^k(f^0_k)dt\big) \geq \Delta T\Vert f^{r,i}-f^0\Vert_2^2/4\Big\}. \nonumber
    \end{align}
    Then,  whatever the link function $\varphi$, there exists $T_0>0$, $c>0$, both independent of $f^{r,i}$,  such that for $T\geq T_0$  we have:
    \begin{align}
        &\E_0[\onee_{\Omega_{\tau,T}}\phi_{f^{r,i}}] \lesssim e^{-c T\varepsilon_T^2r}, \nonumber \\
        &\underset{f:\Vert f - f^{r,i}\Vert_2\leq \frac{\Delta \Vert f^{r,i} -f^0\Vert_2}{4\eta}}{\sup}\E_0\big[\E_f\big[\onee_{\Omega_{\tau,T} \cap  \Xi_T(r)}(1-\phi_{f^{r,i}})\big \vert \mathcal{G}_{0-}\big]\big] \lesssim e^{-c T\varepsilon_T^2\log(T)}.\nonumber
    \end{align}
\end{lemma}

\begin{proof}[Proof of Lemma \ref{test_l2}]
We first bound the type I error.
For $1\leq n\leq n_T -1$, we set
\begin{align}
    W_n^k &: =\int_{\tau_n}^{\chi_n} \tilde \lambda^k_t(f^{r,i}_k-f^0_k)(dN_t^k -\lambda_t^k(f^0_k)dt)\nonumber\\
    &=-\int_{{\tau_n}}^{\chi_n}\tilde \lambda^k_t(f^{r,i}_k-f^0_k)\lambda_t^k(f^0_k)dt + (\nu^{r,i}_k-\nu^0_k) + \onee_{\{\tau_{n+1}> U^{(2)}_n, N^k(\{U^{(2)}_n\})=1\}}\nonumber\\
    &\hspace{2cm}\times\Big((\nu^1_k-\nu^0_k) + \sum_{l=1}^K\onee_{\{N^l(\{U^{(1)}_n\})=1\}}(h^{r,i}_{l,k} -h^0_{l,k})(U^{(2)}_n -U^{(1)}_n)\Big).\nonumber
\end{align}
The variables $ (\sum_{k=1}^K W_n^k)_{n\geq 1}$ are i.i.d. and centered by Lemma \ref{W_centered}. Thus, the process $(\sum_{n=1}^{\bar n} \sum_{k=1}^K W_n^k)_{\bar n\geq 1}$  is a centered martingale with respect to the filtration generated by the $(W_n^1,...,W_n^K)_{n\geq 1}$.  Moreover, we have
\begin{align}
    \P_0\big(\Omega_{\tau,T}\cap \{\phi_{f^{r,i}}=1\} \big)&= \P_0\bigg(\Omega_{\tau,T}\cap \Big\{\sum_{n=1}^{n_T-1}\sum_{k=1}^K W_n^k \geq T\Delta \Vert f^{r,i} -f^0\Vert_2^2/4\Big\}\bigg)\nonumber\\
    &\leq \P_0\bigg(\Omega_{\tau,T}\cap \Big\{\underset{\bar n\leq 2T/\bar\tau}{\max }\sum_{n=1}^{\bar n}\sum_{k=1}^K W_n^k \geq  T\Delta \Vert f^{r,i} -f^0\Vert_2^2/4\Big\}\bigg).\nonumber
\end{align}
To obtain the exponential decay of the type I error, we apply the Bernstein inequality for martingales of \cite{van_zanten} to the right-hand side of last inequality. To do so, we have to bound on $\Omega_{\tau,T}$ the random variables $(\sum_{k=1}^K W_n^k)_{n}$ and also, since these variables are i.i.d. and centered,  to upper bound the variance of their sum which is $\bar n\E_0\big[\big(\sum_{k=1}^K W_1^k\big)^2\big]$.  First, with similar computations as in Equation (29) of \cite{Sulem_concentration}, for $k\in \K$ we have that
\begin{equation}\label{absolute_bound_W}
\begin{aligned}
    \vert W_1^k\vert &\leq 2\Vert f^{r,i}-f^0\Vert_\infty + \int_{\tau_1}^{\chi_1}\vert \tilde \lambda_t^k(f^{r,i}_k-f^0_k)\vert \lambda_t^k(f^0_k)dt\\
    &\leq 2\Vert f^{r,i}-f^0\Vert_\infty  + 2(A+1+U^{(1)}_1-\tau_1) \Vert f^0\Vert_\infty\Vert f^{r,i}-f^0\Vert_\infty.
\end{aligned}
\end{equation}
In particular, we have on the event $\Omega_{\tau,T}$, for all $k\in \K$ and $n\leq n_T-1$,
\begin{align}
    \big\vert \sum_{k=1}^K W_n^k\big \vert\lesssim\log(T)\Vert f^0\Vert_\infty\Vert f^{r,i}-f^0\Vert_\infty.\nonumber
\end{align}
Secondly, for the variance, we have
\begin{align}
    \E_0[(W_1^k)^2]&\lesssim  \E_0\bigg[\Big(\int_{\tau_n}^{\chi_n}\tilde \lambda_t^k(f^{r,i}_k-f^0_k)\lambda_t^k(f^0_k)dt\Big)^2\bigg]+ (\nu^1_k -\nu^0_k)^2\nonumber\\
    &\hspace{0.5cm}+ \sum_{l=1}^K\E_0\Big[(h^{r,i}_{l,k} - h^0_{l,k})^2(U^{(2)}_1 -U^{(1)}_1)\onee_{\{\tau_{2}> U^{(2)}_1,N^l(\{U^{(1)}_1\})=1, N^k(\{U^{(2)}_1\})=1\}}\Big].
\end{align}
With Equation (30) of \cite{Sulem_concentration}, we further have 
\begin{align}
    E_0\bigg[\Big(\int_{\tau_n}^{\chi_n}\tilde \lambda_t^k(f^{r,i}_k-f^0_k)\lambda_t^k(f^0_k)dt\Big)^2\bigg] \lesssim \Vert f^0\Vert_\infty^2\Vert f^{r,i}-f^0\Vert_2^2.\nonumber
\end{align}
Then, by the cluster representation of the Hawkes process (see \cite{hawkes_oakes}), $U^{(2)}_1$ is either an immigrant point and, conditionally on the event "$U^{(2)}_1$ is marked by $k$", $U^{(2)}_1- U^{(1)}_1$ follows an exponential distribution with parameter $\nu^0_k$; or $U^{(2)}_1$ is an offspring of $U^{(1)}_1$ and conditionally and on the event "$\tau_{2}> U^{(2)}_1$, $U^{(1)}_1$ is marked by $l$ and $U^{(2)}_1$ is marked by $k$", $U^{(2)}_1- U^{(1)}_1$ follows a distribution with a density proportional to $x\mapsto h^0_{l,k}(x)e^{-\int_0^x h_{l,k}^0(u)du}$. Whence, for some finite $c>0$ we have
\begin{align}
     &\sum_{l=1}^K\E_0\Big[(h^{r,i}_{l,k} - h^0_{l,k})^2(U^{(2)}_j -U^{(1)}_j)\onee_{\{\tau_{2}> U^{(2)}_1,N^l(\{U^{(1)}_1\})=1, N^k(\{U^{(2)}_1\})=1\}}\Big] \nonumber\\
     &\leq c\Vert f^0\Vert_1\Vert \boh^{r,i} - \boh^0\Vert_2^2,\nonumber
\end{align}
and therefore  $\E_0[(W_1^k)^2]\leq(1+ c\Vert f^0\Vert_1+ \Vert f^0\Vert_\infty^2)\Vert f^{r,i} -f^0\Vert_2^2$. We can now apply Theorem~3.3 of \cite{van_zanten} which gives that for some $c'>0$ that depends only on $K$ and $\bar\tau $:
\begin{align}
    &\P_0\big(\Omega_{\tau,T}\cap \{\phi_{f^{r,i}}=1\} \big)\nonumber\\
    &\leq 2\exp\bigg(\frac{-c'\Delta^2T\Vert f^{r,i}-f^0\Vert_2^2 }{(1+ c\Vert f^0\Vert_1+\Vert f^0\Vert_\infty^2) +\Delta\log(T)\Vert f^0\Vert_\infty\Vert f^{r,i}-f^0\Vert_\infty} \bigg).\nonumber
\end{align}
If $\Vert f^{r,i}-f^0\Vert_\infty \leq (1+c\Vert f^0\Vert_1+ \Vert f^0\Vert_\infty^2)/(\Delta\Vert f^0\Vert_\infty \log(T))$, as $l\geq \log(T)$, then we have for $T$ large enough
\begin{align}
    \P_0\big(\Omega_{\tau,T}\cap\{\phi_{f^{r,i}}=1\} \big)\leq 2\exp\bigg(-\frac{c'\Delta^2 T\varepsilon_T^2r^2}{2(1+c\Vert f^0\Vert_1+ \Vert f^0\Vert_\infty^2)}\bigg)\leq \exp\bigg(-\Delta^2T\varepsilon_T^2r\bigg).\nonumber
\end{align}
Now, we consider that $\Vert f^{r,i}-f^0\Vert_\infty \geq (1+ c\Vert f^0\Vert_1+\Vert f^0\Vert_\infty^2)/(\Delta\Vert f^0\Vert_\infty \log(T))$ (and in particular for $T$ large enough, $\Vert f^{r,i}-f^0\Vert_\infty \geq \varepsilon_T$). In this case we have
\begin{align}
    \P_0\big(\Omega_{\tau,T}\cap \{\phi_{f^{r,i}}=1\} \big)&\leq 2\exp\bigg(\frac{-c'\Delta T\Vert f^{r,i}-f^0\Vert_2^2 }{2\log(T)\Vert f^0\Vert_\infty\Vert f^{r,i}-f^0\Vert_\infty} \bigg).\nonumber
\end{align}
When $\varphi$ is  the identity function, with Lemma \ref{prem_1} we find
\begin{align}
    \P_0\big(\Omega_{\tau,T}\cap\{\phi_{f^{r,i}}=1\} \big)&\leq 2\exp\bigg(-\frac{c'\Delta T\Vert f^{r,i}-f^0\Vert_2^2}{4D\Vert f^0\Vert_\infty\log(T)\sqrt{J_T}\Vert f^{r,i}-f^0\Vert_2}\bigg) \nonumber \\
    &\leq 2\exp\bigg(-\frac{c'\Delta T\varepsilon_T r }{4D\Vert f^0\Vert_\infty\log(T)\sqrt{J_T}}\bigg)\nonumber\\
    &\leq 2\exp\bigg(-\frac{c'\Delta T\varepsilon_T^2 r }{4D\Vert f^0\Vert_\infty\log(T)\sqrt{J_T}\varepsilon_T}\bigg) \leq 2\exp\big(-c'\Delta T\varepsilon_T^2 r \big), \nonumber
\end{align}
for $T$ large enough since $\log(T)\sqrt{J_T}\varepsilon_T\rightarrow 0$. When  $\varphi$ is not the identity function, we recall that by definition of $\mathcal{F}_T(r)$, we have $\Vert  \boh^{r,i}\Vert_\infty \leq r_T$ and if $\Vert \boh^{r,i} - \boh^0\Vert_2\leq \log(T)r_T\varepsilon_T$ then $\Vert \boh^{r,i} -\boh^0\Vert_\infty \leq G$. So, if $l\geq \log(T) r_T$ we have
\begin{align}
    \P_0\big(\Omega_{\tau,T}\cap\{\phi_{f^{r,i}}=1\} \big)\leq 2\exp\bigg(-\frac{c'\Delta T\varepsilon_T^2r^2}{2\Vert f^0\Vert_\infty\log(T)r_T}\bigg)  \leq  2\exp\bigg(-\frac{c'\Delta^2 T\varepsilon_T^2r}{2\Vert f^0\Vert_\infty}\bigg) ,\nonumber 
\end{align}
and if $\lfloor\log(T)\rfloor\leq r\leq  \log(T)r_T$,
\begin{align}
    \P_0\big(\Omega_{\tau,T}\cap\{\phi_{f^{r,i}}=1\} \big)\leq 2\exp\Big(-\frac{c'\Delta T\varepsilon_T^2r^2}{2G\Vert f^0\Vert_\infty\log(T)}\Big)\leq 2\exp\bigg(-\frac{c'\Delta^2 T\varepsilon_T^2r}{2G\Vert f^0\Vert_\infty}\bigg).\nonumber 
\end{align}
It terminates the proof for the type I error.

For the type II error, let $f\in \mathcal{F}_T(r)$ be such that $\Vert f^{r,i}-f\Vert_2\leq \Vert f^{r,i}-f^0\Vert_2\Delta/(4\eta)$. Let $\bar \Omega_T = \Xi_T(r)\cap\Omega_{\tau,T}$ (with $ \Xi_T(r)$ and $\Omega_{\tau,T}$ defined at the start of Section \ref{sec_tests}), we have
\begin{align}
    &\P_f\big(\bar \Omega_T\cap\{\phi_{f^{r,i}}=0\}\big \vert \mathcal{G}_{0-}\big)\nonumber\\
    &=\P_f\bigg(\bar \Omega_T \cap\Big\{\sum_{k=1}^K\int_0^T\onee_{A_2(T)}\tilde \lambda_t^k(f^{r,i}_k-f^0_k)(dN_t^k-\lambda_t^k(f^0_k)dt)\leq \frac{\Delta}{4} T\Vert f^{r,i}-f^0\Vert_2^2\Big\}\Big \vert \mathcal{G}_{0-}\bigg) .\nonumber 
\end{align}
Recall that $\P_f$ is the distribution of a stationary ReLU Hawkes process whose intensity is given by $\lambda_t(f_k) = \tilde \lambda_t(f_k)_+$. Recall also that by definition of $\mathcal{F}_R^e$ we have  for all $k\in \K$, $\nu_k -\max_{l\in \K} \Vert h^-_{l,k}\Vert_\infty>0$. Let $n\geq 1$ and $t\in [\tau_n, \chi_n]$. Assume without loss of generality that $ U^{(1)}_N$ is marked by $l$, then we obtain:
\begin{align}
    \lambda^k_t(f_k) = \Big(\nu_k + \onee_{t>U^{(1)}_n}h_{l,k}(t-U^{(1)}_n)\Big)_+ = \nu_k + \onee_{t>U^{(1)}_n}h_{l,k}(t-U^{(1)}_n = \tilde \lambda_t^k(f_k).\nonumber
\end{align}
On the event $\Xi_T(r)$, with the previous remark, Cauchy-Schwarz inequality and  with the inequality (\ref{dom_dT}) we have:
\begin{align}
    &\frac{\Delta}{4}T\Vert f^{r,i}-f^0\Vert_2^2 - \sum_{k=1}^K\int_0^T\onee_{A_2(T)}\tilde \lambda_t^k(f^{r,i}_k-f^0_k)( \lambda_t^k(f_k)-\lambda_t^k(f^0_k))dt\nonumber\\
    &=\frac{\Delta}{4}T\Vert f^{r,i}-f^0\Vert_2^2 - \sum_{k=1}^K\int_0^T\onee_{A_2(T)}\tilde \lambda_t^k(f^{r,i}_k-f^0_k)\tilde  \lambda_t^k(f_k-f^0_k)dt\nonumber\\
    &=\frac{\Delta}{4} T\Vert f^{r,i}-f^0\Vert_2^2 -Td_T^2(f^{r,i},f^0) - \sum_{k=1}^K\int_0^T\onee_{A_2(T)}\tilde \lambda_t^k(f^{r,i}_k-f^0_k) \tilde \lambda_t^k(f_k-f^{r,i}_k)dt\nonumber\\
    &\leq \frac{\Delta}{4} T\Vert f^{r,i}-f^0\Vert_2^2 -  T\Delta \Vert f^{r,i}-f^0\Vert_2^2 + T\sqrt{d_T^2(f^{r,i},f^0)d_T^2(f^{r,i},f)}\nonumber\\
    &\leq -\frac{3\Delta}{4} T\Vert f^{r,i}-f^0\Vert_2^2 + T\eta \Vert f^{r,i}-f^0\Vert_2\Vert f-f^{r,i}\Vert_2
    \leq -\frac{\Delta}{2}T\Vert f^{r,i}-f^0\Vert_2^2,\nonumber
\end{align}
as soon as $\Vert f-f^{r,i}\Vert_2\leq \Delta \Vert f^{r,i}-f^0\Vert_2/(4\eta)$.
Let
\begin{align}
    Y^k_n :=-\int_{\tau_n}^{\chi_n}\tilde \lambda_t^k(f^{r,i}_k-f^0_k) \big(dN_t^k -\lambda_t^k(f_k)dt\big).\nonumber
\end{align}
Then the above inequality, together with the inclusion
\begin{equation}
    \Omega_{\tau,T}\subset \big\{U^{(1)}_n-\tau_n\geq 2v^{-1}\log(T),\hspace{0.1cm} 1\leq n\leq 2T/\bar \tau\big\},\nonumber
\end{equation} leads to
\begin{align}
    &\P_f\big(\bar \Omega_T \cap\{\phi_{f^{r,i}}=0\}\big \vert \mathcal{G}_{0-}\big)\nonumber\\
    &\leq\P_f\bigg(\bar \Omega_T \cap \Big\{\sum_{k=1}^K\int_0^T\hspace{-0.3cm}\onee_{A_2(T)}\lambda_t^k(f^{r,i}_k-f^0_k)(dN_t^k-\lambda_t^k(f_k)dt)\leq -\frac{\Delta}{2} T\Vert f^{r,i}-f^0\Vert_2^2\Big\}\Big \vert \mathcal{G}_{0-}\bigg) \nonumber \\
    &\leq\P_f\bigg(\Omega_{\tau,T}\cap \Big\{\sum_{k=1}^K\sum_{n=1}^{n_T-1} Y^k_n\geq \frac{\Delta}{2} T\Vert f^{r,i}-f^0\Vert_2^2\Big\}\Big \vert \mathcal{G}_{0-}\bigg) \nonumber \\
    &\leq \sum_{k=1}^K\P_f\bigg(\big\{U^{(1)}_n-\tau_n\geq 2v^{-1}\log(T),\hspace{0.1cm} 1\leq n\leq 2T/\bar \tau\big\}\nonumber\\
    &\hspace{5cm}\bigcap \Big\{\underset{\bar n\leq 2T/\bar\tau}{\max}\sum_{n=1}^{\bar n} Y^k_n\geq \frac{\Delta}{2K} T\Vert f^{r,i}-f^0\Vert_2^2\Big\}\Big \vert \mathcal{G}_{0-}\bigg)\nonumber\\
    &= \sum_{k=1}^K\P_f\bigg(\big\{U^{(1)}_n-\tau_n\geq 2v^{-1}\log(T),\hspace{0.1cm} 1\leq n\leq 2T/\bar \tau\big\}\nonumber\\
    &\hspace{5cm}\bigcap \Big\{\underset{\bar n\leq 2T/\bar\tau}{\max}\sum_{n=1}^{\bar n} Y^k_n\geq \frac{\Delta}{2K} T\Vert f^{r,i}-f^0\Vert_2^2\Big\}\bigg),\nonumber 
\end{align}
and for the last equality we have used that for $n\geq 1$, the variables $Y_n^k$ and $U^{(1)}_n-\tau_n$ are independent of $\mathcal{G}_{0-}$. Now, we can proceed for the $(Y^k_n)_n$ under $\P_f$ exactly as for the $(W^k_n)_n$ under $\P_0$ and with Theorem 3.3 of \cite{van_zanten} we find that some $c''$ that depends only on $K$ and $\bar\tau$ we have:
\begin{align}
    &\P_f\big(\bar \Omega_T \cap\{\phi_{f^{r,i}}=0\}\big \vert \mathcal{G}_{0-}\big)\nonumber\\
    &\leq 2\exp\bigg(\frac{-c''\Delta^2T\Vert f^{r,i}-f^0\Vert_2^2 }{(1+ c\Vert f\Vert_1+ \Vert f\Vert_\infty^2) +\Delta\log(T)\Vert f\Vert_\infty\Vert f^{r,i}-f^0\Vert_\infty} \bigg)\nonumber
\end{align}
This last upper bound can be further upper bounded using that $\Vert f\Vert_\infty\leq \Vert f-f^0\Vert_\infty + \Vert f^0\Vert_\infty$. As $\Vert f-f^0\Vert_1\leq C_1\varepsilon_T$ by definition of $\mathcal{F}_T(r)$, we have for $T$ large enough $c\Vert f\Vert_1\leq 2c\Vert f^0\Vert_1$. Then, we distinguish cases and follow the same steps as for type I error.  In particular, when $\varphi$ is not the identity function, we bound $\Vert f-f^0\Vert_\infty$ and $\Vert f^{r,i}-f^0\Vert_\infty$ by $G$ if $r\leq \log(T)r_T$ and by $r_T$ otherwise, and we obtain the result in this case. When $\varphi$ is the identity function, we use Lemma \ref{prem_1} to upper bound $\Vert f -f^0\Vert_\infty$ and $\Vert f^{r,i}-f^0\Vert_\infty$ in terms of their $L_2$ norms and then, using that $\Vert f-f^{r,i}\Vert_2\leq \Delta \Vert f^{r,i}-f^0\Vert_2/(4\eta)$ and that $\varepsilon^2_T J_T\log(T)\rightarrow 0$,  we obtain also the result. It concludes the proof for the type II error.
\end{proof}

The two following lemmas are technical lemmas used in the proof of Theorem \ref{th_concentration_L2}.

\begin{lemma}\label{prem_1}
    Let $f=(\nu,\varphi(\tilde \boh))$ with $\tilde \boh \in \mathcal{\tilde H}(j)^{K^2}$ for some $j\in \mathcal{J}_T$ and such that $\Vert \nu -~\nu^0\Vert_1\leq C_1\varepsilon_T$. There exists some $D>0$ independent of $f$ such that for $T$ large enough,
\begin{align}
    \Vert \tilde \boh-\tilde \boh^0\Vert_\infty \leq D\sqrt{J_T}\big( \Vert \tilde \boh-\tilde \boh^0\Vert_2 +\varepsilon_T\big)\hspace{0.2cm},\hspace{0.3cm}\Vert  f- f^0\Vert_\infty \leq D\sqrt{J_T}\big( \Vert f-\tilde f^0\Vert_2 + \varepsilon_T\big).\nonumber
\end{align}
\end{lemma}

\begin{proof}[Proof of Lemma \ref{prem_1}] 

Since for all $j$, $\mathcal{\tilde H}(j)\subset \mathcal{\tilde H}(Rj)$ (see Section \ref{subsec_prior_construct}), then for all $\tilde \boh \in \mathcal{\tilde H}(j)^{K^2}$ for some $j\in \mathcal{J}_T$, we have   $\tilde \boh\in \mathcal{\tilde H}(m J_T)^{K^2}$, for some $m\in [R]$. Then, by Assumption \hyperref[as_P1]{(P1)}, there exists $\tilde \boh^*\in \mathcal{\tilde H}(m J_T)^{K^2}$ such that for $T$ large enough $\Vert \tilde \boh^* -\tilde \boh^0\Vert_\infty \leq  \varepsilon_T$.  With in addition (\ref{condi_sup}), we obtain
\begin{align}
    \Vert \tilde \boh-\tilde \boh^0\Vert_\infty &\leq \Vert \tilde \boh - \tilde \boh^*\Vert_\infty + \Vert\tilde \boh^* - \tilde \boh^0\Vert_\infty\nonumber\\
    &\leq \sqrt{m}\sqrt{J_T}\Vert \tilde \boh- \tilde \boh^*\Vert_2 + \varepsilon_T\nonumber \\
    &\leq \sqrt{m}\sqrt{J_T}\Vert \tilde \boh- \tilde \boh^0\Vert_2 + \sqrt{m}\sqrt{J_T}\Vert \tilde \boh^*-\tilde \boh_0\Vert_2 + \varepsilon_T\nonumber  \\
    &\leq  \sqrt{m}\sqrt{J_T}\Vert \tilde \boh- \tilde \boh^0\Vert_2 +   (\sqrt{Am}\sqrt{J_T}+ 1)\varepsilon_T.\nonumber 
\end{align}
Thus, there exists $ \underline{D}>0$  independent of $\tilde \boh$ such that $\Vert  \tilde \boh - \tilde \boh^0\Vert_\infty \leq \underline D\sqrt{J_T}\big( \Vert \tilde \boh-\tilde \boh^0\Vert_2 + \varepsilon_T\big)$. Since $\varphi$ is Lipschitz and  since $\Vert \nu -\nu^0\Vert_1 \leq C_1\varepsilon_T$, there exists $D\geq \underline D$ such that $ \Vert  f-~ f^0\Vert_\infty \leq ~D\sqrt{J_T}\big( \Vert f- f^0\Vert_2 + \varepsilon_T\big)$.
\end{proof}

\begin{lemma}\label{W_centered}
 The variables $(W_n^k)_{k\in \K, n\geq 1}$ introduced in the proof of Lemma \ref{test_l2} are centered in expectation.
\end{lemma}

\begin{proof}[Proof of Lemma \ref{W_centered}] The variables $(W_n^k)_{n\geq 1}$ are i.i.d. so it is enough to show that $\E_0[W_1^k]=0$. To do so, let $\mathcal{N}$ be a Hawkes process with parameters $f^0$ and born in $0$,  that is a $\K$-marked process with $\mathcal{N}(]-\infty,0[)=0$ and for $t\geq 0$ the intensity is given by,
\begin{equation}\label{intensity_M}
    \lambda_t^k(f^0,\mathcal{N}) = \nu^0_k + \sum_{l=1}^K\int_0^{t-} h_{l,k}^0(t-u)d\mathcal{N}^l_u.
\end{equation}
For $a\geq 0$, we introduce
\begin{align}
    \bar W_a^k(\mathcal{N})  = \int_{0}^{a} \tilde \lambda^k_t(f^{r,i}_k-f^0_k, \mathcal{N})(d\mathcal{N}_t^k -\lambda_t^k(f^0_k,\mathcal{N})dt),\nonumber
\end{align}
where $ \tilde \lambda_t^k(f^{r,i}_k-f^0_k, \mathcal{N})$ is given by (\ref{intensity_M}) but with $f^0$ replace by $f^{r,i}-f^0$. $(\bar W_a^k(\mathcal{N}))_{a\geq 0} $ is a cadlag martingale with respect to the filtration generated by $\mathcal{N}$. Moreover, by the renewal property, $W_1^k$ is distributed as $W_{\chi'}^k(\mathcal{N})$ where $\chi'$ is the stopping time equal to the minimum between the time of the second point of $M$ and the first renewal time of $\mathcal{N}$.  So it is equivalent to prove that $\E[W_{\chi'}^k(\mathcal{N})]=0$. By the inequality (\ref{absolute_bound_W}), $\E[\vert  W_{\chi'}^k(\mathcal{N}) \vert ] <+\infty$. Moreover, in the same way as $\chi_1$ has exponential moments by Lemma 5.2 of \cite{Sulem_concentration} (which can be applied in our setting, see the start of the proof of Proposition \ref{th_concentration_L1}), $\chi'$ has exponential moments in the sense that there exists $\delta>0$ such that $\E[e^{\chi'\delta}]<+\infty$. As a consequence
\begin{align}
    \E\big[\vert \bar W_a^k(\mathcal{N})\vert \onee_{\chi'>a}\big] &\lesssim \Vert f^{r,i}_k-f^0_k\Vert_\infty(1+a\Vert f^0_k\Vert_\infty) \P(\chi'>a)\nonumber\\
    &\lesssim \Vert f^{r,i}_k-f^0_k\Vert_\infty(1+a\Vert f^0_k\Vert_\infty) e^{-a\delta}\xrightarrow[a\rightarrow +\infty]{} 0.\nonumber
\end{align}
Hence, with these remarks we can apply the martingale optional sampling theorem (see Theorem 2.13 of \cite{markov_pro} for instance) and we have that $\E[W_{\chi'}^k(\mathcal{N})] = E[W_0^k(\mathcal{N})]=0$. It concludes the proof of Lemma \ref{W_centered}.
\end{proof}

Finally, we prove Lemma \ref{sets_XiT} on the stochastic distance $d_T$.

\begin{proof}[Proof of Lemma \ref{sets_XiT}]
Recall $\Omega_{\tau, T}$  defined in (\ref{def_omega_tau}). We recall that for some $x_0'>0$, $\mathcal{N}_r\leq \exp(x_0'T\varepsilon_T^2)$ (see the proof of Theorem \ref{th_concentration_L2}). First, we have
\begin{equation}\label{init_xil}
\begin{aligned}
    &\P_0(\Xi_T^c)\leq \P_0(\Omega_{\tau,T}\cap \Xi_T^c) + P_0(\Omega_{\tau,T}^{c}) \leq \sum_{r=\underline r_T}^{\bar r_T}\P_0(\Omega_{\tau,T}\cap\Xi_T(r)^c) + o(1)  \\
    &\hspace{0.4cm}\leq \exp(x_0'T\varepsilon_T^2)\sum_{r=\underline r_T}^{\bar r_T}\max_{i\leq \mathcal N_r}\P_0\Big( \Omega_{\tau,T}\cap \{d_T^2(f^0, f^{r,i}\big) \leq \Delta\Vert f^0-f^{r,i}\Vert_2\} \Big) + o(1).
\end{aligned}
\end{equation}
For $\underline r_T \leq r\leq \bar r_T$, $i\leq \mathcal N_r$,  $k\in \K$ and $1\leq n\leq n_T-1$, we set 
\begin{align}
    X_n^k = \int_{\tau_n}^{\chi_n}\lambda_t^k(f^0_k-f^{r,i}_k)^2dt,\nonumber
\end{align}
with the variables $\tau_n$ and $\chi_n$ that are defined at the start of Section \ref{sec_tests}. Let $\Delta>0$ that will be fixed later, the variables $(X_n^1,...,X_n^K)_n$ are i.i.d. and
\begin{align}
    &\P_0\big(\Omega_{\tau,T} \cap \big\{d_T^2(f^0, f^{r,i} )
    \leq \Delta\Vert f^0-f^{r,i_*}\Vert_2 \}\big) \nonumber\\
    &= \P_0\bigg(\Omega_{\tau,T}\cap \Big\{\frac{1}{T}\sum_{n=1}^{n_T-1}\sum_{k=1}^K X_n^k \leq \Delta\Vert f^0-f^{r,i}\Vert_2^2\Big\}\bigg) \nonumber\\
    &\leq \P_0\bigg(\Omega_{\tau,T} \cap \Big\{\frac{1}{T}\sum_{n=1}^{n_T-1}\sum_{k=1}^K X_n^k -\E_0[X_n^k ]\leq \Delta\Vert f^0-f^{r,i_*}\Vert_2^2 - \frac{n_T-1}{T}\sum_{k=1}^K \E_0[X_1^k]\Big\}\bigg). \nonumber
\end{align}
Now, we lower bound $\E_0[X^k_1]$. For this, let $\mathbb{Q}$  be the $K$-marked process consisting in $K$ independent component each being a Poisson process on $[0;+\infty[$ with constant intensity equal to $1$. For each $l\in \K$ set
\begin{align}
    \Omega_l &= \Big\{\underset{l'\neq l}{\max} N^{l'}([\tau_1,\tau_2[) =0, N^l([\tau_1,\tau_1+x[)=0, N^l([\tau_1+x,\tau_1+x+A])=1,\nonumber\\
    &\hspace{1cm} N^l([\tau_1+x+A,\tau_2[)=0 \Big\},\nonumber
\end{align}
with $x\in ]0,A[$.  $\mathbb{Q}(\Omega_l)>0$, see page 46 of the supplementary  of \cite{Sulem_concentration}. Mimicking the proof of Lemma A.4 of \cite{Sulem_concentration}, we find that for all $k\in \K$, there exists $C_k>0$ such that
\begin{align*}
    \E_0[X_1^k] &\geq C_k\sum_{l=1}^K\E_{\mathbb{Q}}\bigg[\onee_{\Omega_l}\int_{U_1^{(1)}}^{ U_1^{(1)}+A} \Big(\nu_k^{r,i}-\nu_0^k + (h_{l,k}^{r,i} - h^0_{l,k})(t- U_1^{(1)})\Big)^2dt\bigg]\\
    &= (\nu_k^{r,i}-\nu_0^k)^2AC_k\sum_{l=1}^K\mathbb{Q}(\Omega_l)\nonumber\\
    &\hspace{1cm}+C_k\sum_{l=1}^K \E_{\mathbb{Q}}\bigg[\onee_{\Omega_l}\int_{U_1^{(1)}}^{ U_1^{(1)}+A} (h_{l,k}^{r,i} - h^0_{l,k})^2(t- U_1^{(1)})dt\bigg]\\
    &\hspace{1cm}+ 2(\nu_k^{r,i}-\nu_0^k)C_k\sum_{l=1}^K E_{\mathbb{Q}}\bigg[\onee_{\Omega_l}\int_{U_1^{(1)}}^{ U_1^{(1)}+A} (h_{l,k}^{r,i} - h^0_{l,k})(t- U_1^{(1)})dt\bigg].
\end{align*}
Then, with the same computations as at the end of the page 46 of the supplementary of \cite{Sulem_concentration}, we find that there exists a universal constant $c>0$ such that
\begin{align}
    \E_0[X_1^k]\gtrsim (\nu_k^{r,i}-\nu_0^k)^2 + \Vert \boh_k^{r,i}-\boh^0_k\Vert_2^2 - c\vert \nu_k^{r,i}-\nu_0^k\vert\times  \Vert \boh_k^{r,i}-\boh^0_k\Vert_1.\nonumber
\end{align}
Moreover, since $\mathcal{F}_T(r)\subset\mathcal{W}^1_T$, we have $\Vert f^{r,i}-f^0\Vert_2 \geq r \varepsilon_T$, $\Vert f^{r,i} - f^0\Vert_1\leq C_1\varepsilon_T$ and thus
\begin{align}
    \vert \nu_k^{r,i}-\nu_0^k\vert\times  \Vert \boh_l^{r,i}-\boh^0_k\Vert_1 \leq C_1^2\varepsilon_T^2 \leq \frac{C_1^2}{r^2}\Vert f^{r,i}-f^0\Vert_2^2 \leq \frac{C_1^2}{\lfloor\log(T)\rfloor^2}\Vert f^{r,i}-f^0\Vert_2^2,\nonumber
\end{align}
 where for the last inequality we have used that  $r\geq \lfloor \log(T) \rfloor$. So, for some universal $d>0$ and $T$ large enough
\begin{align}\label{constant_v}
    \sum_{k=1}^K \E_0[X_1^k] \geq d\Vert f^{r,i} - f^0\Vert_2^2.
\end{align}
Letting $\Delta = d/\bar\tau$ and using the definition of $\Omega_{\tau,T}$, we obtain that
\begin{align}
    &\P_0\big( \Omega_{\tau,T}\cap \big\{d_T^2(f^0, f^{r,i} )
    \leq \Delta\Vert f^0-f^{r,i}\Vert_2 \}\big) \nonumber\\
    &\leq\P_0\bigg( \Omega_{\tau,T}\cap \Big\{\frac{1}{T}\sum_{n=1}^{n_T-1}\sum_{k=1}^K X_n^k -\E_0[X_n^k ]\leq -\frac{\Delta}{2}\Vert f^0-f^{r,i}\Vert_2^2\Big\}\bigg) \nonumber\\
    &\leq\P_0\bigg( \Omega_{\tau,T} \cap \Big\{\underset{\bar n\leq 2T/\bar\tau}{\max}\sum_{n=1}^{\bar n}\sum_{k=1}^K -X_n^k -\E_0[-X_n^k ]\geq \frac{\Delta}{2}T\Vert f^0-f^{r,i}\Vert_2^2\Big\}\bigg). \nonumber
\end{align}
Since the variables $(\sum_{k=1}^K -X_n^k -\E_0[-X_n^k ])_n$ are i.i.d. and centered,  the process $(S_{\bar n})_{\bar n\geq 0}$ defined  for $\bar n\geq1 $ by $S_{\bar n}:=\sum_{n=1}^{\bar n}\sum_{k=1}^K -X_n^k -\E_0[-X_n^k ]$ and $S_0=0$,  is a martingale with respect to its natural filtration. By Theorem 3.3 of \cite{van_zanten}, we have for any $x>0$
\begin{align}
    \P_0\bigg( \underset{\bar n\leq\lceil 2T/\bar\tau \rceil }{\max} \vert S_{\bar n}\vert >x , \hspace{0.1cm}\hat v_{\lceil 2T/\bar\tau \rceil }\leq v \hspace{0.1cm}\text{and}\hspace{0.1cm} \hat b_{\lceil 2T/\bar\tau \rceil }\leq b\bigg)\leq  2\exp\Big(-\frac{x^2}{v +xb/3}\Big),\nonumber
\end{align}
with 
\begin{align}
    &\hat v_{\bar n} := \sum_{n=1}^{\bar n} \E_0\Big[\Big(\sum_{k=1}^K X_n^k -\E_0[X_n^k ] \Big)^2\Big]\hspace{0.2cm}\text{and}\hspace{0.2cm}\hat b_{\bar n} := \underset{n\leq \bar n}{\max} \Big\vert \sum_{k=1}^K -X_n^k -\E_0[-X_n^k ]\Big\vert.\nonumber
\end{align}
Similarly to Equation (29) of \cite{Sulem_concentration}, we have $0\leq X_n^k \lesssim (A+1+U^{(1)}_n -\tau_n)\Vert f^{r,i}_k-f^0_k\Vert_2^2$. As a consequence, on $\Omega_{\tau,T}$, $\hat b_{{\lceil 2T/\bar\tau \rceil }} \lesssim\log(T) \Vert f^{r,i}-f^0\Vert_2^2$. Moreover, $\hat v_{\bar n}\lesssim \bar n  \Vert f^{r,i}-f^0\Vert_2^4$. Therefore, we obtain that for some $C>0$,
\begin{align}
    &\P_0\big(\Omega_{\tau,T} \cap \big\{d_T^2(f^0, f^{r,i} )
    \leq \Delta\Vert f^0-f^{r,i}\Vert_2 \}\big) \nonumber\\
    &\leq 2\exp \bigg( -\frac{C\Delta^2T^2\Vert f^{r,i}-f^0\Vert_2^4}{T\Vert f^{r,i}-f^0\Vert_2^4 + \Delta \log(T)T\Vert f^{r,i}-f^0\Vert_2^4}\bigg)\leq 2\exp \bigg( -\frac{C\Delta T}{2\log(T)}\bigg),\nonumber
\end{align}
for $T$ large enough. To conclude, we come back to (\ref{init_xil}) and with the previous bound we now have
\begin{align}
    \P_0(\Xi_T^c) &\leq 2\bar r_T\exp\bigg(T\Big(x_0'\varepsilon_T^2- \frac{C\Delta}{2\log(T)}\Big)\bigg) + o(1) \xrightarrow[T\rightarrow+\infty]{} 0,\nonumber
\end{align}
which terminates the proof of Lemma \ref{sets_XiT}.
\end{proof}

\subsection{Proof of Theorem \ref{main_theorem}}\label{sec:pr:main_th}

In this Section, we prove Theorem \ref{main_theorem} on the (possibly biased)  BvM property for random series priors.

\begin{proof}[Proof of Theorem \ref{main_theorem}]
We adapt the approach of \cite{castillo_rousseau}, especially the proof of their Theorem 2.1. Recall (\ref{def_A}) that defines $\A_T$ for some $M>0$, by Lemma \ref{lemma_concentration_A}  we have that $\Pi\big(\A_T\vert N \big) = 1 + o_{\P_0}(1)$. As explained in  \cite{castillo_rousseau}, to prove (\ref{gen_bvm}) it is enough to prove that the Laplace transform of $\mathcal{L}^\Pi\big(\sqrt{T}(\psi(f) - \hat\psi_T -\mathcal{B}_{J,T})\vert N\big)$ conditionally on  $\A_T$ converges to the Laplace transform of a $\mathcal{N}(0, \Vert \tilde \psi_L^0\Vert_L^2)$ under $\P_0$. Moreover, since $\P_0(\Omega_T)\rightarrow 1$ with $\Omega_T$ defined in (\ref{omega_T}), we may work on the event $\Omega_T$ whenever needed. Let $u\in \R$, we set
\begin{align}
    I_T :=  \E\Big[e^{u\sqrt{T}\big(\psi(f) - \hat\psi_T  -\mathcal{B}_{J,T}\big)}\Big\vert N,\A_T\Big] .\nonumber
\end{align}
First, by the hierarchy of the prior, we have
\begin{align}
    I_T &= \frac{\displaystyle\int_{\A_T}e^{u\sqrt{T}(\psi(f) - \hat\psi_T- \mathcal{B}_{J,T})+ L_T(f)-L_T(f^0)}d\Pi_{f}(f)}{\displaystyle\int_{\A_T}e^{ L_T(f)-L_T(f^0)}d\Pi_{f}(f)} \nonumber\\
    &=\sum_{j\in \mathcal{J}_T}\frac{\int_{\A_{T}(j) }e^{u\sqrt{T}(\psi(f) - \hat\psi_T - \mathcal{B}_{j,T})+ L_T(f)-L_T(f^0)}d\Pi_{f\vert j}(f)}{\int_{\A_{T}(j)}e^{L_T(f)-L_T(f^0)}d\Pi_{f\vert j}(f)}\Pi_J(j\vert N).\nonumber
\end{align}
Let $j\in \mathcal{J}_T$ and $f =(\nu,\boh)=(\nu, \varphi(\tilde \boh))\in \A_T(j)$, we recall that $\tilde f= (\nu, \tilde \boh)$. We have
\begin{align}
    \boh - \boh^0 = (\tilde \boh-\tilde \boh^0).\boldsymbol{\bar\varphi}^0 + \omega_\varphi(\tilde \boh),\nonumber
\end{align}
and Lemma \ref{lemma_linearization} gives  $\Vert \omega_\varphi(\tilde \boh)\Vert_1\lesssim\log(T)^2\varepsilon_T^2$. We rewrite the LAN expansion in terms of $\tilde f-\tilde f^0$ using this linearization of $\varphi$:
\begin{align}\label{LAN_lin}
    L_T(f) -L_T(f^0) = \sqrt{T} W_{T}((\tilde f - \tilde f^0).\boldsymbol{\bar\varphi}^0) -\frac{T}{2}\Vert \tilde f -\tilde f^0\Vert_{L,\varphi}^2 + R_{T, \varphi}(f),
\end{align}
with $ R_{T,\varphi}(f)= R_T(f) +  \sqrt{T}W_{T}\big(0, \omega_\varphi(\tilde \boh) \big)-\frac{T}{2}\Vert \omega_\varphi(\tilde \boh)  \Vert_{L}^2 - T\langle \tilde f -\tilde f^{0}, (0, \omega_\varphi(\tilde \boh) )\rangle_{L}$. Note that if for some $(l,k)$, $h_{l,k}$ take negative values, the LAN expansion is still valid since on $\mathcal{A}_T$ we necessarily have $h_{l,k}\geq -1/\log(T)^2$ (see (\ref{A_inf})) and since we can work on $\Omega_T$ and proceed as in the proof of Lemma \ref{lemma_LAN_convol_RELU}. Similarly,
\begin{align}\label{func_lin}
    \psi(f) - \psi(f^0) = \langle \tilde f - \tilde f^0, \tilde \psi^0_{L,\varphi}\rangle_{L,\varphi} + r_{T, \varphi} (f,f^0),
\end{align}
with $r_{T, \varphi} (f,f^0) = r_{T} (f,f^0) + \langle (0, \omega_\varphi(\tilde \boh)), \tilde \psi^0_{L}\rangle_L$. Now, as in the proof of Theorem 2.1 of \cite{castillo_rousseau} (see the computations at the end of p.2375 and at the beginning of p.2376 and also the second display p.2376 with $\psi_0^{(2)}=0$), using expansions (\ref{LAN_lin}) and (\ref{func_lin}), we obtain
\begin{align}
    &u\sqrt{T}\big(\psi(f) - \hat\psi_T - \mathcal{B}_{j,T}\big)+ L_T(f)-L_T(f^0) \nonumber \\
    &= -u\sqrt{T} \mathcal{B}_{j,T} + \frac{u^2}{2}\Vert \tilde \psi^0_{L,\varphi}\Vert_{L,\varphi}^2+\sqrt{T}W_{T}((\tilde f_u -\tilde f^0).\boldsymbol{\bar\varphi}^0) -\frac{T}{2}\Vert \tilde f_u - \tilde f^0\Vert_{L,\varphi}^2\nonumber\\
    &\hspace{2cm}+R_{T,\varphi}(f) + u\sqrt{T}r_{T, \varphi} (f,f^0) + o_{\P_0}(1) \nonumber \\
    &= -u\sqrt{T} \mathcal{B}_{j,T}+ \frac{u^2}{2}\Vert \tilde \psi^0_{L}\Vert_{L}^2+ L_T(f_u) - L_T(f^0) \\
    &\hspace{2cm}+ (R_{T,\varphi}(f) - R_{T,\varphi}(f_u)) +u\sqrt{T} r_{T, \varphi} (f,f^0) + o_{\P_0}(1).\nonumber
\end{align}
where, 
\begin{align*}
    \tilde f_u= \tilde f - u\tilde \psi^{0}_{L,\varphi}/\sqrt{T} ,\hspace{0.1cm} f_u = (\nu - u\xi_L^0/\sqrt{T}, \varphi(\tilde \boh - u(\bog_L^0./\boldsymbol{\bar\varphi}^0)/\sqrt{T})).
\end{align*}%
Then, using again (\ref{LAN_lin}) on $L_T(f_u)-L_T(f^0)$ and  $L_T(f_{u,j})-L_T(f^0)$, we further have,
\begin{align}
    &L_T(f_u) + R_{T,\varphi}(f) - R_{T,\varphi}(f_u)\nonumber\\
    &= L_T(f_{u,j}) + (L_T(f_u) -L_T(f^0)) - (L_T(f_{u,j})-L_T(f^0)) + R_{T,\varphi}(f) - R_{T,\varphi}(f_u)\nonumber \\
    &=L_T(f_{u,j}) + uW_{T}\big(  (\tilde \psi^{0,j}_{L,\varphi}-\tilde \psi^0_{L,\varphi}).\boldsymbol{\bar\varphi}^0\big) - \frac{u^2}{2}\Vert \tilde \psi^0_{L,\varphi} -\tilde \psi^{0,j}_{L,\varphi}\Vert_{L,\varphi}^2 \nonumber \\
    &\hspace{2cm}+u\sqrt{T}\langle\tilde \psi^{0,j}_{L,\varphi} -\tilde \psi^0_{L,\varphi}, \tilde f -u\tilde \psi^{0,j}_{L,\varphi}/\sqrt{T}- \tilde f^0\rangle_{L,\varphi}+  R_{T,\varphi}(f) - R_{T,\varphi}(f_{u,j}) .\nonumber 
\end{align}
where for the last equality we have used that,
\begin{align*}
    \Vert \tilde f_u-\tilde f_0\Vert_{L,\varphi}^2 &= \Vert \tilde f_{u,j} -\tilde f^0\Vert_{L,\varphi}^2 +\frac{u^2}{T}\Vert \tilde \psi^0_{L,\varphi} -\tilde \psi^{0,j}_{L,\varphi}\Vert_{L,\varphi}^2\\
    &\hspace{1cm}+ \frac{2u}{\sqrt{T}}\langle\tilde \psi^{0,j}_{L,\varphi} -\tilde \psi^0_{L,\varphi}, \tilde f -u\tilde \psi^{0,j}_{L,\varphi}/\sqrt{T}- \tilde f^0\rangle_{L,\varphi}.
\end{align*}%
In addition, since $\tilde \psi^{0,j}_{L,\varphi} -\tilde \psi^0_{L,\varphi}$ is orthogonal to $\R^K\times \mathcal{\tilde H}(j)^{K^2}$ for the inner product $\langle\cdot{,}\cdot\rangle_{L,\varphi}$, we have,
\begin{align}
    &\sqrt{T}\langle \tilde \psi^{0,j}_{L,\varphi} -\tilde \psi^0_{L,\varphi}, \tilde f - u\tilde \psi^{0,j}_{L,\varphi}/\sqrt{T} - \tilde f^0\rangle_{L,\varphi}  + W_{T}\big(  (\tilde \psi^{0,j}_{L,\varphi}-\tilde \psi^0_{L,\varphi}).\boldsymbol{\bar\varphi}^0\big)\nonumber\\
    &= \sqrt{T}\langle \tilde \psi^0_{L,\varphi} -\tilde \psi^{0,j}_{L,\varphi}, \tilde f^0 - \tilde f^{0,j}\rangle_{L,\varphi} + W_{T}\big(  (\tilde \psi^{0,j}_{L,\varphi}-\tilde \psi^0_{L,\varphi}).\boldsymbol{\bar\varphi}^0\big)= \sqrt{T}\mathcal{B}_{j,T}.\nonumber
\end{align}
Therefore,
\begin{align}
    &u\sqrt{T}\big(\psi(f) - \hat\psi_T -  \mathcal{B}_{j,T}\big)+ L_T(f)-L_T(f^0)\nonumber\\
    &= \frac{u^2}{2}\Vert \tilde \psi^{0}_{L}\Vert_{L}^2 +L_T(f_{u,j}) - L_T(f^0)\nonumber\\
    &\hspace{1cm}+ \big(R_{T,\varphi}(f) - R_{T,\varphi}(f_{u,j})\big) + u\sqrt{T}r_T(f,f^0)- \frac{u^2}{2}\Vert \tilde \psi^0_{L,\varphi} -\tilde \psi^{0,j}_{L,\varphi}\Vert_{L,\varphi}^2 +o_{\P_0}(1).\nonumber
\end{align}
 Lemma \ref{control_rem} shows that
 \begin{align*}
     \underset{j\in \mathcal{J}_T}{\max}\underset{f\in \A_{T}(j)}{\sup}\vert R_{T,\varphi}(f) - R_{T,\varphi}(f_{u,j})\vert =o_{\P_0}(1),
 \end{align*}
the proof is based on a chaining argument from \cite{talagrand_chaining} for random variables satisfying Bernstein-type inequality. Next, with similar computations to (\ref{proj_norm_charac}), one can show that
\begin{align}
    \underset{j\in \mathcal{J}_T}{\max}\Vert \tilde \psi^0_{L,\varphi} -\tilde \psi^{0,j}_{L,\varphi}\Vert_{L,\varphi}^2 \lesssim  \underset{j\in \mathcal{J}_T}{\max}\Vert \bog^0_L./\boldsymbol{\bar\varphi}^0 - P^j_{2}(\bog^0_L./\boldsymbol{\bar\varphi}^0)\Vert_{2}^2, \nonumber
\end{align}
the last term being a $o(1)$ by assumption. Moreover, since $\tilde \psi^0_{L}$ is the least favourable direction and by Lemma \ref{lemma_linearization}, we have
\begin{align}
    \vert\langle (0, \omega_\varphi(\tilde \boh)), \tilde \psi^0_{L}\rangle_L\vert = \vert\langle \omega_\varphi(\tilde \boh),  \bog^0_2\rangle_2\vert \leq \Vert \bog^0_2\Vert_\infty \Vert \omega_\varphi(\tilde \boh)\Vert_1 \lesssim \log(T)^2\varepsilon_T^2, \nonumber
\end{align}
and since $\log(T)^2\sqrt{T}\varepsilon_T^2\rightarrow 0$,  we find that $\sup\big\{\sqrt{T}r_{T,\varphi}(f,f^0), f\in \A_T\big\}= o(1)$. Combining these results together, we obtain
\begin{align}
    I_T=e^{o_{\P_0}(1) + \frac{u^2}{2}\Vert \tilde \psi_L^0\Vert_L^2}\sum_{j\in \mathcal{J}_T} \frac{\int_{\A_{T}(j) }e^{L_T(f_{u,j})}d\Pi_{f\vert j}(f)}{\int_{\A_{T}(j)}e^{L_T(f)}d\Pi_{f\vert j}(f)}\Pi_J(j\vert N),\nonumber
\end{align}
and with Condition (\ref{change_var_cond}) we finally have  $I_T =e^{o_{\P_0}(1) + \frac{u^2}{2}\Vert \tilde \psi_L^0\Vert_L^2}$,  which concludes the proof.
\end{proof}

\subsection{Lemmas related to the BvM property}\label{sec:lemma:bvm}
In this section, we state and prove lemmas related to the BvM property. We first prove  the following Lemma \ref{lemma_rk} on the control of the bias term in Theorem \ref{main_theorem} and Corollary \ref{main_coro}.  We also prove (following \cite{nickl_IV_book}) Corollary \ref{coro_posterior_mean} on the asymptotic behaviour of the posterior mean. Then, we also prove Lemma \ref{lemma_concentration_A} which concerns the contraction of the $\tilde \boh$ on $\tilde \boh^0$ when $\varphi$ is not the identity function, under Assumption \hyperref[as_A']{(A')}. Finally,  we prove Lemma \ref{lemma_linearization} on the linearization of $\varphi$ around $\tilde \boh^0$ under \hyperref[as_A']{(A')} which is in particular used in the proof of Theorem \ref{main_theorem}.

\begin{lemma}\label{lemma_rk}
Under the assumptions of Theorem \ref{main_theorem}, if 
\begin{align*}
    \log(T)^3\Vert  \bog^0_L./\boldsymbol{\bar\varphi}^0 - P^j_{2}(\bog^0_L./\boldsymbol{\bar\varphi}^0)\Vert_{2}=o(1) ,
\end{align*}
then we have $\max_{j\in \mathcal{J}_T}W_{T}\big(  (\tilde \psi^{0,j}_{L,\varphi}-\tilde \psi^0_{L,\varphi}).\boldsymbol{\bar\varphi}^0\big) =o_{\P_0}(1)$. In particular, for $j\in \mathcal{J}_T$,
\begin{align*}
    \mathcal{B}_{j,T} = \mathcal{B}_j + o_{\P_0}(T^{-1/2}) \hspace{0.1cm} \textit{with}\hspace{0.1cm}\mathcal{B}_j= -\langle \tilde f^0 - \tilde f^{0,J} , \tilde \psi^{0}_{L,\varphi} -\tilde \psi^{0,J}_{L,\varphi} \rangle_{L,\varphi},
\end{align*}
and  $o_{\P_0}(T^{-1/2})$ independent of $j$. Moreover, we have
 \begin{align}\label{proj_norm_charac}
    \vert \mathcal{B}_j\vert \lesssim \log(T)\varepsilon_T \Vert \tilde \psi^0_{L,\varphi} -\tilde\psi^{0,j}_{L,\varphi}\Vert_{L,\varphi}\lesssim \log(T)\varepsilon_T\Vert  \bog^0_L./\boldsymbol{\bar\varphi}^0 - P^j_{2}(\bog^0_L./\boldsymbol{\bar\varphi}^0)\Vert_{2} .
\end{align}

\end{lemma}
\begin{proof}[Proof of Lemma \ref{lemma_rk}] By assumption, there exists $R>0$ such that $\Vert \psi^0_L \Vert_\infty\leq R/2$ and $\sup\{\Vert \boldsymbol{\bar\varphi}^0.\tilde\psi^{0,j}_{L,\varphi}\Vert_\infty, j\geq 1\}\leq R/2$ so that
\begin{align}
    \Vert (\tilde \psi^{0,j}_{L,\varphi}-\tilde \psi^0_{L,\varphi}).\boldsymbol{\bar\varphi}^0\Vert_\infty = \Vert \tilde\psi^{0,j}_{L,\varphi}. \boldsymbol{\bar\varphi}^0-\tilde\psi^0_L\Vert_\infty\leq R .\nonumber
\end{align}
With this and the first inequality of lemma \ref{Bernstein} we find that for any $x>0$
\begin{align}
    &\P_0\Big(\Omega_T\cap \big\{\underset{j\in \mathcal{J}_T}{\max}\big \vert W_{T}\big(  (\tilde \psi^{0,j}_{L,\varphi}-\tilde \psi^0_{L,\varphi}).\boldsymbol{\bar\varphi}^0\big) \big \vert >x\big\}\Big)\nonumber\\
    &\leq \sum_{j\in \mathcal{J}_T}\P_0\Big(\Omega_T\cap \big\{\big \vert W_{T}\big(  (\tilde \psi^{0,j}_{L,\varphi}-\tilde \psi^0_{L,\varphi}).\boldsymbol{\bar\varphi}^0\big)\big \vert >x\big\}\Big) \nonumber\\
    &\leq \sum_{j\in \mathcal{J}_T}\exp\bigg(\frac{-\log(T)^{-5}Cx^2}{\Vert (\tilde \psi^{0,j}_{L,\varphi}-\tilde \psi^0_{L,\varphi}).\boldsymbol{\bar\varphi}^0\Vert_2^2 + \frac{Rx}{\sqrt{T}}}\bigg).\nonumber
\end{align}
Let $j\in \mathcal{J}_T$. Using the equivalence of the norms $\Vert .\Vert_2$, $\Vert .\Vert_L$ and $\Vert . \Vert_{L,\varphi}$ and the characterisation of an orthogonal projection, we find:
\begin{align}
    \Vert (\tilde \psi^{0,j}_{L,\varphi}-\tilde \psi^0_{L,\varphi}).\boldsymbol{\bar\varphi}^0\Vert_2^2 \lesssim 
    \Vert(\tilde \psi^{0,j}_{L,\varphi}-\tilde \psi^0_{L,\varphi}).\boldsymbol{\bar\varphi}^0\Vert_{L}^2
    &= \Vert   P^{j}_{L,\varphi}(\tilde \psi^0_{L,\varphi})-\tilde \psi^0_{L,\varphi} \Vert_{L,\varphi}^2\nonumber\\
    &\leq \Vert  P^j_{2}(\tilde \psi^0_{L,\varphi})- \tilde \psi^0_{L,\varphi} \Vert_{L,\varphi}^2\nonumber\\
    &\lesssim \Vert  P^j_{2}(\tilde \psi^0_{L,\varphi})- \tilde \psi^0_{L,\varphi} \Vert_{2}^2\nonumber\\
    &= \Vert  P^j_{2}(\bog^0_L./\boldsymbol{\bar\varphi}^0)- \bog^0_L./\boldsymbol{\bar\varphi}^0 \Vert_{2}^2;\nonumber
\end{align}
 Whence, since $\max \{j, j\in \mathcal{J}_T\}\leq J_T,$ we have for some $c,d>0$,
\begin{align}
    &\P_0\Big(\Omega_T\cap \big\{\underset{j\in \mathcal{J}_T}{\max}\big \vert W_{T}\big(  (\tilde \psi^{0,j}_{L,\varphi}-\tilde \psi^0_{L,\varphi}).\boldsymbol{\bar\varphi}^0\big) \big \vert >x\big\}\Big)\nonumber\\
    &\leq J_T\exp\bigg(\frac{-\log(T)^{-5}Cx^2}{c\underset{j\in \mathcal{J}_T}{\max} \Vert \bog^0_L./\boldsymbol{\bar\varphi}^0 - P^j_{2}(\bog^0_L./\boldsymbol{\bar\varphi}^0)\Vert_{2}^2 + \frac{Rx}{\sqrt{T}}}\bigg)\nonumber\\
    &\leq \exp\bigg(d\log(T)\bigg(1-\frac{\log(T)^{-6}x^2}{c\underset{j\in \mathcal{J}_T}{\max} \Vert \bog^0_L./\boldsymbol{\bar\varphi}^0 - P^j_{2}(\bog^0_L./\boldsymbol{\bar\varphi}^0)\Vert_{2}^2 + \frac{Rx}{\sqrt{T}}}\bigg)\bigg)\xrightarrow[T\rightarrow +\infty]{}0,\nonumber
\end{align}
by assumption, which ends the proof of the first assertion. 

For the second assertion, let $j\in\mathcal{J}_T$, we have by Cauchy-Schwarz inequality
\begin{align}
    \vert \mathcal{B}_j\vert \leq \Vert \tilde f^0- \tilde f^{0,j}\Vert_{L,\varphi} \Vert \tilde \psi^{0}_{L,\varphi}- \tilde \psi^{0,j}_{L,\varphi}\Vert_{L,\varphi}.\nonumber
\end{align}
Then, using the same arguments as for the first assertion (namely the equivalence of norms and the characterisation of an orthogonal projection), one can show that $\Vert \tilde f^0- \tilde f^{0,j}\Vert_{L,\varphi}\lesssim\Vert \tilde f^0 - P^j_2(\tilde f^0)\Vert_2$ and $\Vert \tilde \psi^{0}_{L,\varphi}- \tilde \psi^{0,j}_{L,\varphi}\Vert_{L,\varphi}\lesssim \Vert  \bog^0_L./\boldsymbol{\bar\varphi}^0 - P^j_{2}(\bog^0_L./\boldsymbol{\bar\varphi}^0)\Vert_{2} $. With the $L_2$ posterior contraction, we must have $\Vert \tilde f^0 - P^j_2(\tilde f^0)\Vert_2\lesssim \log(T)\varepsilon_T$ and it proves the second assertion.
\end{proof}

We then prove Corollary \ref{coro_posterior_mean} on the efficiency of the posterior mean.
\begin{proof}[Proof of Corollary \ref{coro_posterior_mean}]
 As explained just before Corollary \ref{coro_posterior_mean}, this proof follows closely the proof of Theorem 4.1.10 in \cite{nickl_IV_book} with small changes. We first verify a sufficient condition for uniform integrability, namely that there exists $\epsilon>0$ such that
\begin{align}
    T^{(1+\epsilon)/2}\E\big [\vert \psi(f)-\psi(f^0)\vert^{1+\epsilon}\big\vert N\big] =O_{\P_0}(1).\label{uni_int_cond}
\end{align}
Recall $\mathcal{A}_T$ defined by (\ref{def_A}). In the proof of Lemma \ref{lemma_concentration_A} it is shown that $\mathcal{A}_T'\subset \mathcal{A}_T$ with 
\begin{align}
    \A'_{T} = \big \{f: \Vert f-f^0\Vert_2\leq \log(T)\varepsilon_T, \hspace{0.1cm}\forall (l,k)\in \K^2\hspace{0.1cm}Range( h_{l,k})\subset I_0(\epsilon)\big\}. \nonumber
\end{align}
Moreover,  Equation (26) in \cite{Sulem_concentration} gives that $\P_0(D_T\leq e^{(c_4-1)T\varepsilon_T^2})\rightarrow 1$ ($D_T$ being defined by (\ref{def_DT})). Using that \hyperref[as_P1]{(P1)} and \hyperref[P2]{(P2)} are verified with $c_4$ replaced by $2c_4$ and choosing $C_1$ large enough in Proposition \ref{th_concentration_L1}, an inspection of the proofs of Theorem \ref{th_concentration_L2} and Proposition 3.5 in \cite{Sulem_concentration} reveals that on some event $\Omega^e_T$, $\P_0(\Omega_T^e)\rightarrow1$,
\begin{align*}
    \Pi\big(f: \Vert f-f^0\Vert_2\geq \log(T)\varepsilon_T \big \vert N\big) \lesssim e^{-c_4T\varepsilon_T^2}.
\end{align*}
Combined with Assumption \hyperref[as_A']{(A')} where again $c_4$ is replaced by $2c_4$, we find that on $\Omega^e_T$, $\Pi(\mathcal{A}_T^{'c}\vert N) \lesssim e^{-c_4T\varepsilon_T^2}$ and thus $\Pi(\mathcal{A}_T^{c}\vert N) \lesssim e^{-c_4T\varepsilon_T^2}$. Then, 
\begin{align*}
    &T^{(1+\epsilon)/2}\E\big[\vert \psi(f)-\psi(f^0)\big\vert^{1+\epsilon}\vert N\big] \nonumber\\
    &= T^{(1+\epsilon)/2}\E\big[\vert \psi(f)-\psi(f^0)\vert^{1+\epsilon}\onee_{\mathcal{A}_T^c}\big\vert N\big]+ T^{(1+\epsilon)/2}\E\big[\vert \psi(f)-\psi(f^0)\vert^{1+\epsilon}\onee_{\mathcal{A}_T}\big\vert N\big]\nonumber\\
    &=: (I)+(II).
\end{align*}
For $(I)$, following the same computations as in the two last displays on p.78 in \cite{nickl_IV_book} (but with $\vert \psi(f)-\psi(f^0)\vert^{1+\epsilon}$ instead of $\langle \theta-\theta_0, \psi\rangle^2$), we obtain that for any $v>0$,
\begin{align*}
    \P_0\big((I)>v\big) \lesssim vT^{(1+\epsilon)}e^{-T\varepsilon_T^2}\int \vert \psi(f)-\psi(f^0)\vert^{2+2\epsilon}d\Pi(f)
\end{align*}
and with Condition \ref{cond_func_b}, it proves that $(I)=o_{\P_0}(1)$.

For (II), using that $\forall x\in \R$, $x^2\leq e^{x}+e^{-x}$ and using the computations in the proof of (\ref{main_theorem}) with $u=1$ and $u=-1$, we find that
\begin{align*}
    (II)\leq \Pi(\mathcal{A}_T\vert N)^{-1}2e^{o_{\P_0}(1)+W_T(\tilde \psi^0_L) +\Vert \tilde \psi^0_L\Vert_L^2 /2},
\end{align*}
and since $W_T(\tilde \psi^0_L)= O_{\P_0}(1)$ and $\Pi(\mathcal{A}_T\vert N)\xrightarrow[]{\P_0}1$, it shows that  $(II)=O_{\P_0}(1)$. We have thus proven (\ref{uni_int_cond}).

To conclude it remains to prove that convergence (\ref{def_bvm}) and the uniform integrability condition (\ref{uni_int_cond}) implies convergence of the first moment. This can be done exactly as after the last display on p.79 in \cite{nickl_IV_book} and we finally obtain that $\bar \psi_T$ verifies (\ref{eq_efficient_estimator}) which shows that $\bar \psi_T$  is an efficient estimator by Lemma 25.23 in \cite{VDV_AS}.
\end{proof}

Next, we state and prove Lemma \ref{lemma_concentration_A}  on the contraction of the $\tilde \boh$ on $\tilde\boh^0$ when $\varphi$ is not the identity function.

\begin{lemma}\label{lemma_concentration_A}   Under assumptions of Theorem \ref{main_theorem}, we have $\Pi( \A_T^c\vert N)= o_{\P_0}(1)$ when $\varphi$ is not the identity function; where $\mathcal A_T$ is defined in \eqref{def_A}. 
\end{lemma}
\begin{proof}[Proof of Lemma \ref{lemma_concentration_A}]
By Theorem \ref{th_concentration_L2} and Assumption \hyperref[as_A']{(A')}, we have $\Pi(\A'_{T}\vert N)\rightarrow 1$ under $\P_0$ with
\begin{align}
    \A'_{T} = \big \{f: \Vert f-f^0\Vert_2\leq \log(T)\varepsilon_T, \hspace{0.1cm}\forall (l,k)\in \K^2\hspace{0.1cm}Range( h_{l,k})\subset I_0(\epsilon)\big\}, \nonumber
\end{align}
where $I_0(\epsilon)$ is the $\epsilon$-neighbourhood of $I_0$, the convex hull of $\cup_{l,k} h_{l,k}^0([0,A])$. 
Let $f= (\nu , \varphi(\tilde \boh))\in \A_T'$. By again Assumption \hyperref[as_A']{(A')}, we know that $\varphi^{-1}$ has a bounded derivative on $\varphi^{-1}(I_0(\epsilon))$ and as a consequence:
\begin{align}
    \Vert \tilde \boh -\tilde \boh^0 \Vert _2 \lesssim  \Vert  \boh -\boh^0 \Vert _2 \leq\log(T)\varepsilon_T. \nonumber
\end{align}
Whence, choosing the constant $M$ large enough in the definition of $\mathcal{A}_{T}$, we have  $\mathcal{A}_{T}'\subset \mathcal{A}_{T}$ and we conclude to the result:
 $\Pi(\mathcal A_{T}\vert N)\xrightarrow[T\rightarrow +\infty]{\P_0}1$.
 \end{proof}
Note that by definition of $\A_T$ and by Lemma \ref{prem_1}, we have for $T$ large enough
\begin{align}\label{A_inf}
    \A_T\subset \big\{f: \Vert \tilde f -\tilde f^0\Vert_\infty \leq D'\log(T)\sqrt{J_T}\varepsilon_T\big\}
\end{align} 
for some $D'>0$ and $\log(T)\sqrt{J_T}\varepsilon_T \rightarrow 0$ by assumption. Whence, for $T$ large enough, $\A_T$ contains only elements $f$ such that for all $(l,k)\in \K^2$, 
\begin{align}\label{A_range}
    Range(h_{l,k})\subset I_0(\epsilon)\hspace{0.3cm}\text{and}\hspace{0.3cm}Range(\tilde h_{l,k})\subset \varphi^{-1}(I_0(\epsilon)).
\end{align}

Finally, we turn to the linearization of $\varphi$ around the $\tilde \boh^0$. Let $f=(\nu,\boh)\in \A_T$, $\boh = \varphi(\tilde \boh)$ for some $\tilde \boh$ and for all $(l,k)\in \K^2$ write $h_{l,k} - h^0_{l,k} = (\tilde h_{l,k} - \tilde h^0_{l,k})\bar\varphi^0_{l,k} + \omega_{\varphi}(\tilde h_{l,k})$. By defining $\omega_\varphi(\tilde \boh) = (\omega_\varphi(\tilde h_{l,k}), (l,k)\in \K^2)$, the previous inequalities can be rewritten in vector form as
\begin{align}\label{linearization}
    \boh - \boh^0 = (\tilde \boh-\tilde \boh^0).\boldsymbol{\bar\varphi}^0 + \omega_\varphi(\tilde \boh)
\end{align}
and we have the following control on the remainder term $\omega_\varphi(\tilde \boh)$.

\begin{lemma}\label{lemma_linearization}
Under the assumptions of Theorem \ref{main_theorem}, for $T$ large enough, 
let $f=~(\nu,\boh)\in \A_T$, then  equality (\ref{linearization}) holds with $\Vert \omega_\varphi(\tilde \boh)\Vert_1\lesssim \log(T)^2\varepsilon_T^2.$
\end{lemma}

\begin{proof}[Proof of Lemma \ref{lemma_linearization}] Let $f= (\nu ,\boh)=(\nu ,\varphi(\tilde \boh))\in \A_T$, by (\ref{A_range}) we know that for $T$ large enough (independently of $f$), for all $(l,k)\in \K^2$ and for all $x\in [0,A]$, $ h_{l,k}(x)\in~ I_0(\epsilon)$. By Assumption \hyperref[as_A']{(A')},  we can do a Taylor-Lagrange expansion, for all $(l,k)\in\K^2$ there exists a function $\gamma(\tilde h_{l,k}, \tilde h^0_{l,k})$ in the bracket $[\tilde h_{l,k};\tilde h^{0}_{l,k}]$ such that
\begin{align}
    h_{l,k} - h^{0}_{l,k} =  \varphi(\tilde h_{l,k})  -\varphi(\tilde h^0_{l,k}) =\big(\tilde h_{l,k}-\tilde h^{0}_{l,k}\big)\bar\varphi^0_{l,k}+ \frac{\big(\tilde h_{l,k} - \tilde h^0_{l,k}\big)^2 }{2}\varphi''\big(\gamma(\tilde h_{l,k}, \tilde h^{0}_{l,k})\big)\nonumber
\end{align}
Let $\omega_\varphi(\tilde h_{l,k}) = \frac{1}{2}\big(\tilde h_{l,k} - \tilde h^0_{l,k}\big)^2\varphi''\big(\gamma(\tilde h_{l,k}, \tilde h^{0}_{l,k})\big)$. By assumption, $\varphi''$ is bounded on $\varphi^{-1}(I_0(\epsilon))$ so we have $\Vert \omega_\varphi(\tilde h_{l,k})\Vert_1 \lesssim \Vert \tilde h_{l,k}-\tilde h_{l,k}^0\Vert_2^2\lesssim \log(T)^2\varepsilon_T^2$ and this shows the result.
\end{proof}

\subsection{Lemmas for the application on specific priors} \label{sec:pr:priors}
In this section, we prove three  lemmas used to verify conditions of Theorem \ref{main_theorem} when we apply it to specific prior distributions (random histograms and priors based on wavelet bases) in Section \ref{app_priors}. First, the following Lemma \ref{lemma_LAN_proj_bound} is used to verify the condition "$\sup\{\Vert  \bog^{0,j}_{L,\varphi}\Vert_\infty, \hspace{0.1cm}j\geq 1\}<+\infty$" of Theorem~\ref{main_theorem}. Then, Lemma \ref{change_var_lemma} shows that the change of variable condition (\ref{change_var_cond}) is verified for the priors considered in Section \ref{app_priors}. Finally, we prove in Lemma  \ref{verif_P2_A_wavelets} that Assumptions \hyperref[P2]{(P2)} and \hyperref[as_A']{(A')} are satisfied for the priors based on wavelet bases with $\varphi\geq 0$.

\begin{lemma}\label{lemma_LAN_proj_bound} Let $(\xi,\bog)\in \L_\infty^{K^2}$ and  for $j\geq 1$, let $B_j$ be a linearly independent family in $L_2$ consisting of functions in $L_\infty$ and that verifies (\ref{condi_sup}). As before, let $(\xi^j, \bog^j) = P^j_L(\xi,\bog)$. Assume that there exists some $\beta>1/2$ and some finite $C(\bog)$ such that for all $j\geq 1$
\begin{align}
    \Vert P_2^j(\bog)\Vert_\infty+ j^\beta \Vert P_2^j(\bog) - \bog\Vert_2\leq C(\bog).\nonumber
\end{align}
Then, $\underset{j\geq 1}{\sup}\Vert \bog^j\Vert_\infty<+\infty$.

\end{lemma}

\begin{proof}[Proof of Lemma \ref{lemma_LAN_proj_bound}]
Let $j\geq 1$. First,  with the equivalence between the LAN norm and the $\L_2$ norm and  by characterization of an orthogonal projection, we have
\begin{align}
    \Vert \bog^j -\bog \Vert_2 \leq \Vert (\xi^j,\bog^j) -(\xi,\bog) \Vert_2\lesssim \Vert (\xi^j,\bog^j) -(\xi,\bog) \Vert_L&\leq  \Vert P_2^j(\xi,\bog) -(\xi,\bog) \Vert_L\nonumber\\
    &\lesssim \Vert P_2^j(\xi,\bog) -(\xi,\bog) \Vert_2 \nonumber\\
    &= \Vert P^j_2(\bog) -\bog \Vert_2.\nonumber
\end{align}
Then, with this inequality and the conditions of the Lemma we obtain:
\begin{align}
    \Vert \bog^j\Vert_\infty \leq \Vert \bog^j - P^{j}_2(\bog)\Vert_\infty + \Vert P^{j}_2(\bog)\Vert_\infty&\leq \sqrt{j}\Vert \bog^j - P^{j}_2(\bog)\Vert_2 + C(\bog)\nonumber\\
    &\leq  \sqrt{j}\Vert \bog^j - \bog\Vert_2 + \sqrt{j}\Vert \bog - P^{j}_2(\bog)\Vert_2+  C(\bog)\nonumber\\
    &\lesssim   2\sqrt{j}\Vert \bog - P^{j}_2(\bog)\Vert_2+  C(\bog)\nonumber\\
    &\leq 2j^{1/2-\beta}C(\bog)  +  C(\bog). \nonumber
\end{align}
Since $\beta >1/2$, we can conclude to the result.
\end{proof}

The following lemma shows that the change of variable condition (\ref{change_var_cond}) is verified for the priors considered.

\begin{lemma}\label{change_var_lemma}
    Consider either the setting of Corollary \ref{histo_coro} (random histograms bases), Corollary \ref{wavelets_coro} (truncated wavelets bases in ReLU model) or Corollary \ref{wavelets_coro_varphi} (truncated wavelets bases with $\varphi\geq 0$). In all three cases,  Condition (\ref{change_var_cond}) of Theorem \ref{main_theorem} is satisfied:
    \begin{align}
        \sum_{j\in \mathcal{J}_T}\frac{\int_{\A_{T}(j) }e^{L_T(f_{u,j})}d\Pi_{f\vert j}(f)}{\int_{\A_{T}(j)}e^{L_T(f)}d\Pi_{f\vert j}(f)} \Pi_J(j\vert N)\xrightarrow[T\rightarrow +\infty]{\P_0} 1.\nonumber
\end{align}
\end{lemma}

\begin{proof}[Proof of Lemma \ref{change_var_lemma}] For the moment we do not fix the  prior, it can be either the histogram prior, the wavelet prior in the ReLU model or the wavelet prior with $\varphi\geq 0$. First, note that by definition of $\A_T$ and by Lemma \ref{prem_1}, we have 
\begin{align}
    \A_T\subset \big\{f: \Vert \tilde f -\tilde f^0\Vert_\infty \leq D\log(T)\sqrt{J_T}\varepsilon_T\big\},\nonumber
\end{align} 
for some universal $D>0$ and $\log(T)\sqrt{J_T}\varepsilon_T \rightarrow 0$ by assumption. So for $T$ large enough, $\mathcal{A}_T\subseteq \mathcal{F}_R^e$ (defined at start of Section \ref{subsec_prior_construct}). Moreover, for $\nu = (\nu_1,...,\nu_K)$, let $\pi_{K,\nu}(\nu) := \prod_{k=1}^K \pi_{\nu}(\nu_k) $ and,  given $j$, for $\bth\in \R^{K^2j}$ recall that $\pi_{\bth\vert j}(\bth)= \prod_{(k,l)\in \K^2}\prod_{i=1}^j \pi_{\theta,i}(\theta_{l,k}^i)$. Next, fix $j\in \mathcal{J}_T$ and denote $\bth_{0,j}\in \R^{K^2j}$  the vector satisfying  $\bog^{0,j}_{L,\varphi} = \bth_{0,j}^TB_j$;  we set 
\begin{align}
    \Theta_T(j):= \big\{(\nu,\bth)\in \R^K\times \R^{K^2j}:  (\nu,\varphi(\bth^TB_j)) \in \A_T(j)\big\},\nonumber
\end{align}
and $\Theta_T^S(j) = \Theta_T(j) - u(\xi^{0,j}_{L,\varphi}, \bth_{0,j})/\sqrt{T} $. With the change of variable
\begin{align}
    (\nu,\bth) \rightarrow  (\nu - u\xi^{0,j}_{L,\varphi}/\sqrt{T},\bth - u \bth_{0,j}/\sqrt{T})\nonumber
\end{align}
we have:
\begin{align}
    &\frac{\int_{\A_{T}(j) }e^{L_T(f_{u,j})}d\Pi_{f\vert j}(f)}{\int_{\A_{T}(j)}e^{L_T(f)}d\Pi_{f\vert j}(f)}\nonumber\\
    &=  \frac{\int_{\Theta_{T}(j) }\exp\Big(L_T\Big(\nu - u\xi^{0,j}_{L,\varphi}/\sqrt{T}, \varphi((\bth - u \bth_{0,j}/\sqrt{T})^TB_j)\Big)\Big) \pi_{K,\nu}(\nu)\pi_{\bth\vert j}(\bth) d\nu d\bth}{\int_{\Theta_{T}(j)}\exp\Big(L_T\big(\nu, \varphi(\bth^TB_j)\big)\Big) \pi_{K,\nu}(\nu)\pi_{\bth\vert j}(\bth) d\nu d\bth}\nonumber\\
    &=  \frac{\int_{\Theta_{T}^S(j) }\exp\Big(L_T\big(\nu, \varphi(\bth^TB_j)\big)\Big) \pi_{K,\nu}(\nu + u\xi^{0,j}_{L,\varphi}/\sqrt{T})\pi_{\bth\vert j}(\bth + u\bth_{0,j}/\sqrt{T}) d\nu d\bth}{\int_{\Theta_{T}(j)}\exp\Big(L_T\big(\nu, \varphi(\bth^TB_j)\big)\Big) \pi_{K,\nu}(\nu)\pi_{\bth\vert j}(\bth) d\nu d\bth}.\nonumber
\end{align}
We first study $\pi_{K,\nu}$. $\vert \xi^{0,j}_{L,\varphi}\vert \leq R_0$ for some $R_0>0$ independent of $j$ and $\Vert \nu -\nu^0\Vert_\infty \lesssim \varepsilon_T(\beta)$ so that for $T$ large enough, $\nu^0_k/2\leq \nu_k \leq 3\nu_k^0/2$ for all $k$ over $\Theta_T^S(j)$. Since, $\pi_\nu$ is positive and continuously differentiable on $]0; +\infty[$, $\log (\pi_\nu)$ is $C^1$ on $\prod_k[\nu_k^0/2, 3\nu_k^0/2]$ and  there exists $C_{\pi_\nu}>0$ such that  for $(\nu, \bth)\in \Theta_T^S(j) $, we have
\begin{align} 
    &\big\vert \log\big(\pi_{K,\nu}\big(\nu + u\xi^{0,j}_{L,\varphi}/\sqrt{T}\big)\big) -\log(\pi_{K,\nu}(\nu))\big\vert \nonumber\\
    &\leq \sum_{k=1}^K \big\vert \log\big(\pi_{\nu}\big(\nu_k + u\xi^{0,j}_{L,\varphi,k}/\sqrt{T}\big)\big) -\log(\pi_{\nu}(\nu_k))\big\vert \lesssim KC_{\pi_\nu}R_0/\sqrt{T}.\nonumber
\end{align}
and therefore
\begin{align} \label{bound_change_j_nu}
    \frac{\pi_{K,\nu}\big(\nu + u\xi^{0,j}_{L,\varphi}/\sqrt{T}\big)}{\pi_{K,\nu}(\nu)}  = 1+o(1)
\end{align}
with $o(1)$ independent of $\nu$ and $j$. Now, we show similarly that for $j\in \mathcal{J}_T$ and $(\nu, \bth)\in~ \Theta^S_T(j)$.
\begin{align}\label{bound_change_j}
    \frac{\pi_{\bth\vert j}\big(\bth +  u\bth_{0,j}/\sqrt{T}\big) }{ \pi_{\bth\vert j}(\bth) }=1+ o(1),
\end{align}
with again $o(1)$ independent of $\bth$ and $j$. For this, we treat separately the three priors considered.

For the histogram prior, first note that since $\sup \{\Vert \bog^{0,j}_{L,\varphi}\Vert_\infty , j\geq 1\}<+\infty$, there exists some finite $G_0>0$ such that $\sup\{\Vert \bth_{0,j}\Vert_\infty , j\geq 1\}\leq G_0$. Then, recall that the coefficients are distributed independently according to a distribution with a density $\pi_\theta$ supported on $[\kappa; +\infty[$ with $\kappa$ such that for some $\delta>0 $, we have for all $(l,k)\in \K^2$, $h^0_{l,k}>\kappa +\delta $.  Moreover, given $j\in \mathcal{J}_T$, by the mean value theorem, there exists $K^2$ $j$-histogram functions whose coefficients $\bth_{*,j}$ are between $\kappa +\delta$ and $ \Vert \boh^0\Vert_\infty$ and such that $\Vert \bth_{*,j}^TB_j -\boh^0\Vert_\infty\lesssim \varepsilon_T(\beta)$. Whence, let $j\in \mathcal{J}_T$ and $(\nu,\bth) \in\Theta_T(j)$, we have
\begin{align}
    \Vert \bth- \bth_{*,j}\Vert_\infty \leq \Vert \bth - \bth_{*,j}\Vert_2 \lesssim \sqrt{j}  \Vert \bth^TB_j - \bth_{*,j}^TB_j\Vert_2 \lesssim \sqrt{J_T(\beta)}\log(T)\varepsilon_T(\beta) =o(1).\nonumber
\end{align}
 So, for $T$ large enough (independently of $j$), $\Theta_T^S(j)$ contains only elements $\bth$ with coefficients that are in some bounded interval $B'\subset [\kappa +\delta/2; +\infty[$. Furthermore, $\pi_\theta$ is positive and continuously differentiable on $]\kappa; +\infty[$. In particular, there exists $C_{\pi_\theta}>0$ such that for any $x \in B' $, $\vert \pi'_\theta(x)/\pi_\theta (x)\vert \leq C_{\pi_\theta}$. Thus, let $(\nu, \bth)\in \Theta_T^S(j) $, we have
 \begin{align}
     &\sum_{(l,k)\in \K^2}\sum_{i=1}^j \Big \vert \log(\pi_\theta(\theta_{l,k}^i + u \theta_{0,j,l,k}^i/\sqrt{T})) - \log(\pi_\theta(\theta_{l,k}^i))\Big \vert\nonumber\\
     &\lesssim\sum_{(l,k)\in \K^2}\sum_{i=1}^j  C_{\pi_\theta}\vert u \theta_{0,j,l,k}^i\vert /\sqrt{T} \lesssim C_{\pi_\theta} G_0\log(T)J_T(\beta) /\sqrt{T}= o(1),\nonumber
 \end{align}
 and as a consequence,  (\ref{bound_change_j}) holds for the histogram prior.

For the wavelet priors, we recall that $\bar B_I= \big\{\psi_{i,v}, i\in \{-1,0, ..., I\}, v\in \{0,..., \bar v_i\}\big\} $ and we recall the notations (\ref{not_wav}) that we will use here.  As explained in Section \ref{app_priors}, a prior on $J$ is induced by a prior on $I$ through the relation $J=c(I):=\sum_{i=-1}^I\bar v_i$ and we let $\bar \bog^{0,i}_{L,\varphi}=\bog^{0,c(i)}_{L,\varphi}= \bog^{0,j}_{L,\varphi} $ and we define similarly $\bar \bth_{0,i}$. The following facts will be used repeatedly in the sequel. If $\bog$ are $K^2$ functions belonging to $\mathcal{B}_{\infty,\infty}^\beta([0,A])$, their wavelets coefficients at resolution level $i$, denoted $\bth_{\bog,i}$ verify for all $(l,k)\in \K^2$ and for all $j\in \{-1,...,i\}$, $\Vert \bth_{\bog,i,l,k}^j\Vert_\infty \leq M(\bog)2^{-\max(j,0)(\beta+1/2)}$ where $M(\bog)>0$ depends on the $\mathcal{B}_{\infty,\infty}^\beta$ norm of  $\bog$.  Since $\bar v_j\lesssim  2^j$, it implies that
\begin{align}\label{ineq_wav_infty}
     \Vert  \bth_{\bog,i,l,k}^j \Vert_2\lesssim M(\bog)2^{-\max(j,0)\beta} .
\end{align}
Moreover, we have $\Vert \bth_{\bog,i}^T\bar B_i - \bog \Vert_\infty \lesssim 2^{-i\beta}$. Let $\tilde \bth \in \R^{K^2c(i)}$ such that  $\Vert \bog - \tilde \bth^T\bar B_i\Vert_2\leq \gamma$. First, note that 
\begin{align}
    \Vert \tilde \bth -   \bth_{\bog,i} \Vert_\infty \leq \Vert \tilde \bth -   \bth_{\bog,i} \Vert_2=\Vert \tilde \bth^{T}\bar B_i-   \bth_{\bog,i}^{T}\bar B_i \Vert_2\lesssim \gamma + 2^{-i\beta} \nonumber.
\end{align}
Using the previous inequality, we find that for all $(l,k)\in \K^2$ and $j\in \{-1,0,...,i\}$;
\begin{equation}\label{ineq_gen_wav}
\begin{aligned}
    \Vert \tilde \bth_{l,k}^j \Vert_\infty&\leq \Vert \tilde \bth_{l,k}^j -   \bth_{\bog,i,l,k}^j \Vert_\infty +  \Vert \bth_{\bog,i,l,k}^j\Vert_\infty \\
    &\leq \Vert \tilde \bth -   \bth_{\bog,i} \Vert_\infty + M(\bog)2^{-\max(j,0)(\beta+1/2)}\\
    &\lesssim  \gamma+ 2^{-i\beta} + M(\bog)2^{-\max(j,0)(\beta+1/2)},
\end{aligned}
\end{equation}
and, using  in addition (\ref{ineq_wav_infty}), we obtain similarly,
\begin{equation}\label{ineq_gen_wav_2}
\begin{aligned}
    \Vert \tilde \bth_{l,k}^j \Vert_2&\leq \Vert \tilde \bth_{l,k}^j -   \bth_{\bog,i,l,k}^j \Vert_2 +  \Vert \bth_{\bog,i,l,k}^j\Vert_2&\lesssim  \gamma+ 2^{-i\beta} + M(\bog)2^{-\max(j,0)\beta}\\
    &\lesssim \gamma + M(\bog)2^{-\max(j,0)\beta}.
\end{aligned} 
\end{equation}

For the wavelet prior in the ReLU model,  we proceed similarly to the histogram prior.  First, we have $\Vert \bar \bth_{0,i}\Vert_\infty\leq \Vert \bar \bth_{0,i}\Vert_2= \Vert \bar \bog^{0,i}_{L,\varphi}\Vert_2\lesssim \Vert \bar \bog^{0,i}_{L,\varphi}\Vert_\infty$ and since $\sup \{\Vert \Vert \bar \bog^{0,i}_{L,\varphi}\Vert_2\Vert_\infty , i\geq -1\}<+\infty$ , there exists some finite $G_0>0$ such that $\sup\{\Vert \bar \bth_{0,i}\Vert_\infty , i\geq -1\}\leq G_0$. Then, let $(\nu,\bth) \in\Theta_T(i)$ for some $i\in \mathcal{I}_T$, we have $\Vert \bth^T\bar B_i -\boh^0\Vert_2\lesssim \log(T)\varepsilon_T(\beta)$ and  by (\ref{ineq_gen_wav}) (with $\bog =\boh^0$ and $\tilde \bth = \bth$) we obtain that for some $C_0>0$ independent of $i$ and $\bth$,  $\Vert \bth \Vert_\infty \lesssim C_0$. Thus, for $T$ large enough (independently of $i$), $\Theta_T^S(i)$ contains only elements $\bth$ with coefficients that are in some bounded interval $B''$, and from that, we can  proceed as for the histogram prior and it shows that (\ref{bound_change_j}) is true for the wavelet prior in the ReLU model.
 
Now, we turn to the wavelet prior with $\varphi\geq 0$, we first treat case (i) (Gaussian priors). Since $\bog_{L,\varphi}^0\in  \mathcal{B}^\beta_{\infty,\infty}$,  $\Vert  \bog^{0}_{L,\varphi} - \bar\bth_{0,i}^T\bar B_i\Vert_2=\Vert  \bog^{0}_{L,\varphi} - \bar \bog^{0,i}_{L,\varphi}\Vert_2\lesssim 2^{-i\beta}$. Whence, by applying (\ref{ineq_gen_wav_2}) (with $\bog= \bog^{0}_{L,\varphi}$, $\tilde \bth =\bar \bth_{0,i}$ and $\gamma= 2^{-i\beta}$), we obtain that for all $(l,k)\in \K^2$, and for all $j\in \{-1,...,i\}$, $\Vert  \bar \bth_{0,i,l,k}^j\Vert_2\lesssim 2^{-\max(j,0)\beta}$. Then, let $(\nu,\bth) \in\Theta_T(i)$ for some $i\in \mathcal{I}_T$ ($i\leq I_T(\beta)$), by (\ref{ineq_gen_wav_2}) (with $\bog= \tilde \boh^0$, $\tilde \bth=\bth$ and $\gamma= \log(T)\varepsilon_T(\beta)$) , we have $\Vert \bth_{l,k}^j \Vert_2\lesssim \log(T)2^{-\max(j,0)\beta}$ for all $j\leq i$.  Moreover, we have
\begin{align}
      &\frac{\pi_{\bth\vert i}\big(\bth +  u\bar \bth_{0,i}/\sqrt{T}\big)}{\pi_{\bth\vert i}(\bth)}\nonumber\\
      &=  \exp\bigg(-\hspace{-0.3cm}\sum_{(l,k)\in \K^2}\sum_{ j=-1}^i 2^{\max(j,0)(2\beta_m+1)}\sum_{v=0}^{\bar v_{j}} \frac{2u \theta_{l,k}^{j,v} \bar \theta_{0,i,l,k}^{j,v}}{2\sqrt{T}}+  \frac{(u\bar \theta_{0,i,l,k}^{j,v})^2}{T}\bigg),  \nonumber
\end{align}
and with the previous remarks, and using that $\vert 2xy\vert \leq x^2 + y^2$ and that $\beta >\beta_m$, we find
\begin{align}
    & \Big \vert \sum_{ j=-1}^i 2^{\max(j,0)(2\beta_m+1)}\sum_{v=0}^{\bar v_{j}}\frac{2u \theta_{l,k}^{j,v} \bar \theta_{0,i,l,k}^{j,v}}{\sqrt{T}}+  \frac{(u\bar \theta_{0,i,l,k}^{j,v})^2}{T}\Big \vert\nonumber\\
    &\lesssim  \sum_{ j=-1}^i 2^{\max(j,0)(2\beta_m+1)}\Big(\frac{ \Vert \theta_{l,k}^{j}\Vert_2^2 + \Vert \bar \theta_{0,i,l,k}^{j}\Vert_2^2}{\sqrt{T}}+  \frac{\Vert \bar \theta_{0,i,l,k}^{j}\Vert_2^2}{T}\Big)\nonumber\\
    &\lesssim \frac{\log(T)^2}{\sqrt{T}}\sum_{ j=-1}^{i} 2^{\max(j,0)(2\beta_m+1)} 2^{-\max(j,0)2\beta} \nonumber\\
    &\leq \frac{\log(T)^2}{\sqrt{T}}\sum_{ j=-1}^{I_T(\beta)} 2^{\max(j,0)}\lesssim \frac{\log(T)^2 2^{I_T(\beta)}}{\sqrt{T}}=o(1)\nonumber, 
\end{align}
since $2^{I_T(\beta)}\asymp T^{1/(2\beta+1)}\log(T)^\alpha$ for some $\alpha\geq 0$ and since again $\beta >\beta_m>1/2$. Hence, we can conclude that (\ref{bound_change_j}) is true for the truncated wavelet prior with $\varphi \geq 0$ in case (i) (Gaussian distributions).

Finally, let us study case (ii) (uniform distributions). We recall that in this case, for $j\geq 0$, $\bar \Pi_{j,v}=\mathcal{U}([-\log(T)^2 2^{-j/2}, \log(T)^2 2^{-j/2}])$. As in case (i), we have for all $(l,k)\in \K^2$ and for all $j\in \{-1,...,i\}$, 
$\Vert  \bar \bth_{0,i,l,k}^j\Vert_\infty\leq \Vert  \bar \bth_{0,i,l,k}^j\Vert_2\lesssim 2^{-\max(j,0)\beta}$. Moreover, let $i\in \mathcal{I}_T$ and let $(\nu,\bth) \in\Theta_T(i)$,   (\ref{ineq_gen_wav}) gives that $\Vert \bth_{l,k}^j \Vert_\infty\lesssim \log(T)2^{-\max(j,0)\beta}\nonumber$.  Whence, since $\beta>\beta_m =1/2$ there exists $C>0$ and $T_C$ large enough (both independent of $i$), such that for all $T\geq T_c$, $\Theta_T^S(i)$ contains only elements $\bth$ such that for all $j\in \{0,...,i\}$, $\Vert \bth_{l,k}^j\Vert_\infty \leq C\log(T)2^{-j/2}$. As a consequence, for $T$ large enough we have for all $(\nu, \bth)\in \Theta_T^S(i)$, $\pi_{\bth\vert j}\big(\bth +  u\bth_{0,j}/\sqrt{T}\big)=~\pi_{\bth\vert j}(\bth) $,
which proves (\ref{bound_change_j}) in this case.

So (\ref{bound_change_j}) holds for the three priors and with  (\ref{bound_change_j_nu}) it implies  that, whatever the prior, we have:
\begin{align}
    \frac{\int_{\A_{T}(j) }e^{L_T(f_{u,j})}d\Pi_{f\vert j}(f)}{\int_{\A_{T}(j)}e^{L_T(f)}d\Pi_{f\vert j}(f)} &= \frac{\int_{\Theta_{T}^S(j) }\exp\Big(L_T\big(\nu, \bth^TB_j\big)\Big) \pi_{K,\nu}(\nu)\pi_{\bth\vert j}(\bth) d\nu d\bth}{\int_{\Theta_{T}(j)}\exp\Big(L_T\big(\nu, \bth^TB_j\big)\Big) \pi_{K,\nu}(\nu)\pi_{\Theta\vert j}(\bth) d\nu d\bth}\times \big(1+o(1)\big) \nonumber\\
    &= \Pi\big(\A_T^S(j) \big\vert N, j, \A_T(j)\big) \times \big(1+o(1)\big),\nonumber
\end{align}
with $o(1)$ uniform in $j$ and where $\A_T^S(j) = \A_T(j) - u\tilde \psi^{0,j}_{L,\varphi}/\sqrt{T}$. Let $\A_T^S = \cup_{j\in \mathcal{J}_T} \A_T^S(j)$, as a consequence we find
\begin{align}
    \sum_{j\in \mathcal{J}_T}\frac{\int_{\A_{T}(j) }e^{L_T(f_{u,j})}d\Pi_{f\vert j}(f)}{\int_{\A_{T}(j)}e^{L_T(f)}d\Pi_{f\vert j}(f)}\Pi_J(j\vert N) &= (1 +o(1))\sum_{j\in \mathcal{J}_T}\Pi\big(\A_T^S(j) \vert N, j, \A_T(j)\big) \Pi_J(j\vert N)\nonumber\\
    &= (1 +o(1)) \Pi\big(\A_T^S \vert N,\mathcal{A}_T\big)\nonumber\\
    &= (1 +o_{\P_0}(1)) \Pi\big(\A_T^S \vert N\big),\nonumber
\end{align}
because by Lemma \ref{lemma_concentration_A}, $\Pi(\A_T\vert N)\xrightarrow{\P_0} 1$. Since  $\sup \{\bog^{0,j}_{L,\varphi} \Vert_\infty, j\geq 1\}<+\infty $,  one can show (as in the proof of Theorem 4.2 of \cite{castillo_rousseau}) that by choosing the constant $M$ large enough in the definition of $\A_T$,  $\Pi\big(\A_T^S \vert N\big) = 1 +o_{\P_0}(1)$ and it concludes the proof of Lemma \ref{change_var_lemma}.
\end{proof}

\begin{lemma}\label{verif_P2_A_wavelets}
    Consider the setting of Corollary \ref{wavelets_coro_varphi} (truncated wavelet bases with $\varphi\geq~0$), the prior verifies Assumptions \hyperref[P2]{(P2)} and \hyperref[as_A']{(A')}.
\end{lemma}

Before proving this lemma, we recall a classical Gaussian concentration inequality that can be found in \cite{book_concentration} (Theorem 5.6). Let $X_1,...,X_n$ be i.i.d. standard and centered Gaussian random variables and $F:\R^n\rightarrow \R$ a $L$-Lipschitz function with respect to the Euclidean norm $\Vert .\Vert_2$ on $\R^n$. Then, for any $u>0$,
\begin{align}\label{gaussian_ineq}
    \P\big ( F(X_1,...,X_n) \geq \E[F(X_1,...,X_n)]+u\big)\leq \exp \Big(-\frac{u^2}{2L^2}\Big).
\end{align}

\begin{proof}[Proof of Lemma \ref{verif_P2_A_wavelets}]

We first verify Assumption \hyperref[P2]{(P2)} in case (i) (Gaussian distributions) and then we verify it in case (ii) (uniform distributions). In each case, when verifying \hyperref[P2]{(P2)}, we also verify \hyperref[as_A']{(A')}.

In case (i), for \hyperref[P2]{(P2)}, we begin by verifying (\ref{P2_h+}). Let $r_T = r_T(\beta) = \sqrt{\log(T)T\varepsilon_T(\beta)^2}$. Given the prior construction and since $\varphi$ is Lipschitz and $r_T(\beta)\rightarrow +\infty$, it is enough to show that for $I\leq I_T(\beta)$ and for variables $\boldsymbol{X}=(X_{i,v}, i=1,...,I, v=0,...,\bar v_I)\overset{i.i.d.}{\sim}\mathcal{N}(0,1)$, we have
\begin{align}\label{objP2_1}
    \P\bigg(\Big\Vert \sum_{i= -1}^I\sum_{v=0}^{\bar v_i} 2^{-\max(i,0)(\beta_m+1/2)}X_{i,v}\psi_{i,v} \Big\Vert_\infty >r_T(\beta)/2\bigg)\leq e^{-w_T T\varepsilon_T^2},
\end{align}
for some $w_T\rightarrow +\infty$. Let $\boldsymbol{X}^{(i)} = (X_{i,0}, ..., X_{i, \bar v_i})$. First, there exists a finite integer $B$, such that for any $x\in [0,A]$ and $i\geq -1$, $Card\big(\{v: \psi_{i,v}(x)\neq 0\}\big)\leq B$. Moreover, since $\Vert \psi_{i,v}\Vert_\infty \lesssim 2^{\max(i,0)/2}$, we obtain
\begin{align}
    \Big\Vert \sum_{i= -1}^I\sum_{v=0}^{\bar v_i} 2^{-\max(i,0)(\beta_m+1/2)}X_{i,v}\psi_{i,v} \Big\Vert_\infty \lesssim  \sum_{i= -1}^I 2^{-\max(i,0)\beta_m}\Vert \boldsymbol{X^{(i)}}\Vert_\infty.\nonumber
\end{align}
Furthermore, the application $F:\R^{c(I)}\rightarrow\R$ defined by $F(\boldsymbol{x})= \sum_{i= -1}^I 2^{-\max(i,0)\beta_m}\Vert \boldsymbol{x}^{(i)}\Vert_\infty$ is $c$-Lipschitz with respect to the Euclidean norm for some $c>0$ and $\E[ F(\boldsymbol{X})] \leq C$ for some $C>0$. Whence, since $r_T(\beta)\rightarrow +\infty$, for $T$ large enough and for some $\bar c>0$ we have
\begin{align}
     &\P\bigg(\Big\Vert \sum_{i= -1}^I\sum_{v=0}^{\bar v_i} 2^{-\max(i,0)(\beta_m+1/2)}X_{i,v}\psi_{i,v} \Big\Vert_\infty >r_T(\beta)/2\bigg)\nonumber\\
     &\leq \P\Big(\sum_{i= -1}^I 2^{-\max(i,0)\beta_m}\Vert \boldsymbol{X^{(i)}}\Vert_\infty \geq \bar cr_T(\beta)/2\Big) \leq \P\Big(F( \boldsymbol{X})\geq \E[ F(\boldsymbol{X})]+\bar cr_T(\beta)/4\Big) \nonumber\\
    &\leq \exp\big(- \bar c^2r_T(\beta)^2/(32c^2)\big)\nonumber\\
    &=\exp\big(- \bar c^2\log(T)T\varepsilon_T(\beta)^2/(32c^2)\big)\leq \exp\big(-  \sqrt{\log(T)}T\varepsilon_T(\beta)^2\big),\nonumber
\end{align}
and it proves (\ref{objP2_1}).

We turn to (\ref{P2_G}). Let $\bar c_0= \max_{l,k}\max_{x\in [0,A]} h^0_{l,k}(x)$ and  $\underline c_0= \min_{l,k}\min_{x\in [0,A]} h^0_{l,k}(x)$. For $\epsilon>0$, when it is well-defined, let $U_\epsilon := \varphi^{-1}(]\underbar c_0 -\epsilon ; \bar c_0 + \epsilon[)$. It is assumed that there exists $\epsilon_0>0$ small enough such that $\varphi$ is infinitely differentiable with a positive derivative on $U_{2\epsilon_0}= ]a_{2\epsilon_0}; b_{2\epsilon_0}[$, for some $b_{2\epsilon_0}> a_{2\epsilon_0}$. Let $\boldsymbol{X}$ be Gaussian variables as before and set $\tilde h = \sum_{i= -1}^I\sum_{v=0}^{\bar v_i} 2^{-\max(i,0)(\beta_m +1/2)}X_{i,v}\psi_{i,v}$ for some $I\leq I_T(\beta)$, and $h = \varphi(\tilde h)$. Similarly to (\ref{objP2_1}), given the prior construction and since $\varphi$ is globally Lipschitz,  it is enough to show that for any $(l,k)\in \K^2$, we have 
\begin{align}\label{obj_strong_P2}
    \P\Big( \Vert  h - h^0_{l,k} \Vert_2^2\leq \log(T)^2r_T(\beta)^2\varepsilon_T(\beta)^2,  Range(\tilde h)\not\subset U_{2\epsilon_0} \Big)\leq e^{-w_T T\varepsilon_T^2},
\end{align}
for some $w_T\rightarrow +\infty$. First, consider that for some $x'\in [0,A]$, $\tilde h_{l,k}(x')\leq a_{2\epsilon_0} $. Then, let $x\in [0,A]$, for each $i\geq -1$, there is at most $2B$ functions $\psi_{i,v}$ such that $\psi_{i,v}(x)\neq 0$ or $\psi_{i,v}(x')\neq 0$, and thus
\begin{align}
    \vert \tilde h(x) - \tilde h(x')\vert &\leq  \sum_{i= -1}^I\sum_{v=0}^{\bar v_i} 2^{-\max(i,0)(\beta_m+1/2)}X_{i,v}\big\vert \psi_{i,v}(x) - \psi_{i,v}(x') \big\vert \nonumber\\
    &\lesssim \sum_{i= -1}^I 2^{-\max(i,0)(\beta_m+1/2)}\Vert \boldsymbol{X^{(i)}}\Vert_\infty \underset{v=0,...,\bar v_i}{\max}\big\vert \psi_{i,v}(x) - \psi_{i,v}(x') \big\vert\nonumber\\
    &\lesssim \vert x-x'\vert  \sum_{i= -1}^I 2^{-\max(i,0)(\beta_m+1/2)}\Vert \boldsymbol{X^{(i)}} \Vert_\infty \underset{v=0,...,\bar v_i}{\max}\Vert \psi'_{i,v}\Vert_\infty\nonumber\\
    &\lesssim \vert x-x'\vert  \sum_{i= -1}^I 2^{\max(i,0)(1-\beta_m)}\Vert \boldsymbol{X^{(i)}} \Vert_\infty\nonumber\\
    &\lesssim \vert x-x'\vert  \sum_{i= -1}^I 2^{\max(i,0)/6}\Vert \boldsymbol{X^{(i)}} \Vert_\infty,\nonumber
\end{align}
because $\beta_m = 5/6$. Whence, for some $b>0$ 
\begin{align}
    \tilde h(x) \leq \tilde h(x')  + \vert \tilde h(x) -\tilde h(x')\vert \leq  a_{2\epsilon_0} + \vert x-x'\vert  b\sum_{i= -1}^I 2^{\max(i,0)/6}\Vert \boldsymbol{X^{(i)}} \Vert_\infty. \nonumber
\end{align}
If $\vert x -x'\vert \leq (a_{\epsilon_0} -a_{2\epsilon_0})\big( b\sum_{i= -1}^I 2^{\max(i,0)/6}\Vert \boldsymbol{X^{(i)}} \Vert_\infty \big)^{-1}=: C_0(\boldsymbol{X^{(i)}})$, then $\tilde h(x) < a_{\epsilon_0}$. Moreover, because $\varphi$ is globally non decreasing and strictly increasing on $U_{2\varepsilon_0}$, if  $\tilde h(u)<  a_{\epsilon_0}$, then $ \varphi(\tilde h_{l,k}^0(u)) -\varphi(\tilde h(u)) \geq \epsilon_0$.  As a consequence, 
\begin{align}
    \Vert h^0_{l,k} - h \Vert_2^2 &= \int_0^A \hspace{-0.2cm} \big( \varphi(\tilde h^0_{l,k}(u)) - \varphi(\tilde h(u)) \big)^2 du\nonumber\\
    &\geq  \epsilon_0^2 \int_0^A  \onee_{\big\{\vert u-x'\vert \leq C_0(\boldsymbol{X^{(i)}})\}}du\geq  \epsilon_0^2C_0(\boldsymbol{X^{(i)}}).\nonumber
\end{align}
This reasoning can be reproduced in the case where for some $x'$, $\tilde h(x')\geq b_{2\epsilon_0}$. It proves that for some $\bar b_0>0$ and $T$ large enough
\begin{align}
    &\P\Big( \Vert  h -  h^0_{l,k} \Vert_2^2\leq \log(T)^2r_T(\beta)^2\varepsilon_T(\beta)^2,  Range(\tilde h)\not\subset U_{2\epsilon_0}\Big)\nonumber\\
    &\leq \P\Big( \sum_{i= -1}^I 2^{\max(i,0)/6
    }\Vert \boldsymbol{X^{(i)}} \Vert_\infty \geq \frac{\bar b_0}{\log(T)^2r_T(\beta)^2\varepsilon_T(\beta)^2}\Big).\nonumber
\end{align}
Now, let $\bar F(\boldsymbol{X}) =  \sum_{i= -1}^I 2^{\max(i,0)/6}\Vert \boldsymbol{X^{(i)}} \Vert_\infty$.  $\bar F$ is $4\times 2^{I_T(\beta)/6}$-Lipschitz for $T$ large enough and  $ 2^{I_T(\beta)/6}\asymp  T^{\frac{1/6}{2\beta +1}}\log(T)^{\frac{\beta/2 +1/6}{2\beta +1}}$. Moreover, $\E[\bar F(\boldsymbol{X})] \lesssim \log(T)^2T^{\frac{1/6}{2\beta +1}}$ and  $\log(T)^4T^{\frac{1/6}{2\beta+1}} r_T(\beta)^2\varepsilon_T^2 \rightarrow 0$ because $\beta>5/6$. So, for $T$ large enough, with the Gaussian concentration inequality (\ref{gaussian_ineq}), we have
\begin{align*}
    &\P\Big( \sum_{i= -1}^I 2^{\max(i,0)/2}\Vert \boldsymbol{X^{(i)}} \Vert_\infty \geq \frac{\bar b_0}{\log(T)^2r_T(\beta)^2\varepsilon_T(\beta)^2}\Big)\nonumber\\
    &\leq \P\Big(  \bar F(\boldsymbol{X})\geq \E[\bar F(\boldsymbol{X})] + \frac{\bar b_0}{2\log(T)^2r_T(\beta)^2\varepsilon_T(\beta)^2} \Big)\nonumber\\
    &\leq \exp \bigg( \frac{-\bar b_0^2}{8\log(T)^4T^{\frac{1/3}{2\beta +1}}r_T(\beta)^4\varepsilon_T(\beta)^4}\bigg)\leq \exp\big(-\log(T)T\varepsilon_T(\beta)^2\big),\nonumber
\end{align*}
where for the last inequality we have used again that $\beta>\beta_m=5/6$. It concludes the proof for (\ref{P2_G}).

Since  we have proven that (\ref{obj_strong_P2}) holds, Assumption \hyperref[as_A']{(A')} is verified too and it terminates the proof  for case $(i)$.

Now, we treat case (ii) (uniform distributions). In this case $\beta_m= 1/2$. We begin again by verifying (\ref{P2_h+}). Similarly to case $(i)$, given the prior construction and since $\varphi$ is Lipschitz, it is enough to show that for $I\leq I_T(\beta)$ and for variables $\boldsymbol{U}=(U_{i,v}, i=1,...,I, v=0,...,\bar v_I)\overset{i.i.d.}{\sim}\mathcal{U}([-1,1])$, we have
\begin{align}\label{objP2_ii}
    \Big\Vert \sum_{v=0}^{\bar v_{-1}}C_0U_{-1,v}\psi_{-1,v} + \sum_{i= 0}^I\sum_{v=0}^{\bar v_i} \log(T)^2 2^{-i/2}U_{i,v}\psi_{i,v} \Big\Vert_\infty \leq r_T,
\end{align}
for some $r_T>0$. Let $\boldsymbol{U}^{(i)} = (U_{i,0}, ..., U_{i, \bar v_i})$. With similar arguments as for case (i), we find that
\begin{align}
    \Big\Vert \sum_{v=0}^{\bar v_{-1}}C_0U_{-1,v}\psi_{-1,v} + \sum_{i= 0}^I\sum_{v=0}^{\bar v_i} \log(T)^2 2^{-i/2}U_{i,v}\psi_{i,v} \Big\Vert_\infty &\lesssim  \log(T)^2\sum_{i= -1}^I \Vert \boldsymbol{U}^{(i)} \Vert_\infty\nonumber\\
    &\lesssim \log(T)^3.\nonumber
\end{align}
Whence, choosing $r_T= C\log(T)^3$ for some $C>0$ large enough, (\ref{objP2_ii}) is proved.

Then, to verify (\ref{P2_G}), following the same steps as in case (i), it is sufficient that for some $w_T\rightarrow+\infty$
\begin{align}\label{obj_strong_P2_ii}
    \P\Big( \Vert  h - h^0_{l,k} \Vert_2^2\leq C^2\log(T)^6\varepsilon_T(\beta)^2,  Range(\tilde h)\not\subset U_{2\epsilon_0} \Big)\leq e^{-w_T T\varepsilon_T^2},
\end{align}
and this holds as soon as we have for some $\bar b_0'>0$ chosen large enough
\begin{align}\label{concentration_ref}
    \P\Big( \sum_{i= -1}^I 2^{\max(i,0)}\Vert \boldsymbol{U^{(i)}} \Vert_\infty \geq \frac{\bar b_0'}{\log(T)^8\varepsilon_T(\beta)^2}\Big)\leq \exp(-w_T T\varepsilon_T(\beta)^2).
\end{align}
But, note that $\sum_{i= -1}^I 2^{\max(i,0)}\Vert \boldsymbol{U^{(i)}} \Vert_\infty \lesssim 2^{I_T(\beta)} \lesssim T^{1/(2\beta+1)}\log(T)^\alpha$ for some $\alpha>0$. Moreover, since $\beta >1/2$, for any $\Delta>0$ we have for $T$ large enough
\begin{align}
    \frac{1}{\log(T)^8\varepsilon_T(\beta)^2} >\Delta  T^{1/(2\beta+1)}\log(T)^\alpha,\nonumber
\end{align}
so that the probability on the left-hand side of (\ref{concentration_ref}) is equal to $0$ for $T$ large enough. Therefore (\ref{P2_G}) is verified.

Finally, as in case (i), Assumption \hyperref[as_A']{(A')} is verified since (\ref{obj_strong_P2_ii}) holds and it terminates the proof of the lemma.
\end{proof}

\section{Proofs and additional results for Section \ref{sec_lfeq}}\label{proof_sec_lfeq}

\subsection{Proof of Lemma \ref{lemma_linf}}\label{sec:proof:gamma_inf}
First, recall that $\Gamma$ is linear bijective operator of $\R^K\times L_2^{K^2}$.  As $\R^K\times L_\infty^{K^2}\subset \R^K\times L_2^{K^2}$, $\Gamma$ is also a linear bijective operator of $\R^K\times L_\infty^{K^2}$ to $\Gamma(\R^K\times ~L_\infty^{K^2})$. Moreover, it is shown in the proof of Lemma \ref{lemma_lb_p} that for all $(l,k)\in \K^2$ and $g\in \L_2$ we have $\Vert \zeta_{l,j,k}(g)\Vert_\infty \lesssim\Vert g\Vert_1$. So if $g\in L_\infty $, we have  $\Vert \zeta_{l,j,k}(g)\Vert_\infty \lesssim \Vert g\Vert_\infty$. With this last inequality, the upper and lower bounds on $p_{l,k}$ stated in Lemma \ref{lemma_lb_p} and using the expression of Lemma \ref{explicit_gamma}, one can show that  there exists $c>0$ such that for all $(\xi,\bog)\in \R^K\times \L_\infty^{K^2}$, $\Vert \Gamma(\xi, \bog)\Vert_\infty \leq c \Vert (\xi, \bog)\Vert_\infty$. In other words, $\Gamma$ is a bounded operator from $\R^K\times \L_\infty^{K^2}$ to  $\Gamma(\R^K\times \L_\infty^{K^2})\subseteq  \R^K\times \L_\infty^{K^2}$. It remains to show that
\begin{align*}
    \Gamma(\R^K\times \L_\infty^{K^2})=  \R^K\times \L_\infty^{K^2}.
\end{align*}
Let $(\xi,\bog) \in \R^K\times \L_\infty^{K^2}$, using again Lemma \ref{lemma_lb_p} and the converse expression (\ref{converse_expr}), one can again show similarly that $\Gamma^{-1}(\xi, \bog)\in \R^K\times \L_\infty^{K^2}$. It shows that $\Gamma^{-1}(\R^K\times \L_\infty^{K^2})\subseteq \R^K\times \L_\infty^{K^2}$ so $\R^K\times \L_\infty^{K^2}\subseteq  \Gamma(\R^K\times \L_\infty^{K^2})$ and we can conclude to the equality of these two spaces.

Now, we turn to the operator $\Gamma^e$ on $\R^K\times \L^{K^2}_\infty(\R)$. First, we have the following fact that we call $(*)$ for the remaining of the proof: $(\zeta_{l,j,k}^e(g_{j,k}), (l,j,k)\in \K^3)$  and $(\lambda_A^k(0,\bog_k), k\in \K)$ depends only on the values of the functions $\bog$ for $x\in[0,A]$. With $(*)$ it comes that for $\bog\in~ \L^{K^2}_\infty(\R)$, $\Gamma^e(\xi,\bog)_{\vert [0,A]} = \Gamma(\xi,\bog_{\vert [0,A]} )$. Moreover, as Lemma \ref{lemma_lb_p} is written for $p^e_{l,k}$ and $\zeta^e_{l,j,k}$, the same argument used for $\Gamma$ shows that $\Gamma^e$ is a bounded linear operator on $\R^K\times \L^{K^2}_\infty(\R)$. Next, we prove that this operator is injective. Let $(\xi^a, \bog^a),(\xi^b, \bog^b) \in \R^K\times \L^{K^2}_\infty(\R) $ such that $\Gamma^e(\xi^b, \bog^b)  =  \Gamma^e(\xi^a, \bog^a)$, in particular
\begin{align}
    \Gamma(\xi^b, \bog^b_{\vert [0,A]}) = \Gamma^e(\xi^b, \bog^b)_{\vert [0,A]}  =\Gamma^e(\xi^b, \bog^b)_{\vert [0,A]} = \Gamma(\xi^a, \bog^a_{\vert [0,A]})\nonumber
\end{align}
and consequently, $(\xi^b, \bog^b_{\vert [0,A]})=(\xi^a, \bog^a_{\vert [0,A]})$ as $\Gamma$ is injective. By $(*)$, for $i\in \{a,b\}$, the restricted solution $(\xi^i,\bog^i_{[0,A]})$ defines for all $(l,j,k)\in \K^3$ the function $\zeta_{l,j,k}^e(g^i_{j,k})$ on the whole real line, and thus $\zeta_{l,j,k}^e(g^a_{j,k})= \zeta_{l,j,k}^e(g^b_{j,k})$. Then, as the functions $p^e_{l,k}$ are bounded away from $0$ and $+\infty$ by Lemma \ref{lemma_lb_p}, we deduce by definition of $\Gamma^e$ that $\bog^a =\bog^b$ on the whole real line, which ends to show that $\Gamma^e$ is injective. It remains to show that it is also surjective. Let $(\xi',\bog')\in \R^K\times \L_\infty^{K^2}(\R)$, since $\Gamma$ is surjective on $\R^K\times \L_\infty^{K^2}$, there exists $(\xi, \bog_A)\in \R^K\times \L_\infty^{K^2}$ such that $(\xi',\bog'_{[0,A]}) = \Gamma(\xi,\bog_A)$. By $(*)$ again, $\bog_A$ defines the function $\zeta^e_{l,j,k}(g_{A,l,k})$ on the whole real line. Let $(\xi,\bog)$ be defined by $(\xi,\bog_{\vert [0,A]})=(\xi, \bog_A)$ and  for $x\in \R\backslash [0,A]$:
\begin{align}
    g_{l,k}(x)= \frac{g'_{l,k}(x) -\mu^0_l\sum_{j=1}^K\zeta^e(g_{A,j,k})(x)}{\mu^0_lp^e_{l,k}(x)} -\xi_k \hspace{0.1cm},\hspace{0.2cm}\forall (l,k)\in \K^2\nonumber
\end{align}
Then, $(\xi,\bog)\in \R^K\times \L_\infty^{K^2}(\R)$ and we have $\Gamma^e(\xi,\bog) = (\xi',\bog')$. Thus, $\Gamma^e$ is surjective and by Banach-Schauder theorem, $(\Gamma^e)^{-1}$ is bounded.

\subsection{Proof of Lemma \ref{final_lemma_operator}} \label{sec:pr:main_operator}

Let $(\xi,\bog)= (\Gamma^e)^{-1}(\xi', \bog')$ with $(\xi',\bog')\in \R^K\times C^\beta_b(\R)^{K^2}$, we first show that $\Vert \bog - \mathcal{K}_n\bog\Vert_{\infty} \lesssim C(\bog)n^{-\beta}$ for some $C(\bog)>0$ (that depends only on $\bog$). Note that since $C^\beta_b(\R)\subset L_\infty(\R)$, we have $\bog\in \L_\infty(\R)^{K^2}$ by Lemma \ref{lemma_linf}. Recall that $\mathcal{K}_n\tilde \bog:=~ (\mathcal{K}_n\tilde g_{l,k}, (l,k)\in \K^2)$ and we set $\mathcal{K}_n(\tilde \xi,\tilde \bog) := (\tilde \xi, \mathcal{K}_n\tilde \bog)$.  By Lemma \ref{lemma_linf}, $(\Gamma^e)^{-1}$ is a bounded operator on $\R^K\times \L_\infty^{K^2}(\R)$ and thus we have
\begin{align}
    \Vert \bog - \mathcal{K}_n\bog\Vert_{\infty}&= \Vert (\xi,\bog) - (\xi,\mathcal{K}_n\bog)\Vert_{\infty}= \Vert (\Gamma^e)^{-1}\Gamma^e(\xi,\bog) - (\Gamma^e)^{-1}\Gamma^e(\xi,\mathcal{K}_n\bog)\Vert_{\infty}\nonumber\\
    &\lesssim \Vert \Gamma^e(\xi,\bog) - \Gamma^e(\xi,\mathcal{K}_n\bog)\Vert_{\infty}\nonumber\\
    &= \Vert (\xi',\bog') - \mathcal{K}_n(\xi',\bog')+ \mathcal{K}_n(\xi',\bog')- \Gamma^e(\xi,\mathcal{K}_n\bog)\Vert_{\infty}\nonumber\\
    &\leq  \Vert \bog' - \mathcal{K}_n\bog'\Vert_{\infty} + \Vert \mathcal{K}_n\Gamma^e(\xi,\bog) - \Gamma^e(\xi,\mathcal{K}_n\bog)\Vert_{\infty} \nonumber\\
    &\lesssim n^{-\beta} + \Vert \mathcal{K}_n\Gamma^e(\xi,\bog) - \Gamma^e(\xi,\mathcal{K}_n\bog)\Vert_{\infty},\nonumber
\end{align}
where the last inequality  comes from the $\beta$-Hölder continuity assumption on $\bog'$. Now, using the expression of $\Gamma^e$ given in Lemma \ref{lemma_linf}, we obtain that
\begin{equation}\label{decomp_palm}
\begin{aligned}
    \Vert \mathcal{K}_n\Gamma^e(\xi,\bog) - \Gamma^e(\xi,\mathcal{K}_n\bog)\Vert_{\infty}
    &\leq \underset{k\in \K}{\max} \bigg \vert \E_0\bigg[\frac{\lambda_A^k(0, \bog_k - \mathcal{K}_n\bog_{k})}{\lambda_A^k(f^0_k)}\bigg]\bigg\vert\\
    &\hspace{1cm} + \underset{(l,k)\in \K^2}{\max } \mu_l^0\big \Vert \mathcal{K}_n(p_{l,k}^eg_{l,k})- p_{l,k}^e\mathcal{K}_ng_{l,k} \big \Vert_{\infty}\\
    &\hspace{1cm}+  \underset{(l,k)\in \K^2}{\max } \mu_l^0\sum_{j=1}^K \big\Vert \mathcal{K}_n\zeta^e_{l,j,k}(g_{j,k}) - \zeta_{l,j,k}^e(\mathcal{K}_n g_{j,k})\big\Vert_{\infty}.
\end{aligned}
\end{equation}
For the first term on the right-hand side of (\ref{decomp_palm}), let  $\bar p_{l,k}(x) := p^e_{l,k}(x)\onee_{x\in [0,A]}$.  Using formula (\ref{int_eq_palm}) in the Appendix, we have
\begin{align}
    \E_0\bigg[\frac{\lambda_A^k(0, \bog_{k} - \mathcal{K}_n\bog_k)}{\lambda_A^k(f^0_k)}\bigg]&= \sum_{l=1}^K \E_0\bigg[\int_0^A (g_{l,k} - \mathcal{K}_ng_{l,k})(A-s)\frac{dN_s^l}{\lambda_A^k(f^0_k)} \bigg]\nonumber\\
    &= \sum_{l=1}^K\mu^0_l \int_0^A   g_{l,k}(s)p_{l,k}(s)ds - \mu^0_l\int_0^A  (\mathcal{K}_ng_{l,k})(u)p_{l,k}(s)ds \nonumber\\
    &= \sum_{l=1}^K \mu^0_l\int_\R   g_{l,k}(s)\bar p_{l,k}(s)ds - \mu^0_l\int_\R  (\mathcal{K}_n g_{l,k})(s)\bar p_{l,k} (s) ds\nonumber\\
    &= \sum_{l=1}^K \mu^0_l\int_\R   g_{l,k}(s)\bar p_{l,k}(s)ds - \mu^0_l\int_\R g_{l,k} (s)\mathcal{K}_n \bar p_{l,k}(s)ds,\nonumber
\end{align}
where for the last equality we have used that $\mathcal{K}$ is symmetric about $0$. As the functions $\bar p_{l,k}$ are in $C^\beta([0,A])$, as shown in Lemma \ref{smooth_palm} in Section \ref{sec:smooth:p}, and equal to $0$ elsewhere and since $\mathcal{K}$ is supported on $[-1,1]$, we obtain
\begin{align}
    \underset{k\in \K}{\max} \bigg \vert \E_0\bigg[\frac{\lambda_A^k(0, \bog_{k} - \mathcal{K}_n\bog_{k})}{\lambda_A^k(f^0_k)}\bigg]\bigg\vert \lesssim \underset{k\in \K}{\max} \sum_{l=1}^K\Vert g_{l,k}\Vert_\infty\Vert \bar p_{l,k} - \mathcal{K}_n\bar p_{l,k}\Vert_1 \lesssim \Vert \bog\Vert_\infty\big(n^{-\beta} +  n^{-1}\big).\nonumber
\end{align}
For second term on the right-hand side of (\ref{decomp_palm}), since $\mathcal{K}$ is supported on $[-1,1]$ and since the functions $ (p_{l,k}^e, (l,k)\in \K^2)$ are in $C^\beta(\R)$, we have for any $(l,k)\in \K^2$ and $x\in \R$,
\begin{align}
    \vert \mathcal{K}_n(p_{l,k}^eg_{l,k})(x)- p_{l,k}^e(x)\mathcal{K}_ng_{l,k}(x)\vert  &= \bigg \vert \int_{-1}^{1} \mathcal{K}(u)g_{l,k}(x-\frac{u}{n})\big( p_{l,k}^e(x-\frac{u}{n}) - p_{l,k}^e(x) \big)du\bigg \vert\nonumber\\
    &\lesssim n^{-\beta} \Vert \mathcal{K}\Vert _1\Vert g_{l,k}\Vert_{\infty}.\nonumber
\end{align}
It proves that $\max_{(l,k)\in \K^2} \mu_l^0\big \Vert \mathcal{K}_n(p_{l,k}^eg_{l,k})- p_{l,k}^e\mathcal{K}_ng_{l,k} \big \Vert_{\infty}\lesssim n^{-\beta} \Vert \mathcal{K}\Vert_1\Vert \bog\Vert_\infty$.
It remains to study the third term on the right-hand side of (\ref{decomp_palm}). For this, recall the definition of $\lambda_t^{k,e}$ in (\ref{def_lambda_e}) and we set $\Delta_{u/n}(k,x,l)=\lambda_{A}^{k,e}(f^0_k,(x,l)) - \lambda_{A-u/n}^{k,e}(f^0_k,(x,l))$ (as in (\ref{def_delta}) in the proof of Lemma \ref{smooth_palm}).  Let $x\in \R$ and $l\in \K$. To shorten computations, we set for $u\in [-1,1]$, 
\begin{align}
     &v_{u/n}^{(x,l)}(j,k):=  \E_0^{(x,l)}\bigg[\int \big( \onee_{s\in[-\frac{u}{n},A-\frac{u}{n}]} - \onee_{s\in[0,A]} \big)  g_{j,k}(A-s-\frac{u}{n}) \frac{ \onee_{(s,j)\neq (x,l)}dN_s^j}{ \lambda_A^{k,e}(f^0_k,(x,l))}\bigg],\nonumber\\
     &w_{u/n}^{(x,l)}(j,k):=\E_0^{(x,l)}\bigg[\int_{-\frac{u}{n}}^{A-\frac{u}{n}} g_{j,k}(A-s-\frac{u}{n}) \bar \Delta_{u/n}(k,x,l) \onee_{(s,j)\neq (x,l)}dN_s^j\bigg], \nonumber\\
     &\bar\Delta_{u/n}(k,x,l):=\frac{\Delta_{u/n}(k,x,l)\times  \Big(\lambda_A^k(f^0_k) + h^{0,e}_{l,k}(A-x)\onee_{x\notin [0,A[}\Big)^{-1}}{ \lambda_{A-\frac{u}{n}}^k(f_0^k)+h^{0,e}_{l,k}(A- x-\frac{u}{n})\onee_{x+u/n\notin [0,A[} }.\nonumber
\end{align}
With Fubini theorem and formula (\ref{shift_palm}), we have:
\begin{align}
    &\mathcal{K}_n\zeta_{l,j,k}^e(g_{j,k})(x) - \zeta_{l,j,k}^e(\mathcal{K}_ng_{j,k})(x)\nonumber\\
    &= \int \mathcal{K}(u) \E_0^{(x+\frac{u}{n},l)}\bigg[\int_0^A g_{j,k}(A-s)\frac{\onee_{(s,j)\neq (x+\frac{u}{n},l)}dN_s^j}{\lambda_A^k(f_0^k) +  h^{0,e}_{l,k}(A- x-\frac{u}{n})\onee_{x+u/n\notin[0,A[}}\bigg]du \nonumber\\
    &\hspace{0.8cm}-  \E_0^{(x,l)}\bigg[\int_0^A\onee_{(s,j)\neq (x,l)}\int \mathcal{K}(u)g_{j,k}(A-s-\frac{u}{n})du \frac{dN_s^j}{\lambda_A^k(f^0_k) +h^{0,e}_{l,k}(A-x)\onee_{x\notin [0,A[}}\bigg]\nonumber\\
    &= \int \mathcal{K}(u) \E_0^{(x,l)}\bigg[\int_{-\frac{u}{n}}^{A-\frac{u}{n}} g_{j,k}(A-s-\frac{u}{n}) \frac{\onee_{(s,j)\neq (x,l)}dN_s^j}{\lambda_{A-\frac{u}{n}}^k(f_0^k) +  h^{0,e}_{l,k}(A- x-\frac{u}{n})\onee_{x+u/n\notin[0,A[}}\bigg]du \nonumber\\
    &\hspace{0.8cm}- \int \mathcal{K}(u)\E_0^{(x,l)}\bigg[\int_0^A  \onee_{(s,j)\neq (x,l)} g_{j,k}(A-s-\frac{u}{n}) \frac{dN_s^j}{\lambda_A^k(f^0_k) +h^{0,e}_{l,k}(A-x)\onee_{x\notin [0,A[}} \bigg]du \nonumber\\
    &= \int \mathcal{K}(u)\big(v_{u/n}^{(x,l)}(j,k)+w_{u/n}^{(x,l)}(j,k)\big)du.\nonumber
\end{align}
Let $u\in [0,1]$, we have
\begin{align}
    \vert v_{u/n}^{(x,l)}(j,k)\vert &\lesssim \Vert g_{j,k}\Vert_\infty  \E_0^{(x,l)}\big[N^j([-u/n,0]\backslash\{x\}) + N^j([A-u/n,A]\backslash\{x\})\big]\\
    &\lesssim  n^{-1}\Vert g_{j,k}\Vert_\infty,\nonumber
\end{align}
where for the last inequality inequality we have used Lemma \ref{moment_measure_palm}. We turn to $w$. Since the functions $\boh^{0,e}$ are $\beta$-Hölder continuous, proceeding as in the proof of Lemma \ref{smooth_palm}, for $u\in [0,1]$ we find that under $\P^{(x,l)}$:
\begin{align}
    &\big\vert \bar \Delta_{u/n}(k,x,l)\big\vert \lesssim \big\vert  \Delta_{u/n}(k,x,l)\big\vert\nonumber\\
    &\lesssim  n^{-\beta}+ n^{-\beta}\underbrace{\Big(\sum_{i\neq l} N^i([0,A-u/n]) + N^l([0,A-u/n]\backslash\{x\}) \Big)}_{C^1_n}\nonumber\\
    &\hspace{1cm}+ \underbrace{\sum_{i\neq l}N^i\big([-u/n,0]\cup [A-u/n,A]\big) + N^l \big(([-u/n,0]\cup[A-u/n,A])\backslash\{x\}\big)}_{C^2_n}.\nonumber
\end{align}
Thus, using this time the bound (\ref{bound_mom_palm_second}) on the first Palm second moment measure, we obtain
\begin{align}
    \vert w_{u/n}^{(x,l)}(j,k) \vert &\lesssim \Vert \bog_{j,k}\Vert_\infty \E_0^{(x,l)}\bigg[N^j\big([-u/n, A-u/n]\backslash\{x\}\big)\Big(n^{-\beta}+ n^{-\beta}C^1_n+ C^2_n \Big)\bigg]\nonumber\\
    &\lesssim  \Vert \bog_{j,k}\Vert_\infty(n^{-\beta}+n^{-1}).\nonumber
\end{align}
By doing the same for $u\in [-1,0]$, we obtain: 
\begin{align}
    \underset{(l,k)\in \K^2}{\max } \mu_l^0\sum_{j=1}^K \big\Vert \mathcal{K}_n\zeta^e_{l,j,k}(g_{j,k}) - \zeta_{l,j,k}^e(\mathcal{K}_ng_{j,k})\big\Vert_{\infty}\lesssim  n^{-\beta} \Vert \bog\Vert_\infty\Vert \mathcal{K}\Vert_\infty.\nonumber
\end{align} 
and it allows to conclude that $\Vert \bog -\mathcal{K}_n\bog \Vert_\infty \leq C(\bog)n^{-\beta}$ for some constant $C(\bog)>0$ that depends only on $\bog$.

Now, we prove that it implies that the functions $\bog$ belong to $\mathcal{B}^\beta_{\infty,\infty,b}(\R)$ (recall that we already know that the functions $\bog$ are in $L_\infty(\R)$). To do so, because $\Vert \bog -\mathcal{K}_n\bog \Vert_\infty \lesssim C(\bog)n^{-\beta}$, we can apply the Proposition 2.3 of \cite{ker_picard} (taking $p=+\infty$) to the functions $\bog$ and it gives that these functions are in $\mathcal{B}^\beta_{\infty,\infty,b}(\R)$ (Condition (b) of this proposition is verified because we work with a kernel of the form $\mathcal{K}(x-y)$ with $\mathcal{K}$ continuous). It concludes the proof of the lemma.

\subsection{On the Palm moment measures} \label{sec:firstorderPalm}

In this Section, we recall and we prove some properties related to the moment measures under the Palm distributions. We use the notations from the Appendix.

Let $N$ be a $\K$-marked and
boundedly finite point process on $\R$. If $N$ has  a $(n+1)$-th moment measure, the  $n$-th moment measure of the Palm distribution is well defined and is linked to the reduced $(n+1)$-th moment measure $\bar M^{(n+1)}$ of $N$, see Proposition 13.2.VI of \cite{DVJ2}. Specializing to the case of  a linear multivariate stationary Hawkes process $N$ with parameters $f^0\in ]0,+\infty[^K\times \mathcal{H}$, we first study the Palm first moment measure in the following Lemma \ref{moment_measure_palm}.

\begin{lemma}\label{moment_measure_palm}
Let $N$ be a stationary multivariate Hawkes process with parameters $f^0\in~ ]0,+\infty[^K\times \mathcal{H}$, such that $\boh^0\in \L_\infty^{K^2}$. The Palm first moment measure of the process is given by:
\begin{align}
    \E^{(0,l)}_0[N^k(B)] = \onee_{0\in B, l=k} + \int_B m_{l,k}(u) du,\nonumber
\end{align}
with $m_{l,k}$ bounded. Moreover, if the functions $\boh^0$ are in $C^\beta([0,A])$ for some $\beta>0$, then the functions $(m_{l,k}, (l,k)\in \K^2)$ are in $C^\beta(\R)$.
\end{lemma}

For $x\in \R$, by (\ref{shift_palm}), we thus have $\displaystyle \E^{(x,l)}_0[N^k(B)] = \onee_{x\in B, l=k} + \int_B m_{l,k}(u-x) du$.

\begin{proof}[Proof of Lemma \ref{moment_measure_palm}]
By Proposition 13.2.VI of \cite{DVJ2}, $ \E^{(0,l)}_0[N^k(B)] =~ \bar M_{l,k}^2(B)/ \mu^0_l$,
with $\bar M^2$ is the reduced second moment measure of the Hawkes process. Then, using Equation (7) of \cite{hawkes:71:bis}, we obtain that 
\begin{align}
    \bar M_{l,k}^2(B) = \mu^0_l\onee_{0\in B, k=l} + \mu^0_l\mu^0_k\vert B\vert + \int_B \Upsilon_{l,k}( t) dt, \nonumber
\end{align}
with $\Upsilon$ verifying $\Upsilon(-t) = \Upsilon(t)^T$ and satisfying the following functional equation:
\begin{align}
    \Upsilon(t) = \boh^{0,T}(t) D(\mu^0) + (\boh^{0,T}\star \Upsilon)(t) \hspace{0.3cm},\hspace{0.5cm} t>0, \nonumber
\end{align}
where $D(\mu^0)$ is the diagonal matrix with $\mu^0$ as diagonal vector and $\star$ is the matrix product where all the multiplications are replaced by convolutions. Using that the functions $h^0_{l,k}$ are bounded, direct computations show that the functions $\Upsilon_{l,k}$ are also bounded. Whence, 
 $\E^{(0,l)}_0[N^k(.)]$ is the sum of the Dirac  $\onee\{0\in ., k=l\}$ and a measure absolutely continuous with respect to Lebesgue measure with a bounded Radon–Nikodym derivative given by  $ m_{l,k}(t):= \mu_k^0 + \Upsilon_{l,k}(t)/\mu_l^0$. Finally, given the functional equation verified by $\Upsilon$, it is clear that when the functions $\boh^0$ are in $C^\beta([0,A])$, the functions $(m_{l,k}, (l,k)\in \K^2)$ are in $C^\beta(\R)$.
\end{proof}

Then, we briefly study the Palm second moment measure and second-order Palm moment measure which are used in the proof of Lemmas \ref{final_lemma_operator}, \ref{smooth_palm_beta1} and \ref{orderliness_palm}. For these measures, we do not need explicit expressions as in the preceding Lemma \ref{moment_measure_palm}, we just need upper bounds that come directly with again Proposition 13.2.VI of \cite{DVJ2} (as in the previous proof). First, let $B_1$ and $B_2$ be two bounded measurable sets and  $(l,k_1,k_2)\in \K^3$ such that either $0\notin B_1$ or $l\neq k_1$, then we have:
\begin{align}\label{bound_mom_palm_second}
    \E_0^{(0,l)}\big[N^{k_1}(B_1)N^{k_2}(B_2)\big] \lesssim \vert B_1\vert\onee_{l=k_2, 0\in B_2} +\vert B_1 \cap B_2\vert \onee_{ k_1=k_2} + \vert B_1 \vert   \vert B_2 \vert.
\end{align}
Then, let $B$ be a bounded measurable set and $(x,l)\neq (s,j)$, if either $x\notin B$ or $l\neq k$ and if either $s\notin B$ or $j\neq k$, then
\begin{align}\label{second_palm_first_bound}
    \E_0^{(x,s,l,j)}[N^k(B)]\lesssim \vert B\vert .
\end{align}

\subsection{Boundedness and smoothness of $p^e_{l,k}$ and $\zeta^e_{l,j,k}$} \label{sec:smooth:p}

In this section, we prove two lemmas on the functions $p^e_{l,k}$ and the operators $\zeta^e_{l,j,k}$ defined by (\ref{def_plk_e}) and (\ref{def_zeta_e}) respectively. Both depend on the extended functions $\boh^{0,e}$ and we recall that these functions are bounded on $\R$ and for all $(l,k)\in \K^2$, $h^{0,e}_{l,k}\geq -\nu^0_k/2$.

\begin{lemma}\label{lemma_lb_p}
For all $(l,k)\in \K^2$, the function $p_{l,k}^e$ is upper bounded and lower bounded by some positive constant $c_{l,k}$. Moreover, there exists $C>0$ such that for any $g:~\R\rightarrow\R$ such that $g_{\vert [0,A]}$ is in $\L_2([0,A])$, for all $(l,k,j)\in \K^3$,  $\Vert \zeta_{l,j,k}^e(g)\Vert_\infty \leq ~C \Vert g_{\vert [0,A]}\Vert_1$. 
\end{lemma}

\begin{proof}[Proof of Lemma \ref{lemma_lb_p}.]
Let $(l,k,j)\in \K^3$. We begin with the function $p_{l,k}^e$. Clearly, $p_{l,k}^e \leq 2/\nu^0_k$ which gives the upper bound. For the lower bound, we have by Jensen inequality, formula (\ref{shift_palm}) and Lemma \ref{moment_measure_palm}:
\begin{align}
    p_{l,k}^e(A-u) &= \E_0^{(u,l)}\bigg[\frac{1}{ \lambda_A^{k,e}(f^0_k,(u,l))}\bigg]\nonumber\\
    &\geq\E_0^{(u,l)}\bigg[\frac{1}{2\Vert f^0\Vert_\infty+ \Vert \boh^0_k\Vert_\infty N([0,A])}\bigg]\nonumber \\
    &\geq \Big(2\Vert f^0\Vert_\infty + \Vert \boh^0_k\Vert_\infty \E^{(u,l)}_0[N([0,A])]\Big)^{-1}\nonumber \\
    &= \Big(2\Vert f^0\Vert_\infty + \Vert \boh^0_k\Vert_\infty \E^{(0,l)}_0[N([-u,A-u])]\Big)^{-1}\nonumber \\
    &\geq \Big(2\Vert f^0\Vert_\infty + \Vert \boh^0_k\Vert_\infty\big(1+A\sum_{k=1}^K\Vert m_{l,k}\Vert_\infty\Big)^{-1}=:c_{l,k}.\nonumber 
\end{align}
So, we have shown that $p_{l,k}^e \geq c_{l,k}>0$.

For the functions $\zeta_{l,j,k}^e(g)$, with again Lemma \ref{moment_measure_palm} we have
\begin{align}
    \vert \zeta_{l,j,k}^e(g)(A-u)\vert &\leq \frac{2}{\nu^0_k} \E^{(u,l)}_0\bigg[ \int_0^A \onee_{(s,j)\neq (u,l)} \vert g(A-s)\vert dN_s^j\bigg] \nonumber\\
    &\leq \frac{2}{\nu^0_k}\E^{(0,l)}_0\bigg[ \int_{-u}^{A-u} \onee_{(s,j)\neq (0,l)} \vert g(A-u-s)\vert dN_s^j\bigg] \nonumber\\
    &\leq\frac{2}{\nu^0_k} \int_{-u}^{A-u} \onee_{(s,j)\neq (0,l)} \vert g(A-u-s)\vert m_{l,j}(s)ds \leq \frac{2\Vert m_{l,j}\Vert_\infty}{\nu^0_k} \Vert g_{\vert [0,A]}\Vert_1,\nonumber
\end{align}
and it proves the result.
\end{proof}

The next lemma is on the smoothness of the functions $p_{l,k}^e$ when it is assumed that the functions $\boh^0$ are in $C^\beta([0,A])$. We recall that in this case, the extended functions $\boh^{0,e}$ are in $C^\beta(\R)$.

\begin{lemma}\label{smooth_palm}
If the functions $h^0_{l,k}$ belong to $C^\beta([0,A])$ for some  $\beta \in ]0,1]$, then the functions $p_{l,k}^e$  belong to $C_b^\beta(\R)$.
\end{lemma}

\begin{proof}[Proof of Lemma \ref{smooth_palm}.]  By Lemma \ref{lemma_lb_p}, we already know that the functions $p_{l,k}^e$ are in $L_\infty(\R)$, we just have to show that they are also in $C^\beta(\R)$. Let $(l,k)\in \K^2$ and  $x,y\in \R$.  Without loss of generality, assume that $x<y$. First, if $y-x>A$, because $p_{l,k}^e$ is bounded, we have directly $ \vert p_{l,k}^e(y) - p_{l,k}^e(x)\vert \lesssim \vert y-x\vert$. Now, we consider that $y-x\leq A$. For $t\in \R$, let
\begin{align}\label{def_lambda!}
    \lambda_t^{k}(f^0_k, !(x,l)):= \nu^0_k + \sum_{j=1}^K\int_{t-A}^t\onee_{(u,j)\neq  (x,l)} h^0_{j,k}(t-u)dN_u^j.
\end{align}
Since $N$ has a fixed atom at $(x,l)$ under $\P^{(x,l)}$  (see Appendix), we have under $\P^{(x,l)}$
\begin{align}
    \lambda_t^{k,e}(f^0_k,(x,l)) = \lambda_t^k(f^0_k, !(x,l)) + h^0_{l,k}(t-x)\onee_{x\in [t-A,t)}+ h^{0,e}_{l,k}(t-x)\onee_{x\notin [t-A,t)}, \nonumber
\end{align}
and thus
\begin{align}
    p_{l,k}^e(A-x)=\E_0^{(x,l)}\bigg[\frac{1}{\lambda_A^k(f^0_k,!(x,l))+ h^{0,e}_{l,k}(A-x) }\bigg].\nonumber
\end{align}
Then, using  in addition formula (\ref{shift_palm}), we also have:
\begin{align}
    p_{l,k}^e(A-y) &= \E^{(y,l)}_0\bigg[\frac{1}{\lambda_A^k(f^0_k)+ h^{0,e}_{l,k}(A-y)\onee_{y\notin [0,A[}}\bigg]\nonumber\\
    &=\E^{(x,l)}_0\bigg[\frac{1}{\lambda_{A+x-y}^k(f^0_k)+h^{0,e}_{l,k}(A-y)\onee_{y\notin [0,A[}}\bigg]\nonumber\\
    &= \E^{(x,l)}_0\bigg[\frac{1}{\lambda_{A+x-y}^k(f^0_k, !(x,l))+ h^{0,e}_{l,k}(A-y)}\bigg].\nonumber
\end{align}
For $\delta\in \R$, let 
\begin{align}\label{def_delta}
    \Delta_\delta(k,x,l)&:=  \lambda_{A}^{k,e}(f^0_k,(x,l)) - \lambda_{A-\delta}^{k,e}(f^0_k,(x,l)),
\end{align}
and note that under $\P^{(x,l)}$, when $\delta>0$
\begin{equation}\label{decomp_delta}
\begin{aligned}
    &\Delta_\delta(k,x,l)=  h^{0,e}_{l,k}(A-x)- h^{0,e}_{l,k}(A-x-\delta)+ \lambda_{A}^k(f^0_k, !(x,l))-\lambda_{A-\delta}^k(f^0_k, !(x,l))\\
    &=  \big(h^{0,e}_{l,k}(A-x) - h^{0,e}_{l,k}(A-x-\delta)\big) \\
     &+\sum_{j=1}^K\int_{0}^{A-\delta}\onee_{(s,j)\neq (x,l)} (h^0_{j,k}(A-s) - h^0_{j,k}(A-s-\delta))dN^j_s \\ 
                    &+\sum_{j=1}^K-\int_{-\delta}^0 \onee_{(s,j)\neq (x,l)} h^0_{j,k}(A-s-\delta)dN^j_s+ \sum_{j=1}^K\int_{A-\delta}^A \onee_{(s,j)\neq (x,l)} h^0_{j,k}(A-s)dN^j_s\\
\end{aligned}
\end{equation}
and if $\delta<0$, the same decomposition holds after adapting the limits of integration. Since the functions $h^{0,e}_{l,k}$ are $\beta$-Hölder on $\R$, we find
\begin{align}
    &\vert p_{l,k}^e(A-y) - p_{l,k}^e(A-x)\vert\lesssim\nonumber\E_0^{(x,l)}\Big[ \vert\Delta_{y-x}(k,x,l)\vert \Big] \\
    &\hspace{0.5cm}\lesssim \vert y-x\vert^\beta+ \E_0^{(x,l)}\Big[\big\vert  \lambda_A^k(f^0_k,!(x,l)) - \lambda_{A+x-y}^k(f^0_k,!(x,l))\big \vert \Big] \nonumber\\
    &\hspace{0.5cm}\lesssim \vert y-x\vert^\beta +  \vert y-x\vert^\beta \E_0^{(x,l)}\bigg[\sum_{j\neq l}N^j([0, A+x-y]) + N^l([0, A+x-y]\backslash\{x\})\bigg] \nonumber\\
    &\hspace{1.5cm}+ \Vert \boh^0\Vert_\infty\E_0^{(x,l)}\bigg[\sum_{j\neq l}N^j([x-y, 0]\cup [A+x-y,A]) \bigg]\nonumber\\
    &\hspace{1.5cm}+ \Vert \boh^0\Vert_\infty\E_0^{(x,l)}\bigg[ N^l\big(([x-y, 0]\cup [A+x-y,A])\backslash\{x\}\big)\bigg]\nonumber\\
    &\hspace{0.5cm}\lesssim \vert y-x\vert^\beta + \vert y-x\vert, \nonumber
\end{align}
where for the last inequality we have used Lemma \ref{moment_measure_palm} on the Palm first moment measure. It ends to prove that the function $p^e_{l,k}$ is in $C^\beta(\R)$.
\end{proof}

\subsection{Extension to the case $\beta>1$}\label{sec:pr:beta>1}
In this section we aim to prove that when functions $\boh^0$ and $\bog^0_2$ are in $C^\beta([0,A])$ for some $\beta>1$, then the functions $\bog^0_L$ are also in $C^\beta([0,A])$ (it extends Corollary \ref{g0_palm_holder} to the case $\beta>1$). We recall that $\bog^0_L$ verifies
\begin{align}
    g^0_{L,l,k} = \frac{g^0_{2,l,k} - \mu^0_l\sum_{j=1}^K \zeta_{l,j,k}(g^0_{L,j,k}) }{\mu^0_l p_{l,k}} - \xi^0_{L,k} \hspace{0.1cm},\hspace{0.2cm}\forall (l,k)\in \K^2.\nonumber
\end{align}
 We want to show, using the second order Palm distribution, that the functions $p_{l,k}$ and $\zeta_{l,j,k}(g^0_{L,j,k})$ are in $C^\beta([0,A])$ and since the functions $p_{l,k}$ are bounded away from $0$ (Lemma \ref{lemma_lb_p}), it would prove the desired result. In particular, by applying (\ref{iter_palm}), we obtain for the most problematic term $\zeta_{l,j,k}(g^0_{L,j,k})$ a more explicit expression involving the second order Palm distribution:
\begin{align}\label{explicit_zeta}
    \zeta_{l,j,k}(g^0_{L,j,k})(A-x)= \int_0^{A}g^0_{L,j,k}
    (A-s)\E_0^{(x,s,l,j)}\Big[ \frac{1}{\lambda_A^k(f^0_k)} \Big] m_{l,j}(s-x)ds,
\end{align}
where we recall that $\E^{(x,u,l,j)}_0$ is an expectation under the second order Palm distribution (see the Appendix). However, our approach requires certain regularity properties on the second order Palm distribution that we have not been able to establish but that we believe to hold. We conjecture these regularity properties.

Before stating the conjecture, we recall and introduce some notations. We denote by $\Vert . \Vert_{C^\beta}$ the usual Hölder norm on $[0,A]$. A map $x\mapsto \phi(x,s)$ will be sometimes written for short $\phi(\cdot{,}s)$. Furthermore, we denote by $\partial h$ and $\partial^m h$ the first derivative and the $m$-th derivative of a function $h$, respectively. Recall the notation $\lambda_A^k(f^0_k, !(s,j))$ defined by (\ref{def_lambda!}). Let $x,s$ in a neighborhood of $[0,A]$, $(l,j,k)\in\K^3$, $r\in \mathbb{N}$, $0\leq i\leq r-1$ and $c\geq 0$, we define
\begin{align}
    &p^{[2]}_{r,k}(x,s,l,j):=\E^{(x,s,l,j)}_0\bigg[\frac{1}{\lambda_{A}^{k}(f^0_k)^r}\bigg],\nonumber\\
    &p^{![2]}_{r,i,c,k}(x,s,l,j) := \E^{(x,s,l,j)}_0\bigg[\frac{1}{\lambda_{A}^{k}(f^0_k,!(s,j))^{r-i}\big(\lambda_{A}^{k}(f^0_k,!(s,j))+c\big)^{i+1}}\bigg]\nonumber.
\end{align}
In the following conjecture, the regularity properties must hold for any $(l,j,k)\in \K^3$, $r\in ~\{1,..., \lfloor\beta\rfloor\}$, $i\in \{0,...,r-1\}$ and $c\geq 0$.

\begin{conjecture}\label{conj_beta1}
Consider that the functions $\boh^0$ are in $C^\beta([0,A])$ for some $\beta>1$.  We conjecture that for all $x\in [0,A]$, there exists a version of the function $s\mapsto p^{![2]}_{r,i,c,k}(x,s,l,j)$ which is continuous at $A$ and at $0$. Moreover, we conjecture that there are versions of the functions $x\mapsto p^{![2]}_{r,i,c,k}(x,A,l,j)$ and $x\mapsto p^{![2]}_{r,i,c,k}(x,0,l,j)$ that are in $C^{\beta}([0,A])$, and that
\begin{align}
    \int_0^A \Vert p^{[2]}_{r,k}(\cdot{,}s,l,j)\Vert_{C^{\beta}} \hspace{0.1cm}ds <+\infty.\nonumber
\end{align}

\end{conjecture}

Since the second order Palm distribution is symmetric in $(x,s)$,  we think that if one can show that the functions $x\mapsto \E^{(x,s,l,j)}_0[\lambda_{A}^{k}(f^0_k)^{-1}]$ are $\beta$-Hölder, then one can show similarly that this conjecture holds. Note that this conjecture holds in the Poisson case (namely when $\boh^0=0$) since the second order Palm distribution of a Poisson process at $(x,s,l,j)$ is just the convolution between the original Poisson process, a Dirac measure at $(x,l)$ and a Dirac measure at $(s,j)$ (see Lemma 6.15 of \cite{kal} for instance).

Under Conjecture \ref{conj_beta1}, with (\ref{explicit_zeta}) and since by Lemma \ref{moment_measure_palm} there exists $L>0$ such that for all $s\in [0,A]$ the function $x\rightarrow m_{l,j}(s-x)$ is $(L,\beta)$-Hölder, it is clear that the functions $\zeta_{l,j,k}(g^0_{L,j,k})$ are in $C^\beta([0,A])$. Then, we show that the functions $p_{l,k}$ are also in $C^\beta([0,A]) $ under this conjecture. 
\begin{lemma}\label{smooth_palm_beta1}
If the functions $\boh^0$ are in $C^\beta([0,A])$ for some $\beta >1$, then, under Conjecture \ref{conj_beta1}, the functions $(p_{l,k}, (l,k)\in \K^2)$ are also in $C^\beta([0,A])$.
\end{lemma}

To prove this lemma, we introduce some notations to shorten the computations. For a random variable $X_\delta$ that depends on some $\delta\in \R$, we will denote $X_\delta=~ O_{\E^{(x,l)}}(\delta)$ and $X_\delta=~ o_{\E^{(x,l)}}( \delta)$ when, at $\delta=0$, $\E_0^{(x,l)}[\vert X_\delta\vert ] = O( \delta)$ and $\E_0^{(x,l)}[\vert X_\delta \vert ] = o( \delta)$ respectively. For $x\in \R$, $(l,k)\in \K^2$ and $r\in \{1,...,\lfloor\beta\rfloor\}$, we set
\begin{align}\label{def_p1}
    p^{[1]}_{r,k}(A-x,l)= \E^{(x,l)}_0\bigg[\frac{1}{\lambda_{A}^{k}(f^0_k)^r}\bigg].
\end{align}
Note that $ p^{[1]}_{1,k}(x,l) = p_{l,k}(x)$. Finally, we set 
\begin{align}
    &p^{![2],0}_{k}(x,s,l,j):= p^{![2]}_{1,0,h^0_{j,k}(0),k}(x,s,l,j)\hspace{0.3cm}\text{and}\hspace{0.3cm} p^{![2],A}_{k}(x,s,l,j):= p^{![2]}_{1,0,h^0_{j,k}(A),k}(x,s,l,j).\nonumber
\end{align}
\begin{proof}[Proof of Lemma \ref{smooth_palm_beta1}]  We first  differentiate the functions $p_{l,k}$. Let $(l,k)\in \K^2$, $x\in~ [0,A]$ and $ \delta\in \R$ with $\delta >0$ if $x=0$ and $\delta <0$ if $x=A$ and $\delta$ is small enough so that $x+\delta\in ~[0,A]$. With similar computations as in the proof of Lemma \ref{smooth_palm} and since the functions $h^{0}_{l,k}$ are $\beta$-Hölder, we find that for $\delta>0$
\begin{align}
    p_{l,k}(A-x-\delta) - p_{l,k}(A-x)= \E^{(x,l)}_0\bigg[\frac{ \Delta^+_\delta(k,x,l) }{\lambda_{A-\delta}^{k}(f^0_k)\lambda_{A}^{k}(f^0_k)}\bigg],\nonumber
\end{align}
with
\begin{align}
    &\Delta^+_\delta(k,x,l) = \underbrace{\delta \partial h^{0}_{l,k}(A-x)+
   \delta     \sum_{j=1}^K            \int_{0}^{A-\delta}\onee_{(s,j)\neq (x,l)} \partial h^0_{j,k}(A-s) dN^j_s}_{\Delta_\delta(\partial\boh^{0},l,k)}\nonumber \\
    & -\sum_{j=1}^Kh^0_{j,k}(A)\int_{-\delta}^0\onee_{(s,j)\neq (x,l)} dN^j_s+ h^0_{j,k}(0)\int_{A-\delta}^A \onee_{(s,j)\neq (x,l)}dN^j_s + O_{\E^{(x,l)}}(\delta^{\beta_d})\nonumber,
\end{align}
where $\beta_d =\min(2,\beta)>1$. If $\delta<0$, the same decomposition holds after replacing $\Delta^+_\delta(k,x,l)$  by $\Delta^-_\delta(k,x,l)$ given by
\begin{align*}
    &\Delta^-_\delta(k,x,l)=\delta \partial h^{0}_{l,k}(A-x)
   + \delta  \sum_{j=1}^K               \int_{-\delta}^{A^-}\onee_{(s,j)\neq (x,l)} \partial h^0_{j,k}(A-s) dN^j_s\nonumber \\ & +\sum_{j=1}^Kh^0_{j,k}(A)\int_0^{-\delta}\onee_{(s,j)\neq (x,l)} dN^j_s- h^0_{j,k}(0)\int_{A}^{A-\delta} \onee_{(s,j)\neq (x,l)}dN^j_s+ O_{\E^{(x,l)}}(\delta^{\beta_d}).\nonumber
\end{align*}
Using again that the functions $h^{0}_{l,k}$ are $\beta$-Hölder and (\ref{bound_mom_palm_second}), we find when $\delta>0$
\begin{equation}\label{decomp_p1}
\begin{aligned}
     &p_{l,k}(A-x-\delta) - p_{l,k}(A-x) =
     \E_0^{(x,l)}\bigg[ \frac{ \Delta_\delta(\partial\boh^{0},l,k)}{\lambda_A^{k}(f^0_k)^2} \bigg] \\
     & - \sum_{j=1}^Kh^0_{j,k}(A)  \E_0^{(x,l)}\bigg[ \frac{ \int_{-\delta}^{0^-}\onee_{(s,j)\neq (x,l)} dN^j_s }{\lambda_A^{k}(f^0_k) (\lambda_A^{k}(f^0_k) + h^0_{j,k}(A) \int_{-\delta}^{0^-}\onee_{(s,j)\neq (x,l)} dN^j_s)} \bigg]\\
     &  + \sum_{j=1}^Kh^0_{j,k}(0)  \E_0^{(x,l)}\bigg[ \frac{ \int_{A-\delta}^{A^-}\onee_{(s,j)\neq (x,l)} dN^j_s }{\lambda_A^{k}(f^0_k) (\lambda_A^{k}(f^0_k) - h^0_{j,k}(0) \int_{A-\delta}^{A^-}\onee_{(s,j)\neq (x,l)} dN^j_s)} \bigg]+ O(\delta^{\beta_d})
\end{aligned}
\end{equation}
Moreover, by Lemma \ref{orderliness_palm}, for all $j\in\K$ 
     $$\int_{-\delta}^0\onee_{(s,j)\neq (x,l)} dN^j_s = \left(\int_{-\delta}^0\onee_{(s,j)\neq (x,l)} dN^j_s\right)
     \onee_{\int_{-\delta}^0\onee_{(s,j)\neq (x,l)} dN^j_s=1}+ o_{\E^{(x,l)}}(\delta),$$
     so that for all $j$
     \begin{align*}
         \E_0^{(x,l)}&\bigg[ \frac{ \int_{-\delta}^0\onee_{(s,j)\neq (x,l)} dN^j_s }{\lambda_A^{k}(f^0_k)) (\lambda_A^{k}(f^0_k) + h^0_{j,k}(A) \int_{-\delta}^0\onee_{(s,j)\neq (x,l)} dN^j_s)} \bigg] \\
         & = \E_0^{(x,l)}\bigg[ \frac{ \int_{-\delta}^0\onee_{(s,j)\neq (x,l)} dN^j_s }{\lambda_A^{k}(f^0_k) (\lambda_A^{k}(f^0_k) + h^0_{j,k}(A)) } \bigg] + o(\delta).
     \end{align*}
   Doing similarly for the third term on the right-hand side of (\ref{decomp_p1}) and then applying (\ref{iter_palm}), it leads to
   \begin{align}
&p_{l,k}(A-x-\delta) - p_{l,k}(A-x)=  \delta \partial h^{0}_{l,k}(A-x)p^{[1]}_{2,k}(x,l)\nonumber\\
&\hspace{0.4cm}+\delta \sum_j \int_0^{A-}\onee_{(s,j)\neq (x,l)}\partial h^0_{j,k}(A-s)p^{[2]}_{2,k}(x,s,l,j)m_{l,j}(s-x)ds\nonumber\\
&\hspace{0.4cm}-\underbrace{\sum_j h_{j,k}^{0}(A)\int_{-\delta}^0\E_0^{(x,s,l,j)}\bigg[\frac{1}{\lambda_A^{k}(f^0_k)(\lambda_A^{k}(f^0_k)+h_{j,k}^{0}(A))}\bigg]m_{l,j}(s-x)ds}_{F_1} \nonumber\\
&\hspace{0.4cm}+ \underbrace{\sum_jh_{j,k}^{0}(0)\int_{A-\delta}^A\E_0^{(x,s,l,j)}\bigg[\frac{1}{\lambda_A^{k}(f^0_k)(\lambda_A^{k}(f^0_k)-h_{j,k}^{0}(0))}\bigg]m_{l,j}(s-x)ds}_{F_2} + o(\delta).\nonumber
\end{align}
Now, recall from the Appendix that $\P^{(x,s,l,j)} \big(N^{j}(\{s\})=1\big)=1$, so when $s\in [A-\delta, A]$, we have
\begin{align}
    \lambda_A^{k}(f^0_k) &= \lambda_A^{k}(f^0_k, !(s,j)) + h^0_{j,k}(A-s)\hspace{0.1cm}, \hspace{0.2cm}\P^{(x,s,l,j)}-a.s. \nonumber\\
    &=\lambda_A^{k}(f^0_k, !(s,j)) + h^0_{j,k}(0) + O(\delta)\hspace{0.1cm}, \hspace{0.2cm}\P^{(x,s,l,j)}-a.s. ,\nonumber
\end{align}
with $\lambda_A^{k,e}(f^0_k, !(s,j))$ defined by (\ref{def_lambda!}). When $s\in [-\delta, 0[$, $\lambda_A^{k,e}(f^0_k) = \lambda_A^{k,e}(f^0_k, !(s,j))$. Whence,
\begin{align}
    &F_1=-\sum_j h_{j,k}^{0}(A)\int_{-\delta}^0p^{![2],A}_{k}(x,s,l,j)m_{l,j}(s-x)ds, \nonumber\\
    &F_2=\sum_jh_{j,k}^{0}(0)\int_{A-\delta}^A p^{![2],0}_{k}(x,s,l,j)m_{l,j}(s-x)ds + O(\delta^2).\nonumber
\end{align}
When $\delta<0$, doing the same, we find that 
\begin{align}
&p_{l,k}(A-x-\delta) - p_{l,k}(A-x)=  \delta \partial h^{0}_{l,k}(A-x)p^{[1]}_{2,k}(A-x,l) \hspace{4.5cm} \nonumber\\
&\hspace{3.5cm}+\delta \sum_j \int_0^{A-}\onee_{(s,j)\neq (x,l)}\partial h^0_{j,k}(A-s)p^{[2]}_{2,k}(x,s,l,j)m_{l,j}(s-x)ds\nonumber\\
&\hspace{3.5cm}+\sum_j h_{j,k}^{0}(A)\int_0^{-\delta}p^{![2],A}_{k}(x,s,l,j)m_{l,j}(s-x)ds\nonumber\\
&\hspace{3.5cm}- \sum_j h_{j,k}^{0}(0)\int^{A-\delta}_A p^{![2],0}_{k}(x,s,l,j)m_{l,j}(s-x)ds+ O(\delta^2).\nonumber
\end{align}
By Lemma \ref{moment_measure_palm}, the functions $s\rightarrow m_{l,j}(s-x)$ are $\beta$-Hölder. Thus, with in addition Conjecture \ref{conj_beta1}, we have that 
\begin{align}
    &\partial p_{l,k}(A-x) = -\underset{\delta \rightarrow 0}{\lim}\frac{p_{l,k}(A-x-\delta)-  p_{l,k}(A-x)}{\delta}\nonumber\\
    &=-\partial h^{0}_{l,k}(A-x)p^{[1]}_{2,k}(x,l) -\sum_j \int_0^{A-}\onee_{(s,j)\neq (x,l)}\partial h^0_{j,k}(A-s)p^{[2]}_{2,k}(x,s,l,j)m_{l,j}(s-x)ds\nonumber\\
    &\hspace{0.8cm}+\sum_j h_{j,k}^{0}(A)p^{![2],A}_{k}(x,0,l,j)m_{l,j}(-x)-\sum_j h_{j,k}^{0}(0) p^{![2],0}_{k}(x,A,l,j)m_{l,j}(A-x).\nonumber
\end{align}
Now, we study the regularity of this derivative.  Conjecture \ref{conj_beta1} implies that if $ x\mapsto p^{[1]}_{2,k}(x,l)$ is $(\beta-1)$-Hölder, then $\partial p^e_{l,k}$ is also $(\beta-1)$-Hölder, which proves the result. So it remains to study the regularity of $ x\mapsto p^{[1]}_{2,k}(x,l)$. Using $a^r - b^r = (a - b)\sum_{i=0}^{r-1} a^{r-1-i} b^i $, we have for $r\geq 2$ and $\delta>0$
\begin{align}
    p^{[1]}_{r,k}(A-x-\delta,l) -  p^{[1]}_{r,k}(A-x,l)&=\sum_{i=0}^{r-1}E^{(x,l)}\bigg[\frac{\Delta^+_\delta(k,x,l)}{\lambda_{A-\delta}^{k,e}(f^0_k)^{1+i}\lambda_{A}^{k,e}(f^0_k)^{r-i}}\bigg] .\nonumber
\end{align}
Thus, using what we just did for $p_{l,k}(x)= p^{[1]}_{1,k}(x,l)$, one can show similarly that for $r\in~ \{2,...,\lfloor\beta \rfloor\}$, the function $x\rightarrow p^{[1]}_{r,k}(x,l)$ is differentiable with
\begin{align}
    &\partial p^{[1]}_{r,k}(A-x,l) =-r\partial h^{0}_{l,k}(A-x)p^{[1]}_{r+1,k}(x,l)\nonumber\\
    &\hspace{0.7cm}-r\sum_j \int_0^{A-}\onee_{(s,j)\neq (x,l)}\partial h^0_{j,k}(A-s)p^{[2]}_{r+1,k}(x,s,l,j)m_{l,j}(s-x)ds\nonumber\\
    &\hspace{0.7cm}+\sum_{i=0}^{r-1}\sum_j h_{j,k}^{0}(A)p^{![2]}_{r,i,h^0_{j,k}(A),k}(x,0,l,j)m_{l,j}(-x)\nonumber\\
    &\hspace{0.7cm}-\sum_{i=0}^{r-1}\sum_j h_{j,k}^{0}(0) p^{![2]}_{r,i,h^0_{j,k}(0),k}(x,A,l,j)m_{l,j}(A-x),\nonumber
\end{align}
so that if $p^{[1]}_{r+1,k}(\cdot{,}l)$ is Lipschitz then $\partial p^{[1]}_{r,k}(\cdot{,}l)$ is Lipschitz. Moreover, by Lemma \ref{lemma_powerm}, for all $r\geq2$, $p^{[1]}_{r,k}(\cdot{,}l)$ is Lipschitz. In particular, $ p^{[1]}_{\lfloor \beta \rfloor+1,k}(\cdot{,}l)$ is Lipschitz and thus $\partial^{(\lfloor \beta \rfloor-1)} p^{[1]}_{2,k}(\cdot{,}l)$ is Lipschitz. This proves that $p^{[1]}_{2,k}(\cdot{,}l)$ is $\beta$-Hölder, which concludes the proof.
\end{proof}

We state and prove two results used in the previous proof of Lemma \ref{smooth_palm_beta1}. The first one is on the orderliness of the process under the first order Palm distribution.

\begin{lemma}\label{orderliness_palm}
 Let $(x,l)\in \R\times \K$, $z\in \R$ and $\delta\in \R$. Let $B_{z,\delta}:= \big\{ s\in \R: \vert s-z\vert \leq \vert \delta\vert \big\}$ and $B_{z,\delta}^{-x}=B_{z,\delta}\backslash\{x\}$. Then,  $\P_0^{(x,l)}\big(N^l(B_{z,\delta}^{-x})\geq 2 \big)\lesssim \delta^2$ and  $\P_0^{(x,l)}\big(N^j(B_{z,\delta})\geq 2 \big)\lesssim \delta^2$ for $j\neq l$.
\end{lemma}

\begin{proof}[Proof of Lemma \ref{orderliness_palm}]
First, note that when $j\neq l$, $\P_0^{(x,l)}\big(N^j(B_{z,\delta}) = N^j(B_{z,\delta}^{-x})\big)=1$, so we can just consider $B_{z,\delta}^{-x}$. Let $(l,j)\in \K^2$,
\begin{align*}
    \mathbb P_0^{(x,l)}\big(N^j (B_{z,\delta}^{-x})\geq 2 \big) &\leq  \mathbb E_0^{(x,l)}\left[\onee_{N^j (B_{z,\delta}^{-x})\geq 2}\Big(N^j(B_{z,\delta}^{-x})^2- N^j(B_{z,\delta}^{-x})\Big)\right]\\
    &= \mathbb E_0^{(x,l)}\left[N^j(B_{z,\delta}^{-x})^2- N^j(B_{z,\delta}^{-x})\right]\\
    & = \int_{B_{z,\delta}^{-x}} \E_0^{(x,s,l,j)}\left[ N^j (B_{z,\delta}^{-x})- 1\right]m_{l,j}(s-x)ds,
\end{align*}
where for the last equality we have used (\ref{iter_palm}). Moreover, for $s\in B_{z,\delta}^{-x}$, since under  $\P_0^{(x,s,l,j)}$ the process has a fixed atom at $(s,j)$ and with (\ref{second_palm_first_bound}), we have
$$\mathbb E_0^{(x,s,l,j)}\left[ N^j (B_{z,\delta}^{-x}) - 1\right] = \mathbb E_0^{(x,s,l,j)}\left[N^j(B_{z,\delta}^{-x}\backslash\{s\})\right]\lesssim \delta ,$$
which leads to the proof.
\end{proof}

The second lemma is on the Lipschitz continuity of the functions $x\mapsto p^{[1]}_{r,k}(x,l)$  defined in (\ref{def_p1}) at start of the proof of Lemma \ref{smooth_palm_beta1}. 
\begin{lemma}\label{lemma_powerm}
If the functions $\boh^0$ are in $C^\beta([0,A])$ for some $\beta\geq 1$, then for all $r\geq 2$, the functions $x\mapsto p^{[1]}_{r,k}(x,l)$ are Lipschitz.
\end{lemma}

\begin{proof}[Proof of Lemma \ref{lemma_powerm}]
 Let $r\geq 2$. First, as in proof of Lemma \ref{smooth_palm}, it is sufficient to show that for $x,y$ such that $0\leq y-x\leq A$, $\vert p^{[1]}_{r,k}(y,l) -p^{[1]}_{r,k}(x,l) \vert \lesssim \vert y-x\vert$. Using the same computations as at start of the proof of Lemma \ref{smooth_palm} and in addition that $a^r - b^r =(a - b)\sum_{k=0}^{r-1} a^{r-1-k} b^k =: (a-b) P_r(a,b)$, we find
 \begin{align}
     \Big \vert p^{[1]}_{r,k}(y,l) -p^{[1]}_{r,k}(x,l)\Big \vert &\leq \E_0^{(x,l)}\bigg[\frac{\vert \Delta_{y-x}(k,x,l) \vert \times P_r\Big(\lambda_{A+x-y}^{k,e}(f^0_k,(x,l)),\lambda_{A}^{k,e}(f^0_k,(x,l))\Big) }{\lambda_{A+x-y}^{k,e}(f^0_k,(x,l))^r\lambda_{A}^{k,e}(f^0_k,(x,l))^r}\bigg]\nonumber\\
     &\leq \frac{r}{(\nu^0_k)^{r+1}}\E_0^{(x,l)}\big[\vert \Delta_{y-x}(k,x,l) \vert\big]\nonumber
 \end{align}
 where for the last inequality we have used that for $a,b\geq \nu^0_k$, $P_r(a,b)/(a^rb^r) \leq r/(\nu^0_k)^{r+1}$. Then, it is shown in the proof of Lemma \ref{smooth_palm} that $\E_0^{(x,l)}\big[\vert \Delta_{y-x}(k,x,l) \vert\big]\lesssim \vert y-x\vert$ and it concludes the proof of the Lemma.
\end{proof}

In conclusion, we have the following corollary.

 \begin{corollary}\label{g0_holder_beta1}
 If the functions $\boh^0$ and $\bog^0_2$ are in $C^\beta([0,A])$ for some $\beta>1$, then, under Conjecture \ref{conj_beta1}, the functions $\bog^0_L$ are also in $C^\beta([0,A])$.
 \end{corollary}

\section{Technical lemmas}\label{sec:pr:tech_lemmas}

In this section, we state and prove two technical lemmas. First, Lemma \ref{Bernstein} gives Bernstein-type inequalities that are used in Lemma \ref{lemma_rk} to control the bias term appearing in Theorem \ref{main_theorem}. These Bernstein-type inequalities are also used to prove Lemma \ref{control_rem} which is the second technical lemma of this section and which shows that the difference between the LAN remainders in the proof of Theorem \ref{main_theorem} is uniformly controlled.

\subsection{Bernstein-type inequalities}

The following lemma gives three Bernstein-type inequalities, derived from results of \cite{hansen_al} and \cite{reynaud_roy_hawkes}. The first one is for an integral with respect to $dN_t - \lambda_t(f^0)dt$ and thus the proof exploits the martingale property of this quantity. The second is for an integral with respect to $dN_t$ whose integrand is centered in expectation, it is based on an approximation of the Hawkes process by i.i.d. sequences (see Section 3 of \cite{reynaud_roy_hawkes}). The third is a direct combination of the two first.

\begin{lemma}\label{Bernstein}
Let $N$ be a stationary multivariate linear Hawkes process with parameters $f^0$ such that the functions $h^0_{l,k}$ are bounded and supported on $[0,A]$. Let $\alpha>0$ and with this $\alpha$ define $\Omega_T$ as in Lemma \ref{omega_T}. There exists a constant $C>0$ depending only on $(f^0, \alpha, K, A)$ such that for $T$ large enough, for any $f,\bar f\in \R^K\times \L_\infty^{K^2}$ and $u>0$,
\begin{align}
    &\P_0\bigg(\Omega_T \cap \Big \{\Big\vert \sum_{k=1}^K \int_0^T \lambda_t^k(\bar f_k)\lambda_t^k(f_k) \big(dN_t^k -\lambda_t^k(f^0_k)dt\big)\Big\vert> u \Big\}\bigg) \nonumber\\
    &\hspace{1cm}\leq 2\exp\bigg(\frac{-C\log(T)^{-5}u^2}{ T\Vert \bar f\Vert_\infty^2  \Vert f\Vert_2^2+ u\Vert \bar f\Vert_\infty\Vert f\Vert_\infty}\bigg). \nonumber
\end{align}
Moreover, there exists an event $\Omega_T'\subset \Omega_T$, $\P_0(\Omega_T')\xrightarrow[]{}1$, and a constant $C'>0$ that depends only on $(f^0, \alpha, K, A)$  such that for $T$ large enough, for any $f,\bar f\in \R^K\times \L_\infty^{K^2}$ and $u>0$,
\begin{align}
    &\P_0\bigg( \Omega_T'\cap \Big \{\Big\vert \sum_{k=1}^K \int_0^T \lambda_t^k(\bar f_k)\lambda_t^k(f_k) -\E_0\big[\lambda_A^k(\bar f_k)\lambda_A^k(f_k)\big]dt  \Big\vert> u \Big\}\bigg)\nonumber\\
    &\hspace{1cm}\leq 4\exp\bigg( \frac{-C'\log(T)^{-3}u^2}{T\Vert \bar f \Vert_\infty^2 \Vert f \Vert_2^2 + u\Vert \bar f\Vert_\infty \Vert  f\Vert_\infty} \bigg).\nonumber
\end{align}
Combining the two previous inequalities,  we obtain that there exists a constant $C''>0$ such that for $T$ large enough, for any $f,\bar f\in \R^K\times \L_\infty^{K^2}$ and $u>0$,
\begin{align}
    &\P_0\bigg( \Omega_T' \cap \Big \{\Big\vert \sum_{k=1}^K \int_0^T \lambda_t^k(\bar f_k)\lambda_t^k(f_k)dN_t^k -T\E_0\big[\lambda_A^k(\bar f_k)\lambda_A^k(f_k)\lambda_t^k(f^0_k)\big]  \Big\vert> u \Big\}\bigg)\nonumber\\
    &\leq 6\exp\Bigg(\frac{- C''\log(T)^{-5}u^2}{T\Vert \bar f\Vert_\infty^2\Vert f\Vert_2^2+ u\Vert \bar f\Vert_\infty\Vert f\Vert_\infty\big)}\Bigg). \nonumber
\end{align}

\end{lemma}

\begin{proof}[Proof of Lemma \ref{Bernstein}]
We begin by proving the first inequality. Let,
\begin{align}
    \hat v_T = \displaystyle\sum_{k=1}^K\int_0^T \lambda_t^k(\bar f_k)^2\lambda_t^k(f_k)^2 \lambda_t^k(f^0_k)dt \hspace{0.3cm} \text{and}\hspace{0.3cm} \hat B_T = \underset{k\in\K, \hspace{0.1cm}t\in[0,T]}{\sup}\vert \lambda_t^k(\bar f_k)\lambda_t^k(f_k)\vert. \nonumber
\end{align}
By applying Theorem 3.3 of \cite{van_zanten} to the $\mathcal{G}$-martingale
\begin{align}
    (M_T^0)_T:=\bigg(\sum_{k=1}^K \int_0^T \lambda_t^k(\bar f_k)\lambda_t^k(f_k) \big(dN_t^k -\lambda_t^k(f^0_k)dt\big)\bigg)_{T\geq 0}, \nonumber
\end{align}
we obtain for $u>0$,
\begin{align}
     \P_0\Big(\big \vert M_T^0\big\vert >u, \hspace{0.2cm}\hat v_T \leq v \hspace{0.1cm}\text{and} \hspace{0.1cm} \hat B_T\leq B\Big) \leq 2\exp\Big(\frac{-u^2}{v+ uB/3}\Big).\nonumber
\end{align}
First, on the event $\Omega_T$, for $T$ large enough independently of $f$ and $\bar f$,
\begin{align}
    \hat B_T \leq K^2 C_\alpha^2\log(T)^2\Vert \bar f\Vert_\infty \Vert f\Vert_\infty. \nonumber
\end{align}
Secondly, let $r_T = K^3 C_\alpha^3\log(T)^3\Vert \bar f \Vert_\infty^2 \Vert f^0\Vert_\infty$. Using that $(x+y)^n\leq 2^{n-1}(x^n+y^n)$ and Jensen inequality, we have on $\Omega_T$ for $T$ large enough,
\begin{equation}\label{bound_variance_1}
\begin{aligned}
    \hat v_T &\leq r_T\sum_{k=1}^K \int_0^T\lambda_t^k(f_k)^2dt\leq 2r_T \sum_{k=1}^K \int_0^T \nu_k^2 +\Big(\sum_{l=1}^K\int_{t-A}^{t-} h_{l,k}(t-u)dN_u^l\Big )^2dt  \\
    &\leq 2r_TT\Vert \nu\Vert_2^2+ 2^{K}r_T\sum_{k=1}^K\sum_{l=1}^K\int_0^T  N^l([t-A,t[)\int_{t-A}^{t-} h_{l,k}(t-u)^2 dN_u^l dt ,
\end{aligned}
\end{equation}
and with Fubini-Tonelli we further obtain on $\Omega_T$,
\begin{equation}\label{bound_variance}
\begin{aligned}
    &2^{K}r_T\sum_{k=1}^K\sum_{l=1}^K\int_0^T  N^l([t-A,t[)\int_{t-A}^{t-} h_{l,k}(t-u)^2 dN_u^l dt \\
    &\leq 2^{K}C_\alpha r_T\log(T)\sum_{k=1}^K\sum_{l=1}^K\int_0^T  \int_{t-A}^{t-} h_{l,k}(t-u)^2 dN_u^l dt \\
    &=   2^{K}C_\alpha r_T \log(T)\sum_{k=1}^K\sum_{l=1}^K\int_{-A}^T  \int_{u+}^{u+A} h_{l,k}(t-u)^2 dt dN_u^l \\
    &= 2^K  C_\alpha r_T\log(T)\sum_{k=1}^K\sum_{l=1}^K N^l([-A,T])\Vert h_{l,k}\Vert_2^2 \\
     &= 2^K  C_\alpha r_T\log(T)(\lceil T/A\rceil+1)\underset{l\in \K, t\in [0,T]}{\sup}N^l([t-A,t])\sum_{k=1}^K\sum_{l=1}^K \Vert h_{l,k}\Vert_2^2 \\
    &\leq  2^{K+1} C_\alpha K r_T   (\lceil T/A\rceil+1)\log(T)^2\Vert \boh\Vert_2^2.
\end{aligned}
\end{equation}
Let $C_0 = c C_\alpha^4 \Vert  f^0\Vert_\infty$ and $c=2^{K+3} K^4/A$, we have proved that for $T$ large enough, on $\Omega_T$, $\hat v_T \leq C_0\log(T)^5  T\Vert \bar f \Vert_\infty^2\Vert f\Vert_2^2,$ and thus
\begin{align}
    \Omega_T\subset \big\{\hat v_T\leq C_0 \log(T)^5  T\Vert \bar f \Vert_\infty^2\Vert f\Vert_2^2 \hspace{0.2cm}\text{and}\hspace{0.2cm} \hat B_T\leq K^2 C_\alpha^2\log(T)^2\Vert \bar f\Vert_\infty \Vert f\Vert_\infty \big\}. \nonumber
\end{align}
As a consequence, it proves that 
\begin{align}\label{interm_eq_first_bernstein}
    \P_0\Big(\Omega_T \cap \big \{\vert M_T^0\vert > u \big\}\Big)\leq 2\exp\bigg(\frac{-\log(T)^{-5}u^2}{ C_0T\Vert \bar f\Vert_\infty^2  \Vert f\Vert_2^2+ uK^2 C_\alpha^2\Vert \bar f\Vert_\infty\Vert f\Vert_\infty}\bigg) ,
\end{align}
and letting $C=\max(C_0, K^2 C_\alpha^2)^{-1}$, we have proved the first inequality.

For the second inequality, we adapt the proof of Proposition 3 of \cite{hansen_al} to obtain a Bernstein-type inequality on a certain event $\Omega_T'=\Omega_T^1\cap \Omega_T^2$ defined just after. For a real Borel set $A$, we denote by $N_{\vert A}$ the random measure $N$ restricted to $A$: $N_{\vert A}(B) = N(B\cap A)$. To begin with, 
\begin{align}
    Z_t := \sum_{k=1}^K \lambda_t^k(\bar f_k)\lambda_t^k(f_k) - \E_0[\lambda_A^k(\bar f_k)\lambda_A^k(f_k)],\nonumber
\end{align}
depends only on $N_{\vert [t-A,t[}$ and can be written as $Z_0\circ \mathfrak
{S}_t(N)$ with $\mathfrak
{S}_t$ the time shift operator (it means that $Z_t$ depends on $N_{\vert [t-A,t[}$ as $Z_0$ depends on $N_{\vert [-A,0[}$). Let $c>0$ fixed later, $d_T = \lfloor cT/\log(T)\rfloor $ and $x=x_T=T/2d_T$. Note that  $x$ is of order $\log(T)$ and so for $T$ large enough $x>A$. We split
\begin{align*}
    &\Big\vert \sum_{k=1}^K \int_0^T \lambda_t^k(\bar f_k)\lambda_t^k(f_k) -\E_0[\lambda_A^k(\bar f_k)\lambda_A^k(f_k)]dt\big)\Big\vert \nonumber\\
    &\leq \Big\vert \sum_{q=0}^{d_T-1} \int_{2qx}^{2qx +x}\hspace{-0.2cm} Z_0\circ \mathfrak
{S}_t(N)  dt \Big\vert + \Big\vert\sum_{q=0}^{d_T-1}\int_{2qx+x}^{2qx +2x} \hspace{-0.2cm}Z_0\circ \mathfrak
{S}_t(N) dt\big)\Big\vert  = :\Delta_1+ \Delta_2,
\end{align*} 
and we treat each term separately but using a similar argument. We first study $\Delta_1$. There exists a sequence $M^{x,1}=(M^{x,1}_q)_q$ of i.i.d. point processes such that for all $q$, $M^{x,1}_q$ has the same distribution as $N_{\vert [2qx-A,2qx+x]}$. Moreover, there exists $c>0$ small enough (in the definition of $x$) such that  
\begin{align}\label{first_om}
    \P\Big(\exists q\leq d_T-1,\,  M^{x,1}_q \neq N_{\vert [2qx-A,2qx+x]}\Big) = o(1).
\end{align}
Hence, if $\Omega_{1,T} = \Omega_T \cap \{  \forall q\leq d_T-1,\,  M^{x,1}_q = N_{\vert [2qx-A,2qx+x]}\}$ then $\mathbb P( \Omega_{1,T}^c) =o(1)$; see Section 3 of \cite{reynaud_roy_hawkes} for details on the construction of $(M_q^{x,1})_q$ and for a proof of (\ref{first_om}). On $\Omega_{1,T}$, since $\sup_{t\in [0;T]}N([t-A, t[)\leq KC_\alpha\log T$, we also have $\max_q \sup_{t\in [0;T]}M_q^{x,1}([t-A, t[)\leq KC_\alpha\log T$. Next, let $F_q = \int_{2qx}^{2qx +x} Z_0\circ \mathfrak
{S}_t(M^{x,1}_q) dt$, the variables $(F_q)_q$ are independent and centered so they form a martingale difference sequence with respect to their natural filtration denoted by  $\mathfrak{F}$. Moreover,
\begin{align}
    \P_0\bigg( \Omega_T^1 \bigcap\Big\{\Big\vert\sum_{q=0}^{d_T-1}  \int_{2qx}^{2qx +x} Z_0\circ \mathfrak
{S}_t(N)  dt\Big\vert >\frac{u}{2}\Big\}\bigg) &= \P_{M^{x,1}}\bigg( \Omega_T^1 \bigcap\Big\{ \Big\vert\sum_{q=0}^{d_T-1}  F_q\Big\vert >\frac{u}{2}\Big\}\bigg).\nonumber
\end{align}
For $T$ large enough (independently of $f$ and $\bar f$), on the event $\Omega_T^1$,
\begin{align}
    \vert F_q\vert &\leq  x\hspace{0.1cm}\underset{t\in [2qx,2qx+x]}{\sup} \vert Z_0\circ \mathfrak {S}_t(N)\vert \nonumber\\
    &\leq x\Vert \bar f\Vert_\infty\Vert f\Vert_\infty\Big( \underset{t\in [2qx,2qx+x]}{\sup} \big(1+M^{x,1}_q([t-A,t[)\big)^2 + \E_0\big[(1+ N([0,A]))^2\big]\Big) \nonumber \\
    &\leq  2K^2C_\alpha^2\log(T)^2x\Vert \bar f\Vert_\infty\Vert f\Vert_\infty=  \frac{4K^2C_\alpha^2}{c}\log(T)^3\Vert \bar f\Vert_\infty\Vert f\Vert_\infty.\nonumber
\end{align}
Then, by independence $\E_{M^{x,1}}[F_q^2\vert \mathfrak{F}_{q-1}] = \E_{M^{x,1}_q}[F_q^2]$. As $M^{x,1}_q$ is equal in distribution to $N_{\vert [2qx-A,2qx+x]}$, we also have $\E_{M^{x,1}_q}[F_q^2]= \E_0[F_q^2]$. By Jensen inequality,
\begin{align}
    \E_0[F_q^2] &=  \E_0\bigg[ \Big( \sum_{k=1}^K\int_{2qx}^{2qx+x} \lambda_t^k(\bar f)\lambda_t^k(f)  - \E_0[\lambda_A^k(\bar f)\lambda_A^k(f)]dt \Big)^2 \bigg] \nonumber\\
    &\leq 2^{K-1}\sum_{k=1}^K\E_0\bigg[ \Big( \int_{2qx}^{2qx+x} \lambda_t^k(\bar f)\lambda_t^k(f)  - \E_0[\lambda_A^k(\bar f)\lambda_A^k(f)]dt \Big)^2 \bigg]\nonumber\\
    &\leq 2^{K}x\sum_{k=1}^K\E_0\bigg[  \int_{2qx}^{2qx+x} \lambda_t^k(\bar f)^2\lambda_t^k(f)^2  + \E_0[\lambda_A^k(\bar f)\lambda_A^k(f)]^2dt  \bigg]. \nonumber
\end{align}
Let $k\in \K$, by Cauchy-Schwarz inequality,
\begin{align}
    \E_0[\lambda_A^k(\bar f)\lambda_A^k(f)]^2 \leq \E_0[\lambda_A^k(\bar f)^2]\E_0[\lambda_A^k(f)^2]\leq K\Vert \bar f\Vert_\infty^2 \E_0\big[(1+N([0,A]))^2\big]\E_0[\lambda_A^k(f)^2].\nonumber
\end{align}
In the same way as in the proof of (\ref{(b)}), one can show that there exists a finite positive constant $C_2'$, independent of $f$, such that  $\sum_{k=1}^K \E_0[\lambda_t^k(f)^2] \leq C_2' \Vert f \Vert _2^2$. Hence, for some $c(K)>0$,
\begin{align}
   \E_0[F_q^2] \leq c(K)x^2\Vert \bar f\Vert_\infty^2 \Vert f\Vert_2^2 + 2^{K}x\sum_{k=1}^K\E_0\bigg[  \int_{2qx}^{2qx+x} \lambda_t^k(\bar f)^2\lambda_t^k(f)^2 dt  \bigg]. \nonumber 
\end{align}
Next, using similar computations to those used in  (\ref{bound_variance}) (in particular Fubini theorem to switch $dN_u$ and $dt$), we find that
\begin{align}
    &\E_0\bigg[  \int_{2qx}^{2qx+x} \lambda_t^k(\bar f)^2\lambda_t^k(f)^2 dt  \bigg]\leq \E_0\bigg[ \underset{l\in \K, t'\in [2qx,2qx+x]}{\sup}\lambda_{t'}^l(\bar f)^2 \int_{2qx}^{2qx+x} \lambda_t^k(f)^2 dt  \bigg]\nonumber\\ 
    &\leq \Vert \bar f\Vert_\infty^2\E_0\bigg[ \underset{l\in \K, t'\in [2qx, 2qx+x]}{\sup}(1+ N^l([t'-A,t'])^2 \int_{2qx}^{2qx+x} \lambda_t^k(f)^2 dt  \bigg]\nonumber\\
    &\leq 2x\nu_k^2\Vert \bar f\Vert_\infty^2\E_0\big[ \underset{l\in \K, t'\in [2qx, 2qx+x]}{\sup}(1+ N^l([t'-A,t'])^2\big]\nonumber\\
    &+ 2^K\Vert \bar f\Vert_\infty^2\E_0\bigg[ \underset{ t'\in [2qx, 2qx+x]}{\sup}(1+ N([t'-A,t'])^2 \sum_{l=1}^K \int_{2qx}^{2qx+x}\hspace{-0.1cm}\Big(\int_{t-A}^{t-} \hspace{-0.2cm}h_{l,k}(A-u)dN_u^l \Big)^2\hspace{-0.15cm}dt \bigg]\nonumber\\
    &\leq 2x\nu_k^2\Vert \bar f\Vert_\infty^2\E_0\big[ \underset{l\in \K, t'\in [2qx, 2qx+x]}{\sup}(1+ N^l([t'-A,t'])^2\big]\nonumber\\
    &+ 2^K \Vert h_k\Vert^2_2 \Vert \bar f\Vert_\infty^2\E_0\big[ \underset{ t'\in [2qx, 2qx+x]}{\sup}(1+ N([t'-A,t'])^3N([2qx-A,2qx+x])\big]\nonumber\\
    &\leq 2x\Vert f\Vert_2^2\Vert \bar f\Vert_\infty^2\E_0\big[ \underset{ t'\in [0, T]}{\sup}(1+ N([t'-A,t'])^2\big]\nonumber\\
    &+ 2^K\Vert f\Vert_2^2 \Vert \bar f\Vert_\infty^2\E_0\big[ \underset{ t'\in [0,T]}{\sup}(1+ N([t'-A,t'])^3 N([2qx-A,2qx+x])\big].\nonumber
\end{align}
Finally, let $p\geq 0$, note that by Lemma \ref{omega_T}, we have for $T$ large enough
\begin{align}
    &\E_0\Big[\underset{t\in [0,T]}{\sup}(1+N([t-A,t[))^p\Big]\nonumber\\
    &= \E_0\Big[\underset{t\in [0,T]}{\sup}(1+N([t-A,t[))^p\onee_{\Omega_T}\Big]+ \E_0\Big[\underset{t\in [0,T]}{\sup}(1+N([t-A,t[))^p\onee_{\Omega_T^c}\Big] \nonumber \\
    &\leq 2K^pC_\alpha^p\log(T)^p.\nonumber
\end{align}
Putting all together, we obtain that for $T$ large enough
\begin{align*}
    &\sum_{q=0}^{d_T-1} \E_0[F_q^2\vert \mathfrak{F}_{q-1}]\nonumber\\
    &\lesssim \Vert \bar f \Vert_\infty^2 \Vert f \Vert_2^2 \bigg(d_Tx^2 + d_Tx\log(T)^2\nonumber\\
    &\hspace{3cm}+\E_0\Big[\underset{t\in [0,T]}{\sup}(1+N([t-A,t]))^2 \sum_{q=0}^{d_T-1}N([2qx-A,2qx+x])\Big] \bigg)\nonumber\\
    &\lesssim \Vert \bar f \Vert_\infty^2 \Vert f \Vert_2^2 \bigg(T\log(T) + T\log(T)^2 + \E_0\Big[\underset{t\in [0,T]}{\sup}(1+N([t-A,t]))^2 N([-A,T])\Big] \bigg)\nonumber\\
    &\lesssim \Vert \bar f \Vert_\infty^2 \Vert f \Vert_2^2 \bigg(T\log(T)^2 + T\E_0\Big[\underset{t\in [0,T]}{\sup}(1+N([t-A,t]))^3\Big]\bigg)\lesssim T\log(T)^3\Vert \bar f \Vert_\infty^2 \Vert f \Vert_2^2.\nonumber
\end{align*}
Therefore, again by Theorem 3.3 of \cite{van_zanten} (see also Equation (1.1) of this article for a formulation in discrete time), for some finite constant $C'>0$ and for $T$ large enough, both independent of $f$ and $\bar f$, we have
\begin{align}
    \P_{M^{x,1}}\bigg( \Omega_T^1 \cap\Big\{ \Big\vert\sum_{q=0}^{d_T-1}  F_q\Big\vert >\frac{u}{2}\Big\}\bigg) &\leq 2\exp\bigg( \frac{-C'\log(T)^{-3}u^2}{T\Vert \bar f \Vert_\infty^2 \Vert f \Vert_2^2 + u\Vert \bar f\Vert_\infty \Vert  f\Vert_\infty} \bigg).\nonumber
\end{align}
It remains to study the term $\Delta_2$. To do so, we proceed as for $\Delta_1$ but with this time a sequence $(M^{x,2}_q)_q$ such that for all $q$, $M^{x,2}_q$ that has the same distribution as $N_{\vert [2qx+x-A,2qx+2x]}$ and we define $\Omega_{2,T} := \Omega_T \cap \{  \forall q\leq d_T-1,\,  M^{x,2}_q = N_{\vert [2qx+x-A,2qx+2x]}\}$. Then, everything done for $\Delta_1$ holds similarly for $\Delta_2$ and, recalling that we have set $\Omega_T'=\Omega_T^1\cap \Omega_T^2$, we conclude that
\begin{align}
    &\P_0\bigg(\Omega_T'\cap\Big\{\Big\vert \sum_{k=1}^K \int_0^T \lambda_t^k(\bar f_k)\lambda_t^k(f_k) -\E_0[\lambda_A^k(\bar f_k)\lambda_A^k(f_k)]dt\big)\Big\vert> u \Big\}\bigg)\nonumber\\
    &\leq 4\exp\bigg( \frac{-C'\log(T)^{-3}u^2}{T\Vert \bar f \Vert_\infty^2 \Vert f \Vert_2^2 + u\Vert \bar f\Vert_\infty \Vert  f\Vert_\infty} \bigg). \nonumber 
\end{align}
This ends the proof of the second inequality.
\end{proof}

\subsection{Control of the remainders in the LAN expansions for the proof of Theorem \ref{main_theorem}}

The following lemma is used in the proof of Theorem \ref{main_theorem} to get rid of remainder terms coming from the LAN expansions. The proof relies on a chaining argument from \cite{talagrand_chaining} for variables satisfying Bernstein-type inequalities.

\begin{lemma}\label{control_rem} Under the assumptions of Theorem \ref{main_theorem}, the remainder term $R_{T,\varphi}$ defined in the proof of Theorem \ref{main_theorem} verifies
\begin{align}
    \underset{j\in \mathcal{J}_T}{\max}\underset{f\in \A_{T}(j)}{\sup}\vert R_{T,\varphi}(f) - R_{T,\varphi}(f_{u,j})\vert =o_{\P_0}(1)\nonumber
\end{align}
\end{lemma}

\begin{proof}[Proof of Lemma \ref{control_rem}]
 First, to shorten some computations, given $H = (H_t^k)_{t,k}$ a marked, stationary and predictable process having a second moment, we define the process $H\bullet (N -\E_0)_T$ by 
\begin{align}
    H\bullet (N -\E_0)_T =   \sum_{k=1}^K \int_0^T H_t^k  dN_t^k -T\sum_{k=1}^K \E_0[H_A^k\lambda_A^k(f^0_k)].\nonumber
\end{align}
Moreover, a process $\big(\tilde \lambda_t^k(f)\tilde \lambda_t^k(f')/\tilde \lambda_t^k(f'')\big)_{t,k}$ will be denoted $\tilde \lambda(f)\tilde \lambda(f')/\tilde\lambda(f'')$. Let $\Omega_T'$ be the event defined in Lemma \ref{Bernstein}, as $\P_0(\Omega_T')\xrightarrow[]{} 1$, it is sufficient to prove that for all $x>0$,
\begin{align}
   \P_0^{\Omega_T'}\Big( \underset{j\in \mathcal{J}_T}{\max}\underset{f\in \A_{T}(j)}{\sup}\vert R_{T,\varphi}(f) - R_{T,\varphi}(f_{u,j})\vert >x\Big) \rightarrow 0,
\end{align}
where $\P_0^{\Omega_T'}$ is the conditional probability given $\Omega_T'$: $\P_0^{\Omega_T'}(.) = \P_0(\Omega_T'\cap .)/\P_0(\Omega_T')$. Now we recall that for $f\in \A_T$, 
\begin{equation}\label{recall_remainder}
\begin{aligned}
    &R_{T,\varphi}(f) = T\Delta Q_T(f-f^0)+ \sum_{k=1}^K \int_0^T  V\Big(\frac{\lambda_t^k(f_k-f^0_k)}{\lambda_t^k(f^0_k)} \Big)\frac{dN_t^k}{\lambda_t^k(f^0_k)} + \tilde R_{T,\varphi}(f), \\
    &\Delta Q_T(f-f^0)=\frac{1}{2}\Vert f-f^0\Vert_L^2  - \frac{1}{2T} \sum_{k=1}^K \int_0^T \frac{\lambda_t^k(f_k-f^0_k)^2}{\lambda_t^k(f^0_k)} \frac{dN_t^k}{\lambda_t^k(f^0_k)}, \\
    &\tilde R_{T,\varphi}(f)=  \sqrt{T}W_{T}\big(0, \omega_\varphi(\tilde \boh) \big)-\frac{T}{2}\Vert \omega_\varphi(\tilde \boh)  \Vert_{L}^2 - T\langle \tilde f -\tilde f^{0}, (0, \omega_\varphi(\tilde \boh) )\rangle_{L},
\end{aligned}
\end{equation}
with $V(x) = \log(1+x) -x +x^2/2=x^2R(x)$ for $x>-1$. We will study separately the difference between the evaluation in $f_{u,j}$ and in $f$ of each term appearing in (\ref{recall_remainder}) for some integer $j\in \mathcal{J}_T$ and some $f\in \A_T(j)$. Next, for each term, we will obtain a uniform control of the difference over $f\in \A_T(j)$ and then over $\A_T$.  We begin with the difference  $T(\Delta Q_{T}(f_{u,j}-f^0) -\Delta Q_{T}(f-f^0))$.

\textit{Difference} $T\vert \Delta Q_{T}(f_{u,j}-f^0) -\Delta Q_{T}(f-f^0)\vert  $. Let $j\in \mathcal{J}_T$ and let $f = (\nu,\varphi(\tilde \boh))\in~ \A_{T}(j)$, set as before $\tilde f = (\nu,\tilde \boh)$ and $\tilde f\varphi^0 = (\nu, \tilde \boh.\boldsymbol{\bar\varphi}^0)$. We first use the linearization (\ref{linearization}) (justified by lemma \ref{lemma_linearization}) to rewrite the difference:
\begin{align}\label{diff_lin_delta_Q}
&T\big(\Delta Q_{T}(f_{u,j} -f^0) - \Delta Q_{T}(f-f^0)\big)= -\frac{u^2}{2T}\frac{\tilde \lambda(\tilde \psi^{0,j}_{L,\varphi}.\boldsymbol{\bar\varphi}^0)^2}{\lambda(f^{0})^2}\bullet (N -\E_0)_T\nonumber\\
&\hspace{2cm}+\frac{u}{\sqrt{T}} \frac{\tilde \lambda((\tilde f - \tilde f^{0}).\boldsymbol{\bar\varphi}^0)\tilde\lambda(\tilde \psi^{0,j}_{L,\varphi}.\boldsymbol{\bar\varphi}^0)}{\lambda(f^{0})^2}\bullet (N -\E_0)_T \nonumber \\
    &\hspace{2cm}+\frac{u}{\sqrt{T}}\frac{\tilde\lambda(\tilde \psi^{0,j}_{L,\varphi}.\boldsymbol{\bar\varphi}^0)\tilde\lambda(0, \omega_\varphi(\tilde \boh_{u,j}))}{\lambda(f^0)^2}\bullet(N-\E_0)_T\\
    &\hspace{2cm}- \frac{\tilde\lambda((\tilde f-\tilde f^0).\boldsymbol{\bar\varphi}^0)\tilde\lambda(0,\omega_\varphi(\tilde \boh_{u,j})- \omega_\varphi(\tilde \boh))}{\lambda(f^0)^2}\bullet(N-\E_0)_T\nonumber \\
    &\hspace{2cm}-\frac{\tilde\lambda(0, \omega_\varphi(\tilde \boh))\tilde\lambda(0,\omega_\varphi(\tilde \boh_{u,j})- \omega_\varphi(\tilde \boh))}{\lambda(f^0)^2}\bullet(N-\E_0)_T\nonumber\\
    &\hspace{2cm}- \frac{1}{2} \frac{\tilde\lambda(0, \omega_\varphi (\tilde \boh_{u,j})-\omega_\varphi (\tilde \boh))^2}{\lambda(f^{0})^2}\bullet (N -\E_0)_T.\nonumber
\end{align}
We study in detail  the first term and the second term on the right-hand side of (\ref{diff_lin_delta_Q}). Then, we will show that the third  can treated similarly to the second term. Finally, we will treat the last three terms with simple inequalities on $\omega_\varphi$ and using what we did for the second and third terms. Note that in the case where $\varphi$ is the identity function, the last four terms are equal to $0$ and for two first terms, the functions $\bar\varphi^0_{l,k}$ are constant and equal to $1$.

The first term on the right-hand side of (\ref{diff_lin_delta_Q}) is independent of $f$. By assumption we have for some $c>0$ independent of $j$, $\Vert \tilde \psi^{0,j}_{L,\varphi} .\boldsymbol{\bar\varphi}^0\Vert_\infty <c$. Moreover, by the equivalence of norms and since an orthogonal projection is norm-decreasing, we find
\begin{align}
    \Vert \tilde \psi^{0,j}_{L,\varphi}. \boldsymbol{\bar\varphi}^0\Vert_2\lesssim \Vert \tilde \psi^{0,j}_{L,\varphi} \Vert_{L,\varphi} \lesssim \Vert \tilde \psi^{0}_{L,\varphi} \Vert_{L,\varphi}\nonumber
\end{align}
Hence, with the third Bernstein inequality of Lemma \ref{Bernstein}, we obtain that for some $B_1>0$, for any $x>0$,
\begin{align}
    &\P_0\bigg(\Omega_T'\cap \Big\{\Big\vert \frac{1}{2T}\frac{\lambda(\tilde \psi^{0,j}_{L,\varphi}.\boldsymbol{\bar\varphi}^0)^2}{\lambda(f^{0})^2}\bullet (N -\E_0)_T\Big\vert >x\Big\}\bigg)\nonumber\\
    &\leq 6\exp\bigg(\frac{-B_1Tx^2}{\log(T)^5\big(c^2\Vert \tilde \psi^{0}_{L,\varphi} \Vert_{L,\varphi}^2 + xc^2\big)} \bigg).\nonumber
\end{align}
This upper bound does not depend on $j$ so, 
\begin{align}
    &\P_0\bigg(\Omega_T'\cap \Big\{\underset{j\in \mathcal{J}_T}{\max}\Big\vert \frac{1}{2T}\frac{\lambda(\tilde \psi^{0,j}_{L,\varphi}.\boldsymbol{\bar\varphi}^0)^2}{\lambda(f^{0})^2}\bullet (N -\E_0)_T\Big\vert >x\Big\}\bigg)  \nonumber  \\
    &\leq 6J_T\exp\bigg(\frac{-B_1Tx^2}{\log(T)^5\big(c^2\Vert \tilde \psi^{0}_{L,\varphi} \Vert_{L,\varphi}^2 + xc^2\big)} \bigg)=o(1).\nonumber
\end{align}
It shows that the first term on the right-hand side of (\ref{diff_lin_delta_Q}) is a $o_{\P_0}(1)$ uniformly in $j$.

For the second term on the right-hand side of (\ref{diff_lin_delta_Q}), we set
\begin{align}
    H_2(\tilde f,T) := \frac{1}{\sqrt{T}}\frac{\tilde \lambda((\tilde f).\boldsymbol{\bar\varphi}^0)\tilde\lambda(\tilde \psi^{0,j}_{L,\varphi}.\boldsymbol{\bar\varphi}^0)}{\lambda(f^{0})^2},\nonumber
\end{align}
and so we have to show that,
\begin{align}\label{goal_2}
    \underset{j\in \mathcal{J}_T}{\max}\hspace{0.05cm}\underset{\tilde f\in \tilde \A_{T}(j)}{\sup} \Big\vert H_2(\tilde f-\tilde f^0,T) \bullet (N - \E_0)_T  \Big\vert \xrightarrow[T\rightarrow +\infty]{\P_0^{\Omega_T'} }0,
\end{align}
where $\tilde \A_T(j)$ is such that $\A_T(j) = \{\varphi(\tilde f), \tilde f\in \tilde \A_T(j)\}$. To do so, we first fix $j\in \mathcal{J}_T$ and we study the increments of the process $\big(H_2(\tilde f-\tilde f^0, T) \bullet (N -\E_0)_T\big)_{ \tilde f\in \tilde \A_{T}(j)} $ (which is a kind of empirical process) in order to conclude with a chaining argument.   By the third inequality of lemma \ref{Bernstein} and with the previous remarks on the sup norm and the $\L_2$ norm of $\tilde \psi^{0,j}_{L,\varphi}.\boldsymbol{\bar\varphi}^0$, there exists $B_2>0$ such that for any $ \tilde f, \tilde f'\in \tilde \A_{T}(j)$ and $x>0$
\begin{align}
    &\P_0\Big( \Omega_T'\cap \Big\{\Big \vert \big(H_2(\tilde f-\tilde f^0,T) - H_2(\tilde f '-\tilde f^0,T)\big)\bullet(N-\E)_T\big \vert >x\Big\}\Big) \nonumber\\
    &=\P_0\Big( \Omega_T'\cap \Big\{\Big \vert H_2(\tilde f-\tilde f ',T)\bullet(N-\E)_T\big \vert >x\Big\}\Big) \nonumber\\
    &\lesssim \exp\bigg( \frac{-B_2\log(T)^{-5}x^2}{\Vert \tilde f - \tilde f'\Vert_2^2 + xT^{-1/2}\Vert \tilde f - \tilde f'\Vert_\infty}\bigg)\nonumber\\
    &\lesssim  \exp\bigg( \frac{-B_2\log(T)^{-5}x^2}{\Vert \tilde f - \tilde f '\Vert_2^2 + x\sqrt{j}T^{-1/2}\Vert \tilde f - \tilde f'\Vert_2}\bigg)\label{bernstein_incr},
\end{align}
where the last inequality comes from (\ref{condi_sup}). With inequality (\ref{bernstein_incr}), we can apply Theorem~1.2.7 of \cite{talagrand_chaining} (more precisely the penultimate inequality in the proof of this theorem), which combined with inequality 1.49 again from \cite{talagrand_chaining} gives the following bound (\ref{first_chain}) on the deviation of the supremum around some $\tilde f^j_*\in \tilde \A_T(j)$. Let $d_{\Vert .\Vert_2}$ be the distance induced by $\Vert .\Vert_2$. There exists a universal constant $\bar L>0$ such that for any $x>0$ 
\begin{equation}\label{first_chain}
\begin{aligned}
    &\P_0^{\Omega_T'}\Bigg(\underset{ \tilde f\in  \tilde \A_{T}(j)}{\sup} \Big\vert  \Big(H_2(\tilde f -\tilde f^0,T)- H_2(\tilde f^j_*-\tilde f^0,T)\Big) \bullet (N - \E_0)_T  \Big\vert >x\bigg) \\
    &\leq \bar L\exp\Bigg( \frac{-\bar L^{-1}B_2\log(T)^{-5}x}{E_1(T,j) + E_2(T,j)}\Bigg)
\end{aligned}
\end{equation}
with 
\begin{align}
    &E_1(T,j) = \sqrt{j}T^{-1/2}\inf \underset{ \tilde f\in \tilde \A_{T}(j)}{\sup}\sum_{n\geq 0}2^{n}d_{\Vert .\Vert_2}(\tilde f, \tilde A_n) \nonumber\\
    &E_{2}(T,j) =\inf \underset{\tilde f\in \tilde \A_{T}(j)}{\sup}\sum_{n\geq 0}2^{n/2}d_{\Vert .\Vert_2}(\tilde f, \tilde A_n)\nonumber
\end{align}
and where the infimum is each time taken on the sequences $(\tilde A_n)_n$ of subsets of $\tilde \A_{T}(j)$ such that $Card(\tilde A_0)=1$ and $Card(\tilde A_n)\leq 2^{2^n}$ for $n\geq 1$. We first bound the term $E_1(T,j)$. Let $\tilde f_c = \bth_c^TB_j \in \tilde \A_{T}(j)$, by definition we have  $$\tilde \A_{T}(j)\subset \big\{\tilde f\in [0,+\infty[^K\times \mathcal{\tilde H}(j)^{K^2}; \Vert \tilde f -\tilde f_c\Vert_2\leq 2M\log(T)\varepsilon_T\big\}:= \tilde \A_{T}(j)^{ext}.$$
$\tilde \A_{T}(j)^{ext}$ is a $\Vert .\Vert_2$-ball in $[0,+\infty[^K\times \mathcal{\tilde H}(j)^{K^2}$ and an internal covering of $\A_{T}(j)^{ext}$ is an external covering of $\tilde \A_T(j)$. So,   with (\ref{condi_rho}), there exists $c>0$ such that   any $r>0$, 
\begin{align}
    \mathcal{N}(r, \tilde \A_{T}(j), \Vert.\Vert_2)\leq \mathcal{N}\big(r/2, \tilde \A_{T}(j)^{ext}, \Vert.\Vert_2\big)&\lesssim  \mathcal{N}\big(r/2, B_{\Vert . \Vert}(\bth_c, 2cM\log(T)\varepsilon_T), \Vert.\Vert\big) \nonumber \\
    &\leq \max \Big(1, \frac{12cM\log(T)\varepsilon_T}{r}\Big)^{K^2j}, \nonumber
\end{align}
where $\Vert . \Vert$ is the usual Euclidean norm on $\R^{K^2j}$. Let $c'= 12cM$, we choose $r_{n,j,T}=c'\log(T)\varepsilon_T 2^{-2^n/(K^2j)}$ so that the set of centering points of the covering of $\tilde \A_{T}(j)$ with radius $r_{n,j,T}$, denoted $\tilde A_{n,T}$, verifies $Card(\tilde A_{n,T})\leq 2^{2^n}$. Consequently,
\begin{align}
    E_1(T,j)&\leq \sqrt{j}T^{-1/2}\sum_{n\geq 0}2^{n}r_{n,j,T}\leq K^2 c'\log(T)T^{-1/2}\varepsilon_T J_T^{3/2}\sum_{n\geq 0}\frac{2^{n}}{K^2j}2^{-2^n/(K^2j)}. \nonumber
\end{align}
where for the last inequality we have used that $j\leq J_T$. Let $n_j:= \lceil \log(3K^2j)/\log(2)\rceil$, if $n\geq n_j$ then $2^n/(K^2j)\geq 3$. Moreover, the map $x\mapsto x2^{-x}$ is bounded by $1$ on $[0,+\infty[$ and is decreasing on $[2,+\infty[$. Whence,
\begin{align}
    \sum_{n\geq 0} \frac{2^n}{K^2j} 2^{-\frac{2^n}{K^2j}} \leq  n_j +\sum_{n\geq n_j} \frac{2^n}{K^2j} 2^{-\frac{2^n}{K^2j}}&\leq   n_j +\int_{n_j-1}^{+\infty} \frac{2^x}{K^2j} 2^{-\frac{2^x}{K^2j}} dx\nonumber\\
    &\leq   n_j +\int_{0}^{+\infty} \frac{e^{-u\log(2)}}{\log(2)} du,\nonumber
\end{align}
and the last inequality comes with the change of variable $u=2^x/(K^2j)$. So, for $T$ large enough, $\sum_{n\geq 0} \frac{2^n}{K^2j} 2^{-\frac{2^n}{K^2j}} \leq 2\log(J_T)$ and thus, $ E_1(T,j) \lesssim  \log(T)^2T^{-1/2}\varepsilon_T J_T^{3/2}$. The last bound being independent of $j\leq J_T$, we even have
\begin{align}
    \underset{j\in \mathcal{J}_T}{\max}\hspace{0.1cm}E_1(T,j)&\lesssim \log(T)^2T^{-1/2}\varepsilon_T J_T^{3/2}\lesssim
    \log(T)\sqrt{J_T}\varepsilon_T,\nonumber
\end{align}
because  by \hyperref[as_P1]{(P1)}, $\log(T)J_T\lesssim T\varepsilon_T^2$ and by (\ref{condi_rate}), $\sqrt{T}\varepsilon_T^2=o(1)$. Now, for the second term appearing in (\ref{first_chain}), namely $E_2(T,j)$, using Dudley inequality (see again \cite{talagrand_chaining}), we find similarly that
$E_2(T,j)\lesssim \log(T)\sqrt{J_T}\varepsilon_T$. 
Therefore, we have proved that for some finite positive constant $L'$, we have for any $j\leq J_T$ and $x>0$
\begin{align}
    &\P_0^{\Omega_T'}\Bigg(\underset{\tilde f\in \tilde \A_{T}(j)}{\sup} \Big\vert \Big(H_2(\tilde f-\tilde f^0,T)-H_2(\tilde f^j_*-\tilde f^0,T)\Big) \bullet (N - \E_0)_T  \Big\vert >x\bigg)\nonumber\\
    &\leq L'\exp\bigg(\frac{-L'x}{\log(T)^{6}\sqrt{J_T}\varepsilon_T}\bigg).\nonumber
\end{align}
We deduce that over $\A_T$:
\begin{align}
    &\P_0^{\Omega_T'}\Bigg(\underset{ f\in  \A_T}{\sup} \Big\vert  H_2(\tilde f-\tilde f^0,T) \bullet (N - \E_0)_T  \Big\vert >x\bigg)\nonumber \\
    &\leq \sum_{j\in \mathcal{J}_T}\P_0^{\Omega_T'}\Bigg(\underset{\tilde f\in \tilde \A_{T}(j)}{\sup} \Big\vert  \Big(H_2(\tilde f -\tilde f^0,T)- H_2(\tilde f^j_*-\tilde f^0,T)\Big) \bullet (N - \E_0)_T  \Big\vert >x/2\bigg)\nonumber \\
    &\hspace{1cm}+\sum_{j\in \mathcal{J}_T} P_0^{\Omega_T'}\Bigg( \Big\vert  H_2( \tilde f^j_*-\tilde f^0,T) \bullet (N - \E_0)_T  \Big\vert >x/2\bigg) \label{res_chaining_1_prem}\\
    &\lesssim J_TL'\exp\bigg(\frac{-L'x/2}{\log(T)^{6}\sqrt{J_T}\varepsilon_T}\bigg)+\sum_{j\in \mathcal{J}_T} P_0^{\Omega_T'}\Bigg( \Big\vert  H_2( \tilde f^j_*-\tilde f^0,T) \bullet (N - \E_0)_T  \Big\vert >x/2\bigg) .\nonumber
\end{align}
For the first term on the right-hand side of (\ref{res_chaining_1_prem}), one can check that by Assumption (\ref{condi_rate}), it goes to $0$ as $T\rightarrow +\infty$. For the second term on the right-hand side of (\ref{res_chaining_1_prem}), recall that by definition of $\A_T(j)$,  $\Vert f^j_*-\tilde f^0\Vert_2^2\lesssim \log(T)^2\varepsilon_T^2$ and with (\ref{condi_sup}) we also have $\Vert f^j_* -\tilde f^0\Vert_\infty\lesssim~ \log(T)\sqrt{J_T}\varepsilon_T=o(1)$. Whence, for some $B_2'>0$ and $T$ large enough, we have for $x>0$,
\begin{align}
    &\sum_{j\in \mathcal{J}_T} P_0^{\Omega_T'}\bigg( \Big\vert  H_2( \tilde f^j_*-\tilde f^0, T) \bullet (N - \E_0)_T  \Big\vert >x/2\bigg) \nonumber\\
    &\leq \sum_{j\in \mathcal{J}_T} 12\exp\bigg( - \frac{B_2'\log(T)^{-5}Tx^2}{T\Vert f^j_* -\tilde f^0\Vert_2^2 + x\sqrt{T}\Vert f^j_* -\tilde f^0\Vert_\infty}\bigg)\nonumber\\
    &\leq 12J_T \exp\bigg( - \frac{B_2'\log(T)^{-5}x^2}{\log(T)^2\varepsilon_T^2 + xT^{-1/2}}\bigg)=o(1).\nonumber
\end{align}
which terminates to prove (\ref{goal_2}).

So far, we have treated the two first terms on the right-hand side of (\ref{diff_lin_delta_Q}), it remains to deal with the following  four terms in (\ref{diff_lin_delta_Q}). These four terms involve the remainder $\omega_\varphi$ and we first recall some facts and prove useful inequalities on it. These inequalities will also be useful to study $\tilde R_{T,\varphi}(f_{u,j}) -\tilde R_{T,\varphi}(f)$ later in the proof. First, recall that
\begin{align}
    \omega_\varphi(\tilde \boh) = \varphi(\tilde \boh) -\varphi(\tilde \boh^0) -(\tilde \boh -\tilde \boh^0).\boldsymbol{\bar\varphi}^0,\nonumber 
\end{align}
and that by Lemma \ref{lemma_linearization}, $\Vert \omega_\varphi(\tilde \boh)\Vert_1\lesssim \log(T)^2\varepsilon_T^2$ on $\A_T$. Take $f, f'\in\A_T $, since $\varphi$ is Lipschitz and the functions $\bar\varphi^0$ are bounded, we have  for any $p\in [1,+\infty]$ 
\begin{align}\label{diff_omega_p}
    \Vert \omega_\varphi(\tilde \boh) - \omega_\varphi(\tilde \boh')\Vert_p \lesssim  \Vert \tilde \boh -\tilde \boh'\Vert_p ,
\end{align}
and thus
\begin{align}\label{diff_omega_j}
    \Vert \omega_\varphi(\tilde \boh_{u,j}) - \omega_\varphi(\tilde \boh)\Vert_p \lesssim \frac{1}{\sqrt{T}}.
\end{align}
We can even be more precise:
\begin{equation}\label{diff_omega_j_precise}
\begin{aligned}
    \omega_\varphi(\tilde \boh_{u,j}) - \omega_\varphi(\tilde \boh) &= \varphi(\tilde \boh_{u,j}) - \varphi(\tilde \boh) - (\tilde \boh_{u,j} - \tilde \boh).\boldsymbol{\bar\varphi}^0\\
    &= (\tilde \boh_{u,j} - \tilde \boh). \big( \varphi'(\gamma(\tilde \boh_{u,j},\tilde \boh)) - \varphi'(\tilde \boh^0)\big)\\
    &= -\frac{u\bog^{0,j}_{L,\varphi}}{\sqrt{T}}.\big( \varphi'(\gamma(\tilde \boh_{u,j},\tilde \boh)) - \varphi'(\tilde \boh^0)\big).
\end{aligned}
\end{equation}
and for all $(l,k)\in \K^2$, $\gamma(\tilde \boh_{u,j},\tilde \boh)_{l,k}$ is a function in the bracket  $[\tilde \boh_{u,j,l,k},\tilde \boh_{l,k}]$. We will use throughout the proof, without recalling it,  that by assumption $\sup \{\Vert \bog^{0,j}_{L,\varphi}\Vert_\infty, j\geq 1\}<+\infty$. Finally, recall (\ref{A_inf})  which gives that for some $D'>0$
\begin{align}
    \A_T\subset \big\{f: \Vert \tilde f -\tilde f^0\Vert_\infty \leq D'\log(T)\sqrt{J_T}\varepsilon_T\big\},\nonumber
\end{align} 
and $\log(T)\sqrt{J_T}\varepsilon_T \rightarrow 0$ by assumption.  Now, we are ready to study individually the last four terms of (\ref{diff_lin_delta_Q}).

For the third term, define for $\tilde f\in \tilde \A_T(j)$
\begin{align}
    H_3(\tilde f,T) := \frac{1}{\sqrt{T}} \frac{\tilde \lambda(\tilde\psi^{0,j}_{L,\varphi}\tilde \lambda(0,\omega_\varphi(\tilde \boh_{u,j}))}{\tilde \lambda(f^{0})^2}\nonumber
\end{align}
Using the third inequality of lemma \ref{Bernstein} and (\ref{diff_omega_p}), we find that for some $B_3,B_3'>0$ (independent of $f, f', T$ and $x$)
\begin{align}
    &\P_0\Big( \Omega_T'\cap \Big\{\Big \vert \big(H_3(\tilde f,T) - H_3(\tilde f ',T)\big)\bullet(N-\E)_T\big \vert >x\Big\}\Big)\nonumber\\
    &\leq 6\exp\bigg( \frac{-B_3x^2}{\log(T)^5\big(\Vert  \omega_\varphi(\tilde \boh_{u,j}) - \omega_\varphi(\tilde \boh'_{u,j})\Vert_2^2 + xT^{-1/2}\Vert \omega_\varphi(\tilde \boh_{u,j}) - \omega_\varphi(\tilde \boh'_{u,j})\Vert_\infty\big)}\bigg) \nonumber \\
    &\leq 6\exp\bigg( \frac{-B_3'x^2}{\log(T)^5\big(\Vert \tilde f - \tilde f\Vert_2^2 + xT^{-1/2}\Vert \tilde f - \tilde f'\Vert_\infty\big)}\bigg) \nonumber \\
    &\leq  6\exp\bigg( \frac{-B_3\log(T)^{-5}x^2}{\Vert \tilde f - \tilde f '\Vert_2^2 + x\sqrt{j}T^{-1/2}\Vert \tilde f - \tilde f'\Vert_2}\bigg).\nonumber
\end{align}
Using the same arguments as for $H_2$,
\begin{align}
    \P_0^{\Omega_T'}\bigg(\hspace{0.05cm}\underset{\tilde f\in  \A_T}{\sup} \Big\vert  H_3(\tilde f) \bullet (N - \E_0)_T  \Big\vert >x\bigg)\xrightarrow[T\rightarrow +\infty]{} 0\nonumber.
\end{align}

For the fourth, fifth and sixth terms, we could again study the increment of the associated process and conclude with a chaining argument. But thanks to (\ref{diff_omega_j_precise}) and (\ref{A_inf}), we can be more concise (and rougher).  For the fourth term, we have
\begin{equation}\label{eq_4th}
\begin{aligned}
    &\bigg \vert \frac{\tilde\lambda((\tilde f-\tilde f^0).\boldsymbol{\bar\varphi}^0)\tilde \lambda(0,\omega_\varphi(\tilde \boh_{u,j})- \omega_\varphi(\tilde \boh))}{\lambda(f^0)^2}\bullet(N-\E_0)_T\bigg\vert \\
    &\leq \bigg \vert \sum_{k=1}^K \int_0^T\frac{\tilde \lambda_t^k((\tilde f_k-\tilde f^0_k).\boldsymbol{\bar\varphi}^0_k)\tilde \lambda^k_t(0,\omega_\varphi(\tilde \boh_{u,j,k})- \omega_\varphi(\tilde \boh_k))}{\lambda_t^k(f^0_k)^2} dN_t^k\bigg \vert\\
    &\hspace{1cm}+ T\sum_{k=1}^K\E_0\bigg[\Big\vert \frac{\tilde\lambda_A^k((\tilde f_k-\tilde f^0_k).\boldsymbol{\bar\varphi}^0_k)\tilde \lambda_A^k(0,\omega_\varphi(\tilde \boh_{u,j})_k- \omega_\varphi(\tilde \boh)_k)}{\lambda_A^k(f^0_k)}\Big \vert \bigg]
\end{aligned}
\end{equation}
Now, using in order Cauchy-Schwarz inequality, the equivalence between the LAN norm and the $L_2$ norm, (\ref{diff_omega_j_precise}) and that by assumption $\varphi'$ is Lipschitz on $\mathcal{\tilde A}_T(j)$ for $T$ large enough, we obtain the following bound on the second term of the right-hand side of (\ref{eq_4th}):
\begin{equation}\label{4th_1}
\begin{aligned}
    &T\sum_{k=1}^K\Vert (\tilde f_k -\tilde f^0_k).\boldsymbol{\bar\varphi}^0_k\Vert_L \Vert \omega_\varphi(\tilde \boh_{u,j,k})- \omega_\varphi(\tilde \boh_k)\Vert_L\\
    &\lesssim T\sum_{k=1}^K\Vert (\tilde f_k -\tilde f^0_k).\boldsymbol{\bar\varphi}^0_k\Vert_2 \Vert \omega_\varphi(\tilde \boh_{u,j,k})- \omega_\varphi(\tilde \boh_k)\Vert_2 \\
    &\lesssim \log(T)T\varepsilon_T\Big\Vert \frac{u\bog^{0,j}_{L,\varphi}}{\sqrt{T}}\varphi'(\gamma(\tilde\boh_{u,j} ,\tilde \boh)) - \varphi'(\tilde \boh^0)\Big\Vert_2\\
    &\lesssim \log(T)\sqrt{T}\varepsilon_T\Vert \varphi'(\gamma(\tilde\boh_{u,j} ,\tilde \boh)) - \varphi'(\tilde \boh^0)\Vert_2\\
    &\lesssim \log(T)\sqrt{T}\varepsilon_T\Vert \gamma(\tilde\boh_{u,j} ,\tilde \boh) -\tilde \boh^0\Vert_2\\
    &\lesssim \log(T)\sqrt{T}\varepsilon_T\big(\Vert \tilde\boh_{u,j} -\tilde \boh^0\Vert_2+ \Vert \tilde \boh -\tilde \boh^0\Vert_2\big)\\
    &\lesssim\log(T)^2\sqrt{T}\varepsilon_T^2=o(1).
\end{aligned}
\end{equation}
Next, for the first term on the right-hand side of (\ref{eq_4th}), with similar arguments (in particular (\ref{diff_omega_j_precise})) we find that:
\begin{align}
    &\bigg \vert \sum_{k=1}^K \int_0^T\frac{\tilde \lambda_t^k((\tilde f_k-\tilde f^0_k).\boldsymbol{\bar\varphi}^0_k)\tilde \lambda_t^k(0,\omega_\varphi(\tilde \boh_{u,j,k})- \omega_\varphi(\tilde \boh_k))}{\lambda_t^k(f^0_k)^2} dN_t^k\bigg \vert\nonumber \\
    &\leq \sum_{k=1}^K \int_0^T\frac{\tilde \lambda_t^k((\vert \tilde f_k-\tilde f^0_k\vert).\boldsymbol{\bar\varphi}^0_k)\tilde \lambda_t^k(0,\vert \omega_\varphi(\tilde \boh_{u,j,k})- \omega_\varphi(\tilde \boh_k)\vert )}{\lambda_t^k(f^0_k)^2} dN_t^k\nonumber\\
    &\lesssim  \frac{1}{\sqrt{T}}\sum_{k=1}^K \int_0^T\frac{\tilde \lambda_t^k((\vert \tilde f_k-\tilde f^0_k\vert).\boldsymbol{\bar\varphi}^0_k)\tilde \lambda_t^k(0,\vert \varphi'(\gamma(\tilde \boh_{u,j}, \tilde \boh))-\varphi'(\tilde \boh^0))\vert )}{\lambda_t^k(f^0_k)^2} dN_t^k\nonumber\\
    &\lesssim  \frac{1}{\sqrt{T}}\sum_{k=1}^K \int_0^T\frac{\tilde \lambda_t^k((\vert \tilde f_k-\tilde f^0_k\vert).\boldsymbol{\bar\varphi}^0_k)\tilde \lambda_t^k\big(0,\vert \tilde \boh_{u,j} - \tilde \boh^0\vert + \vert \tilde \boh -\tilde \boh^0\vert \big) \big)}{\lambda_t^k(f^0_k)^2} dN_t^k.\nonumber
\end{align}
Moreover, similarly to (\ref{4th_1}), we have on the event $\Omega_T$
\begin{align}
    T\E_0\bigg[ \frac{\tilde \lambda_t^k((\vert \tilde f_k-\tilde f^0_k\vert).\boldsymbol{\bar\varphi}^0_k)\tilde \lambda_t^k\big(0,\vert \tilde \boh_{u,j} - \tilde \boh^0\vert + \vert \tilde \boh -\tilde \boh^0\vert  \big)}{\lambda_t^k(f^0_k)}\bigg] \lesssim \log(T)^2\sqrt{T}\varepsilon_T^2\nonumber.
\end{align}
Whence, on the event $\Omega_T$, we have for the fourth term:
\begin{align}\label{4th_ccl}
    &\bigg\vert \frac{\tilde \lambda((\tilde f-\tilde f^0).\boldsymbol{\bar\varphi}^0)\tilde \lambda(0,\omega_\varphi(\tilde \boh_{u,j})- \omega_\varphi(\tilde \boh))}{\lambda(f^0)^2}\bullet(N-\E_0)_T\bigg\vert \nonumber \\
    &\lesssim \frac{1}{\sqrt{T}}\bigg\vert \frac{\tilde \lambda((\vert \tilde f-\tilde f^0\vert ).\boldsymbol{\bar\varphi}^0)\tilde \lambda(0,\vert \tilde \boh_{u,j} - \tilde \boh^0\vert + \vert \tilde \boh -\tilde \boh^0\vert)}{\lambda(f^0)^2}\bullet(N-\E_0)_T\bigg\vert  +  \log(T)^2\sqrt{T}\varepsilon_T^2.
\end{align}
Using triangular inequality and what we did for the second and third term of (\ref{diff_lin_delta_Q}), one can obtain a uniform control over $\A_T(j)$ of the first term on the right-hand of side of (\ref{4th_ccl}). We can then as before deduce that
\begin{align}
    \underset{j\in \mathcal{J}_T}{\max}\underset{f\in \A_T(j)}{\sup}\bigg\vert \frac{\tilde \lambda((\tilde f-\tilde f^0).\boldsymbol{\bar\varphi}^0)\tilde \lambda(0,\omega_\varphi(\tilde \boh_{u,j})- \omega_\varphi(\tilde \boh))}{\lambda(f^0)^2}\bullet(N-\E_0)_T\bigg\vert \xrightarrow[T\rightarrow +\infty]{\P_0}0 .\nonumber 
\end{align}

The fifth term can be treated as the fourth one.

The sixth term can be directly bounded. Indeed,
\begin{align}\label{eq_5th}
    &\Big \vert \frac{\tilde \lambda(0, \omega_\varphi (\tilde \boh_{u,j})-\omega_\varphi (\tilde \boh))^2}{\lambda(f^{0})^2}\bullet (N -\E_0)_T \Big \vert \nonumber\\
    &\leq \sum_{k=1}^K\int_0^T\frac{\tilde \lambda_t^k(0, \omega_\varphi (\tilde \boh_{u,j})_k-\omega_\varphi (\tilde \boh)_k)^2}{\tilde \lambda_t^k(f^{0}_k)^2} dN_T^k + T\sum_{k=1}^K\E_0\bigg[\frac{\tilde \lambda_A^k(0, \omega_\varphi (\tilde \boh_{u,j})_k-\omega_\varphi (\tilde \boh)_k)^2}{\lambda_A^k(f^{0}_k)}\bigg].
\end{align}
Using (\ref{diff_omega_j_precise}) and (\ref{A_inf}),  we find that on the event $\Omega_T$:
\begin{align}
    &\sum_{k=1}^K\int_0^T\frac{\tilde \lambda_t^k(0, \omega_\varphi (\tilde \boh_{u,j})_k-\omega_\varphi (\tilde \boh)_k)^2}{\lambda_t^k(f^{0}_k)^2} dN_T^k\nonumber\\
    &\lesssim N([0,T]) \log(T)^{2}\Vert \omega_\varphi (\tilde \boh_{u,j})-\omega_\varphi (\tilde \boh)\Vert_\infty^2\nonumber \\
    &\lesssim \frac{N([0,T])}{T} \log(T)^{2}\Vert \varphi'(\gamma(\tilde \boh_{u,j},\tilde \boh))-\varphi'(\tilde \boh^0)\Vert_\infty^2\nonumber \\
    &\lesssim \log(T)^2 \big(\Vert \tilde \boh_{u,j} -\tilde \boh^0\Vert_\infty^2 + \Vert \tilde \boh -\tilde \boh^0\Vert_\infty^2 \big) \lesssim \log(T)^{4}J_T\varepsilon_T^2\xrightarrow[T\rightarrow +\infty]{}0.\nonumber
\end{align}
In the same way, the expectation on the right-hand side of (\ref{eq_5th}), can be uniformly bounded independently of $f$ and $j$ by some sequence $u_T\rightarrow 0$. Thus, it proves that
\begin{align}
    \underset{j\in \mathcal{J}_T}{\max}\underset{f\in \A_T(j)}{\sup}\Big \vert \frac{\tilde \lambda(0, \omega_\varphi (\tilde \boh_{u,j})-\omega_\varphi (\tilde \boh))^2}{\lambda(f^{0})^2}\bullet (N -\E_0)_T \Big \vert  \xrightarrow[T\rightarrow +\infty]{\P_0} 0. \nonumber
\end{align}

Coming back to (\ref{recall_remainder}), it remains to study the term involving the function $V$ and $\tilde R_{T,\varphi}$.\\

\textit{Term involving function $\mathit{V}$}: We recall that for $x>-1$, $V(x) = \log(1+x)-x+x^2/2$, $V'(x) = x^2/(1+x)$ and $V''(x) = (x^2+2x)/(1+x)^2$.  We first fix $j\in \mathcal{J}_T$ and $f\in \A_T(j)$. To begin, note that  by (\ref{A_inf}),  on the event $\Omega_T$, for any $t\in [0,T]$ and $k\in \K$, we have
\begin{align}\label{rq_V_1}
    \Big\vert \frac{\tilde \lambda_t^k(f_k-f^0_k)}{\lambda_{t}^k(f^0_k)}\Big\vert \lesssim \log(T) \Vert f-f^0\Vert_\infty \lesssim \log(T)\sqrt{J_T}\varepsilon_T =o(1),
\end{align}
and
\begin{align}\label{rq_V_2}
    \Big\vert \frac{\tilde \lambda_t^k(f_{u,j,k}-f_k)}{\lambda_{t}^k(f^0_k)}\Big\vert &\lesssim \log(T) \Vert f_{u,j}-f\Vert_\infty \lesssim  \frac{\log(T)}{\sqrt{T}}=o(1).
\end{align}
Then, by Taylor-Lagrange expansion, for all $k\in \K$ and $t\geq 0$, there exists a random variable
\begin{align}
    \Gamma_t^k(f) \in \bigg[ \frac{\tilde \lambda_t^k(f_{u,j,k}-f^{0}_k)}{\lambda_{t}^k(f^0_k)}, \frac{\tilde \lambda_t^k(f_k-f^{0}_k)}{\lambda_{t}^k(f^0_k)}\bigg],\nonumber
\end{align}
such that
\begin{align}\label{dl_V}
     &\sum_{k=1}^K \int_0^T V\bigg(\frac{\tilde \lambda_t^k(f_{u,j,k}-f^{0}_k)}{\lambda_{t}^k(f^0_k)}\bigg)- V\bigg(\frac{\tilde \lambda_t^k(f_k-f^{0}_k)}{\lambda_{t}^k(f^0_k)}\bigg)  \frac{dN_t^k}{\lambda_{t}^k(f^0_k)}\nonumber \\
     &=  \sum_{k=1}^K \int_0^T\frac{\tilde \lambda_t^k(f_{u,j,k}-f_k)}{\lambda_{t}^k(f^0_k)} V'\Big(\frac{\tilde \lambda_t^k(f_k-f^0_k)}{\lambda_{t}^k(f^0_k)}\Big)  \frac{dN_t^k}{\lambda_{t}^k(f^0_k)} \\
     &\hspace{1cm}+ \sum_{k=1}^K \int_0^T \frac{\tilde \lambda_t^k(f_{u,j,k}-f_k)^2}{2\lambda_{t}^k(f^0_k)^2}V''\big(\Gamma_t^k(f)\big) \frac{dN_t^k}{\lambda_{t}^k(f^0_k)}\nonumber.
\end{align}
For the first term on the right-hand side of (\ref{dl_V}), with the two previous remarks (\ref{rq_V_1}) and (\ref{rq_V_2}) and the fact that for $\vert x\vert <1/2$, $\vert V'(x)\vert \leq 2x^2$,   we have  on the event $\Omega_T$ for $T$ large enough
\begin{align}
    &\bigg \vert \sum_{k=1}^K \int_0^T\frac{\tilde \lambda_t^k(f_{u,j,k}-f_k)}{\lambda_{t}^k(f^0_k)} V'\Big(\frac{\tilde \lambda_t^k(f_k-f^0_k)}{\lambda_{t}^k(f^0_k)}\Big)  \frac{dN_t^k}{\lambda_{t}^k(f^0_k)} \bigg\vert\nonumber\\
    &\lesssim\frac{\log(T)}{\sqrt{T}}\sum_{k=1}^K \int_0^T \frac{\tilde \lambda_t^k(f_k-f^0_k)^2}{\lambda_{t}^k(f^0_k)^2} dN_t^k\nonumber\\
    &\lesssim  \frac{\log(T)}{\sqrt{T}} \frac{\tilde \lambda(f-f^0)^2}{\lambda(f^{0})^2} \bullet (N-\E_0)_T +  \log(T) \sqrt{T} \Vert f-f^0\Vert_L^2 \nonumber \nonumber\\
    &\lesssim   \frac{\log(T)}{\sqrt{T}} \frac{\tilde \lambda(f-f^0)^2}{\lambda(f^{0})^2} \bullet (N-\E_0)_T + \log(T)^3 \sqrt{T} \varepsilon_T^2,\nonumber
\end{align}
and by assumption $\log(T)^3 \sqrt{T} \varepsilon_T^2 \rightarrow 0$. Furthermore, linearizing again by Lemma \ref{lemma_linearization}, we obtain
\begin{equation} \label{decomp_term1_V}
\begin{aligned}
    &\frac{\log(T)}{\sqrt{T}} \frac{\tilde\lambda(f-f^0)^2}{\lambda(f^{0})^2} \bullet (N-\E_0)_T =  \frac{\log(T)}{\sqrt{T}} \frac{\tilde \lambda((\tilde f-\tilde f^0).\boldsymbol{\bar\varphi}^0)^2}{\lambda(f^{0})^2} \bullet (N-\E_0)_T\\
    &\hspace{3cm}+ \frac{\log(T)}{\sqrt{T}} \frac{\tilde \lambda(0, \omega_\varphi(\tilde h))^2}{\lambda(f^{0})^2} \bullet (N-\E_0)_T  \\
    &\hspace{3cm}+ 2\frac{\log(T)}{\sqrt{T}} \frac{\tilde \lambda((\tilde f-\tilde f^0).\boldsymbol{\bar\varphi}^0)\tilde \lambda(0, \omega_\varphi(\tilde h))}{\lambda(f^{0})^2} \bullet (N-\E_0)_T .
\end{aligned}
\end{equation}
Then, because $\sup_{f\in \A_T}\Vert \tilde f-\tilde f^0\Vert_\infty \leq C$ and $\sup _{f\in \A_T}\Vert  \omega_\varphi(\tilde h)\Vert_\infty \leq C$ for some $C>0$, these three terms on the right-hand side of (\ref{decomp_term1_V}) can be treated in the same way as the second and the third terms of (\ref{diff_lin_delta_Q}). Thus, we have that
\begin{align}
    \underset{j\in \mathcal{J}_T}{\max}\underset{f\in \A_T(j)}{\sup } \frac{\log(T)}{\sqrt{T}}\bigg \vert  \frac{\tilde \lambda(f-f^0)^2}{\lambda(f^{0})^2} \bullet (N-\E_0)_T  \bigg \vert \xrightarrow[T\rightarrow +\infty]{\P_0} 0,\nonumber
\end{align}
and therefore $$\underset{j\in\mathcal{J}_T}{\max}\underset{f\in \A_T(j)}{\sup }  \bigg \vert \sum_{k=1}^K \int_0^T\frac{\tilde \lambda_t^k(f_{u,j,k}-f_k)}{\tilde \lambda_{t}^k(f^0_k)} V'\Big(\frac{\tilde \lambda_t^k(f_k-f^0_k)}{\lambda_{t}^k(f^0_k)}\Big)  dN_t^k \bigg\vert = o_{\P_0}(1).$$
Now, for the second term on the right-hand side of (\ref{dl_V}), using (\ref{rq_V_1}), (\ref{rq_V_2}) and that $\vert V''(x)\vert \leq 4(x^2+2\vert x\vert )$ for $\vert x\vert \leq 1/2$, we can bound it on the event $\Omega_T$ (for $T$ large enough) in the following way,
\begin{align}
    &\bigg \vert \sum_{k=1}^K \int_0^T \frac{\tilde \lambda_t^k(f_{u,j,k}-f_k)^2}{2\lambda_{t}^k(f^0_k)^2}V''\big(\Gamma_t^k(f)\big) \frac{dN_t^k}{\lambda_t^k(f^0_k)} \bigg \vert\nonumber\\
    &\lesssim N([0,T]) \log(T)^3 \Vert f_{u,j} -f\Vert_\infty^2\times  \underset{t\in [0,T], k\in \K}{\sup} \big\vert V''\big(\Gamma_t^k(f)\big)\big\vert\nonumber \\
    &\lesssim \frac{N([0,T]) \log(T)^3}{T} \times  \underset{t\in [0,T], k\in \K}{\sup}\vert \Gamma_t^k(f)\vert^2 +  2\vert \Gamma_t^k(f)\vert \lesssim \log(T)^5 \sqrt{J_T}\varepsilon_T =o(1).\nonumber
\end{align}
Whence,
\begin{align}
    \underset{j\in \mathcal{J}_T}{\max}\underset{f\in \A_T(j)}{\sup}\bigg\vert \sum_{k=1}^K \int_0^T \frac{\tilde \lambda_t^k(f_{u,j,k}-f_k)^2}{2\lambda_{t}^k(f^0_k)^2}V''\big(\Gamma_t^k(f)\big) \frac{dN_t^k}{\lambda_{t}^k(f^0_k)} \bigg\vert \xrightarrow[T\rightarrow +\infty]{\P_0}0.\nonumber
\end{align}

Lastly we have to study the difference of the terms defined by $R_{T,\varphi}$.\\

 \textit{Difference} $\mathit{\tilde R_{T,\varphi}(f_{u,j}) -\tilde R_{T,\varphi}(f)} $. Let $j\in \mathcal{J}_T$ (in particular $j\leq J_T$) and $f \in \A_{T}(j)$. First, recall that by Lemma \ref{lemma_linearization} we have $ \Vert \omega_\varphi(\tilde h_{u,j})\Vert_1^{1/2}+ \Vert \omega_\varphi(\tilde h)\Vert_1^{1/2}\lesssim \log(T)\varepsilon_T $. Moreover, by (\ref{A_inf}), $\Vert \omega_\varphi(\tilde h_{u,j})\Vert_\infty^{1/2}+ \Vert \omega_\varphi(\tilde h)\Vert_\infty^{1/2}$ can be bounded by a independently of $j$ and $\tilde f$. These two remarks, combined with as before (\ref{diff_omega_j_precise}) and the equivalence of norms, give
\begin{align}
    T\Big\vert \Vert \omega_\varphi(\tilde h_{u,j})\Vert_L^2 -\Vert \omega_\varphi(\tilde h)\Vert_L^2\Big\vert&\leq T \big( \Vert \omega_\varphi(\tilde h_{u,j})\Vert_L + \Vert \omega_\varphi(\tilde h)\Vert_L\big) \Vert \omega_\varphi(\tilde h_{u,j}) - \omega_\varphi(\tilde h)\Vert_L \nonumber\\
    &\lesssim T \big( \Vert \omega_\varphi(\tilde h_{u,j})\Vert_2 + \Vert \omega_\varphi(\tilde h)\Vert_2\big) \Vert \omega_\varphi(\tilde h_{u,j}) - \omega_\varphi(\tilde h)\Vert_2  \nonumber \\
    &\lesssim T \big( \Vert \omega_\varphi(\tilde h_{u,j})\Vert_1^{1/2}\Vert \omega_\varphi(\tilde h_{u,j})\Vert_\infty^{1/2}+ \Vert \omega_\varphi(\tilde h)\Vert_1^{1/2}\omega_\varphi(\tilde h)\Vert_\infty^{1/2}\big)\nonumber\\
    &\hspace{0.5cm}\times \frac{1 }{\sqrt{T}} \Vert  \varphi'(\gamma(\tilde \boh_{u,j},\tilde \boh)) - \varphi'(\tilde \boh^0)\Vert_2  \nonumber \\
    &\lesssim  \log(T)\sqrt{T} \varepsilon_T\times \Vert\varphi'(\gamma(\tilde \boh_{u,j},\tilde \boh)) - \varphi'(\tilde \boh^0)\Vert_2\nonumber\\
    &\lesssim \log(T)\sqrt{T} \varepsilon_T^2 =o(1)\nonumber,
\end{align}
by assumption. So, $$\underset{j\in \mathcal{J}_T}{\max}\underset{f\in \A_T(j)}{\sup}T\Big\vert \Vert \omega_\varphi(\tilde h_{u,j})\Vert_L^2 -\Vert \omega_\varphi(\tilde h)\Vert_L^2\Big\vert\xrightarrow[T\rightarrow +\infty ]{}0.$$
Secondly, these arguments also give
\begin{align}
    &T\big \vert \langle \tilde f_{u,j} - \tilde f^{0}, (0, \omega_\varphi(\tilde h_{u,j}) )\rangle_{L} -  \langle\tilde f- \tilde f^{0}, (0, \omega_\varphi(\tilde h) )\rangle_{L} \big \vert \nonumber \\
    &\leq  T \big \vert \langle \tilde f_{u,j}-\tilde f, (0, \omega_\varphi(\tilde h_{u,j}) )\rangle_{L}\big \vert + T \big \vert \langle\tilde f- \tilde f^{0}, (0, \omega_\varphi(\tilde h_{u,j}) -\omega_\varphi(\tilde h) )\rangle_{L} \big \vert \nonumber \\
    &\lesssim  \sqrt{T} \big \vert \langle \tilde \psi^{0,j}_{L,\varphi}, (0, \omega_\varphi(\tilde h_{u,j}) )\rangle_{L}\big \vert + T\Vert \tilde f- \tilde f^{0}\Vert_2 \Vert \omega_\varphi(\tilde h_{u,j}) -\omega_\varphi(\tilde h) \Vert_2 \nonumber \\
    &\lesssim  \sqrt{T} \big \vert \langle \tilde \psi^{0,j}_{L,\varphi}, (0, \omega_\varphi(\tilde h_{u,j}) )\rangle_{L}\big \vert+ \log(T)^2\sqrt{T}\varepsilon_T^2\nonumber .
\end{align}
Now, for $k\in \K$ let $p_{k,n}= \P_0(N^k([0,A])=n)$. There exists $v>0$ such that for any $k\in \K$ the variable $vN^k([0,A])$ has exponential moments (see \cite{leblanc_exp1}). Therefore, $p_{k,n} = O(e^{-vn})$ and thus for any $a,b>0$ the series of general term $n^ap_{k,n}^b$ converges. Moreover, since $\E_0[\lambda_A^k(0,\vert \omega_\varphi(\tilde h_u))\vert]\lesssim \Vert\omega_\varphi(\tilde h_u)\Vert_1\lesssim  \log(T)^2\varepsilon_T^2$, we have
\begin{align}
    \big \vert \langle \tilde \psi^{0,j}_{L,\varphi}, (0,\omega_\varphi(\tilde h_u))\rangle_L\big\vert &\lesssim \sum_{k=1}^K\E_0\Big[\big(1+N([0,A])\big) \lambda_A^k(0,\vert\omega_\varphi(\tilde h_u)\vert) \Big] \nonumber\\
    &\lesssim \log(T)^2\varepsilon_T^2 + \sum_{k=1}^K\E_0\Big[N([0,A]) \lambda_A^k(0,\vert\omega_\varphi(\tilde h_u)\vert) \Big]. \nonumber 
\end{align}
Let $0<\delta<1$ be as in Theorem \ref{main_theorem}, let $k\in \K$, by Hölder inequality ($p=1/(1-\delta)$ and $q=1/\delta$), we find
\begin{align}
    \E_0\Big[N([0,A]) \lambda_A^k(0,\vert\omega_\varphi(\tilde h_u)\vert) \Big] &= \E_0\Big[N([0,A]) \lambda_A^k(0,\vert\omega_\varphi(\tilde h_u)\vert)^\delta\lambda_A^k(0,\vert\omega_\varphi(\tilde h_u)\vert)^{1-\delta}  \Big]\nonumber\\
    &\leq \Vert \omega_\varphi(\tilde h_u)\Vert_\infty^\delta\E_0\Big[N([0,A])^{1+\delta}\lambda_A^k(0,\vert\omega_\varphi(\tilde h_u)\vert)^{1-\delta}  \Big] \nonumber \\
    &\lesssim \E_0\Big[N([0,A])^{1+\delta}\lambda_A^k(0,\vert\omega_\varphi(\tilde h_u)\vert)^{1-\delta}  \Big] \nonumber \\
    &= \sum_{n\geq 0} n^{1+\delta}\E_0\Big[\onee_{N^k([0,A])=n}\lambda_A^k(0,\vert\omega_\varphi(\tilde h_u)\vert)^{1-\delta}  \Big] \nonumber \\
    &\lesssim \sum_{n\geq 0} n^{1+\delta}p_{k,n}^{\delta}\E_0\big[\lambda_A^k(0,\vert\omega_\varphi(\tilde h_u)\vert)  \big]^{1-\delta}\nonumber \\
    &\lesssim \Big(\log(T)^2\varepsilon_T^{2}\Big)^{1-\delta}\sum_{n\geq 0} n^{1+\delta}p_{k,n}^{\delta}\lesssim \log(T)^2\varepsilon_T^{2(1-\delta)}.\nonumber
\end{align}
Consequently, $\sqrt{T}\big \vert \langle \tilde \psi^{0,j}_{L,\varphi}, (0,\omega_\varphi(\tilde h_u))\rangle_L\big\vert  \lesssim\log(T)^2\sqrt{T}\varepsilon_T^{2(1-\delta)}= o(1)$. So, it proves that
\begin{align}
    &\underset{j\in \mathcal{J}_T}{\max}\underset{f\in \A_T(j)}{\sup}T\big \vert \langle \tilde f_{u,j} - \tilde f^{0}, (0, \omega_\varphi(\tilde h_{u,j}) )\rangle_{L} -  \langle\tilde f- \tilde f^{0}, (0, \omega_\varphi(\tilde h) )\rangle_{L} \big \vert=o(1).\nonumber
\end{align}
 Thirdly, we have to show that
\begin{align} \label{rem_omega_third}
    \underset{j\in \mathcal{J}_T}{\sup}\hspace{0.05cm}\underset{\tilde f \in \A_{T}(j)}{\sup}\sqrt{T}\big\vert W_{T}\big(0, \omega_\varphi(\tilde h_{u,j})- \omega_\varphi(\tilde h)\big)\big \vert  \xrightarrow[T\rightarrow +\infty]{\P_0}0.
\end{align}
 To do so, we begin by rewriting in a different way (\ref{diff_omega_j_precise}). Using a Taylor expansion at order 3 around $\tilde h$ (we can by Assumption \hyperref[as_A']{(A')}) we find,
\begin{align}
    &\omega_\varphi(\tilde \boh_{u,j}) - \omega_\varphi(\tilde \boh)\nonumber\\
    &= \varphi(\tilde \boh_{u,j}) - \varphi(\tilde \boh) - (\tilde \boh_{u,j} - \tilde \boh).\boldsymbol{\bar\varphi}^0\nonumber\\
    &= (\tilde \boh_{u,j} - \tilde \boh).(\varphi'(\tilde \boh) -  \varphi'(\tilde \boh^0)) + \frac{1}{2}(\tilde \boh_{u,j} - \tilde \boh)^2\varphi''(\tilde \boh) + \frac{1}{6}(\tilde \boh_{u,j} - \tilde \boh)^3.\varphi'''(\gamma(\tilde \boh_{u,j},\tilde \boh))\nonumber \\
    &= \frac{-u\bog^{0,j}_{L,\varphi}}{\sqrt{T}}.(\varphi'(\tilde \boh) -  \varphi'(\tilde \boh^0)) + \frac{u^2(\bog^{0,j}_{L,\varphi})^2}{2T}.\varphi''(\tilde \boh) - \frac{u^3(\bog^{0,j}_{L,\varphi})^3}{6T^{3/2}}.\varphi'''(\gamma(\tilde \boh_{u,j},\tilde \boh)).\nonumber
\end{align}
and as before $\gamma(\tilde \boh_{u,j},\tilde \boh)$ is a function in the bracket $[\tilde \boh_{u,j},\tilde \boh]$. Whence, to prove (\ref{rem_omega_third}), it is enough to prove that
\begin{align}
    & \underset{j\in \mathcal{J}_T}{\sup}\hspace{0.05cm}\underset{\tilde f \in \A_{T}(j)}{\sup}\big\vert W_{T}\big(0,\bog^{0,j}_{L,\varphi}.(\varphi'(\tilde \boh) -  \varphi'(\tilde \boh^0))\big)\big \vert  \xrightarrow[T\rightarrow +\infty]{\P_0}0 \label{eq_third_1},\\
    &\underset{j\in \mathcal{J}_T}{\sup}\hspace{0.05cm}\underset{\tilde f \in \A_{T}(j)}{\sup}\big\vert \frac{1}{\sqrt{T}}W_{T}\big(0,(\bog^{0,j}_{L,\varphi})^2.\varphi''(\tilde \boh)\big)\big \vert  \xrightarrow[T\rightarrow +\infty]{\P_0}0\label{eq_third_2},\\
    &\underset{j\in \mathcal{J}_T}{\sup}\hspace{0.05cm}\underset{\tilde f \in \A_{T}(j)}{\sup}\frac{1}{T}\big\vert W_{T}\big(0, (\bog^{0,j}_{L,\varphi})^3.\varphi'''(\gamma(\tilde \boh_{u,j},\tilde \boh))\big)\big \vert  \xrightarrow[T\rightarrow +\infty]{\P_0}0 \label{eq_third_3}.
\end{align}
(\ref{eq_third_1}) and (\ref{eq_third_2}) can be proved with the same chaining argument and the similar computations to those used for the second term of (\ref{diff_lin_delta_Q}) (for the Bernstein inequality, use the second inequality of lemma \ref{Bernstein} instead of the third one and then recall that $\varphi'$ and $\varphi''$ are Lipschitz on $\varphi^{-1}(I_{0}(G))$). For (\ref{eq_third_3}), it can be directly uniformly bounded on the event $\Omega_T$, independently of $f$ and $j$ in the following way.  Recall that $\varphi'''$ is bounded on $\tilde \A_T$ and we obtain  on the event $\Omega_T$,
\begin{align}
    \frac{1}{T}\big\vert W_{T}\big(0, (\bog^{0,j}_{L,\varphi})^3.\varphi'''(\gamma(\tilde \boh_{u,j},\tilde \boh))\big)\big \vert \lesssim \frac{1}{T^{3/2}}\big(\log(T)N([0,T]) + \log(T)^2T\big)\lesssim \frac{\log(T)^2}{\sqrt{T}}.\nonumber
\end{align}
Therefore, (\ref{eq_third_3}) is also proved. It concludes the proof of Lemma \ref{control_rem}.
\end{proof}

\section{Control of the number of points of the process}\label{appendix_omega}
We recall here lemma A.1 of \cite{Sulem_concentration} on the number of points of the process.
\begin{lemma}
\label{omega_T}
Let $N$ be a stationary,  $K$-multivariate and non-linear ReLU Hawkes process with parameters $f^0$ such that functions $h^0_{l,k}$ are supported on a known and bounded interval $[0,A]$. For any $Q>0$ and $\alpha>0$, there exists a constant $C_\alpha$, depending only on $f^0$ such that for $T>0$, the set
\begin{align}
    \Omega_T = \bigg\{ \underset{l\in\K}{\max}\hspace{0.1cm}\underset{t\in[0,T]}{\sup} N^l\big([t-A,t]\big) \leq C_\alpha \log(T)\bigg\}\nonumber
\end{align}
satisfies $\P_0(\Omega_T^c) \leq T^{-\alpha}$. Moreover, for any $1\leq q \leq Q$, for $T$ large enough,
\begin{align}
    \E_0\Big[\underset{k\in \K}{\max}\hspace{0.05cm}\underset{t\in [0,T]}{\sup}N^k([t-A,t[)^q\onee_{\Omega_T^c}\Big] \leq 2T^{-\alpha/2}.\nonumber
\end{align}
\end{lemma}

\begin{appendix}
\section*{Palm distributions}

In this appendix, we briefly introduce Palm calculus and we state the  results used in Section \ref{sec_semi_param_eff} and Section \ref{sec_lfeq}, giving each time references for the proofs and for further details. Palm theory is formally presented in the general case by, among others, \cite{DVJ2} (Chapter 13), \cite{kal} (Chapter 7) and \cite{bremaud_book} (Chapter 7 and Chapter 8), we mainly use the terminology of \cite{DVJ2}.  For simplicity and because it is sufficient for our purpose, we introduce Palm distributions for a $K$-marked and boundedly finite point process $N$ on $\R$. In this section, to avoid confusion, we will write $N(dk\times dx)$ instead of $dN^k_x$ which is used elsewhere in the paper. Let's denote by $\mathcal{N}_{\R\times \K}^{\#}$ the set of counting measures on $\R$ with marks in $\K$ (equipped with its Borel $\sigma$-field induced by the "weak hash" topology, see Section 9.1 of \cite{DVJ2} and appendix A2.6 of \cite{DVJ1}).  Assume that the point process $N$ has a $\sigma$-finite first moment measure denoted $M_N$. The Campbell measure $C_N$ is a measure on $\R\times\K\times  \mathcal{N}_{\R\times \K}^{\#}$ (equipped with its product $\sigma$-field) defined 
by
\begin{align}
    C_N(B\times k\times U)= \E\big[N^k(B) \onee_{N\in U}\big].\nonumber
\end{align}
For any measurable function $\Phi:\R\times\K\times \mathcal{N}_{\R\times \K}^{\#} \rightarrow\R^+$ we have
\begin{align}
    C_N \Phi:= \E\bigg[\int \Phi(x,k,N) N(dk\times dx)\bigg] = \int \Phi(x,k,\bar N)C_N( dk\times dx\times d\bar N).\nonumber
\end{align}
The Palm kernel $(\P^{(x,k)}, (x,k)\in \R\times \K)$ is then defined as the kernel from $\R\times \K$ to $\mathcal{N}_{\R\times \K}^{\#}$ obtained by disintegrating the Campbell measure with respect to the first moment measure $M_N$ and we have
\begin{align}\label{int_eq_palm}
    C_N \Phi = \int \E^{(x,k)}[\Phi(x,k,N)] M_N(dk\times dx),
\end{align}
$\E^{(x,k)}$ being an expectation under the Palm distribution $\P^{(x,k)}$; see Proposition 13.1.IV and the preceding discussion in \cite{DVJ2}. When $N$ is simple, which we consider to be the case from now on, $\P^{(x,k)}$ can be interpreted as the conditional distribution of $N$ given that there is a point at $x$ with mark $k$ (the latter event being of probability $0$ in general, this only holds approximately, see \cite{kal} Section 6.6 for further details). In particular,  $\P^{(x,k)} \big(N^k(\{x\})=1\big)=~1$. By disintegrating the second order Campbell measure $C^2_N$ defined  on $\R^2\times\K^2\times  \mathcal{N}_{\R\times \K}^{\#}$  by
\begin{align}
    C^2_N(B_1\times B_2\times k_1\times k_2\times U)= \E[N^{k_1}(B_1)N^{k_2}(B_2) \onee_{N\in U}],\nonumber
\end{align}
with respect to the second moment measure (when it is $\sigma$-finite), we obtain the second order Palm kernel $(\P^{(x_1,x_2,k_1,k_2)}, (x_1,x_2,k_1,k_2)\in \R^2\times \K^2)$. Because of this, the kernel $(\P^{(x,k)}, (x,k)\in \R\times \K)$ is sometimes called the first order Palm kernel. Then, for any measurable function $\Phi:\R^2\times\K^2\times \mathcal{N}_{\R\times \K}^{\#} \rightarrow\R^+$ we have
\begin{align}
    C_N^2 \Phi = \int \E^{(x_1,x_2,k_1,k_2)}[\Phi(x_1,x_2, k_1,k_2,N)] M_N^2(dk_1\times dk_2\times dx_1\times dx_2).\nonumber
\end{align}
where $M_N^2$ is the second-order moment measure. The second order Palm distributions are symmetric in the sense that $\P^{(x_1,x_2,k_1,k_2)}= \P^{(x_2,x_1,k_2,k_1)}$. Moreover, $$\P^{(x_1,x_2,k_1,k_2)} \big(N^{k_1}(\{x_1\})=1, N^{k_2}(\{x_2\})=1\big)=1,$$ 
see Lemma 6.2 of \cite{kal}. By the \textit{iteration principles} (see Corollary 6.24 of \cite{kal} or alternatively \cite{kal_article}), we have
\begin{align}
    &\E^{(x_1,k_1)}\bigg[\int \Phi(x_2,k_2,N) N(dk_2\times dx_2)\bigg] \nonumber\\
    &=\int \E^{(x_1,x_2,k_1,k_2)}\big[\Phi(x_2,k_2,N)\big] M_N^{(x_1, k_1)}(dk_2\times dx_2)\label{iter_palm},
\end{align}
where $M_N^{(x_1, k_1)}$ is the Palm first moment measure: $M_N^{(x_1, k_1)}(B\times l) = \E^{(x_1, k_1)}[N^l(B)]$.

Assume now that the point process $N$ is stationary and denote by $S_xN$ the point process $N$ whose points  are all shifted by $x$. In this case, the Palm distributions are all equal up to a shift (\cite{DVJ2}, Section 13.4):
\begin{align}\label{shift_palm}
    \P^{(x,k)}(N\in .) = \P^{(0,k)}(S_xN\in .).
\end{align}

Finally, even if we do not use it in this paper, we point out that for a stationary Hawkes process, by Theorem 8.4.18 of \cite{bremaud_book}, the Palm distribution $\P^{(0,k)}$ is equal to the convolution between the stationary distribution of the process and a "reweighted and shifted" version of the distribution of  a cluster rooted in $(0,k)$. This  result is stated in more general terms in Lemma 6.16 in \cite{kal} and has a generalisation to higher order Palm distributions, see Theorem 6.30 in \cite{kal} or Section 5 in \cite{kal_article}.

\end{appendix}

\begin{acks}[Acknowledgments]
We are thankful to two anonymous referees for their remarks and suggestions, which improved the paper. We thank Marc Hoffmann and Vincent Rivoirard for pointing out \cite{ker_picard} to us. We also thank Richard Nickl and Aad van der Vaart for useful discussions on the regularity of the least favourable direction.

\end{acks}

\begin{funding}
MD's PhD position is funded by the Imperial College London - CNRS joint PhD programme. JR received funding from the European Research
Council (ERC) under the European Union’s Horizon 2020 research and innovation programme (grant agreement No 834175).
\end{funding}\vspace{0.4cm}

\textbf{Declaration on the use of AI.} No mathematical content was generated or found with the assistance of artificial intelligence.


\bibliographystyle{imsart-number} 
\bibliography{bib_hawkes.bib}       

\end{document}